\documentclass[11pt,twoside,openright]{extbook}

\usepackage{rounakstyle}

\newcommand{\chRounak}[1]{{#1}}

\usepackage{times}

\usepackage{accsupp}
\usepackage[toc,page]{appendix}

\usepackage[contents={},opacity=1,scale=1,color=black]{background}
\usepackage{cite}
\usepackage{booktabs}
\usepackage{pdfpages}

\usepackage{calc}
\usepackage{color}
\usepackage{csquotes}

\usepackage{dsfont}

\usepackage{etoolbox}

\usepackage{fancyhdr}
\usepackage{float}

\usepackage{import}

\usepackage{mathtools}

\usepackage{setspace}
\usepackage{silence}
\usepackage{standalone}
\usepackage{subcaption}

\usepackage{tabularx}
\usepackage{tikz}
\usepackage{tikzpagenodes}
\usepackage{titlesec}
\usepackage{titletoc}
\usepackage[colorinlistoftodos,prependcaption,textsize=tiny]{todonotes}
\usepackage{totcount}
\usepackage{lettrine}

\usepackage{diagbox}

\usepackage{xargs}
\allowdisplaybreaks

\usepackage{pdflscape}
\usepackage{enumerate}
\usepackage{algorithm}
\usepackage{algpseudocode}
\usepackage{algorithmicx}

\definecolor{purplish}{rgb}{0.41, 0.41, 0.64}
\definecolor{webbrown}{rgb}{.6,0,0}
\definecolor{NiceRed}{rgb}{0.41,0, 0}

\definecolor{Neonpink}{rgb}{1, 0.0627, 0.9412}
\definecolor{Neongreen}{rgb}{0.0588, 1, 0.3137}
\definecolor{Neonblue}{rgb}{0.1215686 0.3176471 1.0000000}
\definecolor{Neonorange}{rgb}{1, 0.3725, 0.1216}

\hypersetup{
	colorlinks=false, 
	linkcolor=blue,  
}

\newcolumntype{x}[1]{>{\centering\let\newline\\\arraybackslash\hspace{0pt}}p{#1}}
\fancypagestyle{headings}{
\fancyhead[RE]{\hspace{1.4em}\textit{\nouppercase{\leftmark}}}
\fancyfoot[LE]{\thepage}
\fancyfoot[RO]{\thepage}
\fancyhead[LO]{\textit{\nouppercase{\!\rightmark}}\hspace{1.4em}}
\fancyfoot[C]{}

}

\fancypagestyle{plain}{
\fancyhead{}
\fancyfoot[LE]{\thepage}
\fancyfoot[RO]{\thepage}

}

\fancypagestyle{empty}{
\fancyhead{}
\fancyfoot{}

}

\titleformat{\chapter}[display]{\normalfont}{\vspace{-2cm}\Large\itshape\chaptertitlename~\thechapter}{0.7pc}{\raggedleft\vspace{0.7pc}\huge}[\vspace{0.5pc}{\titlerule[1pt]}]

\newcommand{\emptychapterdist}{-2cm}

\newif\ifMaterial

\newlength\LabelSize
\AtBeginDocument{%
\regtotcounter{chapter}
\setlength\LabelSize{\dimexpr\textheight/\totvalue{chapter}\relax}
\ifdim\LabelSize>2.5cm\relax
  \global\setlength\LabelSize{2.5cm}
\fi
}

\newcommand{\chapterNum}{%
\if \thechapter A%
	8%
\else%
	\if \thechapter B%
		9%
	\else%
		\thechapter%
	\fi%
\fi%
}
\definecolor{thesisthumbs}{RGB}{214, 214, 214}

\newcommand{\StartThumbs}{%
\Materialtrue%
\AddEverypageHook{%
\ifMaterial%
	\ifodd\value{page} %
		\backgroundsetup{
		angle=0,
		position={current page.east|-current page text area.north east},
		hshift=-0.5cm,
		vshift=-1cm-((\chapterNum-1)/7)*15cm,
		contents={%
			\tikz\node[fill=thesisthumbs, anchor=west, text width=18mm,
			align=center, text height=0.5\LabelSize, text depth=20pt,font=\Huge\normalfont] 
		{{\centering \color{white}\thechapter\hspace{7mm}}};
		}%
		}
	\else
		\backgroundsetup{
		angle=0,
		position={current page.west|-current page text area.north west},
		hshift=0.5cm,
		vshift=-1cm-((\chapterNum-1)/7)*15cm,
		contents={%
			\tikz\node[fill=thesisthumbs,anchor=west, text width=18mm,
			align=center, text height=0.5\LabelSize, text depth=20pt,font=\Huge\normalfont] 
		{{\centering \hspace{7mm}\color{white}\thechapter}};
		}%
		}
	\fi
\BgMaterial%
\else\relax\fi
}%
}

\newcommand{\EndThumbs}{\Materialfalse}
\numberwithin{equation}{section}

\makeatletter
\pretocmd{\chapter}{\addtocontents{toc}{\protect\addvspace{14\p@}}}{}{}
\makeatother

\begin{document}

\pagestyle{empty}

\vspace*{\fill}
\begin{center}
\huge{Stochastic processes on preferential attachment models
}
\end{center}
\begin{center}
\large{Understanding global structures from local properties}
\end{center}
\vspace*{\fill}
\clearpage



This work is supported in part by the Netherlands Organisation for Scientific Research (NWO) through the Gravitation {\sc Networks} grant \(024.002.003\) and European Union's Horizon \(2020\) research and innovation programme under the Marie Sk\l{}odowska-Curie grant agreement no.\ \(945045\).

    	\begin{figure}[!htb]%
    		\hspace{2cm}
    		{\includegraphics[height=3cm]{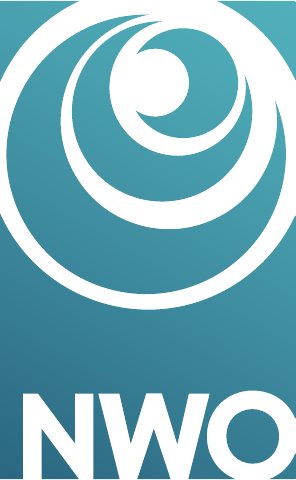}}
    		\hspace{2cm}
		{\includegraphics[height=3cm]{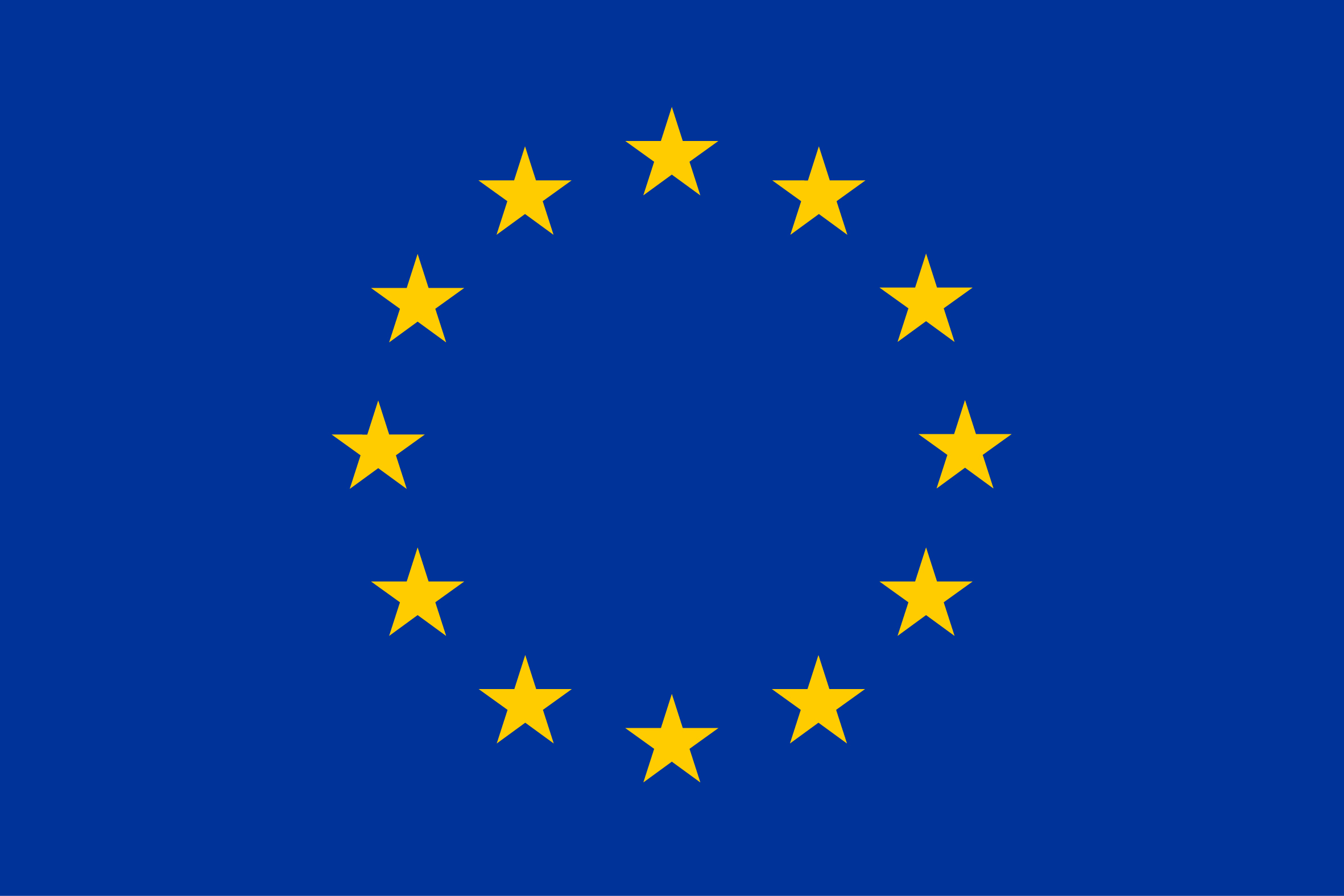}}
    		\newline
    		
    		\vspace{0.5cm}
    		\hspace{4cm}
		{\includegraphics[height=3cm]{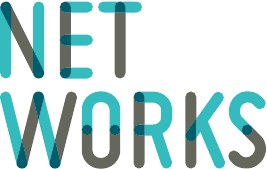}}
    		\hspace{1cm}                   
    	\end{figure}

\vspace*{\fill}

\begin{flushleft}
\textcopyright\, Rounak Ray, 2024\\
Stochastic processes on preferential attachment models
\end{flushleft}

\begin{flushleft}
A catalogue record and digital copies (PDF) are available from the \\ Eindhoven University of Technology library\\
ISBN: 978-90-386-6211-4\\
\end{flushleft}


\begin{flushleft}
Printed by ProefschriftMaken \(\mid\mid\) \url{www.proefschriftmaken.nl} \\
Cover by Prasmita Dey 
\end{flushleft}

\vspace*{\fill}
\begin{center}
\Large
Stochastic processes on preferential attachment models
\end{center}
\vspace{10mm}
\begin{center}
\uppercase{Proefschrift}
\end{center}
\vspace{16mm}
\hspace{0.05\textwidth}%
\begin{minipage}{0.9\textwidth}
\begin{center}
ter verkrijging van de graad van doctor aan de Technische Universiteit
Eindhoven, op gezag van de rector magnificus prof.\hspace{1.5pt}dr.~ S.K. Lenaerts,
voor een commissie aangewezen door het College voor Promoties, in het
openbaar te verdedigen op woensdag 4 december 2024 om 16:00 uur
\end{center}
\end{minipage}
\vspace{14mm}
\begin{center}
door
\end{center}
\vspace{8mm}
\begin{center}
Rounak Ray
\end{center}
\vspace{8mm}
\begin{center}
geboren te Coochbehar, India
\end{center}
\vspace*{\fill}
\clearpage

\noindent\hspace{0.05\textwidth}%
\begin{minipage}{0.9\textwidth}
\noindent%
Dit proefschrift is goedgekeurd door de promotoren en de samenstelling van de promotiecommissie is als volgt:

\vspace{5ex}
\begin{tabular}{ll}
voorzitter:  &prof.\hspace{1.5pt}dr.~J.N. Kok\\
Promotoren: &prof.\hspace{1.5pt}dr.\hspace{1.5pt} R. W. van der Hofstad\\
			&
dr.\hspace{1.5pt} R. S. Hazra (Universiteit Leiden)\\
Promotiecommissieleden:       &prof.\hspace{1.5pt}dr.\hspace{1.5pt} R.W. van der Hofstad\\
			&\hspace{1.5pt}dr.\hspace{1.5pt} R. S. Hazra (Universiteit Leiden)\\

		&prof.\hspace{1.5pt}dr.~P. M\"{o}rters (Universit\"{a}t zu K\"{o}ln)\\
             &prof.\hspace{1.5pt}dr.~M. Deijfen (Universitet Stockholms)\\
             &prof.\hspace{1.5pt}dr.~A.P. Zwart\\
\end{tabular}
\vspace{2ex}

\end{minipage}

\vspace*{\fill}
\noindent%
\textit{Het onderzoek dat in dit proefschrift wordt beschreven is uitgevoerd in overeenstemming met de TU/e Gedragscode Wetenschapsbeoefening.}

\pagenumbering{Roman}
\cleardoublepage

\vspace{10cm}


\begin{center}
	\includegraphics[width=0.95\linewidth]{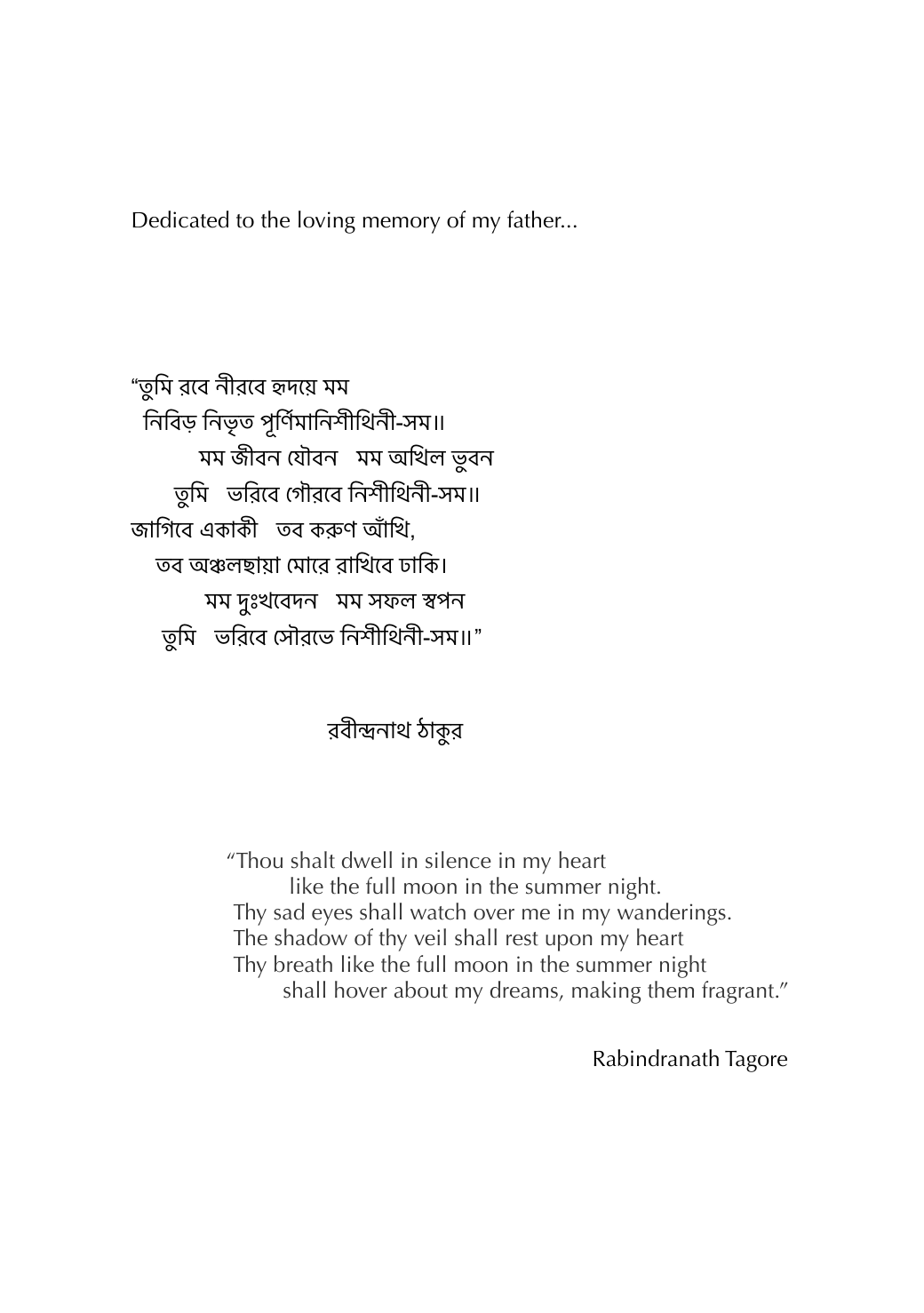}
\end{center}

\pagestyle{plain}

\cleardoublepage\currentpdfbookmark{Acknowledgements}{acknowledgements}

\chapter*{ 
\vspace*{\emptychapterdist}Acknowledgements}
\vspace{10pt}

As I write this Acknowledgement, my defence is exactly a month away. I delayed writing this, both intentionally and unintentionally. I wanted to ensure I didn’t miss the chance to thank anyone who contributed to this thesis, and, admittedly, I am also habitually behind schedule. Looking back, I realise that every person I’ve met has shaped this journey. Pursuing a PhD in the Netherlands has been a remarkable learning experience in both mathematics and life, and I would like to thank all of you through these brief stories and lessons we shared.

Since childhood, I’ve always been drawn to research. For PhD applications abroad, I needed recommendations. One of my seniors, \emph{Sohom Bhattacharya} (Sohom B), advised me to do a summer project with \emph{Prof. Manjunath Krishnapur} at IISc. \emph{Prof. Krishnapur} introduced me to the fascinating world of random walks and trees. Looking back, this was the beginning of my PhD journey in probability theory.

During my Master’s degree at the Indian Statistical Institute (ISI) in Kolkata, I was actively seeking someone who could provide a recommendation for my PhD applications. I admit that my initial reason for meeting \emph{Rajat daa} was this purpose alone, but when he suggested I attend his summer course on random graphs, I was immediately captivated. My first impression was, “It’s tough maths, but not tougher than Statistics,” and it was then that I decided to pursue Probability as a field of research. At that time, I had only a vague sense of the research world and a somewhat naïve vision of my future.

Toward the end of my first year, I approached \emph{Rajat daa} for a Master’s project, specifically in random walks. He entrusted me with a project close to his heart, and I enjoyed working with my collaborators \emph{Dr. Alessandra Cipriani} and \emph{Dr. Biltu Dan}. However, I soon encountered my first reality check: preparing for the GRE and TOEFL, necessary for US PhD applications. Despite multiple attempts, I struggled with these exams, and with the onset of COVID, all US universities imposed a hiring freeze. With rejection letters piling up, I felt stranded with no PhD offers.

I turned again to \emph{Rajat daa}, who encouraged me to apply for open position at the NETWORKS COFUND PhD programme. He gave me invaluable support and recommendations. I was extremely nervous during my interview with NETWORKS COFUND, where I met \emph{Remco (van der Hofstad)} and \emph{Frank (den Hollander)}. I was struck by how approachable these eminent professors were and later realised it is their genuine nature to bring out the best in their students. After my selection at NETWORKS COFUND, I began to feel that perhaps this was where I was meant to be, and that all those rejections had led me to the right place.

Just as my career was taking off, my father (\emph{Baba}) was diagnosed with terminal-stage cancer. He had sparked my interest in mathematics and research during my childhood, but by the time I started my research journey, he was nearing the end. Despite his condition, he urged me to go to the Netherlands to join the programme, while my mother (\emph{Maa}) accompanied him to Mumbai for treatment. The doctors there gave him a median life expectancy of just six months, meaning he likely wouldn’t live to see me complete my PhD, a dream we had shared. Six months later, when he passed away, the entire SPOR community was there to support me. Upon returning from India, I received a heartfelt card from everyone in SPOR, and that’s when this place truly began to feel like family.

The support of \emph{Suman daa, Manish, Arpan daa}, and \emph{Neeladri} was invaluable during those times. \emph{“WhatsApp is better than emails”} became our motto, helping us organise meet-ups and cooking challenges. Although I did my theoretical cooking course before coming to the Netherlands, I actually cooked for the first time here, with \emph{Arpan daa} as my guinea pig. My passion for food soon made cooking a favourite pastime, and \emph{Arpan daa}'s `critical' feedback motivated me to keep experimenting.

In the first few months, these friends were my only companions, but as lockdowns eased, I met more fantastic people in the department, like \emph{Joost, Martijn}, and \emph{Zsolt}. Since \emph{Joost} also works on preferential attachment models, we share a unique bond. Together, we went on long bike rides, travelled to Paris for a conference, and I presented my work there for the first time—masks on and all. It was daunting but made easier by their support. I also fondly remember Friday evenings, where I had my first drinks, spurred on by this fantastic group under the able leadership of \emph{Pim (van der Hoorn)} and \emph{Joost}.

When the restrictions were lifted, I met my other officemates, \emph{Samaneh} and \emph{Marta}. I’m grateful to \emph{Samaneh} for being my tea partner — our tea-time conversations ranged from international politics to the similarities between Hindi and Urdu. Our `dart-throw' competitions, even during times of intense pressure, were both spirited and, surprisingly, relaxing.

This was also when I joined a fantastic teaching team for a unique course with \emph{Alberto, Sandra, Tim, Martijn}, and \emph{Hans}. Special thanks to \emph{Tim, Sandra}, and \emph{Hans} for sparing me the task of learning SAS! Over three years, \emph{Martijn} was succeeded by \emph{Beno\^{i}t} and \emph{Stan}, but they all kindly allowed me to focus on the theoretical aspects of the course. \emph{Mike}’s course was also invaluable, helping me revise tools like coupling and stochastic dominance, which I used frequently during my PhD.

Winning arguments with \emph{Manish} wouldn’t have been as entertaining without \emph{Tom, Fatuma}, and \emph{Mark}. Your office became my go-to spot for trash-talking. Special thanks to \emph{Tom} for regularly bringing brownies and cakes, which were my afternoon snacks most of the time.

\emph{Haodong} is someone I can always rely on for finding loopholes in proofs or identifying counterintuitive examples. I hope we can get at least one research paper from our work at ASML. 

I’ve rarely seen anyone as enthusiastic as \emph{Nelly} when tackling new challenges. I’m also impressed by her simple yet effective technique for remembering family relations in Bengali—a perspective I hadn’t thought of before. Her teaching philosophy inspires me, and I aspire to be as dedicated a teacher while also excelling as a researcher.
I feel fortunate to have \emph{Pranab daa} in the neighbouring office. Our conversations in Bengali about teaching, the Dutch healthcare system, elections, Durga Puja, thesis work, and late-night office were always both refreshing and enjoyable.

\emph{Remco}, words cannot express my gratitude for standing by me at every stage of these four years of my PhD. You took the time and effort to clarify every single doubt, even the ones that felt stupid (though you always say, ``No doubts are stupid!"). I’ve learned countless life lessons from you. Whenever I needed help, whether with maths or other issues, your door was always open. Our discussions—whether while biking, driving, or travelling by train—have consistently sparked fruitful ideas. Daylong conversations with \emph{Rajat daa} also helped immensely, filling in gaps in my arguments and catching up on ISI gossips. \emph{Frank}, although you weren’t my direct supervisor, you were nothing short of one to me. I have always benefited greatly from your advice and support.

Next, I would like to thank my committee members; \emph{Bert Zwart, Peter M\"{o}rters, Maria Deijfen} and the chair \emph{J.N. Kok}. Thank you all for taking time out of your busy schedule to read the thesis, join the defence and preparing questions.

Although I wasn’t regular at the EURANDOM lunch table, the conversations were often interesting, enriched by the insights of \emph{Purva, Mike, Connor, Stan, Sanne, Ellen, Ivo, Thomas, Ignachio (Nacho), Alexander}, and \emph{Nelly}.
Special thanks go to \emph{Noela} and \emph{Pim} for organising events and activities for the Probability cluster. Although my interactions with \emph{Qiu, Isabel,} and \emph{Kai-Chun}, the newest members of the cluster, have been brief, I’ve enjoyed their company.  \emph{Lorenzo}, I trust you’ll lead the office with excellence. And \emph{Vinay}, now that we’ve found a shared interest in card tricks and games, I think it’s time to make it a tradition!

The NETWORKS training weeks were a unique experience for all affiliates. Although informally dubbed `NOTWORKS' weeks, we engaged in some serious activities—bowling, dart-throwing, card games, swimming, sauna, and of course, partying and lastly attending mini-courses. The `after \(11\)' parties at our secret hideouts are unforgettable, with \emph{Gianluca, Tim, Roshan,} and \emph{Marta} as the core club members. \emph{Marieke}, I don’t know how you manage to remember every detail about each NETWORKS affiliate, but you do so effortlessly, and your support with travel-budget tracking and event organisation is unmatched.

Attending conferences with our gang—\emph{Nandan, Federico, Daoyi, Twan, Oliver}, and \emph{Manish}—was a lot of fun. Although there were some consequences, I am glad that you didn't care at all about those. Highlights like walking to the geographic centre of the Netherlands, getting drunk and breaking glass, late-night chats outside the hotel in Lunteren, and our time in Saint Flour were made all the more memorable with friends like \emph{Riccardo, Lars}, and \emph{Maike}.

I would like to express my deep gratitude for the invaluable support from \emph{Chantal, Astrid, Marianne}, and \emph{Ellen}, the secretaries of the SPOR cluster. During the organisation of the lecture series in 2023, \emph{Marianne} and \emph{Marieke} handled all administrative tasks seamlessly, ensuring the process went smoothly. \emph{Chantal} and \emph{Astrid} remain my go-to contacts for meeting scheduling, office-related matters, and arranging my defence ceremony. Despite your heavy workloads and the fact that I am not directly your PhD student, \emph{Sem}, you provided valuable guidance whenever I needed it during the lecture series organisation. Thank you all for your exceptional support.

Outside the exceptional workplace at TU/e, I was fortunate to be surrounded by a wonderful Indian community, including \emph{Suman daa, Sreetama di, Sampad, Manish, Arpan daa, Shivesh, Purva, Nitant, Sreehari, Bharti, Sanjana, Rajeshwari}, and \emph{Neeladri}. Our weekly market trips, birthday celebrations, and Holi and Diwali festivities, along with our Durga Puja excursions to Blixembosch, made me feel almost at home. Playing cricket on the field and in our living room at night, and staying up late with \emph{Arpan daa, Manish, Shivesh, Nitant}, and \emph{Purva} playing \(29\) and other card games, was an incredible amount of fun. Making important life decisions became easier thanks to the constant guidance from \emph{Suman daa, Nivedita di}, and \emph{Sreetama di}. 

Thanks to seniors, like \emph{Arnab daa (both Chowdhury and Chakraborty), Sandipan daa, Indranil daa, Arkajyoti daa, Shaswat daa, Mayukh daa, Tamal daa, Rajarshi daa, PPG, Chinmay daa, Deepan daa, Sagnik daa, Dipanjan daa, Aliva di, Shaswati di, Srijita di, Basundhara di, Tulip di, Suprita di} and \emph{Agnishikha di}, celebrating Durga Puja in the Netherlands felt just like home.

I could not have reached this point without friends like \emph{Manaswini, Sayantan}, and \emph{Sayan} from ISI, Kolkata, who supported me through the intense curriculum. I fondly miss those good old ISI days, though they are somewhat compensated by our occasional evening texts and WhatsApp video calls. Although the Portugal plan with \emph{Sayan} didn’t work out, I hope the Berkeley plan will be a success.

This acknowledgement would be incomplete without expressing my heartfelt gratitude to \emph{Maa, Baba, Bhai} (my brother), and \emph{Prasmita}. Coming from a small city in India, my journey to the Indian Statistical Institute was far from easy. \emph{Baba} envisioned this path and went beyond his capacity to equip me with the knowledge needed to succeed in those competitive exams. \emph{Maa} ensured that everything at home was well under control while I was away. Without her presence, I could not have even contemplated leaving India when \emph{Baba} was diagnosed with cancer. She has been the source of all my mental strength, teaching me how to stay calm and make decisions in challenging situations, a skill that greatly aided me during difficult phases of my PhD.

When I left home, \emph{Bhai} was still an immature boy, but when the situation demanded it, he demonstrated an extraordinary maturity far beyond his years. His unwavering support at home was unmatched, especially after \emph{Baba} passed away.

\emph{Prasmita}, your constant support over the past 11 years through every challenge has been invaluable. You tolerated all my tantrums and stood by me in my difficult times as well as during moments of success. Without you, I would not have overcome those moments of self-doubt or questioned my own abilities.
The four of you have been the pillars of support and inspiration throughout my journey.

Without all of you, completing this work would not have been possible.

\begin{flushright}
	Rounak Ray,\\
	Eindhoven, November 2024.
\end{flushright}
	
\cleardoublepage\currentpdfbookmark{Contents}{contents}
\renewcommand*\contentsname{\vspace*{\emptychapterdist}Contents}

\hypersetup{linkcolor=darkagenta}
 {\setstretch{1.0}\tableofcontents}
\emergencystretch 3em
\newpage
\cleardoublepage

\pagenumbering{arabic}

\chapter{Introduction}
\StartThumbs
\pagestyle{headings}

In the modern era of information technology, the concept of networks has transcended its traditional boundaries to permeate almost every aspect of our lives. From the Internet, social media, and telecommunications to biological systems, transportation infrastructures, and financial markets, networks provide a robust framework for understanding the complex interactions within and between different entities. Real-life networks exhibit fascinating structural properties that have profound implications for their functionality, resilience, and evolution. Empirical studies on real-life networks reveal that most of these networks {(a)} grow with time; {(b)} are small worlds, meaning that typical distances in the network are small; and {(c)} have power-law degree sequences. 

\paragraph{Real-life networks: an overview}

Real-life networks are composed of nodes (or vertices) and edges (or links) that connect these nodes. These networks can vary greatly in scale and complexity. For example, in a social network, nodes represent individuals or organizations, while edges represent social ties, such as friendships or collaborations. In a biological network, nodes might represent genes or proteins, and edges represent biochemical interactions. Despite the diversity of these systems, they often share common structural features that are critical to their behavior and efficiency.

\paragraph{Sparse random graphs}

To better understand and predict the behavior of real-life networks, researchers have turned to the study of random graphs. A random graph is a graph that is generated by some random process, and it serves as a probabilistic model for understanding the structural properties of networks. Among the various types of random graphs, \textit{sparse random graphs} are of particular interest because they provide a more realistic model for large-scale networks where each node has only a few connections compared to the total number of vertices.

A sparse random graph has a number of edges that is proportional to the number of vertices, rather than the square of the number of vertices, which is common in dense graphs. This sparsity reflects the reality of many networks where each vertex interacts with only a small fraction of other vertices, yet the overall network remains interconnected.

\paragraph{Power-law random graphs: capturing real-world complexity}

Power-law random graphs represent a significant advancement in the modeling of real-life networks, particularly in capturing their degree distributions.
Voitalov et al.\ showed in \cite{voitalov2019} that real-world scale-free networks are definitely not as rare as one would conclude based on the popular but unrealistic assumption that real-world data comes from power laws of pristine purity, void of noise and deviations, defying the concerns raised in \cite{BC19}.
In a power-law distribution, the probability that a node has degree \( k \) follows the relation 
\[ P(k) = L(k)k^{-\tau}~, \] 
where \( \tau \) is a positive constant typically between $2$ and $3$, and \(L(k)\) is a slowly-varying function. This implies that most vertices have a small degree, while a few vertices, often called \emph{hubs}, have a very high degree. These hubs play a critical role in the network’s structure and dynamics, making the network both resilient and vulnerable to targeted attacks.

The power-law degree distribution is a defining characteristic of \textit{scale-free networks}, a class of networks that includes many real-life examples, such as the Internet, social networks, and biological systems. The term \emph{scale-free} refers to the fact that the network’s structural properties are similar across different scales. For instance, the same degree distribution pattern holds whether we examine the network at a local, regional, or global level.

\paragraph{Stochastic processes on graphs}

We are typically not only interested in networks, but also in their functionality. Indeed,
we wish to know how quickly a disease or a rumor spreads, how products compete for
the market in social networks, how damage restricts the network’s functionality, how
people reach consensus when talking to one another, etc. Here we model networks
using random graphs, while we model their functionality by stochastic processes on them. In
recent decades, a tremendous body of work was developed discussing stochastic processes
on random graphs

\begin{center}
	{\textbf{Organisation of this chapter}}
\end{center}

This chapter is organised into three sections, each laying the groundwork for the subsequent analysis of random graphs and stochastic processes on these graphs.

{Section~\ref{chap:intro:sec:random-graph}} begins with a discussion of classical static random graphs, providing a foundational understanding of these models. Following this, we introduce the affine preferential attachment model in a general form, highlighting its dynamic nature and relevance to the real-world networks.
Section~\ref{chap:intro:sec:stoc-proc} shifts focus to stochastic processes on graphs, with an emphasis on random graphs. In this section, we introduce and explore two key processes: percolation and the Ising model. These processes are crucial for understanding the behaviour of complex networks and form a significant part of the analysis presented in this thesis.
Section~\ref{chap:into:sec:tools-techniques} summarises various tools and techniques used to study stochastic processes on random graphs. We discuss the concept of local limits, branching processes, P\'{o}lya urn representations, and the collapsing operator. These tools are essential for approaching the challenges posed by the analysis of stochastic processes on preferential attachment models and other types of random graphs.
This chapter provides a comprehensive introduction to the key concepts, models, and techniques that underpin the research presented in the thesis, setting the stage for the detailed discussions in the subsequent chapters.


\section{Random graphs}\label{chap:intro:sec:random-graph}

A graph $G = (V,E)$ consists of a vertex set $V$, and a set of edges $E\subset \{\{i,j\}: i,j\in V\}$ specifying the connections between different vertices.
For a multigraph, $E$ is a multi-set possibly consisting of multiple-edges between vertices, as well as self-loops. 
Throughout, we will assume that $|V|,|E|<\infty$. 
A random graph model specifies a probability distribution over the space of graphs.
We will consider $n$ vertices labeled by $[n]:=\{1,2,...,n\}$, which will serve as the vertex set of the random graph.

\subsection{Static random graphs}
We now discuss some classical static random graph models, and then move on to define the preferential attachment model in great generality. These graphs has a very rich, but still growing literature, and for detailed investigations of these models and recent developments, we refer the reader to \cite{vdH1,vdH2}.

\paragraph{The Erdős-Rényi model}

One of the most well-known models for generating random graphs is the Erdős-Rényi (ER) model, introduced by Paul Erdős and Alfréd Rényi in \cite{ER59,ER60}. The ER model is simple and elegant, providing a foundational framework for the study of random graphs.

The ER model comes in two closely related versions:

\begin{itemize}
	\item \textbf{G(n, p) model:} In this version, the graph \( G(n, p) \) is constructed by starting with a set of \( n \) vertices. Each possible pair of vertices is then connected by an edge with independent probability \( p \). The probability \( p \) is a fixed parameter of the model, and it determines the expected density of the graph. When \( p \) is close to 0, the graph is sparse, while when \( p \) is close to 1, the graph is dense.
	\item \textbf{G(n, m) model:} In this version, a graph \( G(n, m) \) is formed by choosing exactly \( m \) edges uniformly at random from the set of all possible \( \binom{n}{2} \) edges between \( n \) vertices. This model is often used when the exact number of edges is known, rather than the probability of edge formation.
\end{itemize}

However, for $m\approx np$, the two models are asymptotically equivalent \cite{JLR00}. 
Note that the degree of each vertex is distributed as a $\mathrm{Bin}(n-1,p)$ random variable. 
Thus for sparse regime $p = \lambda/n$, the asymptotic degree of each vertex is $\mathrm{Poisson}(\lambda)$, with fixed average degree~$\lambda$.

\paragraph{Configuration model}

While the Erd\H{o}s-R\'{e}nyi model is valuable for its simplicity, it assumes that all vertices have approximately the same degree in distribution (number of connections), and very few vertices of high degree (i.e.\ no power-law), which is not representative of many real-world networks. The \textit{configuration model} addresses this limitation by allowing for more flexibility in the degree distribution of the graph.

Consider a non-increasing sequence of degrees $\boldsymbol{d} = ( d_i )_{i \in [n]}$ such that $\ell_n = \sum_{i \in [n]}d_i$ is even. 
The configuration model on $n$ vertices having degree sequence $\boldsymbol{d}$ is constructed as follows \cite{B80,BC78}:
\begin{itemize}
	\item[] Equip vertex $j$ with $d_{j}$ stubs, or \emph{half-edges}. Two half-edges create an edge once they are paired. Therefore, initially we have $\ell_n=\sum_{i \in [n]}d_i$ half-edges. Pick any one half-edge and pair it with a uniformly chosen half-edge from the remaining unpaired half-edges and keep repeating the above procedure until all the unpaired half-edges are exhausted. 
\end{itemize}
Let $\rm{CM}$ denote the graph constructed by the above procedure.
Note that $\rm{CM}$ may contain self-loops or multiple edges. It was shown in \cite{AHH19,JL18} that, under very general assumptions, the asymptotic probability of the graph being simple is positive.

The configuration model is particularly useful for studying networks with heterogeneous degree distributions, such as power-law distributions, which are common in many real-life networks, including social, biological, and technological systems. Configuration model is used to study the impact of degree heterogeneity on various network properties, such as connectivity, resilience, and the spread of information or diseases in \cite{T07,BST09,BST10}.

\paragraph{Inhomogeneous random graph}

Inhomogeneous random graphs generalize the idea of random graphs by allowing the probability of edge formation to vary across different pairs of vertices. This model is more flexible and can capture the varying strengths of interactions or affinities between vertices that are often observed in real-life networks. A detailed analysis of the properties of this random graph model has been provided in \cite{BJR07} under a very general setup. 

In an inhomogeneous random graph, each vertex \( i \) is assigned a weight \( w_i \), and the probability \( p_{ij} \) that an edge exists between vertices \( i \) and \( j \) depends on their weights. A common form for this probability is given by

\[
p_{ij} = f(w_i, w_j)~,
\]

where \( f \) is a function that typically increases with \( w_i \) and \( w_j \). For example, in some models, \( p_{ij} \) might be proportional to the product of the weights, \( p_{ij} \propto w_i  w_j \), which can lead to graphs where vertices with higher weights (representing more \emph{important} or \emph{active} entities) are more likely to be connected. Depending on the choice of the function \( f \), there are several models in the literature namely \emph{Norros-Reittu model} \cite{NR06}, \emph{Chung-Lu model} \cite{CL02,CL03}, \emph{Generalised random graph} \cite{BDM06}. Choosing \( f \) function to be a constant \( p\in (0,1) \), we obtain the Erd\H{o}s-R\'{e}nyi model back from inhomogeneous random graphs.

There are several other models available in the literature to accommodate different features of real-life networks, e.g.\ the random geometric graph was introduced in \cite{G61} to capture the underlying geometry of the networks. This network has been in use to study social networks in \cite{WPR06,NA73,LC18,WW90} as it can encode actual physical distance, such as two servers in adjacent cities, as well as a more abstract form of similarity, e.g., users with similar interests or hobbies.

\subsection{Preferential attachment models}\label{chap:intro:sec:PA}

The Barab\'asi-Albert model of \cite{ABrB,albert1999} is {one of the most popular random graph model for dynamic real-life networks} due to the fact that, through a simple dynamic, its properties resemble real-world networks. This model has been generalised in many different ways, creating a wide variety of {\em preferential attachment models} (PAMs).

By preferential attachment models we denote a class of random graphs with a common dynamics: At every time step, a new vertex appears in the graph, and it connects {to} $m\geq1$ existing vertices, with probability proportional to a function of the degrees \FC{of the vertices}. In other words, when vertex $n\in\N$ appears, all of its edges connects to vertex $i\leq n$ with probability
\eqn{
	\label{for:heur:PAfunction}
	\pr(n\rightsquigarrow  i \mid \mathrm{PA}_{n-1}) \propto f(d_i(n-1)),
}	
where \FC{$\PA_{n}$ is the preferential attachment graph with $n$ vertices,} $d_i(n-1)$ denotes the degree of vertex $i$ in the graph $\PA_{n-1}$, and $f$ is some \FC{preferential attachment} function. We {thus} in fact {consider a whole} PA class of random graphs since every function $f$ defines a different model.
The original Barab\'asi-Albert model is retrieved with $f(k) = k$. The literature often considers the so-called {\em affine} PAM, where $f(k) = k+\delta$, for some constant $\delta>-m$.

The constant $\delta$ allows for flexibility in the graph structure. In fact, the power-law exponent of the degree distribution is given by $\tau= 3+\delta/m$ \cite{BRST01, Der2009,vdH1}, and in general, for $m\geq 2$, the typical distance and the diameter are of order $\log\log n$ when $\tau\in(2,3)$, while they are of order $\log n$ when $\tau>3$. When $\tau=3$, distances and diameter are of order $\log n/\log\log n$ {instead} \cite{BolRio04b,CarGarHof,DSMCM,DSvdH}.

\begin{figure}
	\centering
	\includegraphics[width=0.8\linewidth]{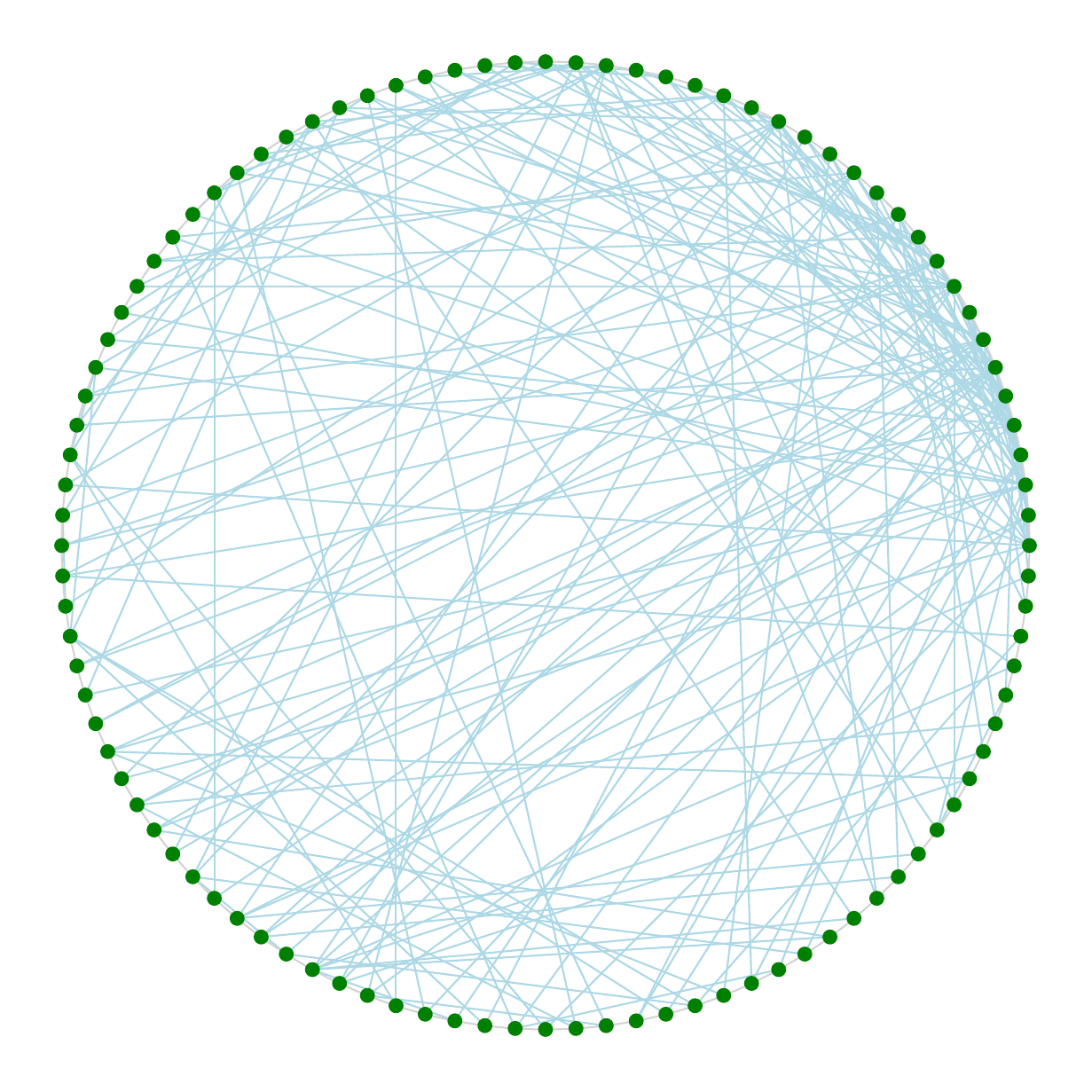}
	\caption{Realisation of preferential attachment model on \(100\) vertices}	
\end{figure}

In {PAMs}, the degree of a vertex increases over time and higher degree vertices are prone to attract the edges incident to new vertices, increasing their degrees even further. In literature, this is sometimes referred {to as} \textit{rich-get-richer} effect. The models where the vertex degrees are determined by a \textit{weight} associated to it are sometimes called \textit{rich-by-birth} {models}. 

In \cite{DEH09}, the authors have considered a model incorporating both of these effects. The model is a {PAM} with random {out-degrees}, i.e., every vertex joins the existing network with a random number of edges that it connects to the existing vertices preferentially to their degrees at that time, as done in usual preferential attachment models. Here, the authors have shown that this system also shows the power-law degree distribution. {Jordan \cite{jordan06} considered} a special case of this particular model to analyze their results on degree distribution of the random graph. 

Cooper and Frieze \cite{cooper2003} have shown  that a further generalised version of this PAM also {has a} power-law degree distribution. Since most real-life networks are dynamic in nature and have power-law degrees, the PAM is often used to model them. However, in such networks it is never the case that every vertex joins with exactly the same out-degree, and instead i.i.d.\ out-degrees are more realistic. Gao and van der Vaart \cite{GV17} have shown that the maximum likelihood estimator of the fitness parameter $\delta$ of PAM with random out-degrees is asymptotically normal and has asymptotically minimal variance.
There have been further generalisations by allowing younger vertices to have higher degrees in PAMs through a random {\em fitness} parameter. More precisely, individual fitness factors \((\eta_i)_{i\geq 1}\) can be assigned to vertices, resulting in the probabilities in \eqref{for:heur:PAfunction} being proportional to \(\eta_i f(d_i(n-1))\), thus yielding PAMs \emph{with multiplicative fitness}, or proportional to \(d_i(n-1) + \eta_i\), giving rise to \emph{additive fitness} \cite{B07,Bianconi,Borgs,der16,der2014}. Usually, \((\eta_i)_{i\geq 1}\) is chosen as a sequence of i.i.d.\ random variables.

\paragraph{Model definition}\label{chap:intro:sec:model}
{Several versions} of preferential attachment models are available in the literature. We generalize these definitions to the case of a random initial {number of edges that connect to the already available graph. We refer to these edge numbers as {\em out-degrees}, even though we consider our model to be {\em undirected.}} 
Let $(m_i)_{i\ge 3}$ be {a sequence of i.i.d.\ copies of $M$ which is a $\N$-valued random variable} with finite $p$-th moment for some $p>1$, and $\delta >-{\infsupp}(M)$ be a fixed real number, {where $\operatorname{supp}(M)$ denotes the support of the random variable $M$}.
In our models, every new vertex $v$ joins the graph with $m_v$ many edges incident to it. 
In case of self-loops, degree of the new vertex increases by $2$.
{Thus,} we can consider the existing models as a degenerate case of ours. Define $m_1=m_2=1$, $\boldsymbol{m}=(m_1,m_2,m_3,\ldots)$ and
\[
m_{[l]} = \sum\limits_{i\leq l} m_i.
\]
To describe the {edge-connection probabilities} and to simplify our calculations, we {frequently work with the conditional law given $\boldsymbol{m}=(m_i)_{i\ge 1}$}. 
The conditional measure is denoted by $\prob_m$, i.e., for any event $\mathcal{E}$,
\eqn{\label{eq:conditional_measure:M} \prob_m(\mathcal{E})= \prob\left( \left. \mathcal{E} \right| \boldsymbol{m} \right).}
{Conditionally} on $\boldsymbol{m}$, we define some {versions of the preferential attachment models,} special cases of which {are} equivalent to the models~(a-g) described in \cite{GarThe}. We consider the initial graph to be $G_0$ with $2$ vertices with degree $a_1$ and $a_2$, respectively. {Throughout the remainder of this thesis, we fix}  $\delta>-\infsupp(M)$.
\paragraph{\bf Model (A)} {This is} the generalised version of \cite[Model (a)]{vdH1}. Conditionally on $\boldsymbol{m}$, for $v\in\N$ and $j=1,\ldots,m_v$, the attachment probabilities are given by
\eqn{
	\label{eq:edge_connecting:model(rs):2}
	\prob_m\left(\left. v\overset{j}{\rightsquigarrow} u\right| \PArs_{v,j-1}(\boldsymbol{m},\delta) \right) = \begin{cases}
		\frac{d_u(v,j-1)+\delta}{c_{v,j}}\qquad\qquad \mbox{ for $v>u$},\\
		\frac{d_u(v,j-1)+1+\frac{j}{m_u}\delta}{c_{v,j}}\qquad \mbox{ for $v=u$},
	\end{cases}
}
where $v\overset{j}{\rightsquigarrow} u$ denotes that vertex $v$ connects to $u$ with its $j$-th edge, $\PArs_{v,j}(\boldsymbol{m},\delta)$ denotes the graph with $v$ vertices, \FC{with} the $v$-th vertex \FC{having} $j$ out-edges, and $d_u(v,j)$ denotes the degree of vertex $u$ in $\PArs_{v,j}(\boldsymbol{m},\delta)$. We identify $\PArs_{v+1,0}(\boldsymbol{m},\delta)$ with $\PArs_{v,m_v}(\boldsymbol{m},\delta)$. The normalizing constant $c_{v,j}$ in \eqref{eq:edge_connecting:model(rs):2} equals
\eqn{\label{eq:normali}
	c_{v,j} := a_{[2]}+2\left( m_{[v-1]}+j-2 \right) - 1+ (v-1)\delta + \frac{j}{m_v}\delta \, ,
}
where $a_{[2]}=a_1+a_2$. We denote the above model by $\PArs_v(\boldsymbol{m},\delta)$. If we consider $M$ as a degenerate distribution {that is equal to $m$ a.s.}, then $\PArs_v(\boldsymbol{m},\delta)$ is essentially equivalent to $\PA_v^{\sss ({m},\delta)}(a)$ defined in \cite{vdH2}. $\PA_v^{({m,0})}(a)$ was {informally introduced in \cite{albert1999} and first studied rigorously in \cite{BolRio04b}.}
Later the model for general $\delta$ was described in \cite{ABrB}.
\smallskip
\paragraph{\bf Model (B)} {This is} the generalised version of \cite[Model (b)]{vdH1}. Conditionally on $\boldsymbol{m}$, the attachment probabilities are given by
\eqn{\label{eq:edge_connecting:model(rt):2}
	\prob_m\left(\left. v\overset{j}{\rightsquigarrow} u\right| \PArt_{v,j-1}(\boldsymbol{m},\delta) \right) = \begin{cases}
		\frac{d_u(v,j-1)+\delta}{c_{v,j}}\qquad\qquad \mbox{ for $v>u$},\\
		\frac{d_u(v,j-1)+\frac{(j-1)}{m_u}\delta}{c_{v,j}}\qquad \mbox{ for $v=u$},
\end{cases}}
where again $\PArt_{v,j}(\boldsymbol{m},\delta)$ denotes the graph with $v$ vertices, \FC{with} the $v$-th vertex \FC{having} $j$ out-edges, and $d_u(v,j)$ denotes the degree of vertex $u$ in $\PArt_{v,j}(\boldsymbol{m},\delta)$. We identify $\PArt_{v+1,0}(\boldsymbol{m},\delta)$ with $\PArt_{v,m_v}(\boldsymbol{m},\delta)$. The normalizing constant $c_{v,j}$ in \eqref{eq:edge_connecting:model(rt):2} now equals
\[
c_{v,j} = a_{[2]}+2\left( m_{[v-1]}+j-3 \right) +(v-1)\delta + \frac{(j-1)}{m_v}\delta \,.
\]
We denote the above model by $\PArt_v(\boldsymbol{m},\delta)$. For $M$ a degenerate distribution {which is equal to $m$ a.s.}, we obtain $\PA_v^{\sss ({m},\delta)}(b)$ described in \cite{vdH2} from $\PArt_v(\boldsymbol{m},\delta)$.
\FC{\begin{remark}[Difference between models (A) and (B)]
		\rm{From the definition of the models (A) and (B), the models are different in that the first edge from every new vertex can create a self-loop in model (A) but not in model (B). Note that the edge probabilities are different for all $j$.}\hfill$\blacksquare$
\end{remark}}
\paragraph{\bf Model (D)} This model is the generalised {version} of the sequential model described in \cite{BergerBorgs}. We start with a fixed graph $G_0$ of size $2$ and degrees $a_1$ and $a_2$, respectively. Denote the graph {by} $\PAri_{v}(\boldsymbol{m},\delta)$ when {the graph has $v$ vertices.}

For $v\geq 3$, vertex $v$ enters the system with $m_v$ many out-edges {whose} edge-connection probabilities are given by
\eqn{\label{sec:model:eq:1}
	\prob_m\left(\left. v\overset{j}{\rightsquigarrow} u \right| \PAri_{v,j-1} \right) = 
	\frac{d_u(v,j-1)+\delta}{c_{v,j}}\quad \text{for}\quad v>u,
}
where $\PAri_{v,j}(\boldsymbol{m},\delta)$ denotes the graph with $v$ vertices, \FC{with} the $v$-th vertex \FC{having} $j$ out-edges, $d_u(v,j)$ is the degree of the vertex $u$ in the graph $\PAri_{v,j}$, and
\begin{equation*}
	c_{v,j}= a_{[2]}+2\left( m_{[v-1]}-2 \right)+(j-1) + (v-1)\delta\,.
\end{equation*}
Again we identify $\PAri_{v+1,0}(\boldsymbol{m},\delta)$ with $\PAri_{v,m_v}(\boldsymbol{m},\delta)$. For $M$ a degenerate random variable {that is equal to $m$ a.s.}, we obtain $\PA_v^{\sss ({m},\delta)}(d)$, {as} described in \cite{vdH2}.
{\paragraph{\bf Model (E)} This model is the independent model proposed {by Berger et al.\ in \cite{DEH09}}. We start with the same initial graph $G_0$. We name the graph $\PArf_{v}(\boldsymbol{m},\delta)$ when there are $v$ vertices in the graph. Conditionally on $\boldsymbol{m},$ every new vertex $v$ is added to the graph with $m_v$ many {out-edges}. Every {out-}edge $j\in[m_v]$ from $v$ connects to the vertex $u\in[v-1]$ independently with the {probabilities}
	\eqn{\label{eq:model:f:1}
		\prob_m\left( \left.v\overset{j}{\rightsquigarrow}u\right|\PArf_{v-1}(\boldsymbol{m},\delta) \right) = \frac{d_u(v-1)+\delta}{a_{[2]}+2(m_{[v-1]}-2)+(v-1)\delta},}
	where $d_{u}(v-1)$ is the degree of the vertex $u$ in $\PArf_{v-1}$. Note that the intermediate degree updates are not there in this graph. Similarly as {in} model (D), no self-loops are allowed here.}

\paragraph{\bf Model (F)} The independent model does not allow for self-loops, but multiple {edges} between two vertices {may still occur}. {Such multigraphs are, in many applications, unrealistic, and we next propose a model where the graphs obtained are {\em simple}, i.e., without self-loops and multiple edges.} Model (F) is a variation of the independent model where the new vertices connect to the existing vertices in the graph independently {but {\em without replacement}} and the edge-connection probabilities are {similar to} \eqref{eq:model:f:1}. Indeed, for $j\geq 2$, the normalization factor changes a little due to the fact that vertices that have already appeared as neighbours are now forbidden. The degenerate case of our PAM, i.e., with fixed out-degree, is studied in \cite{basu2014} for analyzing the largest connected component in the strong friendship subgraph of evolving online social networks.

\begin{remark}[{The missing model (c)}]
	{\rm There is a description of model (c) in \cite[Section 8.2]{vdH1}, which {reduces to model (a), so we refrain from discussing it further in this Chapter}. 
	}\hfill$\blacksquare$ 
\end{remark}

\section{Stochastic processes}\label{chap:intro:sec:stoc-proc}

Stochastic processes on random graphs combine the randomness inherent in graph structures with dynamic processes that live on these graphs. A stochastic process is a collection of random variables that evolves over time or space, and when these processes are applied to random graphs, they model a variety of phenomena in complex networks, such as random walks, the spread of information, diseases, or cascading failures. 

Random graphs, provide a flexible framework for studying how these processes behave in networks with different structural properties. The interplay between the topology of the graph and the stochastic dynamics can lead to rich and often surprising behaviors. For instance, the speed and pattern of disease spread in a population can be highly sensitive to the degree distribution of the underlying social network. In the next section, we study properties of percolation and Ising models on graphs.

\subsection{Percolation on graphs}

\paragraph{Percolation theory}

Percolation theory is a fundamental concept in the study of random processes living on networks and graphs. It examines the behaviour of connected clusters in a random graph as edges or vertices are added or removed randomly. Percolation is particularly useful in understanding the robustness and connectivity of networks, the spread of diseases, and various phase transitions in statistical physics.

Given a graph $G$, bond (or site) percolation refers to the deletion of each edge (or vertex) independently with probability $1-p$. Throughout this discussion, we focus on bond percolation, hence we will simply refer to it as percolation. The resulting graph is denoted by $G(p)$. In the case of percolation on random graphs, the deletion of edges is independent of the structure of the underlying random graph.

The \emph{percolation process} is defined as the graph-valued stochastic process $(G(p))_{p \in [0,1]}$, which is coupled through the so-called Harris coupling. More precisely:

\begin{itemize}
	\item[] Associate an independent uniform $[0,1]$ random variable $U_e$ to each edge $e$ of the graph $G$. The graph $G(p)$ is generated by retaining edge $e$ if and only if $U_e \leq p$. By keeping the uniform random variables fixed while varying $p$, we obtain a coupling between the graphs $(G(p))_{p \in [0,1]}$.
\end{itemize}

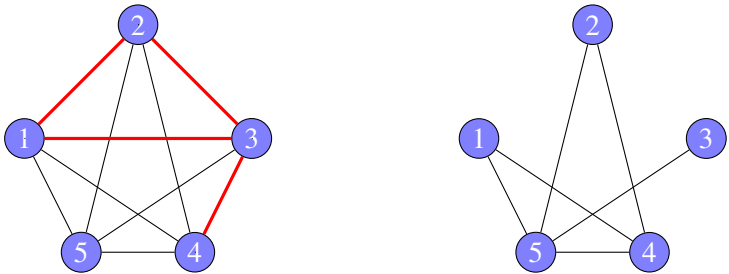
\begin{figure}[h!]
	\centering
	\begin{tikzpicture}[scale=1, auto, every node/.style={circle, draw, fill=blue!50, inner sep=2pt, text=white}]
		\node (1) at (0, 1.5) {1};
		\node (2) at (1.5, 3) {2};
		\node (3) at (3, 1.5) {3};
		\node (4) at (2.25, 0) {4};
		\node (5) at (0.75, 0) {5};
		
		\foreach \i in {1,2,3,4}
		\foreach \j in {\i,...,5}
		\draw (\i) -- (\j);
		
		\draw[red, very thick] (1) -- (2);
		\draw[red, very thick] (1) -- (3);
		\draw[red, very thick] (2) -- (3);
		\draw[red, very thick] (3) -- (4);
		
		\begin{scope}[xshift=6cm]
			\node (1) at (0, 1.5) {1};
			\node (2) at (1.5, 3) {2};
			\node (3) at (3, 1.5) {3};
			\node (4) at (2.25, 0) {4};
			\node (5) at (0.75, 0) {5};
			
			\draw (1) -- (4);
			\draw (1) -- (5);
			\draw (2) -- (4);
			\draw (2) -- (5);
			\draw (3) -- (5);
			\draw (4) -- (5);
		\end{scope}
	\end{tikzpicture}
	\caption{Percolation on a complete graph.}
	\label{fig:percolation}
\end{figure}

Percolation has been extensively studied on infinite connected graphs, such as the hypercubic lattice. This is a very simple model that exhibits a \emph{phase transition}. For an in-depth exploration of percolation theory, we refer the reader to \cite{JL18, FvdHdHH17, CHL09, Grimmet99, RvdHSF} and the references therein.


Percolation is also employed to model vaccination strategies on a network, aiming to prevent the spread of an epidemic. A detailed account of these applications can be found in \cite{Newman-book, BBCS05, N02}.

\paragraph{Percolation on infinite trees}

Percolation on infinite trees offers a simpler setting to study phase transition phenomena, due to the absence of loops and the regular structure of trees. An infinite tree is a graph where any two nodes are connected by exactly one path, and the number of nodes is infinite.

For an infinite tree, the percolation process begins by assigning each edge an independent probability \( p \) of being open. The key question is whether an infinite connected component of open edges exists. The critical probability \( p_c \) is defined as the infimum of the values of \( p \) for which an infinite connected component appears with non-zero probability. Let \( (\Tbold,o) \) be an infinite tree rooted at $o$, and \( (\Tbold(p),o) \) denote the percolated cluster of \( \Tbold(p) \) containing \(o\). Then, \(p_c\) is formally defined as
\[
	p_c = \inf\big\{ p\in[0,1]:~ \prob\big((\Tbold(p),o)~\mbox{is infinite}\big)>0 \big\}~.
\]

For $d$-regular infinite trees, where each vertex has the same degree \( d \), the critical probability \( p_c \) can be calculated explicitly as

\[
p_c = \frac{1}{d-1}.
\]

When \( p < p_c \), all connected components are finite, and no infinite cluster exists, whereas for \( p > p_c \), an infinite cluster appears with positive probability.
Percolation on branching processes is a useful technique for analysing the emergence of large connected components in random structures. It is important to note that, if the root is retained in site percolation, performing bond and site percolation on branching processes are equivalent operations.
While percolating a branching process, we retain or delete individuals (or equivalently, edges in the associated tree) independently, with a given probability \( p \). 

The critical probability \( p_c \) in percolation on branching processes is the threshold above which an infinite cluster (corresponding to an infinitely large connected component) emerges. Below this threshold, all components remain finite. The analysis of percolation on branching processes is particularly elegant because the branching structure allows for precise calculations of critical probabilities and cluster sizes, and due to the fact that a percolated branching process remains a branching process.

For a simple Galton-Watson branching process with a mean offspring number \( \mu \), the critical percolation threshold \( p_c \) is determined by the condition that the expected number of surviving offspring per individual equals \(1\), i.e., \( p_c \mu = 1 \). Thus,

\[
p_c = \frac{1}{\mu}.
\]

Heuristically, the inverse of the growth rate is the critical percolation threshold, for an exponentially growing branching process.

\paragraph{Percolation on finite random graphs}

In contrast to infinite trees, finite random graphs introduce additional complexity due to their finite size and the presence of loops. Percolation on finite random graphs involves studying the formation of large connected components as edges are randomly included with probability \( p \).

In this setting, the concept of a phase transition is still relevant but manifests differently due to the finite size of the graph. As \( p \) increases, a critical threshold \( p_c \) is reached, beyond which a \emph{giant} connected component emerges, containing a substantial fraction of the total number of vertices. Below \( p_c \), most components remain small and fragmented. 

Now, we define the critical percolation threshold for a sequence of graph \(\{G_n\}_{n\geq 1}\). For any \(p\in[0,1]\), let \(\Ccal_1^{(n)}(p)\) denote the largest connected component of the \(p\)-percolated graph \(G_n\). Then the critical percolation threshold \(p_c\) for the graph sequence \(\{G_n\}_{n\geq 1}\) is defined as
\[
	p_c=\inf\Big\{ p\in[0,1]:~\chRounak{\liminf\limits_{n\to\infty}\prob\Big(\frac{|\Ccal_1^{(n)}(p)|}{n}>\vep\Big)=1}~\mbox{for some }\vep>0 \Big\}~.
\]

The behavior of percolation on finite random graphs is often studied in the context of specific random graph models, such as the Erd\H{o}s-R\'{e}nyi models, or configuration models. It's noteworthy that Erd\H{o}s-R\'{e}nyi random graphs are essentially percolated complete graphs, and hence a percolated Erd\H{o}s-R\'{e}nyi graph is again an Erd\H{o}s-R\'{e}nyi graph. In the Erd\H{o}s-R\'{e}nyi graph \( G(n, p) \), where \( n \) is the number of vertices and each edge is included with probability \( p \), \cite{ER59,G59} proved that the critical probability \( p_c \) is given by

\[
p_c = \frac{1}{n-1}.
\]

For \( p > p_c(1+\vep) \) and any $\vep>0$, a giant component emerges, typically containing \( O(n) \) vertices. For \( p < p_c \), the graph consists mostly of small components. 

For random $d$-regular graphs \cite{ABS04}, $p_c = 1/(d-1)$.
The percolation phase transition on \emph{random regular graphs} and configuration models were studied in \cite{F07,J09} when the empirical degree distribution has a finite second moment in the large network limit, as $n\to\infty$.
See also \cite{FJP22,JP18} for some recent results on more general degree sequences.
In this thesis, we study our recent work on computing the critical percolation threshold for preferential attachment models explicitly.


\subsection{Ising model}\label{subsec:intro:def-ising}
The Ising model originated in the early \(20\)-th century as a simple yet powerful framework to study phase transitions and collective behaviour in magnetic systems. It was first introduced by Wilhelm Lenz in \cite{L1920} and further developed by his student, Ernst Ising, in \cite{I1925}. The model was initially conceived to explain ferromagnetism, a phenomenon where atomic spins align to create a permanent magnet.


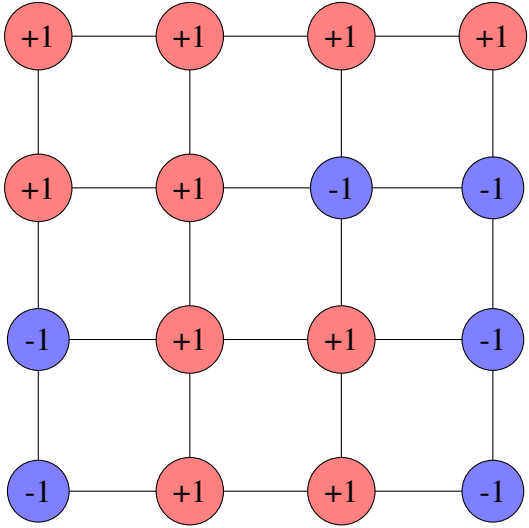
\begin{figure}[H]
\centering
\begin{tikzpicture}
    \foreach \x in {0, 2, 4} {
        \foreach \y in {0, 2, 4, 6} {
            \draw (\x, -\y) -- (\x+2, -\y);
        }
    }
    \foreach \x in {0, 2, 4, 6} {
        \foreach \y in {0, 2, 4} {
            \draw (\x, -\y) -- (\x, -\y-2);
        }
    }

    \node[draw, circle, minimum size=0.5cm, fill=red!50, font=\large] at (0, 0) {+1};
    \node[draw, circle, minimum size=0.5cm, fill=red!50, font=\large] at (2, 0) {+1};
    \node[draw, circle, minimum size=0.5cm, fill=red!50, font=\large] at (4, 0) {+1};
    \node[draw, circle, minimum size=0.5cm, fill=red!50, font=\large] at (6, 0) {+1};
    
    \node[draw, circle, minimum size=0.5cm, fill=red!50, font=\large] at (0, -2) {+1};
    \node[draw, circle, minimum size=0.5cm, fill=red!50, font=\large] at (2, -2) {+1};
    \node[draw, circle, minimum size=0.5cm, fill=blue!50, font=\large] at (4, -2) {-1};
    \node[draw, circle, minimum size=0.5cm, fill=blue!50, font=\large] at (6, -2) {-1};
    
    \node[draw, circle, minimum size=0.5cm, fill=blue!50, font=\large] at (0, -4) {-1};
    \node[draw, circle, minimum size=0.5cm, fill=red!50, font=\large] at (2, -4) {+1};
    \node[draw, circle, minimum size=0.5cm, fill=red!50, font=\large] at (4, -4) {+1};
    \node[draw, circle, minimum size=0.5cm, fill=blue!50, font=\large] at (6, -4) {-1};
    
    \node[draw, circle, minimum size=0.5cm, fill=blue!50, font=\large] at (0, -6) {-1};
    \node[draw, circle, minimum size=0.5cm, fill=red!50, font=\large] at (2, -6) {+1};
    \node[draw, circle, minimum size=0.5cm, fill=red!50, font=\large] at (4, -6) {+1};
    \node[draw, circle, minimum size=0.5cm, fill=blue!50, font=\large] at (6, -6) {-1};

\end{tikzpicture}
\caption{Ising model realisation on a \(4\times4\) grid.}
\end{figure}

In its most basic form, the Ising model considers a graph structure where each site (or node) hosts a spin that can take one of two values, typically represented as \(+1\) or \(-1\). These spins interact with their nearest neighbours, favouring alignment to minimise the system's energy. The interplay between thermal fluctuations, that tend to randomise spin orientation, and the interactions, which seek to align spins, gives rise to interesting collective phenomena.

In the Ising model, we can observe a phase transition in terms of the abrupt change in the system's macroscopic properties as the temperature crosses a critical threshold, often known as the \emph{Curie temperature} in physics literature. At high temperatures, thermal fluctuations dominate, causing the spins to orient randomly, resulting in a disordered phase with no net magnetisation. However, as the temperature decreases and falls below a critical value (known as the critical temperature), the system undergoes a phase transition to an ordered phase. 
In the low-temperature phase, the spins tend to align, leading to a spontaneous magnetisation where the system exhibits a net magnetic moment. This phase transition is a hallmark of the Ising model, illustrating how microscopic interactions can lead to emergent macroscopic behaviour.

\paragraph{Ising model on finite graphs}
Let $(G_n)_{n\geq 1}$ be a sequence of (possibly random) graphs, where $G_n=(V_n,E_n)$ has vertex set $V_n=\{1,2,\ldots,n\}=[n]$, and some (possibly random) set of edges $E_n$. Each vertex in $G_n$ is assigned an Ising spin value in $\{-1,+1\}$. A spin configuration of $G_n$ is denoted by $\sigma=(\sigma_1,\ldots,\sigma_n)\in\{-1,+1\}^n$. For any $\beta\geq 0$ and $B_v\in\R$ for all $v\in V_n$, the Ising model on $G_n$ is given by the Boltzmann distribution
\begin{equation}\label{eq:def:boltzmann:seq}
	\mu_n(\sigma)=\frac{1}{Z_n(\beta,\underline{B})}\exp\left\{ \beta\sum\limits_{\{u,v\}\in E_n}\sigma_u\sigma_v +\sum\limits_{v\in V_n}B_v\sigma_v \right\},
\end{equation}
where $Z_n({\beta,\underline{B}})$ is the normalising constant, also known as the partition function, given by
\begin{equation}\label{eq:def:partition-function}
	Z_n({\beta,\underline{B}})=\sum\limits_{\sigma\in\{-1,+1\}^n}\exp\left\{ \beta\sum\limits_{\{u,v\}\in E_n}\sigma_u\sigma_v +\sum\limits_{v\in V_n}B_v\sigma_v \right\}.
\end{equation}
For any function $f:\{-1,+1\}^n\to \R,$ define $\langle f(\sigma) \rangle_{\mu_n}$ as
\begin{equation}\label{eq:def:quenched-expectation}
	\langle f(\sigma) \rangle_{\mu_n} =\sum\limits_{\sigma\in\{-1,+1\}^n} f(\sigma)\mu_n(\sigma).
\end{equation}
Lastly, we define the thermodynamic quantity of main interest, \textit{pressure per particle}, as
\begin{equation}\label{eq:def:pressure-per-particle:initial}
	\psi_n(\beta,\underline{B})=\frac{1}{n}\log Z_n(\beta,\underline{B}).
\end{equation}

\paragraph{Ising model on the lattice}

Ernst Ising proved in \cite{I1925} that the Ising model on \(\mathbb{Z}\) does not exhibit a phase transition, meaning the critical temperature in this case is \(\infty\). For the Ising model on \(\mathbb{Z}^2\), Onsager obtained the explicit thermodynamic limits of several quantities in \cite{O44}, which were later extensively generalised by Montroll, Potts, and Ward in \cite{MPW63}. From two dimensions onwards, the Ising model is known to exhibit a phase transition. The behaviour of the Ising model in higher dimensions is explored in detail in \cite{BKLS89}.

\paragraph{Ising model on the random graphs}
Alongside the Ising model on lattice structures, it has also been studied on complete graphs and other types of graphs, as discussed in \cite{KPW13, EE88, CG16}. The Ising model on random graphs introduces a second level of randomness due to the inherent variability of the graph itself. This leads to two distinct settings: \emph{quenched} and \emph{annealed}. In the \emph{annealed} setting, the Ising model is examined after averaging out the randomness of the graph. In contrast, the \emph{quenched} setting involves studying the Ising model conditioned on a specific realisation of the random graph. For a detailed analysis of Ising models on random graphs, see \cite{B06, CG13, E06}.

Beyond its original application to ferromagnetism, the Ising model has become relevant in a wide range of disciplines, including biology, sociology, and computer science. It is employed to model binary systems with local interactions, such as gene regulation networks, social dynamics, and optimisation problems. For applications in biological sciences, see \cite{KTS20,WB16}; for applications in social sciences, refer to \cite{FS19}; and for applications in computer science and optimisation, see \cite{CZSHT18,BKNKOS22}. The Ising model's mathematical elegance and capacity to model critical behaviour ensure its status as a cornerstone of theoretical physics and applied mathematics.

\section{Tools and techniques}\label{chap:into:sec:tools-techniques}

In the study of sparse random graphs, various tools and techniques have been developed to analyse their complex structures and dynamics. Due to the sparse edge density characteristic of these graphs, the neighbourhood of a uniformly chosen vertex often resembles a branching process, particularly when considering small, fixed neighbourhoods. This resemblance was rigorously formalised through the concept of local convergence, introduced by Benjamini and Schramm in \cite{Benjamini_LL}. Local convergence has since become a fundamental method for studying the limiting behaviour of local structures in random graphs, such as degree distribution and clustering.

Beyond local properties, some global characteristics such as the critical percolation threshold or the behaviour of the Ising model can also be derived using local limits, albeit under specific model constraints. These global properties provide crucial insights into the overall behaviour of the graph, bridging the gap between local structure and global dynamics.

The preferential attachment model presents additional challenges due to its dynamic nature, where edges are added sequentially and depend on the existing network structure. This sequential dependence introduces complexities that require specialized techniques to handle. One common approach is the collapsing method, which transforms preferential attachment trees into graphs, simplifying the analysis. Another useful technique is the P\'{o}lya urn representation from \cite{BBCS05}, which is employed to achieve conditional independence between edge-connection events, facilitating the study of these models. Additionally, the continuous-time branching process representation is often utilised to analyse preferential attachment models, as the literature on branching processes is more mature and offers a wealth of applicable results.

In this section, we will explore these tools and techniques in detail, discussing how they are applied to the study of sparse random graphs and, more specifically, preferential attachment models. These methods are essential for gaining a deeper understanding of the structure and behaviour of these complex networks.

\subsection{Branching processes}

Branching processes are a class of random models that describe the evolution of populations where individuals reproduce independently according to some probability distribution. Branching processes are classified into two major settings: the discrete-time branching process and the continuous-time branching process.

\paragraph{Discrete-time branching process}
The simplest discrete-time branching process is the Galton-Watson process, which models a population where each individual in one generation produces a random number of offspring in the next generation, independently of other individuals. The process starts with a single ancestor, and the population evolves through generations.

In general, when a branching process is infinite, it tends to grow exponentially (and sometimes even doubly exponentially). For branching processes with exponential growth, the exponential growth parameter is denoted by \( \lambda \). In cases where the mean number of offspring \( \mu > 1 \), the population typically grows exponentially, and \( \lambda \) represents the rate of this growth. Heuristically, if \( Z_n \) denotes the size of the \( n \)-th generation, the exponential growth parameter satisfies

\[
\lim_{n \to \infty} \E[Z_n]^{1/n} = \lambda~,
\]
\noindent
indicating that \( \E[Z_n] \) grows approximately as \( \lambda^{ n} \) for large \( n \). The exponential growth parameter plays a crucial role in determining the long-term behaviour of the process, including whether the population will grow indefinitely or face extinction.

While the basic Galton-Watson process assumes a single type of individual, many real-world populations are more complex, involving individuals of different types, each with its own reproduction behaviour. This leads to multitype branching processes, where the population consists of individuals of several types, and the offspring distribution depends on the type of the parent.

In a multitype branching process with countably many types, the evolution of the population can be described by a vector \( \mathbf{Z}_n \) representing the number of individuals of each type in the \( n \)-th generation. The process is governed by a matrix of reproduction probabilities, where each entry specifies the expected number of offspring of a particular type produced by an individual of another type.

The exponential growth parameter generalises to multitype branching processes with countably many types, where it is defined as the largest eigenvalue of the mean offspring matrix. As discussed in \cite{AN72,H63}, this parameter again determines the asymptotic growth rate of the population, providing insights into the long-term dynamics of multitype populations.

A further extension of the multitype branching process with countably many types is the multitype branching process with a continuous type space. The mean offspring matrix ensemble for this extension is known as the mean offspring operator. Under certain regularity conditions, \cite{AN72,H63,A2000,LPP95} proved that the spectral radius of the mean offspring operator determines the exponential growth of this branching process.

\paragraph{Continuous-time branching process}

A continuous-time branching process (CTBP) is a stochastic model used to describe the evolution of populations where individuals reproduce and die according to random processes that occur continuously over time. Unlike discrete-time branching processes, where events happen at fixed intervals, in CTBPs, the time between events (such as births and deaths) is typically modelled by exponential distributions, allowing for greater flexibility in capturing real-world dynamics.

In a CTBP, each individual lives for a random amount of time before either reproducing or dying, and the offspring produced by an individual can also reproduce or die, creating a branching structure over time. The process starts with an initial population, and the size and composition of the population evolve as individuals are added or removed. The key parameters of a CTBP include the birth rate, death rate, and the distribution of the number of offspring produced.

One of the fundamental properties of CTBPs is the \emph{Malthusian parameter}, which determines the long-term growth rate of the population. This parameter is analogous to the exponential growth parameter in discrete-time branching processes. If this parameter is positive, the population tends to grow exponentially over time; if it is negative, the population is likely to die out. Significant contributions in this area have been made by Jagers and Nerman in \cite{JN84} and Nerman in \cite{N81}.

CTBPs are widely used in various fields, including biology, epidemiology, and network theory, to model phenomena where the timing of events is crucial. Since preferential attachment trees grow over time, \cite{AGS08} provided a continuous-time branching process ensemble for preferential attachment trees. This CTBP representation of preferential attachment trees has been utilised in analysing degree distributions and PageRank asymptotics in preferential attachment trees in \cite{BO22, BDO23, GHL20, RTV07}.

Branching processes and their multitype extensions are useful tools in various fields. In biology, they model the spread of genes or species; in epidemiology, they describe the transmission of infectious diseases; and in network theory, they help to understand the growth of complex networks. For applications of branching processes in biological sciences, see for example \cite{KA02,CMMP06}. The ability to capture both the probabilistic nature and the heterogeneity of populations makes branching processes an essential component of modern applied mathematics.


\subsection{Local limit}
Motivated by the growing importance of dynamic networks in real-world applications, the mathematics community has devoted significant attention to the study of graph limits over the past five years. Much of this work has focused on dense graphs, particularly through the lens of graph homomorphisms. A notable series of papers \cite{BCLSV06, BCLSV08, BCLSV12, LS06} has developed a comprehensive framework for graph limits defined via graph homomorphisms, demonstrating their equivalence to limits defined through other approaches \cite{BCLSV08, BCLSV12}. While the majority of these results pertain to dense graphs, the seminal paper \cite{BCLSV06} also introduces a notion of graph limits for sparse graphs with bounded degrees, again using graph homomorphisms. Furthering this line of inquiry, Borgs et al. \cite{BCKL13}, employing expansion methods from mathematical physics, established general results on graph limits for sparse graphs. Another significant contribution is found in \cite{BR11}, which addresses the limits of graphs that are neither dense nor sparse in the conventional sense, focusing on graphs whose average degrees tend to infinity.

Local convergence techniques have become a key methodology to study sparse
random graphs. These important properties include the thermodynamic quantities like pressure per particle, magnetization, internal energy etc.
However, convergence of many random graph properties does not directly
follow from local convergence. A notable, and important, such random graph property is the
size and uniqueness of the giant component.

\paragraph{Giant is almost local}
In \cite[Conjecture~1.2]{BNP11}, Schramm conjectured that the critical percolation threshold exhibits locality for a sequence of vertex-transitive infinite graphs converging locally to a limiting graph. Various attempts have been made to prove this conjecture for different subclasses of infinite graphs \cite{BNP11,S21,ABS22,KLS20,vdH23}. In particular, van der Hofstad in \cite{vdH23} provided a simple necessary and sufficient condition for the conjecture to hold, while \cite{BNP11,S21,KLS20} proved it for bounded-degree expanders. Recently, Easo and Hutchcroft proved Schramm's conjecture in \cite{EH23}. Furthermore, \cite{ABS22} established the conjecture for large-set expander graphs with bounded average degree and extended the result to include directed percolation as well. This result extends Schramm's conjecture beyond the scope of vertex-transitive graphs. All these results provide an estimate of the asymptotic size of the largest and second largest connected components in the percolated graph sequence. Since this asymptotic result is not a direct consequence of local convergence, the giant is said to be \emph{almost local}. van der Hofstad and Pandey have shown in \cite{HP24} that under certain connectivity condition, giants in random directed graphs are also \emph{almost local}.

\paragraph{Locality of Ising pressure per particle}

Unlike the critical percolation threshold or the largest connected component, Ising thermodynamic quantities are local properties of random graphs in quenched settings. Dembo and Montanari proved in \cite{DM10} that these properties are local for random graphs with finite variance degree sequences in quenched settings. Later, Dommers et al. extended these results in \cite{DGvdH10} to cases with infinite variance degree sequences. Finally, Dembo, Montanari, and Sun generalised this result to all random graphs under a uniform integrability condition on the degree sequence in \cite{DMS13}. They considered factor models, including Ising and Potts models as special cases, where this locality result holds true.

The annealed setting is often easier to handle and is therefore used as an approximation to the quenched setting. Examples include spin glasses \cite{CG13}, such as the Sherrington-Kirkpatrick model, where the annealed partition function becomes trivial due to the independence of the couplings, and random polymers, where the comparison between quenched and annealed settings is used to distinguish between relevant and irrelevant disorder \cite{CdHP12}. In the physics literature, it is generally believed that the annealed Ising model on a random graph is equivalent to an inhomogeneous mean-field model where the coupling between spins is proportional to their degrees; see Bianconi \cite{B02} and Dorogovtsev et al. \cite{dorogovtsev2008}. The annealed setting of the Ising model was first studied on random graphs in \cite{DGGvdHP16}. Leveraging the locally tree-like property, Can examined the annealed Ising model on random regular graphs in \cite{C17}, and later Can et al. \cite{CGGvdH22} studied the model on configuration models in both random and fixed degree settings. For fixed degree settings, they found the physics folklore to hold true, whereas for random degree settings of the configuration model, the conjecture did not hold.

\subsection{Tree-like structure of preferential attachment models}
Sparse Erd\H{o}s–R\'{e}nyi graphs, where the average degree remains constant as the number of vertices grows, the local limit is a Poisson branching process. This means that, locally, the graph around any randomly chosen vertex resembles a tree where the number of offspring (neighbours) of each node follows a Poisson distribution. Similarly, other classes of sparse random graphs, such as configuration models and generalised random graphs, also admit local limits that are branching processes. In these cases, the degree distribution of the vertices influences the branching process, leading to different types of local tree-like structures that capture the graph's local behaviour as the graph size tends to infinity.

\cite{BergerBorgs} identified the P\'{o}lya point tree as the local limit of preferential attachment models for fixed \( m \) and positive \( \delta \). \cite{tiffany2021} proved local convergence for preferential attachment trees with i.i.d.\ fitness parameters and also determined the rate of convergence. These results consistently identify the local limit as a multi-type branching process with a continuous type space, a topic that we explore in detail in Chapter~\ref{chap:LWC}. Additionally, \cite{DM13} identified a killed branching random walk as the local limit of the sub-linear Bernoulli preferential attachment model, as defined in \cite{vdH2}. Collectively, these findings suggest that, locally, almost all preferential attachment models have tree-like structures.

Continuous time branching processes (CTBPs) are known to embed preferential attachment trees, which are typically defined as a discrete sequence of random graphs, into continuous time \cite{A07, AGS08, B07, RTV07}. The heuristic idea is to construct a CTBP whose tree in continuous time coincides with a preferential attachment tree along the sequence of birth times. This approach offers the advantage of leveraging the extensive theory of CTBPs to investigate the properties of preferential attachment trees.

\subsection{Tools for preferential attachment models}

Contrary to other random graph models, preferential attachment models are dynamic graphs with a power-law degree distribution. Furthermore, the edge-connection events in preferential attachment models are highly dependent. While these properties make the model more representative of real-life networks, they also complicate the analysis of its local and global properties. To tackle problems related to preferential attachment models, two key tools are essential: the P\'{o}lya urn representation and the collapsing operator.

\paragraph{P\'{o}lya urn scheme}

In early twentieth century, P\'{o}lya proposed and analyzed the following model known as the P\'{o}lya urn model in \cite{EP1923}. This is a classic stochastic process that models the evolution of a system where the likelihood of certain outcomes increases as those outcomes occur more frequently. The process begins with an urn containing balls of different colours. At each step, a ball is drawn at random, its colour is noted, and then it is returned to the urn along with additional balls of the same colour. This reinforcement mechanism makes it more likely to draw balls of the same colour in future steps, leading to a form of positive feedback. Observe that, despite being correlated, the ball drawing events exchangeable. We can then use de Finnetti's lemma to show that this system is equivalent to the following simpler to analyse stochastic process:
\begin{itemize}
	\item[]For every colour \(i\), choose a parameter (which we call \emph{strength} or \emph{attractiveness}) \( p_i \), and at each step, independently of our decision in previous steps, put the new ball in urn \(i\) with probability \(p_i\). With suitably chosen prior distribution for \(p_i\), this procedure is equivalent to the P\'{o}lya urn process.
\end{itemize}

The P\'{o}lya urn process is a fundamental tool in probability theory, with applications ranging from the study of network growth models to understanding preferential attachment mechanisms. It captures the essence of "rich-get-richer" dynamics, where initial advantages can amplify over time, making it particularly useful in modelling systems with self-reinforcing behaviours.

The edge-connection process in preferential attachment models is closely related to the P\'{o}lya urn scheme in the following sense: every new connection that a vertex gains can be represented by a new ball added in the urn corresponding to that vertex, and every new vertex creates a new colour in the urn. Berger et al. used this idea to obtain the P\'{o}lya urn representation of the sequential models of the preferential attachment models in \cite{BBCS05}, where conditionally on a sequence of $\Beta$ random variables, the edge-connection events are independent.

\paragraph{Collapsing operator}\label{chap:intro:sec:collapsing-operator}
Collapsing is a valuable technique for studying affine preferential attachment models. This facilitates the transition from tree-based versions of these models to their non-tree counterparts.
Here we {generalize} the collapsing {procedure} discussed in \cite{vdH2}. Let us choose $\boldsymbol{r} = (r_1,r_2,\ldots)\in\N^\N$ and {let} $\mathcal{H}$ be the set of all finite vertex-labelled {graphs}. Then the collapsing operator $\mathcal{C}_{\boldsymbol{r}}$, {acting} on $\mathcal{H}$ {and collapsing groups of vertices of size $r_i$ into vertex $v_i$,} is defined as follows:
\begin{itemize}
	\item[{\btr}] \FC{l}et $G\in\mathcal{H}$ be a vertex-labelled graph of size $n$ and $n\in \left(r_{[k-1]}, r_{[k]}\right]$ for some $k\in\N$, where 
	\[
	r_{[l]}=\sum\limits_{i\leq l} r_i~;
	\]
	\item[{\btr}] \FC{g}roup the vertices $\left\{ r_{[i-1]}+1,\ldots, r_{[i]} \right\}$ and name the groups $v_i$ for all $i< k$, while the group $v_k$ contains the vertices $\left\{r_{[k-1]}+1, \ldots , n\right\}$;
	\item[{\btr}] \FC{c}ollapse the vertices of each of these $v_k$ groups {into one vertex,} and name the new vertex $k$;
	\item[{\btr}] \FC{e}dges originating and ending in same group form self-loops in the new graph;
	\item[{\btr}] \FC{e}dges between two different groups form edges between the respective vertices in the new graph.
\end{itemize}
Note that if we fix $r_1=r_2=1$ and $r_i=m$ for all $i\geq 3$ and suitable $G_0$, then we get back the collapsing procedure used first in \cite{BolRio04b}, and further discussed in \cite{vdH2}.
We now discuss the construction of models (A) and (B) through this collapsing operator, {conditionally on the i.i.d.\ out-degrees described by $\boldsymbol{m}$.}

\begin{figure}[h!]
	\centering	
	\begin{tikzpicture}
		\node[circle, draw, fill=blue!50] (1) at (0,0) {1};
		\node[circle, draw, fill=blue!50] (2) at (-2,-1.5) {2};
		\node[circle, draw, fill=blue!50] (3) at (0,-1.5) {3};
		\node[circle, draw, fill=purple!50] (4) at (0,-3) {4};
		\node[circle, draw, fill=purple!50] (5) at (2,-1.5) {5};
		\node[circle, draw, fill=purple!50] (6) at (-3.5,-3) {6};
		\node[circle, draw, fill=yellow!50] (7) at (-1,-4.5) {7};
		\node[circle, draw, fill=yellow!50] (8) at (1.5,-3.5) {8};
		\node[circle, draw, fill=yellow!50] (9) at (-2.5,-4.5) {9};
		\node[circle, draw, fill=green!50] (10) at (2.5,-4) {10};
		\node[circle, draw, fill=green!50] (11) at (3.5,-4.5) {11};
		\node[circle, draw, fill=green!50] (12) at (0,-5) {12};
		
		\draw[->] (2) -- (1);
		\draw[->] (3) -- (1);
		\draw[->] (4) -- (3);
		\draw[->] (5) -- (1);
		\draw[->] (6) -- (2);
		\draw[->] (7) -- (4);
		\draw[->] (8) -- (5);
		\draw[->] (9) -- (2);
		\draw[->] (10) -- (5);
		\draw[->] (11) -- (5);
		\draw[->] (12) -- (4);
		
		\node[below=of 7.south, align=center] { (a) Pre-collapsed tree};
		
		\node[circle, draw, fill=blue!50] (C1) at (6.25,-0.5) {1};
		\node[circle, draw, fill=purple!50] (C2) at (8.25,-2.5) {2};
		\node[circle, draw, fill=yellow!50] (C3) at (6.25,-4.5) {3};
		\node[circle, draw, fill=green!50] (C4) at (4.25,-2.5) {4};
		
		\draw[->, loop above, looseness=10] (C1) to (C1); 
		\draw[->, loop left, looseness=14] (C1) to (C1); 

		\draw[->, bend left=20] (C2) to (C1); 
		\draw[->, bend right=20] (C2) to (C1); 
		\draw[->] (C2) to (C1); 
		
		\draw[->] (C3) to (C2); 
		\draw[->, bend right=20] (C3) to (C2); 
		
		\draw[->] (C3) to (C1); 
		
		\draw[->, bend left=20] (C4) to (C2); 
		\draw[->] (C4) to (C2); 
		\draw[->, bend right=20] (C4) to (C2); 
		
		\node[below=of C3.south, align=center] { (b) Collapsed branching process};
		
		\end{tikzpicture}
\label{fig:collapsing}
\caption{Collapsing process with \( \bfr=(3,3,3,3,\ldots)\)}
\end{figure}
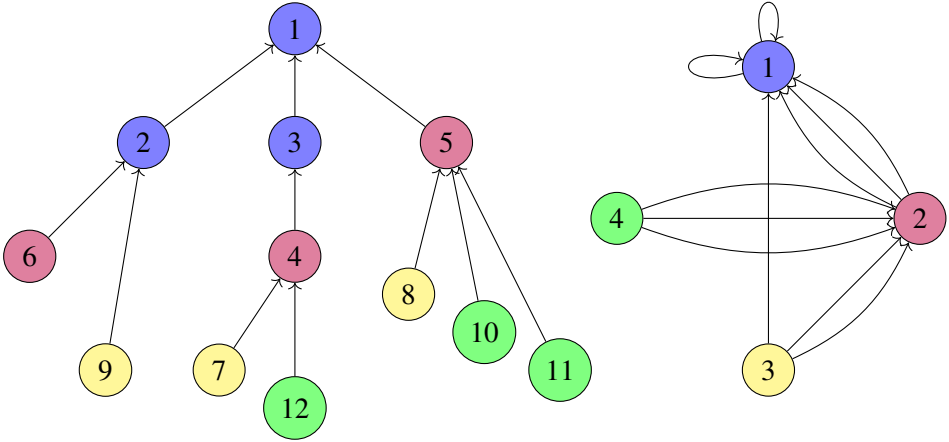

\paragraph{\textbf{PAM construction by collapsing}}

{We start with a vertex-labelled graph $G_0$ of size $2$ and degrees $a_1$ and $a_2$, respectively. First, we explain the construction of model (A) using collapsing.
	
	Every $v\geq 3$ comes with exactly one edge. Given $\boldsymbol{m}=(1,1,m_3,\ldots)$, the incoming vertex $v=m_{[i-1]}+j$ for some $i\in[3,n)$ and $j\leq m_i$, connects to $u\in [v]$ with probability
	\eqn{\label{eq:edge_connecting:model(rs):1}
		\prob_m\left( v\rightsquigarrow u\mid \PArs_{v-1}(\boldsymbol{m},1,\delta) \right) = \begin{cases}
			\frac{d_u(v-1)+\delta(u)}{c_{v,j}}&\mbox{ when $v>u$,}\\
			\frac{1+\delta(u)}{c_{v,j}}&\mbox{ when $v=u$}.
	\end{cases}}
	Here $\delta(u) = {\delta/m_k}$ when $u\in\left( m_{[k-1]},m_{[k]} \right],$
	\[
	c_{i,j} = a_{[2]}+2\left( m_{[i-1]}+j-2 \right) - 1+ (i-1)\delta + \frac{j}{m_i}\delta~\FC{,}
	\]
	\FC{and }$\PArs_{v-1}(\boldsymbol{m},1,\delta)$ denotes the graph formed after the $(v-1)$-st vertex is added, with $d_u(v-1)$ denoting the degree of the vertex $u$ in $\PArs_{v-1}(\boldsymbol{m},1,\delta)$. Continue the process until the $m_{[n]}$-th vertex is added. 
	We obtain $\PArs_n(\boldsymbol{m},\delta)$ by {applying} $\mathcal{C}_{\boldsymbol{m}}$ on $\PArs_{m_{[n]}}(\boldsymbol{m},1,\delta)$.
	Therefore, for model (A), the conditional edge-connection probabilities are given by \eqref{eq:edge_connecting:model(rs):2}, {as required.}
	
	{To construct $\PArt_n(\boldsymbol{m},\delta)$ by a collapsing procedure}, we do not allow {for} self-loops for $m=1$ as we did in \eqref{eq:edge_connecting:model(rs):1}.
	The rest of the process remains the same. Starting from the same initial graph $G_0$, conditionally on $\boldsymbol{m}$, every $v\geq 3$ comes with exactly one edge. The incoming vertex $v=m_{[i-1]}+j$ for some $i\in[3,n]$ and $j\leq m_i$, connects to $u\in [v-1]$ {with the same probability as in \eqref{eq:edge_connecting:model(rs):1}. Since here we do not allow for any self-loop, the normalising constant in the denominator now is} 
	\[
	c_{i,j} = a_{[2]}+2\left( m_{[i-1]}+j-3 \right) +(i-1)\delta + \frac{(j-1)}{m_i}\delta~\FC{,}
	\]
	\FC{and} $\PArt_{v}(\boldsymbol{m},1,\delta)$ denotes the graph formed after the $v$-th vertex is added, with $d_u(v)$ denoting the degree of the vertex $u$ in $\PArt_{v}(\boldsymbol{m},1,\delta)$. Continue the process until the $m_{[n]}$-th vertex is added. We obtain $\PArt_n(\boldsymbol{m},\delta)$ by applying $\mathcal{C}_{\boldsymbol{m}}$ on $\PArt_{m_{[n]}}(\boldsymbol{m},1,\delta)$.
	Therefore, for model (B), the edge-connection probabilities are given by \eqref{eq:edge_connecting:model(rt):2}. 
	\begin{remark}[Initial graph is preserved in collapsing]
		{\rm Observe that we always choose $r_1=r_2=1$ while performing the collapsing operator on the pre-collapsed preferential attachment graphs. This is done intentionally to preserve the structure of the initial graph $G_0$. If we start with an initial graph of size $\ell\geq 1$ then we can choose $r_1=r_2=\ldots=r_\ell=1$ to preserve the initial graph structure in both collapsed and pre-collapsed preferential attachment graphs.}\hfill$\blacksquare$
	\end{remark}
}
\medskip
\begin{center}
	{\textbf{Structure of the thesis}}
\end{center}

The structure of this thesis is organised into two main parts, each addressing a distinct aspect of preferential attachment models. 

\textbf{Part I} focuses on the local limit of generalised preferential attachment models and consists of three chapters. In Chapter~\ref{chap:LWC}, we introduce the concept of local limit in detail and describe the random P\'{o}lya point tree. We also demonstrate that the P\'{o}lya point tree described in \cite{BergerBorgs} is a special case of our more general random P\'{o}lya point tree. Chapter~\ref{chap:PU_representation} presents the P\'{o}lya urn representation of sequential preferential attachment models and proves that this representation is distributionally equivalent to the model. Chapter~\ref{chap:convergence} concludes Part I by proving the local convergence of generalised preferential attachment models to the random P\'{o}lya point tree. As a consequence of this local limit, we compute the degree distribution of a uniformly chosen vertex and its neighbours.

\textbf{Part II} investigates stochastic processes on preferential attachment models, with a particular emphasis on percolation and the Ising model. Chapter~\ref{chap:percolation_threshold_PPT} begins by computing the explicit critical percolation threshold for P\'{o}lya point trees and proves that the inverse of the spectral radius of the mean offspring operator of P\'{o}lya point tree is the critical percolation threshold in the positive $\delta$ regime. During this proof, we establish the continuity of the survival probability of the P\'{o}lya point tree with respect to the percolation probability. In Chapter~\ref{chap:local-giant}, we extend this result to preferential attachment models, showing that they share the same critical percolation threshold as the P\'{o}lya point tree, leveraging their properties as large-set expanders with bounded average degree and the continuity of the survival probability of their local limit. Chapter~\ref{chap:subcritical} studies the size of the largest connected component in sub-critical percolated preferential attachment models, demonstrating that it is significantly larger than the maximum degree in the graph. This chapter also includes simulation results to support our lower bound argument. Finally, Chapter~\ref{chap:ising} examines the quenched Ising model on preferential attachment models, where we derive explicit thermodynamic limits for pressure per particle, magnetisation per vertex, and internal energy. We also identify the explicit inverse critical temperature for the quenched Ising model on these graphs.


\EndThumbs
\cleardoublepage
\part{The Local limit}
\StartThumbs

\chapter{Local Convergence}
\label{chap:LWC}
\begin{flushright}
	\footnotesize{}Based on:\\
	\textbf{Universality of the local limit of\\
		preferential attachment models}\cite{RRR22}\\
\end{flushright}
\vspace{0.1cm}
\begin{center}
	\begin{minipage}{0.7 \textwidth}
		\footnotesize{\textbf{Abstract.}
		This chapter describes the local limit definitions for rooted graphs. Here we describe both the marked and un-marked local convergence. Later, we define the \emph{random P\'{o}lya point tree}, which we prove to be the local limit of preferential attachment models in general. We also describe the relation between local limit described in \cite{BergerBorgs} and random P\'{o}lya point tree.}
	\end{minipage}
\end{center}
\vspace{0.1cm}

\section{Introduction}

Local convergence of rooted graphs was introduced by Benjamini and Schramm in \cite{Benjamini_LL} and by Aldous and Steele in \cite{Aldous2004}. A large class of random graphs are referred to as \emph{locally tree-like graphs}, meaning that the neighbourhood of the majority of vertices is structured like a tree (up to a certain distance). This concept can be formalised using the notion of local convergence (i.e., the Benjamini-Schramm limit), introduced in \cite{Aldous2004,Benjamini_LL}. Local convergence has proven to be an extremely versatile tool for understanding the geometry of graphs. For example, the asymptotic number of spanning trees, the partition function of the Ising model, and the spectral distribution of the adjacency matrix of the graph can all be computed in terms of the local limit. The basic idea is to explore the neighbourhood of a uniformly chosen vertex of the graph up to a finite distance to understand its distributional and geometric properties.

Marked local convergence of random graphs is relatively new to the literature. \cite{GHL20} introduced this notion to analyse the local convergence of PageRanks in a directed setting. For an overview of the theory and applications of this concept, we refer to, e.g., \cite[{Chapter 2}]{vdH2}.

\newpage
\begin{center}
	{\textbf{Organisation of the chapter}}
\end{center}
This chapter is organised as follows. In Section~\ref{chap:LWC:sec:local-convergence:def}, we define local convergence in detail. In Section~\ref{chap:LWC:sec:RPPT}, we provide a detailed description of the random P'{o}lya point tree. Lastly, in Section~\ref{chap:LWC:sec:relation-literature}, we establish the connection between the P'{o}lya point tree defined in Section~\ref{chap:LWC:sec:RPPT} and the one defined in \cite{BergerBorgs}.
\section{Local convergence}\label{chap:LWC:sec:local-convergence:def}

A graph $G=(V(G), E(G))$ ({possibly} infinite) is called \textit{locally finite} if every vertex $v\in V(G)$ has finite degree (not necessarily uniformly bounded). A pair $(G,o)$ is called a {{\em rooted graph}, where $G$ is} rooted at $o\in V(G)$. For any two vertices $u,v\in V(G),~d_G(u,v)$ is defined as the length of the shortest path from $u$ to $v$ in $G$, {i.e., the minimal number of edges needed to connect $u$ and $v$}. 

For a rooted graph $(G,o)$, the {\em $r$-neighbourhood} of the root $o$, {denoted by $B_r^{\sss(G)}(o)$, is defined as the graph rooted at $o$ with vertex and edge sets given by}
\eqn{\label{eq:def:nbhd:root}
	\begin{split}
		V\big( B_r^{\sss(G)}(o) \big) =& \{ v\in V(G)\colon d_G(o,v)\leq r \},\quad\text{and}\\
		E\big( B_r^{\sss(G)}(o) \big) =& \left\{ \{ u,v \}\in E(G)\colon u,v\in V\big( B_r^{\sss(G)}(o) \big) \right\}\,.
\end{split}}

Let $(G_1,o_1)$ and $(G_2,o_2)$ be two rooted graphs with $G_1=(V(G_1),E(G_1))$ and $G_2=(V(G_2),E(G_2))$, which are locally finite. Then we say that $(G_1,o_1)$ is {\em{rooted isomorphic}} to $(G_2,o_2)$, {which we denote as $(G_1,o_1)\simeq (G_2,o_2)$}, {when} there exists a graph isomorphism between $G_1$ and $G_2$ that maps $o_1$ to $o_2$, i.e., {when} there exists a bijection $\phi\colon V(G_1)\mapsto V(G_2)$ such that 
\eqn{\label{def:rooted:isomorphism}
	\begin{split}
		&\phi(o_1)=o_2,\quad\text{and} \quad\{u,v\}\in E(G_1) \iff \{\phi(u),\phi(v)\}\in E(G_2).
\end{split}} 

{\em Marks} are generally images of an injective function $\mathcal{M}$ acting on the vertices of a graph $G,$ as well as the edges. 
\sloppy{These} marks take values in a complete separable metric space $(\Xi,d_{\Xi})$. Rooted graphs having only vertex marks, are called \textit{vertex-marked} rooted graphs, {and are denoted} by $(G,o,\mathcal{M}(G))= (V(G),E(G),o,\mathcal{M}(V(G)))$. We say that two graphs $(G_1,o_1,\mathcal{M}_1(G_1))$ and $(G_2,o_2,\mathcal{M}_2(G_2))$ are vertex-marked-isomorphic, {which we denote as $(G_1,o_1,\mathcal{M}_1(G_1))\overset{\star}{\simeq}(G_2,o_2,\mathcal{M}_2(G_2))$,} {when} there exists a bijective function $\phi\colon V(G_1)\mapsto V(G_2)$ such that
\begin{itemize}
	\item[$\rhd$] $\phi(o_1)=o_2$;
	\item[$\rhd$] $(v_1,v_2)\in E(G_1)$ implies that $(\phi(v_1),\phi(v_2))\in E(G_2)$;
	\item[$\rhd$] $\mathcal{M}_1(v)=\mathcal{M}_2(\phi(v))$ for all $v\in V(G_1)$.
\end{itemize}

We {let} $\mathcal{G}_\star$ be the vertex-marked isomorphism invariant class of rooted graphs. Similarly as {\cite[Definition 2.11]{vdH2}} that one can define a metric $d_{\mathcal{G}_\star}$ as
\eqn{\label{for:def:distance:1}
	d_{\mathcal{G}_\star}\left( (G_1,o_1,\mathcal{M}_1(G_1)),(G_2,o_2,\mathcal{M}_2(G_2)) \right) = \frac{1}{1+R^\star}\, ,}
where
\eqan{\label{for:def:distance:2}
	R^\star = \sup\{ r\geq 0: B_r^{(G_1)}(o_1)\simeq B_r^{(G_2)}(o_2),\mbox{and there exists a bijective} &\nn\\
	\mbox{function $\phi$ from $V(G_1)$ to $V(G_2)$ such that }&\\
	d_{\Xi}(\mathcal{M}_1(u),\mathcal{M}_2(\phi(u)))\leq \frac{1}{r},
	\forall u\in V(B_r^{(G_1)}(o_1))&\},\nn}
which makes $\left(\mathcal{G}_\star,d_{\mathcal{G}_\star}\right)$ a Polish space.

We {next} define the notion of marked local convergence of vertex-marked random graphs on this space. {\cite[Theorem~2.14]{vdH2} describes various} notions of local convergence.
For the next two definitions of local convergence, we consider that $\{ G_n \}_{n\geq 1}$ is a sequence of (possibly random) vertex-marked graphs with $G_n=\left( V(G_n),E(G_n),\mathcal{M}_n(V(G_n)) \right)$ that are finite (almost surely). Conditionally on $G_n,$ let $o_{n}$ be chosen uniformly at random from $V(G_n)$. Note that $\{ (G_n,o_n) \}_{n\geq 1}$ is a sequence of random variables defined on $\mathcal{G}_\star$. Then {vertex-marked local convergence is defined as follows:}
\begin{Definition}[Vertex marked local  convergence]\label{def:vertex:marked:local:convergence}
	\smallskip\noindent
	{ \begin{itemize}
			\item[(a)] {\bf (Vertex-marked local weak convergence)} The sequence of (possibly random) graphs, $\{(G_n,o_n)\}_{n\geq 1}$ is said to converge {{{vertex-marked locally weakly}}} to a random element $(G,o,\mathcal{M}(V(G)))\in \mathcal{G}_\star$ having probability law $\mu_\star$, if, for every $r>0$, and every $(H_\star,\mathcal{M}_{H_\star}(H_\star))\in\mathcal{G}_\star$, as $n\to\infty$,
			\eqan{\label{eq:def:mark:local:weak:convergence}
				&\prob\left( d_{\mathcal{G}_\star}\big( (B_r^{\sss(G_n)}({o_n}),o_n,\mathcal{M}(V(B_r^{\sss(G_n)}({o_n}))),(H_\star,\mathcal{M}_{H_\star}(H_\star)) \big)\leq \frac{1}{1+r} \right)\nn\\
				&\hspace{0.35cm}\to \mu_\star\left( d_{\mathcal{G}_\star}\big( (B_r^{\sss(G)}({o}),o,\mathcal{M}(V(B_r^{\sss(G)}({o}))),(H_\star,\mathcal{M}_{H_\star}(H_\star)) \big)\leq \frac{1}{1+r} \right).}
			\item[(b)] {\bf (Vertex-marked local convergence)}
			The sequence of (possibly random) graphs $\{ G_n \}_{n\geq 1}$ is said to converge {{{vertex-marked locally}}} in probability to a random element $(G,o,\mathcal{M}(V(G)))\in \mathcal{G}_\star$ having probability law $\mu_\star$, when for every $r>0$, and for every $(H_\star,\mathcal{M}_{H_\star}(H_\star))\in\mathcal{G}_\star$, as $n\to\infty$,
			\eqan{\label{eq:def:mark:local:convergence}
				&\frac{1}{n}\sum\limits_{\omega\in[n]}\one_{\left\{ d_{\mathcal{G}_\star}\big( (B_r^{\sss(G_n)}({\omega}),\omega,\mathcal{M}(V(B_r^{\sss(G_n)}({\omega}))),(H_\star,\mathcal{M}_{H_\star}(H_\star)) \big)\leq \frac{1}{1+r} \right\}}\nn\\
				&\hspace{0.1cm}\overset{\prob}{\to} \mu_\star\left( d_{\mathcal{G}_\star}\big( (B_r^{\sss(G)}({o}),o,\mathcal{M}(V(B_r^{\sss(G)}({o}))),(H_\star,\mathcal{M}_{H_\star}(H_\star)) \big)\leq \frac{1}{1+r} \right).}
	\end{itemize}}
	\hspace{-0.5cm}We require the convergence in \eqref{eq:def:mark:local:weak:convergence} and \eqref{eq:def:mark:local:convergence} to hold true for continuity points of $\mu_\star$.
	\hfill$\blacksquare$
\end{Definition}
\section{Definition of random P\'{o}lya point tree}\label{chap:LWC:sec:RPPT}

In this section, {we define the random P\'olya point tree (RPPT) that will act as the vertex-marked local limit of our preferential attachment graphs. We start by defining the vertex set of this RPPT:}

\begin{Definition}[Ulam-Harris set and its ordering]\label{def:ulam}
	{\rm Let $\N_0=\N\cup\{0\}$. {The {\em Ulam-Harris set}} is
		\begin{equation*}
			\mathcal{U}=\bigcup\limits_{n\in\N_0}\N^n.
		\end{equation*}
		For $x = x_1\cdots x_n\in\N^n$ and $k\in\N$ we denote {the element $x_1\cdots x_nk$ by $xk\in\N^{n+1}$.} The~{\em root} of the Ulam-Harris set is denoted by $\emp\in\N^0$.
		
		For any $x\in \mathcal{U}$, we say that $x$ has length $n$ if $x\in\N^n$. {The lexicographic \textit{ordering}} between the elements of the Ulam-Harris set {is as follows:}
		\begin{itemize}
			\item[(a)] for any two elements $x,y\in\mathcal{U}$, {$x>y$ when the length of $x$ is more than that of $y$;}
			\item[(b)] if $x,y\in \N^n$ for some $n$, then $x>y$ if there exists $i\leq n,$ such that $x_j=~y_j$, for all $j<i$ and $x_i>~y_i$.\hfill$\blacksquare$
	\end{itemize}}
\end{Definition}

We use the elements of the Ulam-Harris set to identify nodes in a rooted tree, since the notation in Definition~\ref{def:ulam} allows us to denote the relationships between children and parents, where for $x\in\mathcal{U}$, we denote the $k$-th child of $x$ by element $xk$.
\smallskip

\paragraph{\bf Random P\'olya point tree (RPPT)}
The $\RPPT(M,\delta)$ is an {\em infinite multi-type rooted random tree} where $M$ is an $\N$-valued random variable and $\delta>-\infsupp(M)$. {It is a multi-type branching process, with a mixed continuous and discrete type space. We now describe its properties one by one.}

\paragraph{ \bf Descriptions of the distributions and parameters used}
\begin{enumerate}
	\item[{\btr}] {Define} $\chi = \frac{\E[M]+\delta}{2\E[M]+\delta}$.
	\smallskip
	\item[{\btr}] {Let} $\Gamma_{in}(m)$ denote a Gamma distribution with parameters $m+\delta$ and $1$, where $m\in \N$. 
	\smallskip
	\item[{\btr}] {For $\delta>-\infsupp(M),$ let}  $M^{(\delta)}$ be an $\N$-valued random variable such that $$\prob\left(M^{(\delta)}=m\right) = \frac{m+\delta}{\E[M]+\delta}\prob(M=m),$$ i.e., $M^{(\delta)}+\delta$ is a size-biased version of $M+\delta$.
	\smallskip
	\item[{\btr}] {In particular,} $M^{(0)}$ is the size-biased distribution of $M$.
\end{enumerate}

\paragraph{\bf Feature of the vertices of RPPT}
Below, to avoid confusion, we use `node' for a vertex in the RPPT and `vertex' for a vertex in the PAM. We now discuss the properties of the nodes in $\RPPT(M,\delta)$. Every node except the root in $\RPPT$ has
\begin{enumerate}
	\item[{\btr}] a {\em label} $\omega$ in the Ulam-Harris set $\mathcal{U}$ (recall Definition~\ref{def:ulam});
	\smallskip
	\item[{\btr}] an {\em age} $A_{\omega}\in[0,1]$;
	\smallskip
	\item[{\btr}] a positive number $\Gamma_{\omega}$ called its {\em strength};
	\smallskip
	\item[{\btr}] a {label in $\{{\Old},{\Young}\}$} depending on the age of the {node} and its parent in the tree, {with $\Young$ denoting that the node is younger than its parent, and $\Old$ denoting that the node is older than its parent.}
\end{enumerate}

Note that \emph{age} stands for the birth-time of the given node. Hence, $\omega$ is labeled $\Young$, when $A_{\omega}$ is larger than its parent's age in the tree.
Based on its label being $\Old$ or $\Young$, every {node} $\omega$ has an independent $\N$-valued random variable $m_{-}(\omega)$ associated to it. If $\omega$ has label $\Old$, then
\begin{enumerate}
	\item[{\btr}] $m_{-}(\omega)$ is distributed as $M^{(\delta)}$;
	\smallskip
	\item[{\btr}] given $m_{-}(\omega), \Gamma_{\omega}$ is distributed as $\Gamma_{in}\left( m_{-}(\omega)+1 \right)$.
\end{enumerate}
On the other hand, if $\omega$ has label $\Young$, then
\begin{enumerate}
	\item[{\btr}] $m_{-}(\omega)$ is distributed as $M^{(0)}-1$;
	\smallskip
	\item[{\btr}] given $m_{-}(\omega), \Gamma_{\omega}$ is distributed as $\Gamma_{in}\left( m_{-}(\omega)+1 \right)$.
\end{enumerate}
The root in $\RPPT$ has no label, and its strength $\Gamma_\emp$ is distributed as $\Gamma_{in}(M)$.
\begingroup
\allowdisplaybreaks
\paragraph{\bf Construction of the RPPT} {We next use the above definitions to construct the RPPT using an {\em exploration process}.}
The root $\emp$ is special in the tree. Its age $A_\emp$ is an independent uniform random variable in $[0,1]$. Since $\emp$ has no label, $m_{-}(\emp)$ is distributed as $M$, and $\Gamma_\emp$ is distributed as $\Gamma_{in}(m_{-}(\emp))$.
Then the {children of the root in the} random P\'{o}lya point tree {are} constructed as follows:
\endgroup	
\begin{enumerate}
	\item Sample $U_1,\ldots,U_{m_{-}(\emp)}$ uniform random variables on $[0,1]$, independent of the rest;  
	\item To nodes $\emp1,\ldots, \emp m_{-}(\emp)$ assign ages $U_1^{1/\chi} A_\emp,\ldots,U_{m_{-}(\emp)}^{1/\chi} A_\emp$ and label $\Old$;
	\item \sloppy Assign ages $A_{\emp(m_{-}(\emp)+1)},\ldots, A_{\emp(m_{-}(\emp)+\din_\emp)}$ to the nodes $\emp(m_{-}(\emp)+1),\ldots, \emp(m_{-}(\emp)+\din_\emp)$. These ages are the \FC{occurrence times} given by a conditionally independent Poisson point process on $[A_\emp,1]$ defined by the intensity
	\eqn{\label{def:rho_emp}
		\rho_{\emp}(x) = {(1-\chi)}{\Gamma_\emp}\frac{x^{-\chi}}{A_\emp^{1-\chi}},}
	where $\din_\emp$ is the total number of points of this process. Assign label $\Young$ to them;
	\item Draw an edge between $\emp$ and each of $\emp 1,\ldots, \emp(m_{-}(\emp)+\din_\emp)$; 
	\item Label $\emp$ as explored and the nodes $\emp1,\ldots, \emp(m_{-}(\emp)+\din_\emp)$ as unexplored.
\end{enumerate}	

Then, recursively over the elements in the set of unexplored nodes we perform the following breadth-first exploration:
\begin{enumerate}
	\item Let $\omega$ denote the smallest currently unexplored node;
	\item Sample $m_{-}(\omega)$ i.i.d.\ random variables $U_{\omega1},\ldots,U_{\omega m_{-}(\omega)}$ independently from all the previous steps and from each other, uniformly on $[0,1]$. To nodes $\omega 1,\ldots,\omega m_{-}(\omega)$ assign the ages $U_{\omega1}^{1/\chi}A_\omega,\ldots,U_{\omega m_{-}(\omega)}^{\FC{1/\chi}}A_\omega$ and label $\Old$ and set them unexplored;
	\item Let $A_{\omega(m_{-}(\omega)+1)},\ldots,A_{\omega(m_{-}(\omega)+\din_\omega)}$ be the random $\din_\omega$ points given by a conditionally independent Poisson process on $[A_{\omega},1]$ with intensity
	\eqn{
		\label{for:pointgraph:poisson}
		\rho_{\omega}(x) = {(1-\chi)}{\Gamma_\omega}\frac{x^{-\chi}}{A_\omega^{1-\chi}}.
	}
	Assign these ages to $\omega(m_{-}(\omega)+1),\ldots,\omega(m_{-}(\omega)+\din_\omega)$, {assign them label $\Young$ } and set them unexplored;
	\item Draw an edge between $\omega$ and each one of the nodes $\omega 1,\ldots,\omega(m_{-}(\omega)+\din_\omega)$;
	\item Set $\omega$ as explored.
\end{enumerate}

We call the resulting tree the {\em random P\'olya point tree with parameters $M$ and $\delta$,} and denote it by $\RPPT(M,\delta)$. Occasionally we drop $M$ and $\delta$ while mentioning $\RPPT(M,\delta)$. When $M$ is degenerate {and equal to $m$ a.s.,} we {call this the P\'olya point tree with parameters $m$ and  $\delta$.}

\section{P\'{o}lya point tree: A special case}\label{chap:LWC:sec:PPT}

The P\'{o}lya point tree arises as the local limit of a sequence of preferential attachment models in \cite{BergerBorgs}. In this section, we define the P\'{o}lya point tree as a special case of the random P\'{o}lya point tree, defined in Section~\ref{chap:LWC:sec:RPPT}. To obtain the P\'{o}lya point tree from the random P\'{o}lya point tree, we consider $M$ to have a unit mass at some $m \in \mathbb{N}$ and fix $\delta > -m$. We begin by defining the distributions and parameters used to define the P\'{o}lya point tree and then proceed to define the multi-type branching process.

\paragraph{\bf Descriptions of the distributions and parameters used}
\begin{enumerate}
	\item[{\btr}] Define $\chi = \frac{m+\delta}{2m+\delta}$.
	\smallskip
	\item[{\btr}] Let $\Gamma_{\sf{in}}(m)$ denote a Gamma distribution with parameters $m+\delta$ and $1$.
	\smallskip
	\item[{\btr}] Let $\Gamma_{\sf{in}}^\prime(m)$ denote the size-biased distribution of $\Gamma_{\sf{in}}(m)$, which is also a Gamma distribution with parameters $m+\delta+1$ and $1$.
\end{enumerate}
\smallskip

\paragraph{\bf Features of the nodes in $\PPT$}
To avoid confusion, we use the term `node' for a vertex in the PPT and `vertex' for a vertex in the PAM. We now discuss the properties of the nodes in $\PPT(m,\delta)$. Every node, except the root in the $\PPT$, has:
\begin{enumerate}
	\item[{\btr}] a {\em label} $\omega$ in the Ulam-Harris set $\mathcal{N}$ (recall Definition~\ref{def:ulam});
	\smallskip
	\item[{\btr}] an {\em age} $A_{\omega} \in [0,1]$;
	\smallskip
	\item[{\btr}] a positive number $\Gamma_{\omega}$ called its {\em strength};
	\smallskip
	\item[{\btr}] a label in $\{{\Old},{\Young}\}$ depending on the age of the node and its parent, with $\Young$ denoting that the node is younger than its parent and $\Old$ denoting that the node is older than its parent.
\end{enumerate}
Based on its type, either $\Old$ or $\Young$, every node $\omega$ has a number $m_{-}(\omega)$ associated with it. If $\omega$ is of type $\Old$, then $m_{-}(\omega) = m$, and $\Gamma_{\omega}$ is distributed as $\Gamma_{\sf{in}}^\prime(m)$. If $\omega$ is of type $\Young$, then $m_{-}(\omega) = m - 1$, and given $m_{-}(\omega)$, $\Gamma_{\omega}$ is distributed as $\Gamma_{\sf{in}}(m)$.
\smallskip

\begingroup
\allowdisplaybreaks
\paragraph{\bf Construction of the PPT} Next, we use the above definitions to construct the PPT using an {\em exploration process}. The root is special in the tree. It has label $\emp$ and its age $A_\emp$ is an independent uniform random variable in $[0,1]$. The root $\emp$ has no label in $\{\Old,\Young\}$, and we set $m_{-}(\emp) = m$. The children of the root in the P\'{o}lya point tree are then constructed as follows:
\endgroup
\begin{enumerate}
	\item Sample $U_1, \ldots, U_{m_{-}(\emp)}$ uniform random variables on $[0,1]$, independent of the rest;  
	\item Assign to nodes $\emp1, \ldots, \emp m_{-}(\emp)$ the ages $U_1^{1/\chi} A_\emp, \ldots, U_{m_{-}(\emp)}^{1/\chi} A_\emp$ and type $\Old$;
	\item Assign ages $A_{\emp(m_{-}(\emp)+1)}, \ldots, A_{\emp(m_{-}(\emp)+\din_\emp)}$ to nodes $\emp(m_{-}(\emp)+1), \ldots, \emp(m_{-}(\emp)+\din_\emp)$. These ages are the occurrence times given by a conditionally independent Poisson point process on $[A_\emp,1]$ defined by the random intensity
	\eqn{\label{def:rho_emp:PPT}
		\rho_{\emp}(x) = {(1-\chi)}{\Gamma_\emp}\frac{x^{-\chi}}{A_\emp^{1-\chi}},}
	with $\din_\emp$ being the total number of points of this process. Assign them the type $\Young$;
	\item Draw an edge between $\emp$ and each of $\emp1, \ldots, \emp(m_{-}(\emp)+\din_\emp)$; 
	\item Label $\emp$ as explored and nodes $\emp1, \ldots, \emp(m_{-}(\emp)+\din_\emp)$ as unexplored.
\end{enumerate}	
Then, recursively over the elements in the set of unexplored nodes, we perform the following breadth-first exploration:
\begin{enumerate}
	\item Let $\omega$ denote the smallest currently unexplored node in the Ulam-Harris ordering;
	\item Sample $m_{-}(\omega)$ i.i.d.\ random variables $U_{\omega1}, \ldots, U_{\omega m_{-}(\omega)}$ independently from all previous steps and from each other, uniformly on $[0,1]$. Assign to nodes $\omega 1, \ldots, \omega m_{-}(\omega)$ the ages $U_{\omega1}^{1/\chi}A_\omega, \ldots, U_{\omega m_{-}(\omega)}^{1/\chi}A_\omega$ and type $\Old$, and set them as unexplored;
	\item Let $A_{\omega(m_{-}(\omega)+1)}, \ldots, A_{\omega(m_{-}(\omega)+\din_\omega)}$ be the random $\din_\omega$ points given by a conditionally independent Poisson process on $[A_{\omega},1]$ with random intensity
	\eqn{
		\label{for:pointgraph:poisson:PPT}
		\rho_{\omega}(x) = {(1-\chi)}{\Gamma_\omega}\frac{x^{-\chi}}{A_\omega^{1-\chi}}.
	}
	Assign these ages to $\omega(m_{-}(\omega)+1), \ldots, \omega(m_{-}(\omega)+\din_\omega)$. Assign them the type $\Young$ and set them as unexplored;
	\item Draw an edge between $\omega$ and each of the nodes $\omega 1, \ldots, \omega(m_{-}(\omega)+\din_\omega)$;
	\item Set $\omega$ as explored.
\end{enumerate}
We call the resulting tree the {\em P\'olya point tree with parameters $m$ and $\delta$}, and denote it by $\PPT(m,\delta)$. The vertex marks of $\PPT$ here differ from those in \cite{BergerBorgs}, but the vertex marks in the definition from \cite{BergerBorgs} can always be retrieved from the vertex marks used here, and vice versa.

\section{Relation with literature}\label{chap:LWC:sec:relation-literature}
Berger et al. provided a description of P\'{o}lya point tree in \cite[Section~2.3.2]{BergerBorgs}. In this section, we show that our description of P\'{o}lya point tree, coincides with the description of P\'{o}lya point tree provided in \cite[Section~2.3.2]{BergerBorgs}.

Fix $m \in \mathbb{N}$, and $\delta > -m$. Let us denote the P\'{o}lya point tree from Section~\ref{chap:LWC:sec:PPT} by $\G$, and the one from \cite[Section~2.3.2]{BergerBorgs} by $\G^\prime$. The type space in $\G$ is given by $\Scal = [0,1] \times \{\Old, \Young\}$, where the first component represents the age of the node, and the second component represents the label of the node. The type space of $\G^\prime$ is given by $\Scal^\prime = [0,1] \times \{ L,R\}$, where the first component represents the position of the node, and the second component represents its label. 

We show that $\G$ can be obtained from $\G^\prime$ by performing a change of variables. For any node $\omega^\prime = (x, t)$ in $\G^\prime$, we perform the following change of variables:
\begin{itemize}
	\item Set $a = x^{1/\chi}$;
	\item Define 
	\[
	s = \begin{cases}
		\Old & \text{if } t = L, \\
		\Young & \text{if } t = R;
	\end{cases}
	\]
	\item Define $\delta=2u/m$.
\end{itemize} 
Note that $\chi$ for both $\G$ and $\G^\prime$ are now equal. Then, $\omega = (a, s)$ is the representation of $\omega^\prime$ in $\G$. Note that in $\G^\prime$, the root has position $U^{\chi}$, whereas in $\G$, the root has age $U$, where $U \sim \Unif[0,1]$. In both representations, the root has no label, proving that the definitions are equivalent at the root. Furthermore, note that for any $\omega^\prime \in \Scal^\prime$, $m_{-}(\omega^\prime) = m_{-}(\omega)$, where $\omega$ is the representation of $\omega^\prime$ in $\G$. Moreover, $F$ and $F^\prime$ are equal in distribution to $\Gamma_{(in)}(m)$ and $\Gamma_{(in)}^\prime(m)$, respectively. From the construction, it also follows that $\gamma_{\omega^\prime}$ is equal in distribution to $\Gamma_\omega$.

Now, we still need to show that offspring generation is also preserved under this change of variables between $\G$ and $\G^\prime$.

Let $\omega^\prime$ be a node in $\G^\prime$, and suppose $\omega^\prime$ has position $x$. Then, the ages of the $L$-labelled offspring of $\omega$ are given by $\{U_1 x,\ldots,U_{m_{-}(\omega^\prime)} x\}$, where $\{U_i\}_{i\in\N}$ are i.i.d.\ $\Unif[0,1]$ random variables. The ages of the representative nodes of these offspring in $\G$ are given by $\big\{U_1^{1/\chi} x^{1/\chi}, \ldots, U_{m_{-}(\omega)}^{1/\chi} x^{1/\chi}\big\}$. Denoting $a = x^{1/\chi}$ as the age of the representative of $\omega^\prime$ in $\G$, we obtain the ages of the $\Old$-labelled offspring of $\omega$ as $\big\{U_1^{1/\chi} a, \ldots, U_{m_{-}(\omega)}^{1/\chi} a\big\}$, which matches the construction of the $\Old$-labelled offspring $\omega$ in $\G$, as defined in Section~\ref{chap:LWC:sec:PPT}. Lastly, we need to prove the same for the $\Young$-labelled offspring of $\omega$. Following the construction of $\G^\prime$, we find that the positions of $R$-labelled offspring are the points of a Poisson process with intensity given by
\eqn{\label{eq:equivalence:PPT:1}
	\rho_{\omega^\prime}(y)\,dy=\gamma_{\omega^\prime}\frac{(1-\chi)}{\chi}\frac{y^{(1-\chi)/\chi-1}}{x^{(1-\chi)/\chi}}\,dy~.
}
Now, performing the change of variables $z = y^{1/\chi}$ and $x^{1/\chi} = a$ in \eqref{eq:equivalence:PPT:1}, we obtain
\eqn{\label{eq:equivalence:PPT:2}
	\rho_{\omega^\prime}(y)\,dy = (1-\chi) \Gamma_{\omega} \frac{z^{-\chi}}{a^{1-\chi}}\,dz~,
}
which matches the intensity parameter $\rho_{\omega}(z)$ in \eqref{def:rho_emp:PPT} and \eqref{for:pointgraph:poisson:PPT}. Therefore, the representatives of $R$-labelled offspring of $\omega^\prime$ in $\G$ match the construction of the $\Young$-labelled offspring of $\omega$ in $\G$. Hence, both representations of the P\'{o}lya point tree are equivalent. This essentially proves that our representation is consistent with the literature.

\chapter[P\'{o}lya urn representation of Preferential attachment models]{P\'{o}lya Urn Representation of Preferential Attachment Models}

\label{chap:PU_representation}
\begin{flushright}
\footnotesize{}Based on:\\
\textbf{Universality of the local limit of\\
	preferential attachment models}
	\cite{RRR22}
\end{flushright}
\vspace{0.1cm}
\begin{center}
	\begin{minipage}{0.7 \textwidth}
		\footnotesize{\textbf{Abstract.}
			This chapter establishes the P\'{o}lya urn representation of the generalised preferential attachment models. This representation exploits the inherent exchangeability property of the edge-connection events, despite being heavily dependent. This representation of our preferential attachment models holds only when the connection rules allow for intermediate updates.}
	\end{minipage}
\end{center}
\vspace{0.1cm}

\section{Introduction}
The first and foremost difficulty in dealing with preferential attachment models {is} that the edge-connection {events} are highly dependent. {We now give a representation of our PAMs where these events are \emph{conditionally independent}, extending the results
of Berger et al.\ in \cite{BBCS05} that apply to $\PA_n^{\sss ({m},\delta)}(d)$.} Intuitively, every new edge-connection in the preferential attachment model with intermediate degree updates can be viewed as drawing a ball uniformly from a P\'olya urn with balls {having multiple colours corresponding to the various vertices in the graph.} This is the place where the intermediate degree update comes as a blessing in disguise to us. As a result of this intermediate degree update, the edge-connection events turn out to be exchangeable. Berger et al.\ in \cite{BBCS05} used this exchangeability property to establish a coupling with P\'{o}lya urn graphs, where the edge-connection events are independent conditionally on a sequence of $\Beta$ random variables. They used \emph{de Finetti's theorem} to exploit the exchangeability property. 

Our proof for the P\'{o}lya urn representation relies on an explicit computation of the graph probabilities for both preferential attachment models as well as their respective P\'{o}lya urn representations. This explicit computation allows us to extend the result in \cite{BBCS05} to the generalised preferential attachment models (A), (B), and (D). S\'{e}nizergues also provided a P\'{o}lya urn representation for model (D) in \cite{Delphin21}. Our representation for the same model differs in the choice of parameters for conditioning $\Beta$ random variables.

This P\'{o}lya urn representation is available only for models with intermediate degree updates, i.e., models (A), (B), and (D). If the degrees are not updated after every edge-connection step (such as in models (E) and (F)), the models may seem simpler, but in fact, they are harder to work with since we do not obtain a P\'{o}lya urn description for them.

\begin{center}
	{\textbf{Organisation of the chapter}}
\end{center}
This chapter is organized as follows. In Section~\ref{chap:PUR:sec:PUG}, we introduce the generalised version of the P\'{o}lya urn graphs. In Section~\ref{chap:PUR:sec:CPU}, collapsed P\'{o}lya urn graphs are defined using previously defined P\'{o}lya urn graphs and the collapsing operator defined in Section~\ref{chap:intro:sec:collapsing-operator}. In Section~\ref{chap:PUR:sec:equivalence-model(A)}, we establish the distributional equivalence between model (A) and $\CPU^{\rm (SL)}$. In Section~\ref{chap:PUR:sec:equivalence-other-models}, we extend the equivalence result to models (B) and (D) as well.
\section{P\'{o}lya urn graphs}\label{chap:PUR:sec:PUG}
The sequential preferential attachment models \FC{(A), (B) and (D)} in Section~\ref{chap:intro:sec:model} can be interpreted as an experiment with $n$ urns {corresponding to the vertices in the graph}, where the number of balls in each urn represents the degree of the corresponding vertex in the graph. In this section, 

\begin{Definition}[P\'olya urn graphs]
	\label{def:PU}
	{\rm Define the following:
		\begin{itemize}
			\item[{\btr}] $\left(U_{k} \right)_{k\geq 1}$ is a sequence of i.i.d.\ uniform random variables in $[0,1]$;
			\item[{\btr}] {given $\boldsymbol{m}=(1,1,m_3,m_4,\ldots)$}, $\boldsymbol{\psi}=( \psi_k )_{k\in\N}$ is a sequence of conditionally independent $\Beta$ random variables with support in $[0,1]$, such that $\prob(\psi_k=1)=0$ for all $k\geq 2$ and $\psi_1 = 1$ almost surely;
			\item[{\btr}] let $G_0$ be the initial graph of size $2$ and $n\in\N$ be the size of the graph.
		\end{itemize}
		Define $\mathcal{S}_0^{\sss(n)}=0$ and $\mathcal{S}_n^{\sss(n)}=1$ and for $k\in [n-1]$, define
		\eqn{\begin{split}
				\mathcal{S}_k^{\sss(n)} = (1-\psi)_{(k,n]},\qquad\text{and}\qquad
				\mathcal{I}_k^{\sss(n)} = \left[\mathcal{S}_{k-1}^{\sss(n)}, \mathcal{S}_{k}^{\sss(n)}\right),
		\end{split}}
		where, for $A\subseteq[n]$,
		\eqn{(1-\psi)_A = \prod\limits_{a\in A}(1-\psi_a).}
		Then, {$\PU^{\sss\rm{(SL)}}_n(\boldsymbol{m},\boldsymbol{\psi})$ and $\PU^{\sss\rm{(NSL)}}_n(\boldsymbol{m},\boldsymbol{\psi})$, the {\em P\'olya urn graph} of size $n$ with and without self-loops, respectively, are} defined as follows: 
		\begin{itemize}
			\item[$\rhd$]
			{\FC{In} $\PU^{\sss\rm{(NSL)}}_n(\boldsymbol{m},\boldsymbol{\psi})$, the} $j$-th edge from vertex $k$ is attached to vertex $u\in[k]$ {precisely when}
			\eqn{
				\label{for:PU:NSL}
				U_{m_{[k-1]}+j}\mathcal{S}^{\sss(n)}_{k-1}\in \mathcal{I}_u^{\sss (n)}.
			}
			Observe that self-loops are absent here since  $\left(0,\mathcal{S}^{\sss(n)}_{k-1}\right)$ and $\mathcal{I}_k^{\sss(n)}$ are two disjoint sets\FC{;}
			\item[$\rhd$] \FC{F}or $\PU^{\sss\rm{(SL)}}_n(\boldsymbol{m},\boldsymbol{\psi})$,
			the condition \eqref{for:PU:NSL} is replaced by
			\eqn{
				\label{for:PU:SL}
				U_{m_{[k-1]}+j}\mathcal{S}^{\sss(n)}_{k}\in \mathcal{I}_u^{\sss (n)}.
			}
		\end{itemize}
	}\hfill$\blacksquare$
\end{Definition}

To specify a P\'olya urn graph $\PU_n(\boldsymbol{m,\psi})$, we need to specify the out-edge distribution $M$, and the parameters of the $\Beta$ variables $\boldsymbol{\psi}$.
{Berger et al.\ in \cite{BergerBorgs}} have shown that $\PA_n^{\sss ({m},\delta)}{{(d)}}$ is equal in distribution to $\PU^{\sss\rm{(NSL)}}_n(\boldsymbol{m}^{(1)},\boldsymbol{\psi})$, where $\boldsymbol{\psi}$ is taken as a sequence of {independent} $\Beta$ random variables with certain parameters and $M$ is degenerate at $m$, i.e.\ $\boldsymbol{m}^{(1)} = (1,1,m,m,\ldots)$. Since $\PAri_n(\boldsymbol{m},\delta)$ is a generalised version of $\PA^{{(m,\delta)}}_n(d)$ with i.i.d.\ random out-{degrees}, {we show that }{this model} also has a P\'olya urn description. We sketch an outline in Section~\ref{chap:PUR:sec:equivalence-other-models} {[Theorem~\ref{thm:equiv:PU:PAri}]} that conditionally on $\boldsymbol{m}$, $\PAri_n(\boldsymbol{m},\delta)$ and $\PU^{\sss\rm{(NSL)}}_n(\boldsymbol{m},\boldsymbol{\psi})$ are also equal in distribution. 
On the other hand, it is evident that {the} deterministic versions of model (B) and (D) are equivalent for $m=1$. From the edge-connection probabilities of model (A) and (B), we can observe that they are different only in {whether they give rise to self-loops or not}. {van der Hofstad in \cite[Chapter~5]{vdH2} shows} that model (A) and (B) can be obtained by a collapsing procedure for degenerate $M$. Therefore for obtaining a P\'olya urn graph equivalence for models (A) and (B), a generalisation of this collapsing procedure is helpful, {that we have discussed already in Section~\ref{chap:intro:sec:collapsing-operator}.}

\section{Collapsed P\'{o}lya urn graphs}\label{chap:PUR:sec:CPU}

$\PArt_n(\boldsymbol{m},1,\delta)$ is essentially the same as model (D) when every vertex comes with exactly one out-edge, but $\delta$ is different for every vertex. {Therefore,} we expect that {the} P\'olya urn graph {extends to this graph}. {Similarly to the construction of model (B) through the collapsing procedure defined in \ref{chap:intro:sec:collapsing-operator}}, we collapse the P\'olya urn graph.

{Conditionally on $\boldsymbol{m},$ we construct the collapsed P\'olya Urn graph by using our collapsing construction on the P\'olya Urn Graph defined in Definition~\ref{def:PU} as follows:}

\begin{Definition}[Collapsed Pólya urn graph]
	\label{def:CPU}
	{\rm We first construct $\PU_{m_{[n]}}(\boldsymbol{1},\boldsymbol{\psi})$ with every new vertex having exactly one out-edge and initial graph $G_0$ of size $2$. {Conditionally} on $\boldsymbol{m}=(1,1,m_3,m_4,\ldots),$ the graph $\CPU_n(\boldsymbol{m},\boldsymbol{\psi})$ is defined as $\mathcal{C}_{\boldsymbol{m}}\big( \PU_{m_{[n]}}(\boldsymbol{1},\boldsymbol{\psi}) \big)$. The label SL or NSL for the $\CPU_n$ will be determined by the label of the $\PU_{m_{[n]}}$. We denote the two $\CPU$'s with and without self-loops as $\CPU_n^{\sss\rm{(SL)}}$ and $\CPU_n^{\sss\rm{(NSL)}}$ respectively.}\hfill$\blacksquare$
\end{Definition}

\begin{remark}[{Self-loops for $\CPU_n^{\sss (\mathrm{NSL})}$}]
	{\rm $\CPU_n^{\sss (\mathrm{NSL})}$ may contain self-loops because of collapsing.}\hfill$\blacksquare$
\end{remark}
We end this section by deriving the connection probabilities for $\CPU_n^{\rm{\sss{(SL)}}}$ and $\CPU_n^{\rm{\sss{(NSL)}}}$.
For $k\geq 1$ and $j\in [m_{k}]$, define
\begin{align}\label{eq:def:position:CPU}
	\mathcal{S}_{k,j}^{\sss(n)} = \prod\limits_{l=m_{[k-1]}+j+1}^{m_{[n]}} (1-\psi_l),\qquad\text{and}\qquad
	\mathcal{S}_k^{\sss(n)} =  \mathcal{S}_{k,m_k}^{\sss(n)};
\end{align}
and the intervals $\mathcal{I}_k^{\sss(n)}=\left[\mathcal{S}_{k-1}^{\sss(n)}, \mathcal{S}_{k}^{\sss(n)}\right)$. {Let $\prob_m^{\sss\rm{(SL)}}$ and $\prob_m^{\sss\rm{(NSL)}}$ denote the conditional law given $\boldsymbol{m}$ of $\CPU_n^{\sss\rm{(SL)}}$ and $\CPU_n^{\sss\rm{(NSL)}}$, respectively.} Then from the construction of the $\CPU_n^{\sss\rm{(SL)}}$ it follows that, for $u\geq 3$,
\eqn{\label{eq:edge-connecting_probability:CPU:SL}
	\prob_m^{\sss\rm{(SL)}}\left( \left.u\overset{j}{\rightsquigarrow}v\right|\left( \psi_k \right)_{k\in\N} \right) = \begin{cases}
		\frac{\mathcal{S}_{v}^{\sss(n)}-\mathcal{S}_{v-1}^{\sss(n)}}{\mathcal{S}_{u,j}^{\sss(n)}}& \text{for }u>v,\\
		\frac{\mathcal{S}_{u,j}^{\sss(n)}-\mathcal{S}_{u-1}^{\sss(n)}}{\mathcal{S}_{u,j}^{\sss(n)}}& \text{for }u=v.
\end{cases}}
This probability is for {$\CPU_n^{\sss\rm{(SL)}}$. For $\CPU_n^{\sss\rm{(NSL)}}$ instead}, the expression in \eqref{eq:edge-connecting_probability:CPU:SL} becomes
\eqn{\label{eq:edge-connecting_probability:CPU:NSL}
	\prob_m^{\sss\rm{(NSL)}}\left( \left.u\overset{j}{\rightsquigarrow}v\right|\left( \psi_k \right)_{k\in\N} \right) = \begin{cases}
		\frac{\mathcal{S}_{v}^{\sss(n)}-\mathcal{S}_{v-1}^{\sss(n)}}{\mathcal{S}_{u,j-1}^{\sss(n)}}& \text{for }u>v,\\
		\frac{\mathcal{S}_{u,j-1}^{\sss(n)}-\mathcal{S}_{u-1}^{\sss(n)}}{\mathcal{S}_{u,j-1}^{\sss(n)}}& \text{for }u=v.
\end{cases}}
{Now that we have introduced the relevant random graph models, in the next section, we will show that the collapsed versions of the P\'olya graph models have the same distribution as our preferential attachment models.}
\section{{Equivalence for model} (A)}\label{chap:PUR:sec:equivalence-model(A)}
{Recall that model (A) has i.i.d.\ out-degrees for every vertex. Thus, in order to couple it with our $\CPU$, it must have the same out-degrees.} Given $\boldsymbol{m}= (m_1,m_2,m_3,\ldots)$, with $m_1=m_2=1$, we {thus aim} to couple $\PArs_n(\boldsymbol{m},\delta)$ and $\CPU^{\sss\rm{(SL)}}$.

For $i\in[2],$ {$a_i$ denotes the} degree of vertex $i$ in the initial graph $G_0$, {while for $i>2$ define $a_i\equiv 1$.} Let $\mathcal{H}_n$ be the set of all finite vertex-labelled graphs $G$ of size $n$ and {let} $\mathcal{H}_{\boldsymbol{m}}(G)= \big\{ G_e\in {\mathcal{H}}_{m_{[n]}} \colon \mathcal{C}_{\boldsymbol{m}}(G_e) {\overset{\star}{\simeq}} G \big\}$ {denote the set of graphs that are mapped to $G$ by the collapsing operator $\mathcal{C}_{\boldsymbol{m}}$.} 
For $k\geq 3$ and $l\in\left[ m_{k} \right]$, {define}
\eqn{\label{def:psi:1}
	\psi_{m_{[k-1]}+l} \sim {\Beta}\Big( 1+\frac{\delta}{m_k},\, a_{[2]} + 2\big( m_{[k-1]}+l-3 \big)+(k-1)\delta+\frac{(l-1)}{m_k}\delta \Big),
}
where $a_{[2]}=a_1+a_2$ and, for $k\leq 2$,
\eqn{\label{def:psi:2}
	\begin{split}
		\psi_1 \equiv~1,~
		\mbox{and}\quad \psi_{2} \sim {\Beta}\left( a_2+\delta,a_1+\delta \right).
\end{split}}
Observe that the second parameter of the $\Beta$ random variables in \eqref{def:psi:1} is nothing but $c_{v,l-1}+1$, as defined in \eqref{eq:normali}. {We abbreviate $\boldsymbol{\psi} = (\psi_i)_{i\geq 1}$ for the collection of $\Beta$ variables, where we emphasize that these variables are {\em conditionally independent} given the random out-degrees $\boldsymbol{m}$. Our main result concerning the relation between collapsed P\'olya graphs and model (A) is as follows:}
\begin{Theorem}[Equivalence of model (A) and $\CPU^{\sss\rm{(SL)}}$]\label{thm:equiv:CPU:PArs}
	For any graph $G\in\mathcal{H}_{n}$,
	\eqn{\label{eq:thm:equiv:PA:CPU}
		\prob_m\left( \PArs_n(\boldsymbol{m},\delta) \overset{\star}{\simeq} G \right) = \prob_m\left( \CPU_{n}^{\sss\mathrm{(SL)}}(\boldsymbol{m},\boldsymbol{\psi}) \overset{\star}{\simeq} G \right).}
\end{Theorem}

{We emphasize that Theorem \ref{thm:equiv:CPU:PArs} describes a {\em conditional} result given $\boldsymbol{m}$, i.e., it holds conditionally on $\boldsymbol{m}$.
}

{We prove Theorem \ref{thm:equiv:CPU:PArs} by proving it for the pre-collapsed version of both graphs, and equating their conditional distributions.} The following proposition helps us in equating the conditional probabilities of the pre-collapsed graphs:
\begin{Proposition}[Equivalence of pre-collapsed model (A) and $\PU^{\sss\rm{(SL)}}$]\label{prop:equiv:CPU:PArs}
	Conditionally on $\boldsymbol{m}$, for any graph $H\in\mathcal{H}_{m_{[n]}}$,
	\eqn{\label{prop:eq:equiv:CPU:PArs}
		\prob_m\left( \PArs_n(\boldsymbol{m},1,\delta) \overset{\star}{\simeq} H \right) = \prob_m\left( \PU_{m_{[n]}}^{\sss\mathrm{(SL)}}(\boldsymbol{1},\boldsymbol{\psi}) \overset{\star}{\simeq} H \right).}
\end{Proposition}
\begin{proof}[\textbf{Proof of Theorem~\ref{thm:equiv:CPU:PArs} subject to Proposition~\ref{prop:equiv:CPU:PArs}.}]
	Proposition~\ref{prop:equiv:CPU:PArs} essentially provides us with the pre-collapsing equivalence of the graphs. Since, conditionally on $\boldsymbol{m},~ \left\{\PArs_{m_{[n]}}(\boldsymbol{m},1,\delta) \overset{\star}{\simeq} H\right\} $ are disjoint events for $H\in\mathcal{H}_{\boldsymbol{m}}(G)$, the probability on the RHS ({right hand side}) of \eqref{eq:thm:equiv:PA:CPU} for model (A) can be written in terms of the pre-collapsed graphs as
	\eqn{\label{for:equiv:CPU:PArs:1}
		\begin{split}
			\prob_m\left( \PArs_n(\boldsymbol{m},\delta) \overset{\star}{\simeq} G \right) =& \prob_m\left( \bigcup\limits_{H\in\mathcal{H}_{\boldsymbol{m}}(G)}\left\{ \PArs_{m_{[n]}}(\boldsymbol{m},1,\delta) \overset{\star}{\simeq} H \right\} \right)\\
			=& \sum\limits_{H\in\mathcal{H}_{\boldsymbol{m}}(G)}\prob_m\left(  \PArs_{m_{[n]}}(\boldsymbol{m},1,\delta) \overset{\star}{\simeq} H \right).
	\end{split}}
	Similarly the probability on the LHS ({left hand side}) of \eqref{eq:thm:equiv:PA:CPU} for $\CPU_n^{\sss\rm{(SL)}}$ can be written in terms of $\PU_{m_{[n]}}^{\sss\rm{(SL)}}$ as
	\eqn{\label{for:equiv:CPU:PArs:2}
		\begin{split}
			\prob_m\left( \CPU_n^{\sss\mathrm{(SL)}}(\boldsymbol{m},\boldsymbol{\psi}) \overset{\star}{\simeq} G \right) =& \prob_m\left( \bigcup\limits_{H\in\mathcal{H}_{\boldsymbol{m}}(G)}\left\{\PU_{m_{[n]}}^{\sss\mathrm{(SL)}}(\boldsymbol{1},\boldsymbol{\psi}) \overset{\star}{\simeq} H\right\} \right)\\
			=& \sum\limits_{H\in\mathcal{H}_{\boldsymbol{m}}(G)}\prob_m\left( \PU_{m_{[n]}}^{\sss\mathrm{(SL)}}(\boldsymbol{1},\boldsymbol{\psi}) \overset{\star}{\simeq} H \right).
	\end{split}}
	Now from Proposition~\ref{prop:equiv:CPU:PArs} it follows that the summands in \eqref{for:equiv:CPU:PArs:1} and \eqref{for:equiv:CPU:PArs:2} are equal. Hence, conditionally on $\boldsymbol{m},~\PArs_n$ and $\CPU_n^{\sss\rm{(SL)}}$ are equal in distribution.
\end{proof}

We now move towards the proof of Proposition~\ref{prop:equiv:CPU:PArs}. Berger et al.\ \cite{BergerBorgs} have proved a version of Theorem~\ref{thm:equiv:CPU:PArs} for model (D) and degenerate {out-degrees} using an extension to multiple urns of the P\'olya urn {characterization in terms of conditionally independent events by de Finetti's Theorem. We could adapt this proof. Instead, we} prove Proposition~\ref{prop:equiv:CPU:PArs} by explicitly calculating the graph probabilities of both random graphs and equating them term by term. This proof is interesting in its own right.

Let $v(u)$ denote the vertex to which the out-edge from $u$ connects in $H$. 
Then from \eqref{eq:edge_connecting:model(rs):1},
\eqn{\label{for:equiv:CPU:PArs:4}
	\begin{split}
		&\prob_m\left(  \PArs_{m_{[n]}}(\boldsymbol{m},1,\delta) \overset{\star}{\simeq} H \right)\\
		&\qquad =\prod\limits_{u\in[3,n]}\prod\limits_{j\in[m_{u}]} \frac{d_{v(m_{[u-1]}+j)}(m_{[u-1]}+j-1)+\delta(v(m_{[u-1]}+j))}{a_{[2]}+2(m_{[u-1]}+j-2)+\left( (u-1)+\frac{j}{m_u} \right)\delta-1}~.
\end{split}}
The following lemma simplifies and rearranges the factors in the numerator of \eqref{for:equiv:CPU:PArs:4}:
\begin{Lemma}[Rearrangement of the numerator of \eqref{for:equiv:CPU:PArs:4}]\label{lem:equiv:num:rearrange}
	The numerator of \eqref{for:equiv:CPU:PArs:4} can be rearranged as
	\eqn{\label{eq:lem:num:rearrange:equiv}
		\begin{split}
			&\prod\limits_{u\in[3,n]}\prod\limits_{j\in[m_{u}]}\left( d_{v(m_{[u-1]}+j)}(m_{[u-1]}+j-1)+\delta(v(m_{[u-1]}+j)) \right)\\
			&\qquad=\prod\limits_{k\in[n]}\prod\limits_{l\in[m_{u]}} \prod\limits_{i=a_{m_{[k-1]}+l}}^{d_{m_{[k-1]}+l}(H)-1}\left( i+\frac{\delta}{m_k} \right),
	\end{split}}
	where $d_v(H)$ denotes the degree of the vertex $v$ in the graph $H$.
\end{Lemma}
\begin{proof}
	{Observe that the factors in the numerator of RHS of \eqref{for:equiv:CPU:PArs:4} depend on the receiver's degree. Since the edges from the new vertices connect to one of the existing vertices (or itself), the product in the numerator of \eqref{for:equiv:CPU:PArs:4} can be rewritten as
		\eqn{\label{for:equiv:CPU:PArs:4-1}
			\begin{split}
				&\prod\limits_{u\in[3,n]}\prod\limits_{j\in[m_{u}]}\left( d_{v(m_{[u-1]}+j)}(m_{[u-1]}+j-1)+\delta(v(m_{[u-1]}+j)) \right)\\
				=& \prod\limits_{s\in[1,m_{[n]}]} \prod\limits_{\substack{u\in[3,n],\, j\in[m_u]\\v(m_{[u-1]}+j)=s}} \big( d_s(m_{[u-1]}+j-1)+\delta(s) \big)~.
		\end{split}}
		For the very first incoming edge to $s\in[m_{[n]}]$, we have the factor of $(a_s+\delta(s))$ and for the remaining ones, we have a factor in the RHS of \eqref{for:equiv:CPU:PArs:4-1} with an increment of $1$. On the other hand, for the last incoming edge to $s$ in $H$, we have the factor $(d_s(H)-1+\delta(s))$. For any $s\in[m_{[n]}]$, there exists a unique $k\in[n]$ and $l\in[m_u]$ such that $s=m_{[k-1]}+l$ and $\delta(s)=\delta/m_k$. Therefore the LHS of \eqref{for:equiv:CPU:PArs:4-1} can be further simplified as
		\eqan{\label{for:equiv:CPU:PArs:5}
			&\prod\limits_{s\in[1,m_{[n]}]} \prod\limits_{\substack{u\in[3,n],\, j\in[m_u]\\v(m_{[u-1]}+j)=s}} \big( d_s(m_{[u-1]}+j-1)+\delta(s) \big)\nn\\
			=&\prod\limits_{k\in[n]}\prod\limits_{l\in[m_{u}]} \prod\limits_{i=a_{m_{[k-1]}+l}}^{d_{m_{[k-1]}+l}(H)-1}\left( i+\frac{\delta}{m_k} \right)~,}
	}
	completing the proof of Lemma~\ref{lem:equiv:num:rearrange}.
\end{proof}
By Lemma~\ref{lem:equiv:num:rearrange} and a rearrangement of the numerator, the graph probability in \eqref{for:equiv:CPU:PArs:4} can be written as
\begin{align}
	\label{for:equiv:CPU:PArs:6}
	&\prob_m\left(  \PArs_{m_{[n]}}(\boldsymbol{m},1,\delta) \overset{\star}{\simeq} H \right)\nn\\
	=& \prod\limits_{u\in[n]}\prod\limits_{j\in[m_{u}]} \prod\limits_{i=a_{m_{[u-1]}+j}}^{d_{m_{[u-1]}+j}(H)-1}\left( i+\frac{\delta}{m_u} \right)\\
	&\hspace{1.5cm}\times\prod\limits_{u\in[3,n]}\prod\limits_{j\in[m_{u}]} \frac{1}{a_{[2]}+2(m_{[u-1]}+j-2)+\left( (u-1)+\frac{j}{m_u} \right)\delta-1}.\nonumber
\end{align}
Next, we calculate the graph probabilities of  $\PU_{m_{[n]}}^{\sss\rm{(SL)}}(\boldsymbol{1},\boldsymbol{\psi})$ {and show that these agree}.
To calculate $\prob_m\left( \PU_{m_{[n]}}^{\sss\mathrm{(SL)}}(\boldsymbol{1},\boldsymbol{\psi})\overset{\star}{\simeq} H \right)$, we condition on the $\Beta$ random variables {as well}. We denote the conditional measure by $\prob_{m,\psi}$, i.e., for every event $\mathcal{E}$,
\eqn{
	\label{for:equiv:CPU:PAr:7}
	\prob_{m,\psi}(\mathcal{E}) = \prob\left(\mathcal{E} \mid (m_k)_{k\geq 3},(\psi_k)_{k\geq 1} \right).
}
{Under this conditioning, the edges of $\PU_{m_{[n]}}^{\sss\rm{(SL)}}$ are independent.} First we calculate the conditional edge-connection probabilities for $\PU_{m_{[n]}}^{\sss\rm{(SL)}}(\boldsymbol{1},\boldsymbol{\psi})$:
\begin{Lemma}[Conditional edge-connection probability of $\PU^{\sss\rm{(SL)}}$]\label{lem:edge Probability:Polya Urn graph}
	Conditionally on $\boldsymbol{m}$ and $(\psi_k)_{k\geq 1}$ defined in \eqref{def:psi:1} and \eqref{def:psi:2}, the probability of connecting the edge from $u$ to $v$ in $\PU_{m_{[n]}}^{\sss\rm{(SL)}}(\boldsymbol{1},\boldsymbol{\psi})$ is given by $\psi_v (1-\psi)_{(v,u]}$.
\end{Lemma}
\begin{proof} By the construction of $\PU_{m_{[n]}}^{\sss\rm{(SL)}}(\boldsymbol{1},\boldsymbol{\psi})$,
	\begin{align*}
		\mathcal{S}_k^{\sss(m_{[n]})} &= (1-\psi)_{(k,m_{[n]}]},\qquad\text{and}\qquad 
		|I_k^{\sss(m_{[n]})}|
		=\psi_k(1-\psi)_{(k,m_{[n]}]}.
	\end{align*}
	Taking the ratio {gives}
	\begin{equation*}
		\frac{|I_v^{\sss(m_{[n]})}|}{\mathcal{S}_u^{\sss(m_{[n]})}} = \psi_v\frac{(1-\psi)_{(v,m_{[n]}]}}{(1-\psi)_{(u,m_{[n]}]}}=\psi_v(1-\psi)_{(v,u]}.
	\end{equation*}
	Therefore, $\prob_{m, \psi}\left(u\rightsquigarrow v\right)$ in $\PU_{m_{[n]}}^{\sss\rm{(SL)}}(\boldsymbol{1},\boldsymbol{\psi})$ is given by $\psi_v (1-\psi)_{(v,u]}$.
\end{proof}
In P\'olya urn graphs, conditionally on the $\Beta$ random variables, the edges are added independently, leading to the following lemma:
\begin{Lemma}[Conditional density of $\PU^{\sss\rm{(SL)}}$]\label{lem:PU:conditional:graph_probability}
	For any graph $H\in\mathcal{H}_{m_{[n]}}$,
	\begin{equation}
		\label{eq:PU:conditional:graph_probabilit}
		\prob_{m,\psi}\left( \PU_{m_{[n]}}^{\rm\sss(SL)}(\boldsymbol{1},\boldsymbol{\psi}) \overset{\star}{\simeq} H \right) =  \prod\limits_{s\in[2,m_{[n]}]} \psi_s^{p_s}(1-\psi_s)^{q_s},
	\end{equation}
	where
	\begin{align}\label{def:p:q}
		p_s=&~ d_s(H)-a_s,\qquad\text{and}\qquad 
		q_s=\sum\limits_{u \in\left( m_{[2]}, m_{[n]}\right]}\one_{\{s\in(v(u),u]\}}.
	\end{align}
\end{Lemma}
\begin{proof}
	For every {incoming edge} of the vertex $u$, there is a vertex $w\ge u$ such that $v(w)=u$. Conditionally on {$\boldsymbol{m}$ and $\boldsymbol{\psi}$}, the {edge-connection events} are independent. {We then note that} there are $p_s$ many incoming edges for vertex $s$ and the factor $\psi_s^{p_s}$ comes from that. {Every $u$ such that $s\in(v(u),u]$ gives rise to one factor $1-\psi_s$, giving rise to} $q_s$ many factors $1-\psi_s$. 
\end{proof}
\chRounak{Note that, from definition of $p_s$ and $q_s$,}
\eqan{
	p_{m_{[n]}}&=\one_{\{m_{[n]}~\mbox{creates a self-loop}\}}\label{eq:major:1:1}\\
	\mbox{and,}\qquad q_{m_{[n]}}&=\one_{\{m_{[n]}~\mbox{does not create a self-loop}\}}\label{eq:major:1:2},
}
\chRounak{which essentially leads to the fact that $p_{m_{[n]}}+q_{m_{[n]}}=1$.}
{Next we aim to take the expectation w.r.t.\ $\boldsymbol{\psi}$, and for this, we compute the expectation of powers of $\Beta$ variables:}
\begin{Lemma}[{Integer moments of $\Beta$ distribution}]\label{lemma:Beta Expectation}
	For all $a,b\in \N$ and $\psi\sim\mathrm{Beta}(\alpha,\beta)$,
	\begin{equation}\label{eq:Beta Expectation}
		\E\left[ \psi^a(1-\psi)^b \right] = \frac{(\alpha+a-1)_a(\beta+b-1)_b}{(\alpha+\beta+a+b-1)_{a+b}},
	\end{equation}
	where $(n)_k=\prod\limits_{i=0}^{k-1}(n-i)$ for $k\geq 1$.
\end{Lemma}

\begin{proof}
	A direct calculation for $\Beta$ random variables shows that
	\begin{align*}
		\E\left[ \psi^a(1-\psi)^b \right]=&\frac{B(\alpha+a,\beta+b)}{B(\alpha,\beta)} = \frac{\Gamma(\alpha+\beta)\Gamma(\alpha+a)\Gamma(\beta+b)}{\Gamma(\alpha)\Gamma(\beta)\Gamma(\alpha+\beta+a+b)}\\
		=&\frac{(\alpha+a-1)_a(\beta+b-1)_b}{(\alpha+\beta+a+b-1)_{a+b}}~,
	\end{align*}
	as required.
\end{proof}
As a consequence of Lemma~\ref{lem:PU:conditional:graph_probability} and \ref{lemma:Beta Expectation}, we can calculate $\prob_{m}\left( \PU_{m_{[n]}}^{\rm\sss(SL)}(\boldsymbol{1},\boldsymbol{\psi}) \overset{\star}{\simeq} H \right)$:
\begin{Lemma}\label{lem:PU:probability_calculation}
	For
	$H\in\mathcal{H}_{\boldsymbol{m}}(G)$,
	\eqn{\label{eq:lem:PU:probability:0}
		\prob_{m}\left( \PU_{m_{[n]}}^{\rm\sss(SL)}(\boldsymbol{1},\boldsymbol{\psi}) \overset{\star}{\simeq} H \right) = \prod\limits_{s\in \left[ 2  ,m_{[n]} \right]} \frac{(\alpha_s+p_s-1)_{p_s}(\beta_s+q_s-1)_{q_s}}{(\alpha_s+\beta_s+p_s+q_s-1)_{p_s+q_s}},
	}
	where $\alpha_s$ and $\beta_s$ are the first and second parameters of the $\Beta$ random variables defined in ~\eqref{def:psi:1} and ~\eqref{def:psi:2}, $p_s, q_s$ are {defined} in Lemma~\ref{lem:PU:conditional:graph_probability} and $\infsupp(M)$ is {the minimum of the support} of the random variable $M$.
\end{Lemma}

\begin{proof} 
	Lemma~\ref{lem:PU:conditional:graph_probability} describes the conditional probability distribution of the P\'olya Urn graphs given the independent $\Beta$ random variables. To obtain the unconditional probability, we take the expectation on the RHS of \eqref{eq:PU:conditional:graph_probabilit} with respect to the $\Beta$ random variables, to obtain
	\eqn{ \label{for:lem:PU:probability_calculation:1}
		\prob_m\left( \PU_{m_{[n]}}^{\rm\sss(SL)}(\boldsymbol{1},\boldsymbol{\psi}) \overset{\star}{\simeq} H \right)
		= \E_m \Big[ \prod\limits_{s\in[2, m_{[n]}} \psi_s^{p_s}(1-\psi_s)^{q_s} \Big].
		}
		Since $\psi_{m_{[u-1]}+j}$ are independent $\Beta$ random variables with parameters $\alpha_{m_{[u-1]}+j}$ and $\beta_{m_{[u-1]}+j}$, respectively, as defined in \eqref{def:psi:1}, \eqref{for:lem:PU:probability_calculation:1} simplifies to
		\eqan{\label{for:lem:PU:probability_calculation:2}
	&\prod\limits_{s\in\left[2, m_{[n]} \right]} \E_m\left[\psi_s^{p_s}(1-\psi_s)^{q_s}\right]\\
	=&\frac{(\alpha_2+p_2-1)_{p_2}(\beta_2+q_2-1)_{q_2}}{(\alpha_2+\beta_2+p_2+q_2-1)_{p_2+q_2}}\prod\limits_{s\in\left[ 3,m_{[n]} \right]} \frac{(\alpha_s+p_s-1)_{p_s}(\beta_s+q_s-1)_{q_s}}{(\alpha_s+\beta_s+p_s+q_s-1)_{p_s+q_s}}.\nn}
	This matches the RHS of \eqref{eq:lem:PU:probability:0}.
\end{proof}
In the next lemma, we derive an alternative expression for $q_s$ defined in \eqref{def:p:q} which helps in understanding the RHS of \eqref{eq:lem:PU:probability:0}:
\begin{Lemma}
	The $q_s$ defined in \eqref{def:p:q} can be represented as
	\eqn{\label{eq:lemma:p:q}
		q_s=\begin{cases}
			0 &\mbox{ if }s=1,\\
			d_1(H)-a_1 &\mbox{ if }s=2,\\
			d_{[s-1]}(H) - a_{[2]} - 2\left(s-3\right) &\mbox{ if }s\geq 3,
	\end{cases}}
	where
	\[
	d_{[s-1]}(H) = \sum\limits_{v\in [s-1]}d_v(H), \qquad \mbox{and} \qquad a_{[2]}= a_1+a_2.
	\]
\end{Lemma}

\begin{proof}
	For $s=1$, we observe that $q_1$ is $0$ by definition and, for $s=2$,
	\eqn{\label{for:equiv:CPU:PAr:8}
		\begin{split}
			q_{2} = \sum\limits_{u\in \left( 2,m_{[n]} \right]} \one_{[2\in(v(u),u]]} = \sum\limits_{u\in \left( 2,m_{[n]} \right]} \one_{[v(u)=1]} = d_1(H) - a_1~.
	\end{split}}
	Simplifying $q_s$ for $s\geq 3$ {gives}
	\begingroup
	\allowdisplaybreaks	
	\eqan{\label{for:equiv:CPU:PAr:9}
		q_s&= \sum\limits_{u\in \left(m_{[2]},m_{[n]}\right]} \one_{\left[ s\in (v(u),u] \right]}
		=\sum\limits_{u\in \left(2,m_{[n]}\right]} \one_{\left[ s\in \left(v(u),m_{[n]}\right] \right]} - \sum\limits_{u\in \left(2,m_{[n]}\right]} \one_{\left[ s\in \left(u,m_{[n]}\right] \right]}\nn\\
		&=\sum\limits_{u\in \left(2,m_{[n]}\right]} \one_{\left[ v(u)\in [s-1] \right]} - \sum\limits_{u\in \left(2,m_{[n]}\right]} \one_{\left[ u\in [s-1] \right]} \nn\\
		& =\Big( \sum\limits_{u\in \left[ m_{[n]}\right]} \one_{\left[ v(u)\in [s-1] \right]} - \sum\limits_{u\in \left[2\right]} \one_{\left[ v(u)\in [s-1] \right]} \Big)  \nn\\
		&\hspace{4.5cm}-\Big( \sum\limits_{u\in \left[m_{[n]}\right]} \one_{\left[ u\in [s-1] \right]}
		- \sum\limits_{u\in \left[2\right]} \one_{\left[ u\in [s-1] \right]} \Big)\nn\\
		& =\sum\limits_{v\in[s-1]}\left( d_v^{\sss\rm{(in)}}(H) - d_v^{\sss\rm{(in)}}(G_0) \right) - \left((s-1)-2\right),
	}
	\endgroup
	where $d_v^{\sss\rm{(in)}}(G)$ is the in-degree of vertex $v$ in the graph $G$ and $G_0$ is the initial graph we started with. Let $d_v^{\sss\rm{(out)}}(G)$ denote the out-degree of vertex $v$ in the graph $G$, {so that}
	\[
	d_v^{\sss\rm{(in)}}(G) + d_v^{\sss\rm{(out)}}(G) = d_v(G).
	\]
	{Note that} the vertices in the initial graph $G_0$ do not have out-edges directed to any new incoming vertices. Therefore, $d_v^{\sss\rm{(out)}}(H) = d_v^{\sss\rm{(out)}}(G_0)$ for all $v\in \left[ 2 \right]$.
	
	On the other hand, the new incoming vertices have exactly one out-edge each. Furthermore they are not {part of} the initial graph. {Hence,} by definition $d_v(G_0) = 0$ for all $v>2=m_{[2]}$. Therefore, both $d_v^{\sss\rm{(out)}}(G_0)$ and $d_v^{\sss\rm{(in)}}(G_0)$ are zero {for all $v>2=m_{[2]}$.}
	Hence, for $s\geq 3$,
	\eqan{\label{for:equiv:CPU:PAr:10}
		q_s=& \sum\limits_{v\in[s-1]}\left( d_v^{\sss\rm{(in)}}(H) - d_v^{\sss\rm{(in)}}(G_0) \right) - \left((s-1)-2\right)\nn\\
		=& \sum\limits_{v\in[s-1]}\left( d_v(H) - d_v(G_0) \right) - \sum\limits_{v\in[s-1]}\left( d_v^{\sss\rm{(out)}}(H) - d_v^{\sss\rm{(out)}}(G_0) \right) - \left(s-3\right)\nn\\
		=& \sum\limits_{v\in[s-1]}\left( d_v(H) - d_v(G_0) \right) - \sum\limits_{v\in\left[ 2 \right]}\left( d_v^{\sss\rm{(out)}}(H) - d_v^{\sss\rm{(out)}}(G_0) \right)\nn\\
		&\hspace{5cm} - \sum\limits_{v\in \left( 2,s-1 \right]} d_v^{\sss\rm{(out)}}(H)- \left(s-3\right)\nn\\
		=& \sum\limits_{v\in[s-1]}\left( d_v(H) - d_v(G_0) \right) - 2\left(s-3\right)\nn\\
		=& d_{[s-1]}(H) - a_{[2]} - 2\left(s-3\right)~,}
	as required in \eqref{eq:lemma:p:q}.
\end{proof}
Now, we have all the tools to prove Proposition~\ref{prop:equiv:CPU:PArs}:
\begin{proof}[\textbf{Proof of Proposition~\ref{prop:equiv:CPU:PArs}}]
	We start by simplifying the RHS of \eqref{eq:lem:PU:probability:0} and equating term by term. First, we consider the term corresponding to $s=2$ in the RHS of \eqref{eq:lem:PU:probability:0}.
	Using ~\eqref{for:equiv:CPU:PAr:8} and substituting the values of $\alpha_2$ and $\beta_2$,
	\eqan{\label{for:prop:equiv:PU:PAr:1}
			&\frac{(\alpha_2+p_2-1)_{p_2}(\beta_2+q_2-1)_{q_2}}{(\alpha_2+\beta_2+p_2+q_2-1)_{p_2+q_2}}\nn\\
			=~& \frac{(\delta+d_2(H)-1)_{d_2(H)-a_2}(d_1(H)+\delta-1)_{d_1(G)-a_1}}{(d_{[2]}(G)+2\delta-1)_{d_{[2]}(G)-a_{[2]}}}\nn\\
			=~& \frac{1}{\left( d_{[2]}(H)+2\delta-1 \right)_{d_{[2]}(H)-a_{[2]}}}\prod\limits_{i=a_1}^{d_1(H)-1}(i+\delta)\prod\limits_{i=a_2}^{d_2(H)-1}(i+\delta).
	}
	The factor $(\alpha_{s}+p_{s}-1)_{p_{s}}$ in the numerator of \eqref{eq:lem:PU:probability:0} can be simplified as
	\eqan{\label{for:prop:equiv:PU:PAr:2}
		&(\alpha_{s}+p_{s}-1)_{p_{s}}=\prod\limits_{i=1}^{p_{s}}= \left( \alpha_s-1+i \right)=\prod\limits_{i=a_{s}}^{d_{s}(H)-1} \left( i+ \alpha_s-1 \right).}
	The last equality in \eqref{for:prop:equiv:PU:PAr:2} is obtained by substituting the expression for $p_{m_{[u-1]}+j}$ and a simple change of variables. For any $s\in[3,m_{[n]}]$, we can find a $u\in[3,n]$ and $j\in [m_u]$ such that $s=m_{[u-1]}+j$. Therefore, 
	\eqn{\label{for:prop:equiv:PU:PAr:3}
		\prod\limits_{s\in\left[ 3,m_{[n]} \right]}(\alpha_s+p_s-1)_{p_s} =  \prod\limits_{u\in[3,n]}\prod\limits_{j\in \left[ m_u \right]}\prod\limits_{i=a_{m_{[u-1]}+j}}^{d_{m_{[u-1]}+j}(H)-1} \left( i+ \frac{\delta}{m_u} \right).}
	So far, from \eqref{for:prop:equiv:PU:PAr:1} and \eqref{for:prop:equiv:PU:PAr:3}, we have already obtained the first factor in \eqref{for:equiv:CPU:PArs:6} and an additional term in the denominator.
	Note that, for $s\geq 3$,
	\eqn{\label{for:prop:equiv:PU:PAr:4}
		\begin{split}
			\alpha_s+\beta_s = \beta_{s+1}-1,\qquad\mbox{and}\qquad
			p_s+q_s = q_{s+1}+1.
	\end{split}}
	Therefore the remaining term in \eqref{eq:lem:PU:probability:0} can be rewritten in a telescoping product form {as}
	\eqan{\label{for:prop:equiv:PU:PAr:5}
			&\prod\limits_{s\in\left[3,m_{[n]}\right]}\frac{(\beta_s+q_s-1)_{q_s}}{(\alpha_s+\beta_s+p_s+q_s-1)_{p_s+q_s}}\nn\\
			=~& \prod\limits_{s\in\left[3,m_{[n]}\right]}\frac{(\beta_s+q_s-1)_{q_s}}{(\beta_{s+1}+q_{s+1}-1)_{q_{s+1}+1}}\nn\\
			=~& \prod\limits_{s\in\left[3,m_{[n]}\right]} \frac{1}{\left( \beta_{s+1}-1 \right)}\prod\limits_{s\in\left[3,m_{[n]}\right]}\frac{(\beta_s+q_s-1)_{q_s}}{(\beta_{s+1}+q_{s+1}-1)_{q_{s+1}}}.
	}
	Since $(\beta_{3}+q_3-1)_{q_3}=\left(d_{[2]}(H)+2\delta-1\right)_{d_{[2]}(H)-a_{[2]}}$ and \chRounak{$q_{m_{[n]}}+1=0$}, {we can simplify \eqref{for:prop:equiv:PU:PAr:5} as}
	\eqan{\label{for:prop:equiv:PU:PAr:6}
		&\prod\limits_{s\in\left[3,m_{[n]}\right]}\frac{(\beta_s+q_s-1)_{q_s}}{(\alpha_s+\beta_s+p_s+q_s-1)_{p_s+q_s}}\nn\\
		=&\left(d_{[2]}(H)+2\delta-1\right)_{d_{[2]}(H)-a_{[2]}}\\
		&\qquad\qquad\times\prod\limits_{u\in[3,n]}\prod\limits_{j\in [m_u]}\frac{1}{a_{[2]}+2\left( m_{[u-1]}+j-2 \right)+(u-1)\delta+\frac{j}{m_u}\delta-1}.\nonumber
	}
	Therefore, substituting these simplified expressions from \eqref{for:prop:equiv:PU:PAr:6} and \eqref{for:prop:equiv:PU:PAr:1} in \eqref{eq:lem:PU:probability:0}, we obtain:
	\eqan{\label{for:prop:equiv:PU:PAr:7}
		&\prob_m\left( \PU_{m_{[n]}}^{\sss\rm{(SL)}}(\boldsymbol{1},\boldsymbol{\psi})\overset{\star}{\simeq} H \right)\nn \\
		=& \prod\limits_{u\in[1,n]}\prod\limits_{j\in \left[ m_u \right]}\prod\limits_{i=a_{m_{[u-1]}+j}}^{d_{m_{[u-1]}+j}(H)-1} \left( i+ \frac{\delta}{m_u} \right)\\
		&\hspace{2cm}\times \prod\limits_{u\in[3,n]}\prod\limits_{j\in [m_u]}\frac{1}{a_{[2]}+2\left( m_{[u-1]}+j-2 \right)+(u-1)\delta+\frac{j}{m_u}\delta-1}~,\nn
	}
	which, matches the expression of $\prob_m( \PArs_{m_{[n]}}\left( \boldsymbol{m},1,\delta \right)\overset{\star}{\simeq} H )$ in \eqref{for:equiv:CPU:PArs:6}.
\end{proof}
\section{Equivalence for models (B) and (D)}\label{chap:PUR:sec:equivalence-other-models}
{{Following similar calculations, we can show that models (B) and (D) are equal in distribution to $\CPU^{\sss\rm{(NSL)}}$ and $\PU^{\sss\rm{(NSL)}}$, respectively.}}
$\PArt_n (\boldsymbol{m},\delta)$ is equal in distribution {to} $\CPU_n^{\sss\rm{(NSL)}}(\boldsymbol{m},{\boldsymbol{\psi}})$ where $\boldsymbol{\psi}$ is a sequence of $\Beta$ random variables defined in \eqref{def:psi:1} and \eqref{def:psi:2}:
\begin{Theorem}[Equivalence for model (B) and $\CPU^{\sss\rm{(NSL)}}$]\label{thm:equiv:CPU:PArt}
	Conditionally on $\boldsymbol{m}$, for any graph $G\in\mathcal{H}_{n}$,
	\eqn{\label{eq:thm:equiv:PArt:CPU}
		\prob_m\left( \PArt_n(\boldsymbol{m},\delta) \overset{\star}{\simeq} G \right) = \prob_m\left( \CPU_{n}^{\sss\mathrm{(NSL)}}(\boldsymbol{m},\boldsymbol{\psi}) \overset{\star}{\simeq} G \right).}
\end{Theorem}
Next, we describe the equivalence of $\PAri_n(\boldsymbol{m},\delta)$ and $\PU_n^{\sss\rm{(SL)}}(\boldsymbol{m},\boldsymbol{\psi})$ where $\boldsymbol{\psi}$ is a sequence of independent $\Beta$ random variables defined as
\eqan{\label{eq:def:psi:1}
		&\psi_v\sim\Beta\left( m_v+\delta, a_{[2]}+2\left( m_{[v-1]}-2 \right)+m_v+(v-1)\delta \right)\quad \mbox{for }v\geq 3,\nn\\
		\text{and}\quad&\psi_2\sim\Beta(a_2+\delta,a_1+\delta),\qquad\text{and}~\psi_1=1.
}
\begin{Theorem}[Equivalence of model (D) and $\PU^{\sss\rm{(NSL)}}$]\label{thm:equiv:PU:PAri}
	Conditionally on $\boldsymbol{m}$, for any graph $G\in\mathcal{H}_{n}$,
	\eqn{\label{eq:thm:equiv:PAri:CPU}
		\prob_m\left( \PAri_n(\boldsymbol{m},\delta) \overset{\star}{\simeq} G \right) = \prob_m\left( \PU_{n}^{\sss\mathrm{(NSL)}}(\boldsymbol{m},\boldsymbol{\psi}) \overset{\star}{\simeq} G \right),}
	where $\boldsymbol{\psi}$ is the sequence of $\Beta$ random variables defined in \eqref{eq:def:psi:1}.
\end{Theorem}
The proofs of Theorem~\ref{thm:equiv:CPU:PArt} and \ref{thm:equiv:PU:PAri} follow in exactly the same way as that of Theorem~\ref{thm:equiv:CPU:PArs}. {Here we show that models (D) and (B) are equal in distribution with $\PU^{(\rm{NSL})}$ and $\CPU^{(\rm{NSL})}$ respectively. Recall that, $\CPU^{(\rm{NSL})}$ could be obtained by collapsing a special case $\PU^{(\rm{NSL})}$. So, we prove Lemma~\ref{lem:edge Probability:Polya Urn graph:NSL}-\ref{lem:PU:conditional:graph_probability:NSL} and Corollary~\ref{cor:cond:prob:CPU:NSL} for general choice of $\boldsymbol{m}$ and $\boldsymbol{\psi}$. Later, while proving Theorem~\ref{thm:equiv:CPU:PArt}, we shall use these lemmata and this corollary with $\boldsymbol{m}=\boldsymbol{1}$, whereas for proving Theorem~\ref{thm:equiv:PU:PAri}, we continue with the $\boldsymbol{m}$ of model (D).} {The following} lemmata and propositions are the $\PU^{\sss\rm{(NSL)}}$ {analogues} of the lemmata and propositions proved in Section~\ref{chap:PUR:sec:equivalence-model(A)}. The conditional {edge-connection probabilities} for $\PU^{\sss\rm{(NSL)}}$ {are} given {as follows:}
\begin{Lemma}[Conditional edge-connection {probabilities} of $\PU^{\sss\rm{(NSL)}}$]\label{lem:edge Probability:Polya Urn graph:NSL}
	Conditionally on $\boldsymbol{m}$ and $(\psi_k)_{k\geq 1}$, the probability of connecting an edge from $v$ to $u$ in $\PU_{n}^{\sss\rm{(NSL)}}(\boldsymbol{m},\boldsymbol{\psi})$, is given by $\psi_v (1-\psi)_{(v,u)}$.
\end{Lemma}
The proof to this lemma {is identical to that of} Lemma~\ref{lem:edge Probability:Polya Urn graph}, {using} the definition of  $\PU^{\sss\rm{(NSL)}}$. 
Now, we define the set of all vertex and edge-marked graphs on $n$ vertices by $\mathcal{H}_n^{ALOE}$. Let $\overline{\PU}^{\sss\rm{(NSL)}}$ and $\bar{\PAri}_{n}$ denote the edge-marked version of $\PU^{\sss\rm{(NSL)}}$ and $\PAri_{n}$ respectively.
{Similarly} as Lemma~\ref{lem:PU:conditional:graph_probability} 
the above lemma allows us to compute the conditional {law} of $\PU^{\sss\rm{(NSL)}}$:
\begin{Lemma}[Conditional graph probability of $\overline{\PU}^{\sss\rm{(NSL)}}$]\label{lem:PU:conditional:graph_probability:NSL}
	For any graph $H\in\mathcal{H}_{n}^{ALOE}$,
	\begin{equation}
		\label{eq:PU:conditional:graph_probability}
		\prob_{m,\psi}\left( \overline{\PU}_{n}^{\rm\sss(NSL)}(\boldsymbol{m},\boldsymbol{\psi}) \overset{\star}{\simeq} H \right) =  \prod\limits_{s\in[2,n]} \psi_s^{p_s}(1-\psi_s)^{q_s},
	\end{equation}
	where
	\begin{align*}
		p_s=&~ d_s(H)-f_s,\qquad\text{and}\qquad
		q_s=\sum\limits_{u \in\left( 2,n\right]}\sum\limits_{j=1}^{m_u}\one_{\{s\in(v(u,j),u)\}},
	\end{align*}
	where $v(u,j)$ is the vertex to which the $j$-th out-edge of $u$ connects and $f_s=m_s$ for all $s\geq 3$ and $f_1=a_1$ and $f_2=a_2$ denote the degrees of the vertex $1$ and $2$ in the initial graph $G_0$.
\end{Lemma}
Since the {proof strategy of} this lemma {is identical to that of} Lemma~\ref{lem:PU:conditional:graph_probability}, we omit the proof to this lemma also. {Using} Lemma~\ref{lemma:Beta Expectation}, conditionally on $\boldsymbol{m},$ the graph probability of $\overline{\PU}^{\sss\rm{(NSL)}}$ is computed as follows:

\begin{Corollary}\label{cor:cond:prob:CPU:NSL}
	For
	$H\in\mathcal{H}_n$,
	\eqn{\label{eq:lem:PU:probability}
		\prob_{m}\left( \overline{\PU}_{n}^{\rm\sss(NSL)}(\boldsymbol{m},\boldsymbol{\psi}) \overset{\star}{\simeq} H \right) = \prod\limits_{s\in \left[ 2  ,n \right)} \frac{(\alpha_s+p_s-1)_{p_s}(\beta_s+q_s-1)_{q_s}}{(\alpha_s+\beta_s+p_s+q_s-1)_{p_s+q_s}},
	}
	where $\alpha_s$ and $\beta_s$ are the first and second parameters of the $\Beta$ random variables
	and $p_s, q_s$ are defined in Lemma~\ref{lem:PU:conditional:graph_probability:NSL}.
\end{Corollary}
Following the {steps} in \eqref{for:equiv:CPU:PAr:8}-\eqref{for:equiv:CPU:PAr:10},
we can simplify $q_s$ as
\eqn{
	\label{q_s:for:PU_NSL}
	q_s = d_{[s-1]}(H) - a_{[2]}-2\left( m_{[s-1]}-2 \right)-m_{s}\qquad \text{for}~s\in [3,n],
}
{with $q_2=d_1(H)-a_1$, and $q_s$ satisfies}
\eqn{
	\label{q_:p_s:reccursion}
	p_s+q_s=q_{s+1}+m_{s+1} \qquad \text{for}~s\in[3,n).
}
Note that, each vertex in $\PU_{m_{[n]}}^{\sss\mathrm{(NSL)}}(\boldsymbol{1},\boldsymbol{\psi})$ has only one-out-edge. Therefore, edge-marks are redundant in this case. Now, {we have} all tools to {adapt} Proposition~\ref{prop:equiv:CPU:PArs} to $\PU^{\sss\rm{(NSL)}}$:
\begin{Proposition}[Equivalence of pre-collapsed model (B) and $\PU^{\sss\rm{(NSL)}}$]\label{prop:equiv:CPU:PArt}
	For any graph $H\in\mathcal{H}_{m_{[n]}}$,
	\eqn{\label{prop:eq:equiv:CPU:PArt}
		\prob_m\left( \PArt_n(\boldsymbol{m},1,\delta) \overset{\star}{\simeq} H \right) = \prob_m\left( \PU_{m_{[n]}}^{\sss\mathrm{(NSL)}}(\boldsymbol{1},\boldsymbol{\psi}) \overset{\star}{\simeq} H \right).}
\end{Proposition}
\begin{proof}
	Following a similar calculation {as the one leading to \eqref{for:equiv:CPU:PArs:5}}, 
	\eqan{\label{for:prop:equiv:PA:PU:NSL}
		&\prob_m\left( \PArt_n(\boldsymbol{m},1,\delta) \overset{\star}{\simeq} H \right)\nn\\
		=~& \prod\limits_{u\in[n]}\prod\limits_{j\in[m_{u}]} \prod\limits_{i=a_{m_{[u-1]}+j}}^{d_{m_{[u-1]}+j}(H)-1}\left( i+\frac{\delta}{m_u} \right)\\
		&\hspace{1.5cm}\times\prod\limits_{u\in[3,n]}\prod\limits_{j\in[m_{u]}} \frac{1}{a_{[2]}+2(m_{[u-1]}+j-3)+\left( (u-1)+\frac{j-1}{m_u} \right)\delta}~. \nn
	}
	{For model (B),} we consider $\PU_{m_{[n]}}^{\sss\rm{(NSL)}}(\boldsymbol{1},\boldsymbol{\psi}),$ where $\boldsymbol{\psi}$ is the sequence of $\Beta$ variables defined in \eqref{def:psi:1} and \eqref{def:psi:2}. Therefore, $p_s=d_s(H)-1$ for $s\geq 1$, and 
	\[
	q_s= d_{[s-1]}(H) - a_{[2]}-2(s-3)-1\qquad\text{for}~s\geq 3.
	\]
	By Corollary~\ref{cor:cond:prob:CPU:NSL},
	\eqn{\label{for:prop:equiv:CPU:PArt:1}
		\prob_{m}\left( \PU_{m_{[n]}}^{\rm\sss(NSL)}(\boldsymbol{1},\boldsymbol{\psi}) \overset{\star}{\simeq} H \right) = \prod\limits_{s\in \left[ 2  ,m_{[n]} \right)} \frac{(\alpha_s+p_s-1)_{p_s}(\beta_s+q_s-1)_{q_s}}{(\alpha_s+\beta_s+p_s+q_s-1)_{p_s+q_s}},
	}
	where $\alpha_s$ and $\beta_s$ are the first and second parameters of the $\Beta$ variable $\psi_s$ defined in \eqref{def:psi:1} and \eqref{def:psi:2} respectively. Then, by {\eqref{q_:p_s:reccursion}, the recursion \eqref{for:prop:equiv:PU:PAr:4} again holds.}
		Now, following the calculations {in \eqref{for:prop:equiv:PU:PAr:5}--\eqref{for:prop:equiv:PU:PAr:7},} and substituting the values of $\alpha_s,\beta_s,p_s,q_s$, it follows immediately that
		\eqan{\label{for:prop:equiv:CPU:PArt:3}
			&\prob_{m}\left( \PU_{m_{[n]}}^{\rm\sss(NSL)}(\boldsymbol{1},\boldsymbol{\psi}) \overset{\star}{\simeq} H \right)\nn\\
			=& \prod\limits_{u\in[n]}\prod\limits_{j\in[m_{u}]} \prod\limits_{i=a_{m_{[u-1]}+j}}^{d_{m_{[u-1]}+j}(H)-1}\left( i+\frac{\delta}{m_u} \right)\\
			&\hspace{1.5cm}\times\prod\limits_{u\in[3,n]}\prod\limits_{j\in[m_{u}]} \frac{1}{a_{[2]}+2(m_{[u-1]}+j-3)+\left( (u-1)+\frac{j-1}{m_u} \right)\delta}~,\nn
		}
		{as required.}
	\end{proof}
	\begin{proof}[\bf Proof of Theorem~\ref{thm:equiv:CPU:PArt}]
		Theorem~\ref{thm:equiv:CPU:PArt} follows immediately from Proposition~\ref{prop:equiv:CPU:PArt} {in exactly the same} way {as} Theorem~\ref{thm:equiv:CPU:PArs} follows from Proposition~\ref{prop:equiv:CPU:PArs}.
	\end{proof}
	Unlike model (B), in Theorem~\ref{thm:equiv:PU:PAri}, we couple model (D) to $\PU_{n}^{\rm\sss(NSL)}(\boldsymbol{m},\boldsymbol{\psi})$, where the edge-marks are relevant. Therefore, for proving Theorem~\ref{thm:equiv:PU:PAri}, we first prove that the coupling holds true for edge-marked $\PAri$ and ${\PU}^{\sss \rm{(NSL)}}$. Next, we sum over all possible edge-mark permutations to complete the proof of Theorem~\ref{thm:equiv:PU:PAri}. 
	\begin{proof}[\bf Proof of Theorem~\ref{thm:equiv:PU:PAri}]
		Conditionally on $\boldsymbol{m},$ the distribution of model (D) is
		\eqan{\label{eq:PAri:density:1}
				&\prob_m\left(  \overline{\PA}^{\sss \rm{(D)}}_{n}(\boldsymbol{m},\delta) \overset{\star}{\simeq} H \right) \nn\\
				=~& \prod\limits_{u\in[3,n]}\prod\limits_{j\in[m_{u}]} \frac{d_{v(u,j)}(u,j-1)+\delta}{a_{[2]}+2(m_{[u-1]}-2)+(j-1)+(u-1)\delta}~,
		}
		where $v(u,j)$ is the vertex in $[u-1]$ to which $u$ connects with its $j$-th edge and $d_v(u,j)$ denotes the degree of the vertex $v$ in $\overline{\PA}^{\sss \rm{(D)}}_{u,j}(\boldsymbol{m},\delta).$ Rearranging the numerators of RHS of \eqref{eq:PAri:density:1} gives,
		\eqan{\label{eq:PAri:density:2}
				&\prob_m\left(  \overline{\PA}^{\sss \rm{(D)}}_{n}(\boldsymbol{m},\delta) \overset{\star}{\simeq} H \right)\\
				=~& \prod\limits_{u\in[n]}\left(\prod\limits_{i=f_u}^{d_u(H)-1}(i+\delta)\right)\nn\\
				&\hspace{2cm}\times\prod\limits_{u\in[3,n]}\left(\prod\limits_{j=1}^{m_u}\frac{1}{a_{[2]}+2(m_{[u-1]}-2)+(j-1)+(u-1)\delta}\right).\nn
		}
		\FC{For model (D), we use the $\boldsymbol{\psi}$ defined in \eqref{eq:def:psi:1}}. Using Corollary~\ref{cor:cond:prob:CPU:NSL}, we calculate the conditional {distribution} of $\overline{\PU}_n^{\sss\rm{(NSL)}}(\boldsymbol{m},\boldsymbol{\psi})$ as
		\eqn{\label{PU:density:1}
			\prob_{m}\left( \overline{\PU}_{n}^{\rm\sss(NSL)}(\boldsymbol{m},\boldsymbol{\psi}) \overset{\star}{\simeq} H \right) = \prod\limits_{s\in \left[ 2  ,n \right)} \frac{(\alpha_s+p_s-1)_{p_s}(\beta_s+q_s-1)_{q_s}}{(\alpha_s+\beta_s+p_s+q_s-1)_{p_s+q_s}},}
		where, for $s\geq 3,$
		\eqan{\label{parameters:PU:D}
				\alpha_s = m_s+\delta,\quad&\text{and}\quad \beta_s = a_{[2]}+2\left( m_{[s]}-2 \right)+m_s+(s-1)\delta,\nn\\
				p_s = d_{s}(H)-f_s,\quad&\text{and}\quad q_s = d_{[s-1]}(H)-\left( a_{[2]}+2\left( m_{[s]}-2 \right) \right)+m_s.
		}
		{Therefore, for $s\geq 2$, the recursion relations in \eqref{for:prop:equiv:PU:PAr:4} becomes}
		\eqan{
				\alpha_s+\beta_s=&\beta_{s+1}-m_{s+1},\label{recursion:2}\\
				\mbox{and,}\qquad p_s+q_s=&q_{s+1}+m_{s+1}.\label{recursion:2-1}
		}
		{This gives us all the necessary tools} to prove Theorem~\ref{thm:equiv:PU:PAri}. From the recursion relation in \eqref{recursion:2}, it follows that
		\eqn{\label{for:thm:PU:PAri:1}
			\begin{split}
				\frac{(\beta_s+q_s-1)_{q_s}}{(\alpha_s+\beta_s+p_s+q_s-1)_{q_s+p_s}}=&\frac{(\beta_s+q_s-1)_{q_s}}{(\beta_{s+1}+q_{s+1}-1)_{q_{s+1}+m_{s+1}}}\\
				=& \frac{1}{(\beta_{s+1}-1)_{m_{s+1}}}\frac{(\beta_s+q_s-1)_{q_s}}{(\beta_{s+1}+q_{s+1}-1)_{q_{s+1}}}.
		\end{split}}
		On the other hand, the first factor in RHS of \eqref{PU:density:1} can be  rewritten as
		\eqan{\label{for:thm:PU:PAri:2}
				&\frac{(\alpha_2+p_2-1)_{p_2}(\beta_2+q_2-1)_{q_2}}{(\alpha_2+\beta_2+p_2+q_2)_{p_2+q_2}}\nn\\
				=~&\frac{1}{(\beta_3+q_3-1)_{q_3}}\prod\limits_{i=0}^{m_{3}-1}\frac{1}{a_{[2]}+i+2\delta}\prod\limits_{i\in[2]}\prod\limits_{j=f_i}^{d_i(H)-1}(j+\delta) .
		}
		Hence substituting the simplifications obtained from \eqref{for:thm:PU:PAri:1} and \eqref{for:thm:PU:PAri:2} in \eqref{PU:density:1},
		\eqan{\label{for:thm:PU:PAri:3}
				&\prob_{m}\left( \overline{\PU}_{n}^{\rm\sss(NSL)}(\boldsymbol{m},\boldsymbol{\psi}) \overset{\star}{\simeq} H \right)\\
				=& \prod\limits_{u\in[n]}\prod\limits_{i=f_u}^{d_u(H)-1}(i+\delta)\prod\limits_{u\in[3,n]}\prod\limits_{j=1}^{m_u}\frac{1}{a_{[2]}+2(m_{[u-1]}-2)+(j-1)+(u-1)\delta},\nn
		}
		{that proves both $\overline{\PA}^{\sss \rm{(D)}}_n(\bf{m},\delta)$ and $\overline{\PU}_n^{\sss \rm{(NSL)}}(\bf{m},\delta)$ have the same law. Now, summing over all possible permutations of edge-connections of $H\in\mathcal{H}_n$, Theorem~\ref{thm:equiv:PU:PAri} is proved.}
	\end{proof}
\chapter[Local convergence to random P\'{o}lya point tree]{Local convergence to \\random P\'{o}lya point tree}
\label{chap:convergence}
\begin{flushright}
\footnotesize{}Based on:\\
\textbf{Universality of the local limit of\\
preferential attachment models}\cite{RRR22}
\end{flushright}
\begin{center}
	\begin{minipage}{0.7 \textwidth}
		\footnotesize{\textbf{Abstract.}
			This chapter proves the local convergence of the generalised preferential attachment models to the random P\'{o}lya point tree. We prove this using the second moment method. Additionally, we prove that the local convergence in terms of graph-density convergence. Further, as a consequence of local convergence, we derive the power-law exponent for the degree of a uniformly chosen vertex in the graph and neighbours of this uniformly chosen vertex.
			}
	\end{minipage}
\end{center}
\vspace{0.1cm}

\section{Introduction}\label{chap:convergence:sec:introduction}

Similar to many other random graph models, such as the configuration model, PAMs are called {\em locally tree-like graphs}, meaning that the neighbourhood of the majority of vertices is structured as a tree (up to a certain distance). Berger et al.\ initiated the study of local convergence of PAMs {in \cite{BergerBorgs}}. They showed that the finite neighbourhood of the graph converges to the corresponding neighbourhood of the P\'olya point tree (see {the description of the P\'olya point tree} in Section~\ref{chap:LWC:sec:RPPT}).

The main aim of this chapter is to {extend} the local convergence proof in \cite{BergerBorgs} {to a more general class of PAMs (including random out-degrees and related dynamics), as defined in Section~\ref{chap:intro:sec:model}.} {We achieve this by explicitly computing the} density of neighbourhoods of the PAMs. This helps us to extend the result in \cite{BergerBorgs} to models where one can accommodate negative fitness parameters and random {out-degrees}. The limiting random tree is an extension of the P\'olya point tree described in \cite{BergerBorgs}, which we call the {\em random P\'olya point tree}, as defined in Section~\ref{chap:LWC:sec:RPPT}.

The randomness of the {out-degrees} provides a surprising {size-biasing effect} in the limiting random tree. We show that there is {a} universal description of the limit by considering many possible affine variants of the PAMs. Additionally, we study \FC{the} {\em vertex-marked local convergence in probability} of the PAMs, which is an extension to the local convergence shown in \cite{BergerBorgs}. Here, the marks denote the {\em ages} of the vertices in the tree, and we prove convergence of the joint \emph{densities} of these ages.

The asymptotic degree distribution is a direct consequence of our local limit results. Deijfen et al. computed the power-law-exponent of asymptotic degree distribution for generalised versions of preferential attachment models in \cite{DEH09}. We re-establish the same power-law exponent using local limit properties. We extend these power law exponent results for neighbours of uniformly chosen random vertices for preferential attachment models using our local limit result.

Results on local convergence of related PAMs have been \FC{established} by Y. Y. Lo \cite{tiffany2021}, who analysed the local limit of the preferential attachment trees with i.i.d.\ random fitness parameter $\delta$. Rudas, T\'oth and Valk\'o \cite{RTV07} proved local convergence almost surely for general preferential attachment trees, based on a continuous-time embedding, which gives a continuous-time branching process by Jagers and Nerman \cite{JN84} in full generality. 
\begin{center}
	{\textbf{Organisation of the chapter}}
\end{center}
The chapter is organized as follows. In Section~\ref{chap:convergence:sec:local-convergence-result}, we discuss the local convergence result of generalised preferential attachment models and draw an outline of the proof to this local limit theorem. In Section~\ref{chap:convergence:sec:prelim}, we prove some preliminary results that we use heavily in the proof of the local convergence theorem. Section~\ref{chap:convergence:sec:local_convergence} provides the proof of our local convergence theorem described in Section~\ref{chap:convergence:sec:local-convergence-result} for models (A), (B), and (D). 
Lastly, in Section~\ref{chap:convergence:sec:degree-distribution}, we prove the power-law exponents for the asymptotic-degree distribution as a consequence of our local limit result.

\section{Local convergence of PAM}\label{chap:convergence:sec:local-convergence-result}
In this section, we present the local convergence theorem for generalised preferential attachment models and discuss various aspects of this theorem. Additionally, we provide a brief outline of the proof of the local convergence theorem.

\begin{Theorem}[\bf Local convergence theorem for PA models]\label{theorem:PArs}
	Let $M$ be an $\N$-valued random variable with finite $p$-th moment for some $p>1$ and $\delta>-\infsupp(M)$. Then the preferential attachment models $(A), (B)$ and $(D)$ of converge vertex-marked locally {in probability} to the random P\'olya point tree with parameters $M$ and $\delta$. 
\end{Theorem}
\paragraph{\bf Observations:}We make some remarks about the above result:
\begin{enumerate}
	\item Our proof uses {the finiteness} of the $p$-th moment for the proofs of some of the concentration bounds around the mean. It would be interesting to identify the precise necessary condition for the local limit result to hold.
	\item Berger et al.\ in \cite{BergerBorgs} assumed the fitness parameter $\delta$ to be non-negative, but here {we} allow for negative $\delta$.
	Note that this accommodates infinite-variance degree distributions used in \cite{jordan06,mori2005maximum,waclaw-sokolov}, and suggested in many applied works, see e.g.\ \cite{dorogovtsev2008,dorogovtsev2000,voitalov2019} and the references therein.
	\item {The \emph{vertex-mark} of the vertex $k$ in a preferential attachment model of size $n$ is $k/n$, limit of which represents \emph{age} of nodes in the limiting graph, random P\'{o}lya point tree.}
	\item {Berger et al.\ in \cite{BergerBorgs}} have shown that $\PA_n^{\sss ({m},\delta)}(d)$ converges locally {in probability} to {the} P\'olya point tree with parameters $m$ and $\delta$. 
	{Restricting to degenerate distributions,} our result {can} be viewed as an extension of \cite{BergerBorgs} to all preferential attachment models. Moreover our model considers the case where every vertex comes with an i.i.d.\ number of out-edges which has only finite $p$-th moment for some $p>1$ and we have considered any general starting graph $G_0$ of size $2$. If we do not assume that the initial graph is of size $2,$ then it increases the computational complexity, and hence we avoid this complication.
	\item We provide a proof in detail for Model (A). {The proof for} the models (B) and (D) is very similar and we {only} indicate the {necessary} changes. The fact that all these models have the same local limit is a sign of \emph{universality}.
	\item We prove a \emph{local density} result, which is stronger than in our main result, for models (A), (B) and (D), and interesting in its own right.
\end{enumerate}
\paragraph{\bf Idea of proof of Theorem~\ref{theorem:PArs}}
For any vertex-marked finite graph $\left(\tree,(a_\omega)_{\omega\in V(\tree)}\right)$ and $r\in \N,$ define 
\eqan{
N_{r,n}\left(\tree,(a_\omega)_{\omega\in V(\tree)}\right)=&\sum\limits_{v\in[n]}\one_{\left\{ B_r^{\sss(G_n)}(v)~\simeq~ \tree, ~|v_\omega/n-a_\omega|\leq 1/r~\forall\omega~\in~V(t) \right\}}~,\nn\\
=&\sum\limits_{v\in[n]}\one_{\big\{ d_{\mathcal{G}_\star}\big( \big(B_r^{\sss(G_n)}(v),v,u\mapsto u/n\big),\big(\tree,\emp,(a_\omega)_{\omega\in V(\tree)}\big)\big)\leq 1/r \big\}}\nn
}
where $G_n$ is taken as $\PA_n$ and $v_\omega$ is the vertex in $G_n$ corresponding to $\omega \in V(\tree)$. Then, {to prove} Theorem~\ref{theorem:PArs}, {by} Definition~\ref{def:vertex:marked:local:convergence} it is enough to show that ${N_{r,n}\left(\tree,(a_\omega)_{\omega\in V(\tree)}\right)/n}$ converges in probability to $\mu\left( B_r^{\sss(G)}(\emp)\simeq~\tree, ~|A_\omega-a_\omega|\leq 1/r~\forall\omega~\in~V(t) \right)$, where $\mu$ is the law of the limiting $\RPPT$ graph, and $A_\omega$ is the mark of the vertex in $\RPPT$ corresponding to $\omega\in V(\tree)$ and $\tree$ is a tree. We aim to prove this convergence using a second moment method, i.e., {we will prove that}
\eqan{
	&\frac{1}{n}\E\left[N_{r,n}\left(\tree,(a_\omega)_{\omega\in V(\tree)}\right) \right] \nn\\
	&\hspace{2cm}\to ~\mu\left( B_r^{\sss(G)}(\emp)\simeq~\tree, ~|A_\omega-a_\omega|\leq 1/r~\forall\omega~\in~V(t) \right),\label{eq:second:moment:explain:1}\\
	\text{and}~~&\frac{1}{n^2}\E\left[ N_{r,n}\left(\tree,(a_\omega)_{\omega\in V(\tree)}\right)^2\right] \nn\\
	&\hspace{2cm}\to~ \mu\left( B_r^{\sss(G)}(\emp)\simeq~\tree, ~|A_\omega-a_\omega|\leq 1/r~\forall\omega~\in~V(t) \right)^2.\label{eq:second:moment:explain:2}}
For proving the first, we show that the {\em joint density} of the {ages of the vertices in the} $r$-neighbourhood of a uniformly chosen vertex in $\PAM$ converges to that of $\RPPT$. {Calculating this joint density explicitly is quite} involved because of {the} dependence structure in the edge-connection probabilities of $\PAM$ models. In \cite{BBCS05}, the authors provide a P\'olya urn representation of model (D), {which made the proof in \cite{BergerBorgs} possible} since this representation {implies that the edges are {\em conditionally independent}}. {An essential step in our proof is thus to} construct similar P\'olya urn descriptions for models (A), (B) and (D), which we have already provided in Chapter~\ref{chap:PU_representation} in full detail. With this distributional equivalence in hand, we {can now compute the above joint density,} and show that {it} converges to that of the $\RPPT$.

{For the second moment,} we first expand the {square of the sum arising in the} numerator. From the expansion we observe that along with some vanishing terms, we obtain the joint density of the $r$-neighbourhoods of {\em two} uniformly chosen vertices. Next we prove that the $r$-neighbourhoods of two uniformly chosen vertices are disjoint with high probability (whp). Again, for the joint density calculation, we use the P\'olya urn description of {our} models. Since the edge-connection {events} are conditionally independent {by the P\'olya urn representations}, the two neighbourhoods are {\em conditionally independent} when they are disjoint. Therefore the joint density {factorizes} and we obtain the required result.

Though our main steps for proving the theorem {are the} same as {those in} \cite{BergerBorgs}, our proof techniques differ {significantly}, for example, we avoid the induction argument in the neighbourhood size. We use a coupling of the preferential attachment model with the P\'olya urn graphs through an explicit density computation. To summarize, we have two crucial steps in proving the main theorem:
\begin{itemize}
	\item[(a)] \textbf{Equivalence: }There exist P\'olya urn representations for models (A), (B), and (D), which we have already proved in Chapter~\ref{chap:PU_representation}.
	\item[(b)] \textbf{Convergence: }The joint density of the ages of the vertices in the $r$-neighbourhood of the P\'olya urn graphs converges to that of the RPPT, and this is the main content of this chapter.
\end{itemize}


For model (E), the edges are connected independently and the degrees of the older vertices are updated only after the new vertex is included in the graph with {\em all} its out-edges. This edge-connection procedure is different from other models and we do not have a P\'olya urn representation for this graph. 
We prove the convergence of models with a P\'olya urn representation. 
We do not, however, prove local convergence for models (E) and (F) here. A \emph{heuristic approach} would be to couple models (D) and (E) such that the probabilities of observing the $r$-neighbourhoods of a uniformly chosen vertex in the graphs are asymptotically equal. Since we have proven local convergence for model (D), the convergence for model (E) will follow immediately. Note, however, that this does \emph{not} imply the convergence of the joint density of the ages of vertices in the $r$-neighbourhood for model (E). Model (F) can be addressed similarly to model (E).

\section{Preliminary results}\label{chap:convergence:sec:prelim}
For analysing the convergence to the local limit, we need some analytical results for both $\CPU$ and $\RPPT$. Some of the proofs here follow those in \cite{BergerBorgs}, while others include significant novel ideas. We explain how these results can be reproduced for $\CPU^{\sss\rm{(NSL)}}$ or $\PU$ at the end of this section. This section is organised as follows. In Section \ref{subsec-prel-exp}, we provide auxiliary results on the expected values of random variables. In Section \ref{subsec-pos-conc-Beta-gamma-coupling}, we use these results to analyse the asymptotics of the positions $\mathcal{S}_k^{\sss(n)}$ and for an effective coupling of Beta and Gamma variables. Finally, in Section \ref{sec-attachment-RPPT}, we use these results to study the asymptotics of the attachment probabilities, and we prove some regularity properties of the RPPT.
\subsection{Preliminaries on expectations of random variables}
\label{subsec-prel-exp}
Conditionally on $\boldsymbol{m},$ let $(\chi_k)_{k\in\N}$ be a sequence of independent Gamma random variables with parameters $m_k+\delta$ and $1$.
Since $M$ has finite $p$-th moment, the $p$-th moment of all $\left( \chi_k \right)_{k\geq 1}$ are {finite as well:}
\begin{Lemma}\label{lem:finite_moment}
	{The random variables} $\left( \chi_k \right)_{k\geq 1}$ have {uniformly bounded} $p$-th moment.
\end{Lemma}
\begin{proof}
	Fix $u\in\N$. Let, conditionally on $\boldsymbol{m},\ X_1,\ldots,X_{m_u}$ be independent Gamma random variables with parameters $1+{\delta/m_u}$ and $2\E[M]+\delta$. Hence
	\[
	\E_m\left[ X_1^{p} \right] = \frac{\Gamma\left( 1+\frac{\delta}{m_u}+p \right)}{\Gamma\left( 1+\frac{\delta}{m_u} \right)}\leq \frac{\Gamma(1+|\delta|+p)}{A}<K,
	\]
	where
	\begin{align*}
		A=& \min\left\{ \Gamma\left(1+\left| \frac{\delta}{m} \right|\right),\Gamma\left(1-\left| \frac{\delta}{m} \right|\right) \right\},
	\end{align*}
	and $m=\infsupp(M)$. By the triangle inequality, $\E_m\left[|X_1-\E_m[X_1]|^{p}\right]$ is bounded and the upper bound is independent of $u$ and $m_u$. {By} \cite[Corollary 8.2]{allangut}, for $X_1,\ldots, X_n$ i.i.d.\ mean $0$ random variables with finite $\ell$-th moment, there exists a constant $B_{\ell}<\infty$ depending only on $\ell$ such that
	\eqn{\label{allan:gut:result}
		\E\big[|X_1+\cdots+X_n|^\ell\big]\leq \begin{cases}
			& B_{\ell} n \E[|X_1|^\ell] \quad\mbox{for }1\leq \ell \leq 2,\\
			& B_{\ell} n^{\ell/2}\E\left[|X_1|^{\ell/2}\right] \quad\mbox{for }\ell>2.
	\end{cases}}
	Since $\chi_k=X_1+\cdots+X_{m_k}$ and $X_1,\ldots,X_{m_k}$ are i.i.d.\ random variables, there exists $B_{p}$ finite constant depending only on $p$ such that
	\eqn{\E_m\left[ |\chi_k-\E_m[\chi_k]|^{p} \right] \leq B_{p} m_k^{\FC{p}} \E_m\left[|X_1-\E_m[X_1]|^{p}\right]\leq  C_1m_k^{\FC{p}}~,}
	\FC{for some $C_1>0$}. Using {the} triangle inequality {for the  $L_{p}$-norm},
	\eqn{\label{for:lem:finite_moment:1}
		\E_m\left[ \chi_k^{p} \right] \leq \E_m\left[ \chi_k \right]^{p} + C_1m_k^{\FC{p}} \leq C m_k^{p},}
	\FC{for some $C>0$.} Since we have assumed the existence of the $p$-th moment of $M$, \eqref{for:lem:finite_moment:1} implies that  $\E\left[ \chi_k^{p} \right]$ is uniformly bounded from above.
\end{proof}
Next, we prove an auxiliary lemma that we will use several times, and which is a direct application of the dominated convergence theorem and strong law of large numbers:
\begin{Lemma}
	\label{prop:M-inverse:expectation}
	Let $X_1, X_2,\ldots$ be a sequence of i.i.d.\ random variables with $X_1>c$ for some $c>0$ a.s.\ and finite mean. Then, with $X_{[n]}=X_1+\cdots+X_n$,
	\eqn{\label{eq:prop:M-inverse:expectation}\E\left[ \frac{1}{X_{[n]}} \right] = \left( 1+o(1) \right)\frac{1}{n\E[X_1]}~.}
\end{Lemma}
\begin{proof}
	Note that $X_i\geq c>0$, hence both $1/\E[X]$ and $n/X_{[n]}$ have upper bounds $1/c$. By the strong law of large numbers,
	\[
	\frac{X_{[n]}}{n}\overset{a.s.}{\to} \E[X_1].
	\]
	Since $\E[X_1]>0$ and $n/X_{[n]}\leq 1/c$, by the dominated convergence theorem,
	\begin{align*}
		\E\left[ \frac{n}{X_{[n]}} \right] =& (1+o(1))\frac{1}{\E[X_1]}.
	\end{align*}
	Hence, the lemma follows immediately.
\end{proof}
\subsection{Position concentration and Gamma-Beta couplings}
\label{subsec-pos-conc-Beta-gamma-coupling}
From \eqref{eq:edge-connecting_probability:CPU:SL} we obtain the conditional edge-probabilities for $\CPU^{\sss\rm{(SL)}}$ as
\eqn{\label{edge:connecting:probability:CPU:SL}
		\prob_{m,\psi}\left( u\overset{j}{\rightsquigarrow} v~\text{in }\CPU_n^{\sss\rm{(SL)}} \right) =~ \begin{cases}
			\frac{\mathcal{S}_v^{\sss(n)}}{\mathcal{S}_{u,j}^{\sss(n)}}\Big( 1-\prod\limits_{l\in [m_v]}\left(1-\psi_{m_{[v-1]+l}}\right) \Big)~&\text{for }u>v,\\ 
			1-\prod\limits_{l\in[j]}\left(1-\psi_{m_{[v-1]+l}}\right)~&\text{for }u=v.
		\end{cases}
}
Similarly, from \eqref{eq:edge-connecting_probability:CPU:NSL}, the conditional edge-probabilities {for} $\CPU^{\sss\rm{(NSL)}}$ are given by
\eqn{\label{edge:connecting:probability:CPU:NSL}
	\begin{split}
		\prob_{m,\psi}\left( u\overset{j}{\rightsquigarrow} v~\text{in }\CPU_n^{\sss\rm{(NSL)}} \right) = \begin{cases}
			\frac{\mathcal{S}_v^{\sss(n)}}{\mathcal{S}_{u,j-1}^{\sss(n)}}\Big( 1-\prod\limits_{l\in [m_v]}\big(1-\psi_{m_{[v-1]+l}}\big) \Big),&\text{for }u>v,\\ 
			1-\prod\limits_{l\in[j-1]}\big(1-\psi_{m_{[v-1]+l}}\big),&\text{for }u=v,
		\end{cases}
\end{split}}
and for $\PU^{\sss\rm{(NSL)}}$ it is
\eqn{\label{edge:connecting:probability:PU:NSL}
	\begin{split}
		\prob_{m,\psi}\left( u\overset{j}{\rightsquigarrow} v~\text{in }\PU_n^{\sss\rm{(NSL)}} \right) =~ 
		\psi_v\frac{\mathcal{S}_v^{\sss(n)}}{\mathcal{S}_{u-1}^{\sss(n)}}~.
\end{split}}

{We thus need to analyze} the asymptotics of the $\mathcal{S}_k^{\sss(n)}$ of $\CPU_n^{\sss\rm{(SL)}}(\boldsymbol{m},\boldsymbol{\psi})$ with $\boldsymbol{\psi}$ defined in \eqref{def:psi:1} and \eqref{def:psi:2}, which we do in the following proposition:
\begin{Proposition}[Position concentration for PU and CPU]\label{lem:position:concentration}
	Recall that $\chi={(\E[M]+\delta)}/{(2\E[M]+\delta)}$. Then, for every $\varepsilon,\omega>0,$ there exists $K<\infty$, such that for all $n>K$ and with probability at least $1-\varepsilon$, for both $\PU$ and $\CPU$,
	\eqn{
		\label{eq:lem:position:concentration:1}
		\max\limits_{k\in[n]} \left| \mathcal{S}_k^{\sss(n)} - \left(\frac{k}{n}\right)^\chi \right| \leq \omega,
	}
	and
	\eqn{
		\label{eq:lem:position:concentration:2}
		\max\limits_{k\in[n]\setminus [K]} \left(\frac{k}{n}\right)^{-\chi}\left| \mathcal{S}_k^{\sss(n)} - \left(\frac{k}{n}\right)^\chi \right| \leq \omega.
	}
\end{Proposition}
\medskip
Note that, although the $\Beta$ random variables are different, the position concentration lemma hold true for both $\PU$ and $\CPU$.
\begin{proof}[\bf Proof of Proposition \ref{lem:position:concentration}]
	We follow the line of proof provided in \cite[Proof of Lemma~3.1]{BergerBorgs}
	with the adaptations as required for our case of i.i.d.\ out-degrees.
	\noindent
	\paragraph{Proof for $\CPU$} Fix $\omega,\varepsilon>0$, and let $\bar{\omega}= \log(1+\omega)$. We use the definition of $\mathcal{S}_k^{\sss(n)}$ to bound the error in estimating $\mathcal{S}_k^{\sss(n)}$ by $\left( \frac{k}{n} \right)^{\chi}$. For all $k\in[n-1]$,
	\eqn{\label{for:lem:position:concentration:1}
		{\mathcal{S}}_k^{\sss(n)} = \prod\limits_{l=m_{[k]}+1}^{m_{[n]}} (1-\psi_l) = \exp{\left[ \sum\limits_{l=m_{[k]}+1}^{m_{[n]}} \log (1-\psi_l) \right]},
	}
	with $\mathcal{S}_n^{\sss(n)}\equiv 1$. We concentrate on the argument of the exponential in \eqref{for:lem:position:concentration:1}. Note that
	\eqn{
		\label{for:lem:position:concentration:2}
		\var \left( \log (1-\psi_l) \right) \leq \E \left[ \log^2 (1-\psi_l) \right] \leq \E\left[ \frac{\psi_l^2}{(1-\psi_l)^2} \right].}
	By \eqref{for:lem:position:concentration:2} and Kolmogorov's maximal inequality,
	\eqan{\label{for:lem:position:concentration:3}
		&\prob\Big( \max\limits_{l\in \left[ m_{[n]} -1 \right]} \Big| \sum\limits_{k=l+1}^{m_{[n]}} \log(1-\psi_k) -\E\Big[ \sum\limits_{k=l+1}^{m_{[n]}} \log(1-\psi_k) \Big] \Big| \geq \frac{\bar{\omega}}{2} \Big)\nn\\
		&\hspace{6cm}\leq \frac{4}{\bar{\omega}^2}\E\left[\sum\limits_{i=2}^{m_{[n]}}\E_m\left[ \frac{\psi_i^2}{(1-\psi_i)^2} \right]\right].
	}
	Equation~\eqref{for:lem:position:concentration:3} shows that the maximum of the fluctuations of the argument in \eqref{for:lem:position:concentration:1} can be bounded by the variances of the single terms. By properties of the $\Beta$ distribution, and recalling that, for $u=3,\ldots,n$ and $j\in[m_u]$,
	\[
	\psi_{m_{[u-1]}+j}\sim \Beta \left( 1+\frac{\delta}{m_u}, a_{[2]} + 2\left( m_{[u-1]}+j-3 \right) +(u-1)\delta +\frac{j-1}{m_u}\delta -1 \right),
	\]
	we can {bound}, for $u>2$ and $j\in[m_u]$,
	\eqn{\label{for:lem:position:concentration:4}
		\E_m\Bigg[ \frac{\psi_{m_{[u-1]}+j}^2}{\Big( 1-\psi_{m_{[u-1]}+j} \Big)^2} \Bigg] = \mathcal{O}\left( \big(m_{[u-1]}+j\big)^{-2} \right).}
	Notice that {$m_{[n]}\geq n$ for all $n\geq 1,$} and hence $m_{[n]}\to \infty$ as $n\to\infty$. Equation~\eqref{for:lem:position:concentration:4} assures us that the sum on the RHS of \eqref{for:lem:position:concentration:3} is finite as $n\rightarrow\infty$. Therefore, we can fix $N_1(\bar{\omega})\in\N$ such that $\E\Big[\sum_{i=N_1}^\infty {\psi_i^2}/{(1-\psi_i)^2} \Big]\leq { \varepsilon\bar{\omega}^2/4}.$ As a consequence, bounding the sum on the RHS of \eqref{for:lem:position:concentration:3} by the tail of the series, for $n>N_1,$
	\eqan{\label{for:lem:position:concentration:5}
			&\prob\left( \max\limits_{i\in\left[ m_{[n]}-1 \right]\setminus \left[ m_{[N_1]}-2 \right]}\left| \sum\limits_{l=i+1}^{m_{[n]}} \log(1-\psi_l) - \E\left[ \sum\limits_{l=i+1}^{m_{[n]}} \log(1-\psi_l)  \right]  \right|\geq \frac{\bar{\omega}}{2} \right)\nn\\
			&\hspace{6cm}\leq\frac{4}{\bar{\omega}^2}\E\left[ \sum\limits_{i=N_1}^\infty \frac{\psi_i^2}{(1-\psi_i)^2} \right]\leq \varepsilon~.
	}
	{Next, we wish to compare the expectations of $\sum\limits_{k=i}^{m_{[n]}}\log(1-\psi_k)$ and  $\sum\limits_{k=i}^{m_{[n]}}\psi_k$ for $n$ large enough.} Using the fact that, for $x\in(0,1)$, 	
	\eqn{\label{for:lem:position:concentration:log-result}
	\left|\log(1-x)+x\right|\leq x^2/(1-x),}
	and using \eqref{for:lem:position:concentration:4} we {bound}, for $m_{[N_1]}\leq i\leq m_{[n]}$, 
	\eqn{\label{for:lem:position:concentration:6}
		\Bigg| \E\Big[ \sum\limits_{l=i+1}^{m_{[n]}} \log(1-\psi_l) \Big] + \E\Big[ \sum\limits_{l=i+1}^{m_{[n]}} \psi_l \Big] \Bigg| \leq \E\Bigg[ \sum\limits_{l=i+1}^{m_{[n]}}\frac{\psi_l^2}{(1-\psi_l)} \Bigg] < \infty.
	}
	{Similarly,} there exists $N_2(\omega)\in\N$ such that $\E\big[ \sum\limits_{l=N_2}^{\infty}\psi_l^2/(1-\psi_l) \big]\leq {\bar{\omega}/3}$ and ${1/\sqrt{N_2}}\leq {\bar{\omega}/6}$. On the other hand, for $u>2$ and $j\in[m_u]$,
	\eqn{\label{for:lem:position:concentration:8}
		\begin{split}
			\E_m\left[ \psi_{m_{[u-1]}+j} \right]  &= \frac{1+\frac{\delta}{m_u}}{a_{[2]}+2\left[ m_{[u-1]}+j-3 \right]+(u-1)\delta + \frac{j}{m_u}\delta}\\
			&= \frac{1+\frac{\delta}{m_u}}{2 m_{[u-1]} +(u-1)\delta}\left( 1+ o\left( 1 \right)\right).
	\end{split}}
	Therefore, for all $u>2,$
	\eqan{
		\label{for:lem:position:concentration:9}
		\E\Big[ \sum\limits_{j=1}^{m_u}\psi_{m_{[u-1]}+j}  \Big]&=\E \Big[ \E_m\Big[ \sum\limits_{j=1}^{m_u}\psi_{m_{[u-1]}+j} \Big] \Big]\nn\\
		&= \E\left[ \frac{m_u+\delta}{2m_{[u-1]}+(u-1)\delta} \right]\left( 1+ o\left( 1 \right)\right)\nn\\
		&=(\E[M]+\delta)\E\left[\frac{1}{2m_{[u-1]}+(u-1)\delta} \right]\left( 1+ o\left( 1 \right)\right)\nn\\
		&=\frac{\E[M]+\delta}{(u-1)(2\E[M]+\delta)}\left( 1+ o\left( 1 \right)\right)\nn\\
		&{=\frac{\chi}{u-1}\left( 1+ o\left( 1 \right)\right)},
	}
	{by Lemma~\ref{prop:M-inverse:expectation}.}
	Now, using the bounds in \eqref{for:lem:position:concentration:6} and \eqref{for:lem:position:concentration:9}, for all $N_2\leq k\leq n,$
	\eqn{\label{for:lem:position:concentration:12}
		\begin{split}
			\Bigg| \E\Big[ \sum\limits_{i= m_{[k]}+1}^{m_{[n]}} \log(1-\psi_i) \Big] - \chi\log\left(\frac{k}{n}\right) \Bigg| \leq  o\left( k^{-1/2} \right) + \bar{\omega}/3\leq \frac{\bar{\omega}}{2}.
		\end{split}
	}
	As a consequence, for $n>N_2,$
	\eqn{
		\label{for:lem:position:concentration:13}
		\max\limits_{k\in (N_2,n]} \left| \E\left[ \sum\limits_{i= m_{[k]}+1}^{m_{[n]}} \log(1-\psi_i) \right] - \chi\log\left(\frac{k}{n}\right) \right| \leq \frac{\bar{\omega}}{2}.
	} 
	Let $N_0=\max\{ N_1,N_2 \}$. By \eqref{for:lem:position:concentration:13} and \eqref{for:lem:position:concentration:5}, for $n>N_0$,
	\eqn{\label{for:lem:position:concentration:14}
		\prob\left( \max\limits_{k\in (N_0,n]} \left| \sum\limits_{i= m_{[k]}+1}^{m_{[n]}} \log(1-\psi_i) - \chi\log\left(\frac{k}{n}\right) \right| \geq \bar{\omega} \right) \leq \varepsilon.
	}
	Recalling that $\log \mathcal{S}_k^{\sss(n)} = \sum\limits_{i= m_{[k]}+1}^{m_{[n]}} \log(1-\psi_i)$ and $\bar{\omega}=\log(1+\omega)$,  \eqref{for:lem:position:concentration:14} {implies} that with probability at least $1-\varepsilon,$ for every $i=(N_0,n]$,
	\eqn{\label{for:lem:position:concentration:15}
		\frac{1}{1+\omega}\left( \frac{k}{n} \right)^\chi \leq \mathcal{S}_k^{\sss(n)}\leq (1+\omega)\left( \frac{k}{n} \right)^\chi.}
	Since, $\frac{1}{1+\omega}\geq 1-\omega$, we obtain from \eqref{for:lem:position:concentration:15} that
	\eqn{\label{for:lem:position:concentration:16}
		\prob\Big( \bigcap\limits_{u\in(N_0,n]}\left\{ \left| \mathcal{S}_u^{\sss(n)} - \left( \frac{u}{n} \right)^\chi \right|\leq \omega \left( \frac{u}{n} \right)^\chi \right\} \Big)\geq 1-\varepsilon,
	}
	which proves \eqref{eq:lem:position:concentration:2}.  
	
	{Finally, to} prove \eqref{eq:lem:position:concentration:1}, we observe that, for fixed $\omega>0$ and $\varepsilon>0,\ \left(\frac{N_0}{n}\right)^\chi \leq {\omega/3}$ for large enough $n$. Now,
	\eqn{\label{for:lem:position:concentration:17}
		\max\limits_{u\in[N_0]}\left| \mathcal{S}_u^{\sss(n)} - \left( \frac{u}{n} \right)^\chi \right| \leq \mathcal{S}_{N_0}^{\sss(n)} + \left(\frac{N_0}{n}\right)^\chi \leq \omega,
	}
	as required.
	
	\medskip\noindent
	\paragraph{Proof for $\PU$}
	In P\'olya urn graphs, conditionally on ${\boldsymbol{m}},$ for $u\geq 3$, 
	\eqn{\label{for:prop:pos:conc:PU:0-0}\psi_u\sim\Beta\left( m_u+\delta,a_{[2]}+2(m_{[u-1]}-2)+m_u+(u-1)\delta \right).}
	We prove the position concentration lemma for $\PU$ following the proof for $\CPU$. Using Kolmogorov's inequality as in \eqref{for:lem:position:concentration:3}, 
	\eqan{\label{for:prop:pos:conc:PU:0-1}
		&\prob\Big( \max\limits_{l\in[n-1]\setminus [n_1-2]}\big| \sum\limits_{k=l+1}^n \log (1-\psi_k) - \E\big[ \sum\limits_{k=l+1}^n \log (1-\psi_k) \big] \big|\geq \frac{\bar{\omega}}{2} \Big)\nn\\
		&\hspace{7cm}\leq \frac{4}{\bar{\omega}^2}\E\Big[ \sum\limits_{i=n_1}^n\frac{\psi_i^2}{(1-\psi_i)^2)} \Big]~.}
	Now similarly as in \eqref{for:lem:position:concentration:4} we first bound $\E\Big[ \frac{\psi_i^2}{(1-\psi_i)^2} \Big]$ by $i^{-p}$ and thus we can choose $n_1$ large enough such that the RHS of \eqref{for:prop:pos:conc:PU:0-1} is smaller than $\varepsilon$. Then,
	\eqn{\label{for:prop:pos:conc:PU:0}\E\left[ \frac{\psi_u^2}{(1-\psi_u)^2} \right]\leq \E\Bigg[ \Big(\frac{m_u+\delta}{a_{[2]}+2(m_{[u-1]}-2)+m_u+(u-1)\delta}\Big)^2 \Bigg]~.}
	{ Since the random variable in the RHS of \eqref{for:prop:pos:conc:PU:0} is bounded above by $1,$ we can bound its second moment by its $p$-th moment and obtain the following bound on the LHS of \eqref{for:prop:pos:conc:PU:0}:
		\eqn{\begin{split}
				\E\left[ \frac{\psi_u^2}{(1-\psi_u)^2} \right]\leq&\E\Bigg[ \Big(\frac{m_u+\delta}{a_{[2]}+2(m_{[u-1]}-2)+m_u+(u-1)\delta}\Big)^{p} \Bigg]\\
				\leq& \E\Bigg[ \Big(\frac{m_u+\delta}{a_{[2]}+2(m_{[u-1]}-2)+(u-1)\delta}\Big)^{p} \Bigg]~.
	\end{split}}}
	Now the numerator and the denominator are independent of each other. Using the law of large numbers, similar to the one used to prove Lemma~\ref{prop:M-inverse:expectation}, and Lemma~\ref{prop:M-inverse:expectation},
	\eqan{\label{for:prop:pos:conc:PU:1-1}
		&\E\Bigg[ \Big(\frac{1}{a_{[2]}+2(m_{[u-1]}-3)+(u-1)\delta}\Big)^{1+p} \Bigg] \nn\\
		&\hspace{2cm}= (1+o(1))\frac{1}{u^{1+p}(2\E[M]+\delta)^{1+p}}~.}
	Since $m_u$ has finite $(1+p)$-th moment, there exists a constant $\xi_0>0$ such that
	\eqn{\label{for:prop:pos:conc:PU:1-2}
		\E\left[ \frac{\psi_u^2}{(1-\psi_u)^2} \right]\leq \xi_0 u^{-(1+p)}~,}
	which replaces \eqref{for:lem:position:concentration:4}. Now it remains to adapt a similar result to \eqref{for:lem:position:concentration:9} to complete the proof for $\PU$. By \eqref{for:prop:pos:conc:PU:0-0},
	\eqn{
		\label{for:prop:pos:conc:PU:1}
		\E[\psi_u]=\E\left[\frac{m_u+\delta}{a_{[2]}+2(m_{[u]}-2)+u\delta}\right]~.}
	Note that the numerator and the denominator are not independent here. We can bound the RHS from above as
	\eqan{\label{for:prop:pos:conc:PU:2}
			\E\left[\frac{m_u+\delta}{a_{[2]}+2(m_{[u]}-2)+u\delta}\right]&\leq \E\left[ \frac{m_u+\delta}{a_{[2]}+2(m_{[u-1]}-2)+(u-1)\delta} \right]\nn\\
			&=\frac{\chi}{u-1}(1+{o}(1))~.
	}
	On the other hand for the lower bound, we truncate $m_u$ at $\log (u)$ as
	\eqan{\label{for:prop:pos:conc:PU:3}
			&\E\left[\frac{m_u+\delta}{a_{[2]}+2(m_{[u]}-2)+u\delta}\right]\nn\\
			&\hspace{2cm}\geq\E\left[\frac{m_u\one_{\{m_u\leq \log u\}}+\delta}{a_{[2]}+2(m_{[u-1]}-2)+2\log (u)+u\delta}\right].
	}
	Again we can split the numerator and the denominator since they are now independent. Since $m_u$ has finite mean, $\E[m_u\one_{\{m_u\leq \log u\}}+\delta] = (\E[M]+\delta)(1+o(1)).$ {By Lemma~\ref{prop:M-inverse:expectation},}
	\eqan{\label{for:prop:pos:conc:PU:4}
			\E\left[\frac{m_u+\delta}{a_{[2]}+2(m_{[u]}-2)+u\delta}\right]&\geq \frac{\E[M]+\delta}{(u-1)(2\E[M]+\delta)}(1+o(1))\nn\\
			&=\frac{\chi}{u-1}(1+o(1)).
	}
	Hence from \eqref{for:prop:pos:conc:PU:2} and \eqref{for:prop:pos:conc:PU:4}, we {obtain a} similar result as {in} \eqref{for:lem:position:concentration:8} as
	\eqn{\label{for:lem:position:concentration:PU:01}
		\E[\psi_u]=\E\left[\frac{m_u+\delta}{a_{[2]}+2(m_{[u]}-2)+u\delta}\right] = \frac{\chi}{u-1}(1+o(1)).
	}
	Now, using \eqref{for:lem:position:concentration:log-result} and \eqref{for:lem:position:concentration:PU:01}, we obtain that there exists $N_2\in\N$ (depending on $\bar{\omega}$), such that,
	\eqn{\label{for:lem:position:concentration:PU:02}
	\max\limits_{k\in(N_2,n]}\left| \E\Big[ \sum\limits_{i=k+1}^n \log(1-\psi_k) \Big] - \chi\log\Big(\frac{k}{n}\Big) \right|\leq \frac{\bar{\omega}}{2}~.}
	The remainder of the proof follows exactly the same way as in \eqref{for:lem:position:concentration:14}-\eqref{for:lem:position:concentration:17}.
\end{proof}

In collapsed P\'olya urn graphs, the sequence $\{ \mathcal{S}_{k,j}^{\sss(n)}\colon j\in[m_k] \}$ defined in \eqref{eq:def:position:CPU} is an increasing sequence and from \FC{Proposition}~\ref{lem:position:concentration}, with probability at least $1-\varepsilon,$ for all $n>k>K,$
\eqn{\label{position_concentration:rest}
	\begin{split}
		\mathcal{S}_{k,j}^{\sss(n)} &\leq \mathcal{S}_{k}^{\sss(n)} \leq \left( \frac{k}{n} \right)^\chi (1+2\omega),\\
		\mbox{and }\qquad\mathcal{S}_{k,j}^{\sss(n)} &\geq \mathcal{S}_{k-1}^{\sss(n)} \geq \left( \frac{k-1}{n} \right)^\chi (1-\omega) \geq \left( \frac{k}{n} \right)^\chi (1-2\omega).
\end{split} }
Therefore it is evident that $k$ for sufficiently large, the variation in {$j\mapsto \mathcal{S}_{k,j}^{\sss(n)}$ is minor.} Hence the {differences between the} SL and NSL {versions} of the collapsed P\'olya Urn graphs {will also be minor.} 
The following proposition provides us with a nice coupling between the $\Beta$ random variables $\boldsymbol{\psi}$ and a sequence of Gamma variables. This coupling will be very much useful in analysing the remaining terms in \eqref{edge:connecting:probability:CPU:SL}, \eqref{edge:connecting:probability:CPU:NSL} and \eqref{edge:connecting:probability:PU:NSL}. Before diving into the proposition and its proof, let us denote $p=1+\varrho$ for some $\varrho>0$.
\begin{Proposition}[Beta-Gamma coupling for $\PU$ and $\CPU$]\label{prop:Betagamma:coupling:CPU}
	Consider the sequence $(\psi_k)_{k\in\N}$ for $\CPU$ and define the sequence $(\phi_k)_{k\in\N}$ as
	\eqn{
		\label{def:phi}
		\phi_v:=\sum\limits_{j\in[m_v]} \psi_{m_{[v-1]}+j}.
	}
	Recall that, {conditionally on $\boldsymbol{m}$}, $\chi_k$ has a Gamma distribution with parameters $m_k+\delta$ and $1$.
	Then, {there exist $K_\varepsilon$ and $K_\eta$ such that}
	\begin{itemize}
		\item[(i)] with probability at least $1-\varepsilon$, $\chi_k\leq k^{1-\varrho/2}$ for all $k\geq K_\varepsilon$;
		\smallskip
		\item[(ii)] consider the function $h_k^{\phi}(x)$ such that $\prob(\phi_k\leq h_k^{\phi}(x)) = \prob(\chi_k\leq x)$, where conditionally on $\boldsymbol{m},~\chi_k$  has a Gamma distribution with parameters $m_k+\delta$ and 1. 
		Then, for every $\eta\in\left(0,\tfrac{1}{2}\right),$ there exists a sufficiently large $K_\eta\geq 1$ depending on $\eta$, such that, for all $k\geq K_{\eta}$ and $x\leq k^{1-\varrho/2}$,
		\eqn{\label{for:prop:Betagamma:coupling}
			\frac{1-\eta}{k(2\E[M]+\delta)}x\leq h_k^{\phi}(x)\leq \frac{1+\eta}{k(2\E[M]+\delta)}x.
		}
	\end{itemize}
	Since $(\psi_k)_{k\geq 1}$ is the $\PU$ analogue of the sequence $(\phi_k)_{k\geq 1}$, replacing $(\phi_k)_{k\geq 1}$ by {the} corresponding $(\psi_k)_{k\geq 1}$ of $\PU$ defined in \eqref{eq:def:psi:1}, {a} similar Beta-Gamma coupling holds for $\PU$.
\end{Proposition}
Berger et al.\ \cite[Lemma~3.2]{BergerBorgs} {provide} a {related} result for degenerate $M$, but our {proof technique is a bit different}. For obtaining the upper and lower bounds in \eqref{for:prop:Betagamma:coupling}, we {use a} correlation inequality \cite[Lemma~1.24]{vdH2} and Chernoff's inequality.
\begin{proof}
	{First, we start by proving $(i)$. Using Markov's inequality, 
		\eqan{\label{for:lem:Beta-gamma:1:2}
			\prob\left( \chi_k\geq k^{1-\varrho/2} \right)&=\prob\left( \chi_k^{1+\varrho}\geq k^{(1-\varrho/2)(1+\varrho)} \right)\nn\\
			&\leq \E\left(\chi_k^{1+\varrho}\right) k^{-\left( 1+\varrho(1-\varrho)/2 \right)}.
		}
		By Lemma~\ref{lem:finite_moment}, $\E\left( \chi_k^{p} \right)$ is uniformly bounded. Therefore the RHS of \eqref{for:lem:Beta-gamma:1:2} is summable and we obtain $(i)$ using the Borel-Cantelli lemma.}\\
	{We} proceed to prove $(ii).$
	Let 
	\[
	b_{u,j} = a_{[2]}+2(m_{[u-1]}+j-3)+(u-1)\delta +\frac{(j-1)}{m_u}\delta -1,
	\]
	and, conditionally on $\boldsymbol{m}$, let $\left(\chi_{k}^\prime\right)_{k\in\N}$ be a sequence of independent Gamma variables such that
	\eqn{
		\label{for:lem:Beta-gamma:2}
		\chi_{m_{[u-1]}+j}^\prime \sim {\rm{Gamma}}\left( 1+{\delta/m_{u}}, 1 \right).}
	The function $x\mapsto h_{u,j}^{\phi}(x)$ is defined such that for all $x\leq b_{u,j}^{1-\varrho/2},$
	\[
	\prob\left( \psi_{m_{[u-1]}+j}\leq h_{u,j}^{\phi}(x) \right) = \prob\left( \chi_{m_{[u-1]+j}}^\prime \leq x \right).
	\]
	To prove {part $(ii)$}, we start with the sequence $(\psi_k)_{k\in\N}$ with the aim to couple it with appropriately scaled $\left( \chi_k^\prime \right)_{k\in\N}$. We will prove \eqref{for:prop:Betagamma:coupling} using the following claim:
	\begin{Claim}\label{claim:beta-gamma-1}
		There exists $K_\eta\geq 1$ such that for all $x\leq \left( b_{u,j} \right)^{1-\varrho/2}$ and $u>K_\eta$,
		\eqn{\label{for:lem:Beta-gamma:3}
			\left(1-\frac{\eta}{2}\right)\frac{x}{b_{u,j}}\leq h_{u,j}^{\phi}(x)\leq \frac{x}{b_{u,j}}.}
	\end{Claim}
	\paragraph{\textbf{Proof of part $(ii)$ subject to Claim~\ref{claim:beta-gamma-1}}:}
	From Claim~\ref{claim:beta-gamma-1} and part $(i)$, there exists a $K_\eta$ such that for any $u> K_\eta$, we can couple $\big(\psi_{m_{[u-1]}+j}\big)_{\substack{u>K_\eta\\j\in [m_u]}}$ and $\big(\chi_{m_{[u-1]}+j}^\prime\big)_{\substack{u>K_\eta\\j\in [m_u]}}$ with probability at least $1-\varepsilon$, such that
	\eqn{\label{for:lem:Beta-gamma:4}
		\chRounak{\left( 1-\frac{\eta}{2} \right) \frac{\chi_{m_{[u-1]}+j}^\prime}{b_{u,j}}\leq \psi_{m_{[u-1]}+j}\leq \frac{\chi_{m_{[u-1]}+j}^\prime}{b_{u,j}}~.}
	}
	\chRounak{Rearranging the inequalities in \eqref{for:lem:Beta-gamma:4}, we obtain}
	\eqn{\label{eq:major:3:4}
		\chRounak{\psi_{m_{[u-1]}+j}\leq \frac{\chi_{m_{[u-1]}+j}^\prime}{b_{u,j}}\leq \Big(1+\frac{3}{4}\eta\Big)\psi_{m_{[u-1]}+j}.}}
	Since $b_{u,j}$ is random, we replace it with its expectation and encounter some error. By the law of large numbers,
	\eqn{\label{for:lem:Beta-gamma:5}
		\frac{m_{[u-1]}}{u}\overset{a.s.}{\to}\E[M],}
	and hence, \chRounak{depending on $\eta$ and $\vep$, there exists $K_\eta(\vep)\geq 1$ large enough such that, with probability at least $1-\vep,$ the error {can} be bounded as}
	\eqn{
		\label{for:lem:Beta-gamma:6}
		\chRounak{\left| \frac{b_{u,j}-u(2\E[M]+\delta)}{u(2\E[M]+\delta)} \right|\leq \frac{\eta}{4},\quad \mbox{for all } u>K_\eta(\vep).}
	}
	Therefore, from \eqref{for:lem:Beta-gamma:4} and \eqref{for:lem:Beta-gamma:6}, with probability at least $1-2\varepsilon,$
	{the sequences} $\left(\psi_{m_{[u-1]}+j}\right)_{\substack{u>K_\eta\\j\in m_{[u]}}}$ and $\left(\chi_{m_{[u-1]}+j}^\prime\right)_{\substack{u>K_\eta\\j\in m_{[u]}}}$ can be coupled such that
	\eqn{\left( 1-\eta \right)\psi_{m_{[u-1]}+j}\leq \frac{\chi_{m_{[u-1]}+j}^\prime}{u(2\E[M]+\delta)}\leq (1+\eta) \psi_{m_{[u-1]}+j}.}
	{Thus}, summing over all $j\in[m_u],$ we obtain \eqref{for:prop:Betagamma:coupling} subject to \eqref{for:lem:Beta-gamma:3}. {We are left to proving Claim \ref{claim:beta-gamma-1}:}
	\paragraph{\textbf{Proof of Claim \ref{claim:beta-gamma-1}}}
	To prove the claim it is enough to show that for {all} $x\leq \left( b_{u,j} \right)^{-\varrho/2}$ and $u>K_\eta\geq {1/\sqrt{\eta}},$
	\eqan{\label{for:lem:Beta-gamma:7}
			\prob_m\left( \psi_{m_{[u-1]}+j}\leq (1-\eta)x \right) &\leq \prob_m\left( \chi_{m_{[u-1]}+j}^\prime\leq b_{u,j} x \right)\nn\\
			&\hspace{2.5cm}\leq \prob_m\left( \psi_{m_{[u-1]}+j}\leq x \right).
	}
	\paragraph{\textbf{The upper bound}}
	With $\alpha_u = 1+{\delta/m_u},$ we bound
	\eqn{
		\prob_m\left( \chi_{m_{[u-1]}+j}^\prime\leq b_{u,j} x \right) = \frac{\int\limits_0^{b_{u,j}x} y^{\alpha_u-1}e^{-y}\,dy}{\int\limits_0^{\infty} y^{\alpha_u-1}e^{-y}\,dy}\leq  \frac{\int\limits_0^{b_{u,j}x} y^{\alpha_u-1}e^{-y}\,dy}{\int\limits_0^{b_{u,j}} y^{\alpha_u-1}e^{-y}\,dy}.\nonumber
	}
	Then, using a change of variables and adjusting the missing {factors from the $\Beta$ density,} we get
	\eqn{\label{for:lem:Beta-gamma:8}
		\begin{split}
			\prob_m\left( \chi_{m_{[u-1]}+j}^\prime\leq b_{u,j} x \right)
			\leq \frac{\int\limits_{0}^1\one_{\{y\leq x\}}e^{-b_{u,j}y}\left( 1-y \right)^{-b_{u,j}} f_{m_{[u-1]}+j}(y)\,dy }{\int\limits_{0}^1 e^{-b_{u,j}y}\left( 1-y \right)^{-b_{u,j}} f_{m_{[u-1]}+j}(y)\,dy },
	\end{split}}
	where $f_k$ is the {density} of $\psi_k.$\\
	Note that the numerator is $\E_m\big[ \one_{[\psi_{m_{[u-1}+j}\leq x]}e^{-b_{u,j}\psi_{m_{[u-1}+j}}\big( 1-\psi_{m_{[u-1}+j} \big)^{-b_{u,j}} \big]$ and the denominator is $\E_m\big[ e^{-b_{u,j}\psi_{m_{[u-1}+j}}\big( 1-\psi_{m_{[u-1}+j} \big)^{-b_{u,j}} \big]$. Further, ${z\mapsto\one_{\{z\leq x\}}}$ is non-increasing, {while} $z\mapsto e^{-b_{u,j}z}\left( 1-z \right)^{-b_{u,j}}$ is increasing. {Recall that, for increasing $f$ and decreasing $g$ on the support of $X$,}
	\eqn{
		\label{for:lem:Beta-gamma:9}
		\E[f(X)g(X)]\leq \E[f(X)]\E[g(X)].
	}
	Therefore, using this correlation inequality in the RHS of \eqref{for:lem:Beta-gamma:8}, we have
	\eqn{\label{for:lem:Beta-gamma:10}
		\prob_m\left( \chi_{m_{[u-1]}+j}^\prime\leq b_{u,j} x \right) \leq \prob_m\left( \psi_{m_{[u-1]}+j}\leq x \right),
	}
	{which proves the upper bound in \eqref{for:lem:Beta-gamma:7}.}
	\paragraph{\textbf{The lower bound}}
	The {lower bound} in \eqref{for:lem:Beta-gamma:7} can be handled in a similar, albeit slightly more involved, way. Indeed,
	\eqan{\label{for:lem:Beta-gamma:11}
		&\prob_m \left( \psi_{m_{[u-1]}+j}  \leq x(1-\eta) \right)= \frac{\int\limits_{0}^{x(1-\eta)b_{u,j}}y^{a_u-1}\left( 1-\frac{y}{b_{u,j}} \right)^{b_{u,j}}\,dy}{\int\limits_{0}^{b_{u,j}}y^{a_u-1}\left( 1-\frac{y}{b_{u,j}} \right)^{b_{u,j}}\,dy}\\
		=&\frac{\E_m\left[ \left.\one_{\big\{ \chi_{m_{[u-1]}+j}^\prime\leq b_{u,j}x(1-\eta) \big\}}e^{\chi_{m_{[u-1]}+j}^\prime}\left( 1-\frac{\chi_{m_{[u-1]}+j}^\prime}{b_{u,j}} \right)^{b_{u,j}}\right| \chi_{m_{[u-1]}+j}^\prime\leq b_{u,j}\right]}{\E_m\left[ \left.e^{\chi_{m_{[u-1]}+j}^\prime}\left( 1-\frac{\chi_{m_{[u-1]}+j}^\prime}{b_{u,j}} \right)^{b_{u,j}}\right| \chi_{m_{[u-1]}+j}^\prime\leq b_{u,j}\right]}.\nn}
	Denote $f(z) = \one_{\{z\leq b_{u,j}x(1-\eta)\}}$ and $g(z)= e^z\left( 1-\frac{z}{b_{u,z}} \right)^{b_{u,j}}$. {Note} that $f$ and $g$ are non-increasing and increasing functions, respectively. Therefore, {by} the correlation inequality in \eqref{for:lem:Beta-gamma:9} {once more},
	\eqan{&\prob_m \left( \psi_{m_{[u-1]}+j}  \leq x(1-\eta) \right)\nn\\
		&\hspace{2cm}\leq \prob_m\left( \left. \chi_{m_{[u-1]}+j}^\prime\leq b_{u,j}x(1-\eta) \right|\chi_{m_{[u-1]}+j}^\prime\leq b_{u,j} \right)~.}
	Thus, to prove the required inequality, it suffices to show that for all $u \geq K_\eta$, $j \in m_{[u]}$, and $x \leq (b_{u,j})^{-\varrho/2}$,
	\eqan{\label{for:lem:Betagamma:12}
		&\prob_m\left(\chi_{m_{[u-1]}+j}^\prime \leq b_{u,j}x(1-\eta) \,\Big|\, \chi_{m_{[u-1]}+j}^\prime \leq b_{u,j} \right)\nn\\
		&\hspace{2cm}\leq \prob_m\left( \chi_{m_{[u-1]}+j}^\prime \leq b_{u,j}x \right).}
	
	To simplify notation, instead of proving \eqref{for:lem:Betagamma:12} directly, we demonstrate the same result for any Gamma random variable $Z$ with parameters $\alpha$ and $1$, where $\alpha \in (0,1+|\delta|)$. Observe that proving \eqref{for:lem:Betagamma:12} is equivalent to showing that
	\eqn{\label{for:lem:Betagamma:13}
		E(x) = \prob\left( Z \leq bx \right)\prob\left( Z \leq b \right) - \prob\left( Z \leq bx(1-\eta) \right) \geq 0,}
	where $x \leq b^{-\varrho/2}$ for sufficiently large $b > K_\eta$.
	
	By definition, $E(0)=0.$ {Therefore, it} remains to show that \eqref{for:lem:Betagamma:13} holds for $x\in\left( 0,b^{-\varrho/2}\right)$. We simplify
	\eqn{
		\label{for:lem:Betagamma:14}
		E(x)= \prob\left( (1-\eta)bx\leq Z \leq bx \right) - \prob(Z>b)\prob(Z\leq bx).
	}
	\paragraph{\textbf{Lower bounding the first expression of $E(x)$}}Using the properties of the density of the Gamma distribution, we can lower bound the first term as
	\eqan{\label{for:lem:Beta-gamma:15}
			\prob\left( (1-\eta)bx\leq Z \leq bx \right)\geq& \eta bx \min\left\{ \frac{(bx)^{\alpha-1}}{\Gamma(\alpha)}e^{-bx}, \frac{(bx)^{\alpha-1}}{\Gamma(\alpha)}e^{-bx}(1-\eta)^{\alpha-1}e^{bx\eta} \right\}\nn\\
			=& \eta \frac{(bx)^{\alpha}}{\Gamma(\alpha)}e^{-bx}\min\left\{ 1, (1-\eta)^{\alpha-1}e^{bx\eta}\right\}.
	}
	With the fact that $\eta<\tfrac{1}{2},$ and $e^{\eta bx}>1$, we can lower bound the first term in \eqref{for:lem:Betagamma:14} further by 
	\eqan{\label{for:lem:Beta-gamma:16}
		\prob\left( (1-\eta)bx\leq Z \leq bx \right)&\geq \eta \frac{(bx)^{\alpha}}{\Gamma(\alpha)} e^{-bx}\min \{ 2^{1-\alpha},1 \} \nn\\
		&\geq \eta \frac{(bx)^{\alpha}}{\Gamma(\alpha)} e^{-bx} 2^{1-\lceil \alpha \rceil}.}
	\paragraph{\textbf{Upper bounding the second term of $E(x)$}} By Chernoff's inequality,
	\eqn{\label{for:prop:Beta-gamma:Chernoff}
		\prob\left( Z>b \right)=\prob\left( e^{Z/2}>e^{\frac{b}{2}} \right)\leq \E\left[ e^{Z/2} \right]e^{-\frac{b}{2}}= 2^{\alpha}e^{-\frac{b}{2}}.}
	We use the fact that $e^{-y}\leq 1$ for $y\geq 0$, for upper bounding the distribution function of $Z$ as
	\eqn{\label{for:lem:Beta-gamma:17}
		\prob\left( Z\leq bx \right) = \frac{1}{\Gamma(\alpha)}\int\limits_{0}^{bx} y^{\alpha-1} e^{-y}\,dy\leq \frac{1}{\Gamma(\alpha)}\int\limits_{0}^{bx} y^{\alpha-1} \,dy= \frac{(bx)^\alpha}{\alpha\Gamma(\alpha)}.}
	Therefore, the second term in \eqref{for:lem:Betagamma:14} can be upper bounded using \eqref{for:lem:Beta-gamma:17} and \eqref{for:prop:Beta-gamma:Chernoff}, as
	\eqn{
		\label{for:lem:Betagamma:15}
		\prob(Z>b)\prob(Z\leq bx)\leq 2^{\alpha} e^{-\frac{b}{2}}\frac{(bx)^\alpha}{\alpha\Gamma(\alpha)}.
	}
	Hence from \eqref{for:lem:Beta-gamma:16} and \eqref{for:lem:Betagamma:15}, we obtain
	\eqan{\label{for:lem:Betagamma:16}
		E(x)&\geq~ \eta\frac{(bx)^{\alpha}}{\Gamma(\alpha)} e^{-bx}2^{1-\lceil\alpha\rceil} - 2^{\alpha} e^{-\frac{b}{2}}\frac{(bx)^\alpha}{\alpha\Gamma(\alpha)}\nn\\
		&\hspace{2cm}=~ \frac{(bx)^\alpha}{\Gamma(\alpha)}e^{-\frac{b}{2}}\left[ e^{\left(\frac{1}{2}-x\right)b+\log \eta} 2^{1-\lceil \alpha \rceil} - \frac{2^\alpha}{\alpha} \right].
	}
	Remember that $\eta>b^{-2}$ and, for sufficiently large $b,\ x\leq b^{-\varrho/2}\leq \tfrac{1}{4}.$ Therefore, {by} \eqref{for:lem:Betagamma:16}, for $b\geq K_\eta,$ 
	\eqn{\label{for:lem:Beta-gamma:18}
		E(x)\geq \frac{(bx)^\alpha}{\Gamma(\alpha)}e^{-\frac{b}{2}}\left[  2^{1-\lceil \alpha \rceil} e^{\frac{b}{4}-2\log b} - \frac{2^\alpha}{\alpha} \right].
	}
	Using $\alpha\in (0,1+|\delta|)$ and taking $b$ large enough, the RHS of \eqref{for:lem:Beta-gamma:18} is non-negative and this proves \eqref{for:lem:Betagamma:13}. This completes the proof of the coupling for $\CPU$.
	
	The proof for $\PU$ {is similar}. Indeed, $b_{u,j}$ is replaced by $b_u$ {given by}
	\eqn{\label{for:prop:bgc:PU:1}
		b_u=a_{[2]}+2\left( m_{[u-1]}-2 \right)+m_u+(u-1)\delta.}
	Now, we show that there exists a constant $K_\eta \geq 1$ such that for all $u > K_\eta$,
	\eqn{\left( 1-\frac{\eta}{2} \right)\frac{x}{b_u}\leq h_u^{\psi}(x)\leq \frac{x}{b_u}.}
	To prove this we \FC{need to} show that for $u>K_\eta\geq \frac{1}{\sqrt{\eta}},$ and $x\leq b_u^{-\varrho/2},$
	\[
	\prob_m\left( \psi_u\leq (1-\eta)x \right)\leq \prob_m\left( \chi_u\leq b_u x \right)\leq \prob_m\left( \psi_u\leq x \right),
	\]
	which we have already proved.
\end{proof}
\subsection[Asymptotics of attachment probabilities and\\regularity of RPPT]{Asymptotics of attachment probabilities and regularity of RPPT}
\label{sec-attachment-RPPT}
{The coupling arguments in Proposition~\ref{prop:Betagamma:coupling:CPU} give us a sequence of conditional} Gamma random variables $\left( \hat{\chi}_v \right)_{v\geq 2}$ corresponding to $\left( \phi_v \right)_{v\geq 3}$, such that $\hat{\chi}_v=(2\E[M]+\delta)v\phi_v$ for $v$ large enough ({where we assume that $v>K_\eta$}). Similarly, for P\'olya urn graphs, we can couple $\left( \hat{\chi}_v \right)_{v\geq 2}$ and $\left( \psi_v \right)_{v\geq 2}$.  These relations are crucial {to bound} the edge-connection probabilities in \eqref{edge:connecting:probability:CPU:SL} and \eqref{edge:connecting:probability:CPU:NSL}, {where we recall that $\chi = {(\E[M]+\delta)}/{(2\E[M]+\delta)}$:}

\begin{Lemma}[Bound on attachment probabilities]
	\label{lem:bound:attachment_probability}
	For $n\in\N$, consider $v\in[n]$. Then, for every $\varepsilon>0$ there exists $\omega>0$ such that, for both $\CPU$ and $\PU$, with probability larger than $1-\varepsilon$, for every $k\geq v\geq K_\omega$ and $j\in[m_k]$,
	\eqn{\nonumber
		(1-\omega)\frac{\hat{\chi}_{v}}{(2\E[M]+\delta)v }\Big(\frac{v}{k}\Big)^\chi
		\leq \prob_{m,\psi}\left(k\overset{j}{\rightsquigarrow} v \right)
		\leq (1+\omega) \frac{\hat{\chi}_{v}}{(2\E[M]+\delta)v }\Big(\frac{v}{k}\Big)^\chi,
	}
	where $\hat{\chi}_{v}$ is a Gamma random variable with parameters $m_v+\delta$ and 1.
\end{Lemma}

\begin{proof}
	Fix $n\in\N$ and $\varepsilon,\omega>0$. The edge-connection probabilities for $\CPU^{\sss\rm{(SL)}},\CPU^{\sss\rm{(NSL)}}$ and $\PU^{\sss\rm{(NSL)}}$ are calculated in \eqref{edge:connecting:probability:CPU:SL}, \eqref{edge:connecting:probability:CPU:NSL} and \eqref{edge:connecting:probability:PU:NSL}, respectively.
	By Proposition~\ref{lem:position:concentration}, with probability {at least} $1-\varepsilon/3$,
	\eqn{\label{for:lem:attprob:urn:modela:1:bis}
		\max\limits_{j\in[m_k]}\left|\frac{\mathcal{S}^{\sss(n)}_v}{\mathcal{S}^{\sss(n)}_{(k,j)}} - \left(\frac{v}{k}\right)^{\chi} \right| \leq    \left(\frac{\omega}{2}\right) \left(\frac{v}{k}\right)^\chi	,
	}
	where the value $\chi$ {comes} from the collapsed P\'olya Urn graph model.
	
	{We next} control the remaining terms in the edge-connection probabilities of \eqref{edge:connecting:probability:CPU:SL}, \eqref{edge:connecting:probability:CPU:NSL} and \eqref{edge:connecting:probability:PU:NSL}. For $\CPU,$ we rewrite the remaining expression as
	\eqn{
		\label{for:lem:attprob:urn:modela:2}
		1-\prod_{i=1}^{m_v}\left(1-\psi_{m_{[v-1]}+i}\right)  = \sum_{i=1}^{m_v} \psi_{m_{[v-1]}+i}+ E_m(v)=\phi_v+E_m(v),
	}
	{which implicitly defines the error term $E_m(v)$.} On the other hand, for $\PU,$ the remaining expression turns out to be $\psi_v.$ By the assumptions, we can apply Proposition~\ref{prop:Betagamma:coupling:CPU} to couple the Gamma random variables to $\phi_v$ and $\psi_v$ of $\CPU$ and $\PU$, respectively. As a consequence, with probability {at least} $1-\varepsilon/3$, simultaneously for all $v$ sufficiently large,
	\eqn{\nonumber
		\begin{split}
			&\left( 1-\frac{\omega}{2} \right) \frac{\hat{\chi}_{v}}{2\E[M]+\delta}\leq v\phi_{v} \leq \left( 1+\frac{\omega}{2} \right) \frac{\hat{\chi}_{v}}{2\E[M]+\delta},\\
			\text{and}\quad&\left( 1-\frac{\omega}{2} \right) \frac{\hat{\chi}_{v}}{2\E[M]+\delta}\leq v\psi_{v} \leq \left( 1+\frac{\omega}{2} \right) \frac{\hat{\chi}_{v}}{2\E[M]+\delta},
		\end{split}
	}
	for $\CPU$ and $\PU$, respectively, {where} $\hat{\chi}_{v}$ is a sequence of independent Gamma distribution with parameters $m_v+\delta$ and $1$.
	
	The term $ |E_m(v)|$ in \eqref{for:lem:attprob:urn:modela:2} is bounded by 
	\eqn{
		\label{for:lem:bound:attachment_probability:1}
		\begin{split}
			\sum\limits_{\substack{i\neq j\\i,j\in[m_v]}}\psi_{m_{[v-1]}+i}\psi_{m_{[v-1]}+j} \leq \Big( \sum\limits_{i\in[m_v]}\psi_{m_{[v-1]}+i} \Big)^2\leq (1+2\omega)\left( \frac{\hat{\chi}_v}{v} \right)^2.
		\end{split}
	}
	{By} Proposition~\ref{lem:position:concentration}, {$\hat{\chi}_v\leq v^{1-\varrho/2}$} for all $v>K_\varepsilon,$ with probability at least $1-\varepsilon$. Therefore, {by} \eqref{for:lem:bound:attachment_probability:1}, 
	\eqn{|E_m(v)|\leq \frac{\hat{\chi}_v}{v}O\left(v^{-\varrho/2}\right),}
	which completes the proof.
\end{proof}

\begin{remark}\label{remark:equal-edge-probability}
	{\rm
		Observe that the edge-connection probabilities are essentially the same for the three models: $\CPU_n^{\sss\rm{(SL)}},\ \CPU_n^{\sss\rm{(NSL)}}$ and $\PU_n^{\sss\rm{(NSL)}}$, except for the self-loop creation probability. {The latter} is insignificant for large graphs. Using Lemma~\ref{lem:bound:attachment_probability}, we approximate the attachment probabilities of all these models by the same expression, with an error {that can be effectively taken care of.}}\hfill $\blacksquare$
\end{remark}

We {close this} section by proving a regularity property of the $\RPPT$ {that is similar to \cite[Lemma~3.3]{BergerBorgs} and} which {is} useful in Section~\ref{chap:convergence:sec:local_convergence}:
\begin{Lemma}[Regularity of RPPT]\label{lem:RPPT:property}
	Fix $r\geq 0$ and $\varepsilon>0$. Then there exist $K,C<\infty$ and $\FC{\eta(\varepsilon,r)} >0$ such that, with probability at least $1-\varepsilon$,
	\begin{enumerate}
		\item\label{RPPT_prop:1} $A_\omega\geq \FC{\eta(\varepsilon,r)}$, for all $\omega\in B_r^{\sss (G)}(\emp)$;
		\item\label{RPPT_prop:2} $\left| B_r^{\sss (G)}(\emp) \right|\leq C$;
		\item\label{RPPT_prop:3} $\Gamma_\omega\leq K$, for all $\omega\in B_r^{\sss(G)}(\emp).$
	\end{enumerate}
\end{Lemma}
\begin{proof}
	The lemma holds {trivially} for $r=0$ and ${\varepsilon/4}$, i.e., with probability at least $1-{\varepsilon/4},$
	\begin{itemize}
		\item[(1)] $A_\emp \geq \eta\left( \frac{\varepsilon}{4},0\right)$, where $\eta\left( \frac{\varepsilon}{4},0\right)$ is a positive value depending on $\varepsilon$ and $r=0$;
		\item[(2)] $\big| B_0^{\sss(G)}(\emp) \big|=1$;
		\item[(3)] $\Gamma_\emp\leq K\left( \frac{\varepsilon}{4},0\right)$, where $K\left( \frac{\varepsilon}{4},0\right)$ is a natural number depending on $\varepsilon$ and $r=0$.
	\end{itemize}
	We {thus} prove the lemma for $r=1$. The proof for $r\geq 2$, {then easily} follows by induction on $r.$
	
	Let $U_{\sss(M)}$ be the smallest order statistic of $U_1,\ldots,U_M$, i.e., $U_{\sss(M)}=\min\{ U_1,\ldots,U_M \}$ and $\eta\leq \eta\left( \frac{\varepsilon}{4},0\right)$. To prove (\ref{RPPT_prop:1}), it is enough to show that
	\[
	\prob\Big(\left. A_\emp U_{\sss(M)}^{1/\chi}<\eta~\right|~ A_\emp>\eta( {\varepsilon/4},0)\Big)\leq \frac{\varepsilon}{4}.
	\]
	First, we condition on $M$ and find a lower bound on the LHS. Next, we take an expectation over $M$. Using the density of $U_{\sss(M)}$ and Taylor's inequality,
	\eqan{\label{for:lem:RPPT_prop:1}
			&\prob\big(\left. A_\emp U_{\sss(M)}^{1/\chi}<\eta~\right|~ M=m, A_\emp\geq \eta({\varepsilon/4},0) \big)\nn\\
			&=\frac{1}{(1-\eta({\varepsilon/4},0))}\int\limits_{\eta( {\varepsilon/4},0)}^1 \left( 1-\left( 1-\left( \frac{\eta}{x} \right)^\chi \right)^m \right)\,dx\nn\\
			&\leq \frac{1}{(1-\eta({\varepsilon/4},0))}\int\limits_{\eta({\varepsilon/4},0)}^1 m\left(\frac{\eta}{x}\right)^{\chi}\,dx\leq \frac{m\eta^\chi}{1-\chi}.
	}
	where the last inequality follows from the fact that $\frac{1-\eta({\varepsilon/4},0)^{1-\chi}}{1-\eta({\varepsilon/4},0)}\leq 1$. Integrating over $m$ and choosing $\eta$ suitably,
	\eqn{\label{for:lem:RPPT_prop:2}
		\prob\big(\left. A_{\emp} U_{\sss(M)}^{1/\chi}<\eta~\right|~ A_{\emp}\geq \eta( {\varepsilon/4},0) \big)\leq \frac{\eta^{\chi}}{2\E[M]+\eta}= \frac{\varepsilon}{4}.}
	Denote the $\eta$ in \eqref{for:lem:RPPT_prop:2} by $\eta\left( {\varepsilon},1 \right)$. 
	
	Next, we prove (\ref{RPPT_prop:2}). The number of children of the root is distributed as $(M+\Lambda)$, where $\Lambda$ is the total number of points in an inhomogeneous Poisson process with intensity $\rho_\emp(x).$ Since $M$ is uniformly integrable, there exists $m_0$ such that $\prob(M\leq m_0)\geq 1-\frac{\varepsilon}{8}$. Moreover, it can also be shown that 
	\[
	\prob( M^{\sss(0)}>m_0) = \frac{1}{\E[M]}\sum\limits_{k\geq m_0+1} k\prob(M=k)=\frac{\E\left[ M\one_{\{M>m_0\}} \right]}{\E[M]}\leq\frac{\varepsilon}{8}.
	\]
	Similar results {also hold} for $M^{\sss(\delta)}$, {and are} useful for the induction step when $r>1$. Now, we have the parameter of $\Lambda$ bounded above and hence there exists $C(\varepsilon,1)$ such that 
	\[
	M+\Lambda\leq C(\varepsilon,1).
	\]
	We can choose $K(\varepsilon,1)<\infty$, such that $\prob\left(\Gamma_\omega\leq K(\varepsilon,1)\mbox{ for all }\omega\in B_1^{\sss(G)}(\emp)\right)$ is at least $ 1-\frac{\varepsilon}{4},$ completing the proof.
\end{proof}
\section{Local convergence proof for models (A), (B) and (D)}
\label{chap:convergence:sec:local_convergence}
In this section, we first prove local convergence for the vertex-marked preferential attachment model (A) and then extend this result to other models (models (B) and (D)). For any finite vertex-marked tree $\tree$ with vertex marks in $[0,1]$, let $V(\tree)$ denote the set of vertices of $\tree$, and let $\{ a_\omega \in [0,1] : \omega \in V(\tree) \}$ represent the age-set of the vertices. Fix $r \in \mathbb{N}$ and $\varepsilon > 0$. Let $G_n = \PArs_n(\boldsymbol{m}, \delta)$. If $B_r^{\sss(G_n)}(v) \simeq \tree$, and $v_\omega$ is the vertex in $G_n$ corresponding to vertex $\omega$ of $\tree$, then we define
\[
N_{r,n}\left(\tree, (a_\omega)_{\omega \in V(\tree)}\right) = \sum\limits_{v \in [n]} \one_{\left\{ B_r^{\sss(G_n)}(v) \simeq \tree, \, |v_\omega/n - a_\omega| \leq 1/r \, \forall \omega \in V(\tree) \right\}}~,
\]
where $B_r^{\sss(G_n)}(v)$ is the $r$-neighbourhood of $v$ in $G_n$. With $B_r^{\sss(G)}(\emp)$ denoting the $r$-neighbourhood of the $\RPPT(M, \delta)$ and $A_\omega$ the age in $\RPPT(M, \delta)$ of the node $\omega$ of $\tree$, we aim to show that
\eqn{\label{target:1}
	\frac{N_{r,n}\left(\tree, (a_\omega)_{\omega \in V(\tree)}\right)}{n} \overset{\prob}{\to} \mu\left( B_r^{\sss(G)}(\emp) \simeq \tree, \, |A_\omega - a_\omega| \leq 1/r \, \forall \omega \in V(\tree) \right).
}
To prove \eqref{target:1}, we use the second moment method, i.e., we prove that $\E\left[ N_{r,n}\left(\tree, (a_\omega)_{\omega \in V(\tree)}\right)\right]/n$ converges to the limit and that the variance of $N_{r,n}/n$ vanishes as $n \rightarrow \infty$. Throughout this section, we consider $\eta = \eta(\varepsilon, r)$ as introduced in Lemma~\ref{lem:RPPT:property}.
\subsection{First moment convergence}\label{subsec:first-moment}
Here we prove the convergence of the first moment. Let $\left( \tree, (a_\omega)_{\omega\in V(\tree)} \right)$ be a vertex-marked tree with marks $(a_\omega)_{\omega\in V(\tree)}$ taking values in $[0,1]^{|V(\tree)|}$. {We compute}
\begin{align}\label{firstmoment:aim}
	&{\frac{1}{n}}\E\left[ N_{r,n}\left(\tree,(a_\omega)_{\omega\in V(\tree)}\right)\right]\nn\\
	 &\hspace{2cm}=\prob\left( B_r^{\sss(G_n)}({o})~\simeq~ \tree, ~|v_\omega/n-a_\omega|\leq 1/r~\forall\omega~\in~V(t) \right), 
\end{align}
{where $o\in [n]$ is chosen uniformly at random. We aim to show that this converges to the RHS of \eqref{target:1}.}

Instead of proving the first moment convergence, we prove {the stronger statement that the {\em age densities}} of the vertices in the $r$-neighbourhood of a uniformly chosen vertex in $\PArs_n(\boldsymbol{m},\delta)$ converges pointwise to the age density of the nodes in the $r$\FC{-neighbourhood} of $\RPPT(M,\delta)$.

Define $f_{r,\tree}\left((A_\omega)_{\omega\in V(t)}\right)$ as the \emph{density of the ages in the $\RPPT(M,\delta)$}, when the ordered $r$-neighbourhood $B_r^{\sss(G)}(\emp)$ is in the same equivalence class as $\tree$. Then,
\eqn{\label{eq:RPPT:density:1}
	\mu\left( B_r^{\sss(G)}(\emp)\simeq \tree, A_\omega\in \,da_\omega,~\forall \omega\in V(\tree) \right) = f_{r,\tree}\left((a_\omega)_{\omega\in V(t)}\right)\prod\limits_{\omega\in V(\tree)}\,da_\omega.}
\begin{Theorem}[First moment density convergence theorem]
	\label{thm:prelimit:density}
	Fix $\delta>-\infsupp(M),$
	and consider $G_n=\PArs_n(\boldsymbol{m},\delta)$. Uniformly for all $a_\omega>\eta,$ distinct $v_\omega$, and $\hat{\chi}_{v_\omega}\leq (v_\omega)^{1-\frac{\varrho}{2}}$ for all $\omega\in V(\tree)$,
	\eqn{\label{eq:prelimit:density:1}
		\begin{split}
			&\prob_{m,\psi}\left( B_r^{\sss(G_n)}(v)~\simeq~ \tree,~v_\omega=\lceil na_\omega \rceil~ \forall \omega\in V(\tree) \right)\\
			&\qquad= (1+o_\prob(1))\frac{1}{n^{|V(\tree)|}}g_{r,\tree}\left( \left( {a_\omega} \right)_{\omega\in V(\tree)}; \left( m_{v_\omega},\hat{\chi}_{v_\omega} \right)_{\omega\in V(\tree)} \right),
	\end{split}}
	for some measurable function $g_{r,\tree}\left( \left( a_\omega \right)_{\omega\in V(\tree)}; \left( m_{v_\omega},\hat{\chi}_{v_\omega} \right)_{\omega\in V(\tree)} \right)$, where $(\hat{\chi}_v)_{v\in\N}$ are the Gamma random variables coupled with the corresponding $\Beta$ random variables in Proposition~\ref{prop:Betagamma:coupling:CPU}.
	Consequently, with $\left( \hat{\chi}_{v_\omega} \right)_{\omega\in V(\tree)}$ a conditionally independent sequence of ${\rm{Gamma}}(m_{v_{\omega}}+\delta,1)$ random variables,
	\eqan{\label{eq:prelimit:density:2}
		&\E\left[ g_{r,\tree}\left( \left( a_\omega \right)_{\omega\in V(\tree)}; \left( m_{v_\omega},\hat{\chi}_{v_\omega} \right)_{\omega\in V(\tree)} \right) \right]= f_{r,t}\left((a_\omega)_{\omega\in V(t)}\right).}
\end{Theorem}
By \eqref{eq:prelimit:density:2}, Theorem~\ref{thm:prelimit:density} can be seen as a \textit{local density limit theorem} for the ages of the vertices {in the} $r-$neighbourhoods. This is {significantly} stronger than local convergence of preferential attachment models. Further, when $(\lceil na_\omega \rceil)_{\omega\in V(\tree)}$ are distinct (which occurs whp), $\left( \hat{\chi}_{\lceil na_\omega\rceil} \right)_{\omega\in V(\tree)}$ are {\em conditionally independent} Gamma variables with parameters $m_{v_\omega}+\delta$ and $1$.

We prove Theorem~\ref{thm:prelimit:density} below in several steps. First we calculate the conditional density $f_{r,\tree}$. Using the equivalence of $\CPU$ and model (A) in Proposition~\ref{prop:equiv:CPU:PArs}, we compute the explicit expression in \eqref{eq:prelimit:density:1}. Let $\partial V(\tree)$ denotes the leaf nodes of the tree $\tree$ and $V^\circ(\tree)$ the set of vertices in the interior of the tree $\tree$, i.e., $V^\circ(\tree)= V(\tree)\setminus \partial V(\tree)$. {Then the age densities in $\RPPT(M,\delta)$ are identified as follows:}

\begin{Proposition}[Law of vertex-marked neighbourhood of RPPT]
	\label{prop:RPPT:density} 
	Let $M$ be the law of the out-degrees, and fix $\delta>-\infsupp(M)$. Let $n_{\old}$ and $n_{\young}$ denote the number of $\Old$ and $\Young$ labelled nodes in $\tree$ and $E(\tree)$ the edge-set of the tree $\tree$. Then,
	\eqan{\label{eq:prop:RPPT:density}
		&~f_{r,t}\left((a_\omega)_{\omega\in V(t)}\right)\nn\\
		=&~\chi^{n_{\old}}(1-\chi)^{n_{\young}}\E\left[\prod\limits_{\omega\in V^\circ(t)} m_{-}(\omega)!\, \Gamma_\omega^{d_{\omega}^{({\rm{in}})}(t)}\exp{\left(-\Gamma_\omega \lambda(a_\omega)\right)}\right.\nn\\
		&\hspace{3.5cm}\times\left.\prod\limits_{(\omega,\omega l)\in E(\tree)}(a_\omega \vee a_{\omega l})^{-\chi}(a_\omega \wedge a_{\omega l})^{-(1-\chi)}\right],}
	where $d_{\omega}^{({\rm{in}})}(t) = \#\{ \omega l:~ a_{\omega l}>a_\omega \}$ denotes the number of $\Young$ labelled children of $\omega,$ while $(\Gamma_\omega,~m_{-}(\omega))$ are distributed as in Section~\ref{chap:LWC:sec:RPPT}, and $\lambda$ is a real-valued function on $(0,1)$ defined as 
	\eqn{\label{eq:def:lambda-func}\lambda(x)=\frac{1-x^{1-\chi}}{x^{1-\chi}}.}
\end{Proposition}

To prove {Proposition \ref{prop:RPPT:density},} we need to identify the densities of $(A_\omega)_{\omega\in V(\tree)}$. We start by analysing the density of the age of $\Old$ labelled nodes:

\begin{Lemma}[Conditional age density of $\Old$ labelled children]
	\label{lem:old:density} Conditionally on $a_\omega,$ the age of its parent $\omega\in V^\circ(\tree)$, the density of \FC{an} $\Old$ labelled child on $[0,a_\omega]$ is given by
	\eqn{\label{eq:lem:old:density}
		f_{\old}(x)=\chi x^{-(1-\chi)}a_\omega^{-\chi}.}
\end{Lemma}

\begin{proof}
	From the construction of $\RPPT$, for any $\Old$ labelled node $\omega l$, its age is given by $U^{1/\chi}A_\omega$, where $U$ is {uniform} in $[0,1]$ independently of $A_\omega.$ Therefore, conditionally on $A_\omega=a_\omega,$
	\begin{align}\label{for:lem:old:density:1}
		\prob\left( A_{\omega l}\leq x\mid A_\omega=a_\omega \right) =&~ \prob\left( A_\omega U^{1/\chi}\leq x\mid A_\omega=a_\omega \right) \nn\\
		=&~ \prob\left(U^{1/\chi}\leq x a_\omega^{-1}\right)= a_\omega^{-\chi} x^{\chi}.
	\end{align}
	Now, differentiating the RHS of \eqref{for:lem:old:density:1} with respect to $x,$ we obtain the conditional density of $A_{\omega l}$ in \eqref{eq:lem:old:density}.
\end{proof}
The ages of the $\Young$ labelled {children of $\omega$} follow an inhomogeneous Poisson process with intensity
\begin{equation}
	\label{intensity-younger-children}
	\rho_{\omega}(x) = (1-\chi)\Gamma_\omega\frac{x^{-\chi}}{A_\omega^{1-\chi}}.
\end{equation}
{This leads to the following density result on the ages of all children of $\omega$:}
\begin{Lemma}[Conditional age density of children]
	\label{lem:children:age:density}
	For any $\omega\in V^\circ(\tree),$ conditionally on $(m_{-}(\omega),A_\omega,\Gamma_\omega)$, the density of the ages of the children of $\omega$ is given by
	\eqan{\label{eq:lem:children:age:density}
		&f_\omega\left(\left.(a_{\omega l})_{l\in d_{\omega}(\tree)}\right|m_{-}(\omega),a_\omega,\Gamma_\omega \right)\nn\\
		=&m_{-}(\omega)!\prod\limits_{l=1}^{m_{-}(\omega)}\left[ \chi (a_{\omega l})^{\chi-1}(a_\omega)^{-\chi} \right]\exp{\left(-\Gamma_\omega\lambda(a_\omega)\right)}\nn\\
		&\hspace{2.5cm}\times\prod\limits_{k=1}^{d_{\omega}^{{\rm{(in)}}}(\tree)}\left[ (1-\chi)\left( a_{\omega (m_{-}(\omega)+k)} \right)^{-\chi}(a_\omega)^{-(1-\chi)}\Gamma_\omega \right] ,}
	where $m_{-}(\omega)$ is as defined in Section~\ref{chap:LWC:sec:RPPT}, and $d_{\omega}(\tree)= m_{-}(\omega)+d_{\omega}^{\sss\rm{(in)}}(\tree)$ denotes the number of children of $\omega$.
\end{Lemma}
\begin{proof}
	From the construction of the $\RPPT$, we see that any node $\omega$ have $m_{-}(\omega)$ many $\Old$ labelled children. Recall that, $m_{-}(\omega)\sim M^{(\delta)}$ if $\omega$ has label $\Old$, whereas $m_{-}(\omega) \sim M^{(0)}$ if $\omega$ has label $\Young$, and $m_{-}(\emp)\sim M$. Since the uniform random variables are chosen independently for obtaining the age of the $\Old$ labelled children, using Lemma~\ref{lem:old:density}, 
	\eqan{\label{for:lem:cad:1}
		&f\left( \left.(a_{\omega l})_{l\in [m_{-}(\omega)]}\right|m_{-}(\omega),a_\omega,\Gamma_\omega \right) \nn\\
		&\hspace{1cm}= m_{-}(\omega)!\prod\limits_{l=1}^{m_{-}(\omega)}\left[ \chi (a_{\omega l})^{\chi-1}(a_\omega)^{-\chi} \right].}
	{Indeed, since} there is no particular order for connecting to the older nodes, $\omega$ can connect to its older children with its edges in $m_{-}(\omega)!$ different ways.
	
	The label $\Young$ children have ages coming from a Poisson process with (random) intensity $x\mapsto\rho_\omega(x)$ on $\left[ a_\omega,1 \right]$ {defined in \eqref{intensity-younger-children}}. 
	Therefore for $k\geq2$, {conditionally on $a_{\omega (m_{-}(\omega)+k-1)},~\omega (m_{-}(\omega)+k)$} has an age following a non-homogeneous exponential distribution with intensity
	\eqn{x\mapsto \rho_\omega(x)=(1-\chi)\Gamma_{\omega}\frac{x^{-\chi}}{a_{\omega}^{1-\chi}}\one_{\{x> a_{\omega (m_{-}(\omega)+k-1)}\}}  ~.} 
	
	Additionally, there is one more factor arising in this part of the density, which is the no-further $\Young$ labelled child after $\omega d_\omega(t)$. Conditionally on $a_{\omega d_{\omega}(\tree)}$, this no-further child factor is given by
	\[
	\exp{\left(-\Gamma_\omega\frac{1-a_{\omega d_{\omega}(\tree)}^{1-\chi}}{a_\omega^{1-\chi}}\right)}.
	\]
	This no-further child part is also independent of the ages of the $\Young$ labelled nodes except for $a_{\omega d_\omega(\tree)}$, from the property of the Poisson process. Therefore,
	\eqn{\label{for:lem:cad:2}
		\begin{split}
			&f\left( \left. \left(a_{\omega (m_{-}(\omega)+l)}\right)_{l\in [d_\omega^{\sss\rm{(in)}}(\tree)]}\right|m_{-}(\omega),a_\omega,\Gamma_\omega \right)\\
			&= \exp{\left(-\Gamma_\omega\lambda(a_\omega)\right)}\prod\limits_{k=m_{-}(\omega)+1}^{d_{\omega}(\tree)}\left[ (1-\chi)\left( a_{\omega k} \right)^{-\chi}(a_\omega)^{\chi-1} \Gamma_\omega \right] .
	\end{split}}
	Since the $\Young$ labelled nodes connect one by one sequentially in a particular order, $\omega$ connects to its $\Young$ labelled children in only $1$ way. Now, from the construction, the edges to the $\Old$ labelled children and the edges to the $\Young$ labelled children are created independently (conditionally on $(m_{-}(\omega),A_\omega,\Gamma_\omega)$). Therefore from \eqref{for:lem:cad:1} and \eqref{for:lem:cad:2} and this independence,
	\eqan{\label{for:lem:cad:3}
		&f_\omega\left(\left.(a_{\omega l})_{l\in [d_{\omega}(\tree)]}\right|m_{-}(\omega),a_\omega,\Gamma_\omega \right) \\
		= &f\left( \left.(a_{\omega l})_{l\in [m_{-}(\omega)]}\right|m_{-}(\omega),a_\omega,\Gamma_\omega \right) f\left( \left. \left(a_{\omega (m_{-}(\omega)+l)}\right)_{l\in [d_\omega^{\sss\rm{(in)}}(\tree)]}\right|m_{-}(\omega),a_\omega,\Gamma_\omega \right),\nn}
	which leads to the claimed expression in \eqref{eq:lem:children:age:density}.
\end{proof}

Now we have the required tools to prove Proposition~\ref{prop:RPPT:density}, for which we use induction on  $r$, the depth of the tree:
\begin{proof}[\bf Proof of Proposition~\ref{prop:RPPT:density}]
	{For proving this proposition,} we define some notation \textbf{\textit{that is used in this proof only}}. Define $\tree_r$ to be the $r$-neighbourhood of the root in $\tree$. Therefore,
	\eqn{V^\circ(\tree_{r+1})=V^\circ(\tree_r)\cup \partial V(\tree_r). \nonumber}
	First, we prove the proposition for $r=1$. Here $V^\circ(\tree_1)=\{\emp\}$. By Lemma~\ref{lem:children:age:density},
	\eqan{\label{for:prop:RPPTd:1}
		&f_{\emp}\left(\left.(a_{\emp l})_{l\in [d_{\emp}(\tree)]}\right|m_{-}(\emp),a_\emp,\Gamma_\emp \right)\nn\\
		=& m_{-}(\emp)!\prod\limits_{l=1}^{m_{-}(\emp)}\left[ \chi (a_{\emp l})^{\chi-1}(a_\emp)^{-\chi} \right]\\
		&\hspace{0.5cm}\times\prod\limits_{k=1}^{d_{\emp}^{{\rm{(in)}}}(\tree)}\left[ (1-\chi)\left( a_{\emp (m_{-}(\emp)+k)} \right)^{-\chi}(a_\emp)^{\chi-1}\Gamma_\emp \right] \exp{\left(-\Gamma_\emp\lambda(a_\emp)\right)}.\nn}
	Now, observe that for every edge to an $\Old$ labelled child, we obtain a factor $\chi$ and for every edge to $\Young$ labelled children, we obtain a factor $(1-\chi)$ and a factor $\Gamma_\emp$ in \eqref{for:prop:RPPTd:1}. Further note that $\Old$ labelled children of $\emp$ have smaller age than $a_\emp$ and $\Young$ labelled children have higher age than $a_\emp$. $\tree_1$ has $m_{-}(\emp)$ many $\Old$ labelled nodes and $d_{\emp}^{\sss\rm{(in)}}(\tree)$ many $\Young$ labelled nodes. Therefore \eqref{for:prop:RPPTd:1} can be rewritten as
	\eqan{\label{for:prop:RPPTd:2}
		&f_{\emp}\left(\left.(a_{\emp l})_{l\in {[d_{\emp}(\tree)]}}\right|m_{-}(\emp),a_\emp,\Gamma_\emp \right) \nn\\
		=& \chi^{n_{1,\old}}(1-\chi)^{n_{1,\young}}m_{-}(\emp)!~ \Gamma_\emp^{d_{\emp}^{\sss\rm{(in)}}(\tree)}\exp{\left(-\Gamma_\emp\lambda(a_\emp)\right)}\nn\\
		&\hspace{1cm}\times\prod\limits_{(\emp,\emp l)\in E(\tree_1)}\left[ \left( a_\emp\vee a_{\emp l} \right)^{-\chi}\left( a_\emp\wedge a_{\emp l} \right)^{-(1-\chi)} \right],}
	where $n_{1,\old}$ and $n_{1,\young}$ denote the number of $\Old$ and $\Young$ labelled nodes in $\tree_1$. Since $A_\emp$ has a uniform distribution on $[0,1]$, and is independent of $m_\emp$ and $\Gamma_\emp$,
	\eqan{\label{for:prop:RPPTd:3}
		f_{1,\tree}\left( (a_\omega)_{\omega\in V(t)} \right) =&~ \E\left[ f_{1,\tree}\left(\left.a_\emp,(a_{\emp l})_{l\in {[d_{\emp}(\tree)]}}\right|m_{-}(\emp),\Gamma_\emp \right) \right]\nn\\
		=&~\E\left[f_{\emp}\left(\left.(a_{\emp l})_{l\in {[d_{\emp}(\tree)]}}\right|m_{-}(\emp),a_\emp,\Gamma_\emp \right)\right]\\
		=& \chi^{n_{1,\old}}(1-\chi)^{n_{1,\young}} \E\left[ m_{-}(\emp)!\Gamma_\emp^{d_{\emp}^{\sss\rm{(in)}}(\tree)}\exp{\left(-\Gamma_\emp\lambda(a_\emp)\right)} \right.\nn\\
		&\hspace{2cm} \times \prod\limits_{(\emp,\emp l)\in E(\tree_1)}\left[ \left( a_\emp\vee a_{\emp l} \right)^{-\chi}\left( a_\emp\wedge a_{\emp l} \right)^{-(1-\chi)} \right] \Bigg].\nn
	}
	The first equality comes from the fact that $V^\circ(\tree_1)=\emp$ and hence the proposition is proved for $r=1.$ {We} proceed toward the induction step. Suppose that \eqref{prop:RPPT:density} be true for $r=k\in\N$. We wish to show that the result holds true for $r=k+1$. 
	
	We have the distribution for the ages of the nodes in $V(\tree_k).$ What remains is to compute the density of the boundary conditionally on the ages of $V(\tree_k)$. Now, the nodes in $\partial V(\tree_{k+1})$ are the children of the nodes in $\partial V(\tree_k)$ and their age-distribution is independent of the age of nodes in $V^\circ (\tree_k)$. {By} Lemma~\ref{lem:children:age:density},
	\eqan{\label{for:prop:RPPTd:4}
		&f\left( \left.(a_\omega)_{\omega\in\partial V(\tree_{k+1})}\right|\left( m_{-}(u),a_u,\Gamma_u \right)_{u\in V(\tree_k)} \right)\nn\\
		=&f\left( \left.(a_{\omega l})_{\omega l\in\partial V(\tree_{k+1})}\right|\left( m_{-}(\omega),a_\omega,\Gamma_\omega \right)_{\omega\in \partial V(\tree_k)} \right)\nn\\
		=& \chi^{n_{k+1,\old}} (1-\chi)^{n_{k+1,\young}}\prod\limits_{\omega\in \partial V(\tree_k)} \left[ m_{-}(\omega)! \Gamma_{\omega}^{d_{\omega}^{\sss\rm{(in)}}(\tree_{k+1})}\exp{\left( -\Gamma_\omega\lambda(a_\omega) \right)}\right.\nn\\
		&\hspace{2.5 cm}\times\prod\limits_{(\omega, \omega l)\in E(\tree_{k+1})} \left( a_\omega \vee a_{\omega l} \right)^{-\chi} \left( a_\omega \wedge a_{\omega l} \right)^{-(1-\chi)} \Big],}
	where $n_{k+1,\old}$ and $n_{k+1,\young}$ denote the number of $\Old$ labelled and $\Young$ labelled nodes in $\partial V(\tree_{k+1})$. On the other hand, $\tree_k$ is a rooted tree of depth $k$. Therefore, using induction on $r$, the depth of the tree from the root,
	\eqan{\label{for:prop:RPPTd:5}
		&~f\left( (a_\omega)_{\omega\in V(\tree_k)} \right) \nn\\
		= &~\chi^{n_{[k],\old}}(1-\chi)^{n_{[k],\young}}\E\left[\prod\limits_{\omega\in V^\circ(\tree_k)}  m_{-}(\omega) !\, \Gamma_\omega^{d_{\omega}^{({\rm{in}})}(t)}\exp{\left(-\Gamma_\omega\lambda(a_\omega)\right)}\right.\nn\\
		&\hspace{3.5cm}\times\left.\prod\limits_{(\omega,\omega l)\in E(\tree_k)}(a_\omega \vee a_{\omega l})^{-\chi}(a_\omega \wedge a_{\omega l})^{-(1-\chi)}\right], }
	where $n_{[k],\old}$ and $n_{[k],\young}$ denote the number of $\Old$ labelled and $\Young$ labelled nodes in $V(\tree_k)$. Since $\tree$ is tree, $d_{\omega}^{\sss\rm{(in)}}(\tree_{k+1})=d_{\omega}^{\sss\rm{(in)}}(\tree)$. Moreover $n_{[k],\old}+n_{k+1,\old} =n_{[k+1],\old}$ and  $n_{[k],\young}+n_{k+1,\young} =n_{[k+1],\young}$ and $\left( m_{-}(\omega),\Gamma_\omega \right)_{\omega\in V(\tree)}$ are independent random variables. Therefore from \eqref{for:prop:RPPTd:4} and \eqref{for:prop:RPPTd:5}, we obtain the required result for $r=k+1$ as
	\eqan{
		&~f_{k+1,\tree}\left( (a_\omega)_{\omega\in V(\tree)} \right) \nn\\
		=&~\chi^{n_{[k+1],\old}}(1-\chi)^{n_{[k+1],\young}}\E\left[\prod\limits_{\omega\in V^\circ(\tree)} m_{-}(\omega) !\, \Gamma_\omega^{d_{\omega}^{({\rm{in}})}(t)}\exp{\left(-\Gamma_\omega\lambda(a_\omega)\right)}\right.\nn\\
		&\hspace{3cm}\times\left.\prod\limits_{(\omega,\omega l)\in E(\tree_{k+1})}(a_\omega \vee a_{\omega l})^{-\chi}(a_\omega \wedge a_{\omega l})^{-(1-\chi)}\right].}
	The general claim follows by induction in $r.$
\end{proof}

By Proposition~\ref{prop:RPPT:density}, we have the exact expression for the density of the ages of the nodes {in} the $\RPPT$. To prove Theorem~\ref{thm:prelimit:density}, we compute the expression for $g_{r,\tree}$. Instead of finding $g_{r,\tree}$ directly, we use \FC{Theorem~\ref{thm:equiv:CPU:PArs}}. {To make} our computation simpler, we first introduce edge-marks in the tree and then lift the edge-marks carefully.

Let $\overline{\rm{t}}=\left( \tree, (a_\omega)_{\omega\in V(\tree)}, (e_{\omega,\omega j})_{(\omega,\omega j)\in E(\tree)} \right)$ be the edge-marked version of $\tree$, where $E(\tree)$ is the edge-set of $\tree$. 
We write $\overline{B}_r^{\sss(G_n)}(v)\doteq\overline{\tree}$ to denote that the vertex and edge-marks of the $r$-neighbourhood of the vertex $v$ in $\CPU_n(\boldsymbol{m},\boldsymbol{\psi})$ are given by those in $\tree$.

\begin{Proposition}[Density of vertex and edge-marked $\CPU$]\label{prop:CPU:density}
	Let $o_n$ be a uniformly chosen vertex from $\CPU_n^{\sss\rm{(SL)}}(\boldsymbol{m},\boldsymbol{\psi})$. Then,
	\eqan{\label{eq:prop:CPU:density:1}
		&\prob_{m,\psi}\left( \overline{B}_r^{\sss(G_n)}(\FC{o_n})\doteq\overline{\tree},~v_\omega=\lceil na_\omega \rceil,~\forall \omega\in V(\tree) \right)\nn\\
		=~&(1+o_\prob(1))n^{-|V(\tree)|}\prod\limits_{\omega\in V^\circ(\tree)}\left[\left( \frac{\hat{\chi}_{v_\omega}}{2\E[M]+\delta} \right)^{d_\omega^{\sss\rm{(in)}}(\tree)} \exp{\left( -\hat{\chi}_{v_\omega} \lambda(a_\omega) \right)}\right]\nn\\
		&\hspace{0.5cm}\times\prod\limits_{(\omega,\omega l)\in E(\tree)}\left( a_{\omega}\vee a_{\omega l} \right)^{-\chi}\left( a_{\omega}\wedge a_{\omega l} \right)^{-(1-\chi)}\prod\limits_{\substack{\omega\in V(\tree)\\\omega \text{ is $\Old$ }}}\frac{\hat{\chi}_{v_\omega}}{2\E[M]+\delta}~.}
\end{Proposition}

Before we prove Proposition \ref{prop:CPU:density} we begin with a lemma that gives an estimate of the edge-connection probabilities and this estimate will be useful in the proof of the above proposition. The estimate is given in terms of the coupled Gamma random variables $(\hat{\chi}_v)_{v\geq 2}$ from Proposition \ref{prop:Betagamma:coupling:CPU}:
\begin{Lemma}
	\label{lem:NFE:main}
	Conditionally on $(\boldsymbol{m,\psi}),$ for all $\omega\in V^\circ(\tree)$,
	\eqn{\label{eq:lem:NFE:main}
		\sum\limits_{u,j\colon u\geq v_\omega} p^{\sss(j)}(u,v_\omega) = (1+o_\prob(1))\hat{\chi}_{v_\omega}\lambda(a_\omega).}
\end{Lemma}
\begin{proof}
	For every $u$, there are $m_u$ many out-edges from $u$ and using the expression for the edge-connection probabilities in Lemma~\ref{lem:bound:attachment_probability},
	\eqan{\label{for:lem:NFE:main:1}
		\sum\limits_{u,j:u\geq v_\omega} p^{\sss(j)}(u,v_\omega)=&~ (1+o_\prob(1))\sum\limits_{u\geq v_\omega} m_u \frac{\hat{\chi}_{v_\omega}}{(2\E[M]+\delta)v_\omega}\left( \frac{v_\omega}{u} \right)^{\chi}\nn\\
		=&~(1+o_\prob(1))\frac{\hat{\chi}_{v_\omega}}{(2\E[M]+\delta)v_\omega^{1-\chi}}\sum\limits_{u\geq v_\omega} m_u u^{-\chi}\\
		=&~ (1+o_\prob(1))\frac{\hat{\chi}_{v_\omega}}{(2\E[M]+\delta)\left(\frac{v_\omega}{n}\right)^{1-\chi}}\Big[\frac{1}{n}\sum\limits_{u\geq v_\omega} m_u \left( \frac{u}{n} \right)^{-\chi}  \Big].\nn}
	Define $T_n:=~ \frac{1}{n}\sum\limits_{u\geq v_\omega} m_u \left( \frac{u}{n} \right)^{-\chi}.$ With $v_\omega=\lceil na_\omega \rceil$, we aim to show that 
	\eqn{
		\label{WSLLN}
		T_n\overset{\prob}{\to} \E[M]\int\limits_{a_{\omega}}^1 t^{-\chi}\,dt
	}
	{using a {\em weighted}} strong law of large numbers. Let $(X_i)_{i\geq 1}$ be a sequence of i.i.d.\ random variables with finite mean. Then from \cite[Theorem 5]{choisung87}, a sufficient condition for $\sum\limits_{i=1}^n a_{i,n} X_i~\left(\text{with }\sum\limits_{i=1}^n a_{i,n}=1\right)$ to converge to $\E[X]$ is {that $\max\limits_{i\in [n]} ~a_{i,n} = O(1/n).$}
	
	{In} our case, consider $a_{u,n} = \frac{\frac{1}{n}\left( \frac{u}{n} \right)^{-\chi}}{\frac{1}{n}\sum\limits_{z\geq v_\omega}\left( \frac{z}{n} \right)^{-\chi}},$ {so that}
	\eqn{\label{for:lem:NFE:main:2}
		\max\limits_{u\in [v_\omega,n]} \frac{\frac{1}{n}\left( \frac{u}{n} \right)^{-\chi}}{\frac{1}{n}\sum\limits_{z\geq v_\omega}\left( \frac{z}{n} \right)^{-\chi}}\leq \frac{a_{\omega}^{-\chi}}{n(1-a_\omega^{1-\chi})}=~O\left( \frac{1}{n} \right)  ~.}
	Therefore, the weighted strong law {implies that \eqref{WSLLN} holds.} Note that the denominator here is the Riemann sum approximation of $\int\limits_{a_\omega}^1 t^{-\chi}\,dt$, {so that}
	\eqn{\label{for:lem:NFE:main:3}
		T_n=(1+o_\prob(1))\E[M]\int\limits_{a_\omega}^1 z^{-\chi}\,dz=~(1+o_\prob(1))\frac{\E[M]}{1-\chi}\left[ 1-a_\omega^{1-\chi} \right].}
	Hence, from \eqref{for:lem:NFE:main:1} and \eqref{for:lem:NFE:main:3},
	\eqn{\begin{split}
			\sum\limits_{u,j:u\geq v_\omega} p^{\sss(j)}(u,v_\omega)=&~ (1+o_\prob(1))\frac{\hat{\chi}_{v_\omega}}{(2\E[M]+\delta)a_{\omega}^{1-\chi}} \frac{\E[M]}{1-\chi}\left[ 1-a_\omega^{1-\chi} \right]\\
			=&~ (1+o_\prob(1))\hat{\chi}_{v_\omega}\lambda(a_\omega),
	\end{split}}
	as required.
\end{proof}

\begin{proof}[\bf Proof of Proposition~\ref{prop:CPU:density}]
	The proof of Proposition~\ref{prop:CPU:density} is divided into several steps. Recall that { Lemma~\ref{lem:bound:attachment_probability} yields} the edge-connection probabilities for $\CPU$. 
	\paragraph{\bf Computing the conditional law of $\overline{B}_r^{\sss(G_n)}(o_n)$}
	{The construction of the $\CPU$ implies the conditional independence of the} edge-connection events. Conditionally on 
	$(\boldsymbol{m,\psi})$, let $p^{\sss(j)}(v,u)$ denote the probability of connecting the $j$-th edge from $v$ to $u$. {Then,}
	\eqan{\label{eq:conditional_law:CPU:1}
			&\prob_{m,\psi}\left( \overline{B}_r^{\sss(G_n)}(o_n)\doteq\overline{\tree},~v_\omega=\lceil na_\omega \rceil,~\forall \omega\in V(\tree) \right)\nn\\
			=~&\frac{1}{n}\prod\limits_{u\in V(\tree)}\prod\limits_{\substack{\omega\in V(\tree)\\\omega>u\\\omega\overset{j}{\sim} u}} p^{\sss(j)}(v_\omega,v_u) \chRounak{\prod\limits_{u\in [n]}\prod\limits_{j\in[m_u]}\Bigg[ 1-{\sum\limits_{\substack{\omega\in V^\circ(\bar{\tree})\\(u,j,v_\omega)\notin E(\bar{\tree})}}p^{\sss(j)}(u,v_\omega) \Bigg]}.
	}}
	The factor of ${1/n}$ arises due to the uniform choice of the root and the first product comprises all edge-connection probabilities to make sure that the edges in $\overline{\tree}$ are there in $\overline{B}_r^{\sss(G_n)}(o_n)$ keeping the edge-marks the same. The last product ensures that there is no further edge in the $(r-1)$-neighbourhood of the randomly chosen vertex $o_n$, so that \FC{the vertex and edge-marks of $\overline{B}_r^{\sss(G_n)}(o_n)$ are the same as those in $\overline{\tree}$.}
	\paragraph{\bf The no-further edge probability}
	{We} {continue} by analyzing the second product on the RHS of \eqref{eq:conditional_law:CPU:1} which, for simplicity, we call the \textit{no-further edge probability}. Observe that in this part of the expression, we have not included the edges that are connected in $\tree,$ i.e., we exclude those factors involving $p^{\sss(j)}(v_u,v_\omega),$ for which $\omega\in V^\circ(\tree)$ and $u\overset{j}{\rightsquigarrow}\omega$ in $\tree$. Recall from Lemma~\ref{lem:RPPT:property} that the minimum age of the vertices in $\tree$ is $\eta$ whp. \chRounak{Further, any vertex $u\in[n]$ is allowed to connect to $v_\omega$ if and only if $v_\omega\leq u$.} Therefore by Lemma~\ref{lem:bound:attachment_probability}, for all $u>\eta n,~p^{\sss(j)}(u,v)$ is $o_\prob(1).$ Since $\tree$ is finite, we include a finite number of $p^{\sss(j)}(u,v_\omega)$ factors and hence we can approximate the \textit{no-further edge probability} as
	\eqan{
		&\chRounak{\prod\limits_{u\in [n]}\prod\limits_{j\in[m_u]}\Bigg[ 1-{\sum\limits_{\substack{\omega\in V^\circ(\bar{\tree})\\(u,j,v_\omega)\notin E(\bar{\tree})}}p^{\sss(j)}(u,v_\omega) \Bigg]}}\nn\\
		=&(1+o_\prob(1))\chRounak{\prod\limits_{u\in [n]}\prod\limits_{j\in[m_u]}{\prod\limits_{\substack{\omega\in V^\circ(\bar{\tree})\\(u,j,v_\omega)\notin E(\bar{\tree})}}\Big[ 1-p^{\sss(j)}(u,v_\omega) \Big]}}\label{eq:NFE:01}\\
		=&(1+o_\prob(1))\chRounak{\prod\limits_{\omega\in V^\circ(\bar{\tree})}{\prod\limits_{\substack{u\geq v_\omega}}\prod\limits_{j\in[m_u]}\Big[ 1-p^{\sss(j)}(u,v_\omega) \Big]}. \label{eq:NFE:1}}
		}
	We can approximate
	\eqan{\label{eq:NFE:2}
		&\chRounak{\prod\limits_{u\geq v_\omega}\prod\limits_{j=1}^{m_u} \Big[ 1-p^{\sss(j)}(u,v_\omega) \Big]}\nn\\
		=& \exp{\Big(\Theta(1)\sum\limits_{u,j\colon u\geq v_\omega}p^{\sss(j)}(u,v_\omega)^2\Big)}\exp{\Big( -\sum\limits_{u,j:u\geq v_\omega} p^{\sss(j)}(u,v_\omega) \Big)}.}
	We next investigate the first term in the RHS of \eqref{eq:NFE:2}, \FC{while} the second, and the main, term is estimated using Lemma~\ref{lem:NFE:main}. 
	We prove that the first exponential term in \eqref{eq:NFE:2} is $1+o_\prob(1)$. Using Lemma~\ref{lem:bound:attachment_probability}, similarly as in \eqref{for:lem:NFE:main:1},
	\eqan{\label{eq:NFE:3}
			&\sum\limits_{u,j:u\geq v_\omega} p^{\sss(j)}(u,v_\omega)^2 \nn\\
			=& (1+o_\prob(1))\sum\limits_{u\geq v_\omega} m_u \frac{\hat{\chi}_{v_\omega}^2}{(2\E[M]+\delta)^2v_\omega^2}\left( \frac{v_\omega}{u} \right)^{2\chi} \nn\\
			\leq& (1+o_\prob(1))\frac{\hat{\chi}_{v_\omega}^2}{(2\E[M]+\delta)^2\left({v_\omega}\right)^{2-\chi}}\Big[ \sum\limits_{u\geq v_\omega} m_u\left( \frac{1}{u} \right)^{\chi} \Big] \nn\\
			=& (1+o_\prob(1))\frac{(\hat{\chi}_{v_\omega}^2/n)}{(2\E[M]+\delta)^2\left(\frac{v_\omega}{n}\right)^{2-\chi}}\Big[ \frac{1}{n}\sum\limits_{u\geq v_\omega} m_u\left( \frac{u}{n} \right)^{-\chi} \Big].
	}
	By \eqref{for:lem:NFE:main:3} and recalling that $v_\omega=\lceil na_\omega \rceil$,
	\[
	\frac{1}{n} \sum\limits_{u\geq v_\omega} m_u\left( \frac{u}{n} \right)^{-\chi} \overset{a.s.}{\to} \frac{\E[M]}{1-\chi}\left[ 1-a_\omega^{1-\chi} \right]<C,
	\]
	for some constant $C>0$. It is enough to show that $\frac{\hat{\chi}_{v_\omega}^2}{n}=o_\prob(1)$. {For this, we fix $\varepsilon,\zeta>0$, and note that, for sufficiently large $n$,}
	\eqn{\label{eq:NFE:4}
		\begin{split}
			\prob\left( \frac{\hat{\chi}_{v_\omega}^2}{n}\geq \zeta \right)=\prob\left( \hat{\chi}_{v_\omega}^{p}\geq (n\zeta)^{p/2} \right)\leq {\E\left[ \hat{\chi}_{v_\omega}^{p} \right]}(n\zeta)^{-p/2}\leq \varepsilon.
	\end{split}}
	Therefore,  $\sum\limits_{u,j:u\geq v_\omega} p^{\sss(j)}(u,v_\omega)^2 = o_\prob(1)$ and, by \eqref{eq:NFE:2} and Lemma~\ref{lem:NFE:main},
	\eqn{\label{eq:NFE:5}
		\prod\limits_{\substack{u,j:\\u\geq v_\omega}}\left[ 1-p^{\sss(j)}(u,v_\omega) \right] = (1+o_\prob(1))\exp{\left( -\hat{\chi}_{v_\omega}\lambda(a_\omega) \right)}.}
	\paragraph{\bf Conclusion of the proof}
	{Substituting the} \textit{no-further edge probability} obtained in \eqref{eq:NFE:5} and the conditional probability estimates obtained in Lemma{s}~\ref{lem:bound:attachment_probability} and \ref{lem:NFE:main}, we obtain
	\eqan{\label{eq:conditional_law:CPU:2}
		&\prob_{m,\psi}\left( \overline{B}_r^{\sss(G_n)}\doteq\overline{\tree},~v_\omega=\lceil na_\omega \rceil,~\forall \omega\in V(\tree) \right)\nn\\
		=~& (1+o_\prob(1))\frac{1}{n}\prod\limits_{\omega\in V(\tree)}\prod\limits_{\substack{u\in V(\tree)\\u>\omega\\u\overset{j}{\sim}\omega}}\frac{\hat{\chi}_{v_\omega}}{(2\E[M]+\delta)v_\omega}\left( \frac{v_\omega}{v_u} \right)^{\chi}\prod\limits_{\omega\in V^\circ(\tree)}\exp{\left( -\hat{\chi}_{v_\omega}\lambda(a_\omega) \right)}\nn\\
		=~& (1+o_\prob(1))\frac{1}{n}\prod\limits_{\omega\in V(\tree)}\left(\frac{\hat{\chi}_{v_\omega}}{(2\E[M]+\delta)}\right)^{d_{v_\omega}^{\sss\rm{(in)}}(G_n)}\prod\limits_{\omega\in V^\circ(\tree)}\exp{\left( -\hat{\chi}_{v_\omega}\lambda(a_\omega) \right)}\nn\\
		&\hspace{3.5cm}\times\prod\limits_{(\omega,\omega l)\in E(\tree)} \left( v_\omega\wedge v_{\omega l} \right)^{-(1-\chi)}\left( v_\omega\vee v_{\omega l} \right)^{-\chi},}
	where $d_{v_{\omega}}^{\sss\rm{(in)}}(G_n) = d(v_\omega)-m_{v_\omega}$ is the number of vertices in $\CPU$ connected to $v_\omega$. Observe that $d_{v_\omega}^{\sss\rm{(in)}}(G_n)=d_{\omega}^{\sss\rm{(in)}}(\tree)$ when $\omega$ has label $\Young$, and $d_{v_\omega}^{\sss\rm{(in)}}(G_n)=d_{\omega}^{\sss\rm{(in)}}(\tree)+1$ when $\omega$ has label $\Old$. Further, $d_{v_{\omega}}^{\sss\rm{(in)}}(G_n) = 1$ for $\omega\in \partial V(\tree)$ \FC{with label $\Old$.} Therefore, \eqref{eq:conditional_law:CPU:2} can be re-written as
	\eqan{\label{eq:conditional_law:CPU:3}
		&\prob_{m,\psi}\left( \overline{B}_r^{\sss(G_n)}\doteq\overline{\tree},~v_\omega=\lceil na_\omega \rceil,~\forall \omega\in V(\tree) \right)\\
		=~& (1+o_\prob(1))n^{-(1+|E(\tree)|)}\prod\limits_{\omega\in V^\circ(\tree)}\left(\frac{\hat{\chi}_{v_\omega}}{(2\E[M]+\delta)}\right)^{d_{\omega}^{\sss\rm{(in)}}(\tree)}\prod\limits_{\omega\in V^\circ(\tree)}\exp{\left( -\hat{\chi}_{v_\omega}\lambda(a_\omega) \right)}\nn\\
		&\hspace{2.2cm}\times\prod\limits_{\substack{\omega\in V(\tree)\\\omega\text{ is $\Old$ }}}\frac{\hat{\chi}_{v_\omega}}{(2\E[M]+\delta)}\prod\limits_{(\omega,\omega l)\in E(\tree)} \left( a_\omega\wedge a_{\omega l} \right)^{-(1-\chi)}\left( a_\omega\vee a_{\omega l} \right)^{-\chi}.\nn}
	Since $\tree$ is a tree, $|V(\tree)|=1+|E(\tree)|$ and hence \eqref{eq:conditional_law:CPU:3} leads to \eqref{eq:prop:CPU:density:1}.
\end{proof}

\begin{remark}[Density of vertex-marked $\CPU$]\label{rem:density:CPU}
	\rm{
		 Every vertex $\omega\in V^\circ(\tree)$ has $m_{v_\omega}$ many out-edges that can be marked in $m_{v_\omega}!$ different ways.} If we {sum \eqref{eq:prop:CPU:density:1} out over} the edge-marks, {then} all such {edge-marked} graphs produce the same vertex-marked graph. 
	Observe that we have not considered any of the out-edges from the $\Old$ labelled vertices in $\partial V(\tree)$. Exactly one out-edge from every $\Young$ labelled {vertex} in $\partial V(\tree)$ is considered and its edge-mark can be labelled in $m_{v_\omega}$ many possible ways. Therefore by summing out the edge-marks,
	\eqan{\label{eq:prop:CPU:density:2}
		&\prob_{m,\psi}\left( {B}_r^{\sss(G_n)}(o)\simeq{\tree},~v_\omega=\lceil na_\omega \rceil,~\forall \omega\in V(\tree) \right)\\
		=&~(1+o_\prob(1))n^{-|V(\tree)|}\prod\limits_{\omega\in V^\circ(\tree)}\left[m_{v_\omega}! \left( \frac{\hat{\chi}_{v_\omega}}{2\E[M]+\delta} \right)^{d_\omega^{\sss\rm{(in)}}(\tree)} \exp{\left( -\hat{\chi}_{v_\omega} \lambda(a_\omega) \right)}\right]\nn\\
		&\hspace{0.5cm}\times\prod\limits_{(\omega,\omega l)\in E(\tree)}\left( a_{\omega}\vee a_{\omega l} \right)^{-\chi}\left( a_{\omega}\wedge a_{\omega l} \right)^{-(1-\chi)}\prod\limits_{\substack{\omega\in V(\tree)\\\omega \text{ is $\Old$ }}}\frac{\hat{\chi}_{v_\omega}}{2\E[M]+\delta}\prod\limits_{\substack{\omega\in\partial V(\tree)\\\omega \text{ is $\Old$ }}}m_{v_{\omega}}~.\nn}
	Conditionally on $( \boldsymbol{m},\boldsymbol{\psi} ),$ we have now obtained the density of the $r$-neighbourhood of a randomly chosen vertex of $\PArs_n(\boldsymbol{m},\delta)$. This will be useful in the proof of Theorem~\ref{thm:prelimit:density} below.
	\hfill$\blacksquare$
\end{remark}

From Remark~\ref{rem:density:CPU}, we have explicitly obtained the age density of the $\CPU$ graphs. Now we aim to show that this age density of $\CPU$ converges to that of the $\RPPT$.
\begin{proof}[\bf Proof of Theorem~\ref{thm:prelimit:density}]
	We have obtained the explicit form of $f_{r,\tree}\left( (a_\omega)_{\omega\in V(\tree)}; (\hat{\chi}_{v_\omega})_{\omega\in V(\tree)} \right)$ in Proposition~\ref{prop:RPPT:density}. Remark~\ref{rem:density:CPU} gives us an expression for the LHS of \eqref{eq:prelimit:density:1}. {Thus,} we are left to show that \eqan{\label{for:thm:prelim:density:1}
		&\E\Bigg[ \prod\limits_{\omega\in V^\circ(\tree)}\Bigg[m_{v_\omega}!\left( \frac{\hat{\chi}_{v_\omega}}{2\E[M]+\delta} \right)^{d_\omega^{\sss\rm{(in)}}(\tree)} \exp{\left( -\hat{\chi}_{v_\omega} \lambda(a_\omega) \right)}\Bigg]\nn\\
		&\hspace{5cm}\times\prod\limits_{\substack{\omega\in V(\tree)\\\omega \text{ is $\Old$ }}}\frac{\hat{\chi}_{v_\omega}}{2\E[M]+\delta} \prod\limits_{\substack{\omega\in\partial V(\tree)\\\omega \text{ is $\Young$}}}m_{v_{\omega}}\Bigg]\\
		=& \chi^{n_{\old}}(1-\chi)^{n_{\young}}\E\Bigg[\prod\limits_{\omega\in V^\circ(t)} m_{-}(\omega)^\prime ! ~\left(\hat{\chi}_\omega^\prime\right)^{d_{\omega}^{({\rm{in}})}(t)}\exp{\left(-\hat{\chi}_\omega^\prime\lambda(a_{\omega})\right)}\Bigg],\nn
	}
	where $\hat{\chi}_\omega^\prime\sim{\rm{Gamma}}\left(m_{-}(\omega)^\prime+\delta+\one_{\{\omega\text{ is $\Old$}\}},1\right)$ and $\hat{\chi}_{v_{\omega}}\sim {\rm{Gamma}}\left( m_{v_\omega}+\delta,1 \right)$. The distribution of $m_{-}(\omega)^\prime$ is the same as $M^{(\delta)}$ when $\omega$ has label $\Young$, $M^{(0)}$ when $\omega$ has label $\Old$ and, $M$ when $\omega$ is the root. $m_{v_\omega}$ are i.i.d.\ {copies of $M$.} To prove \eqref{for:thm:prelim:density:1}, we start by simplifying the LHS.
	
	First we rewrite
	\eqn{\label{for:thm:prelim:density:2}
		\prod\limits_{\substack{\omega\in V(\tree)\\\omega\text{ is $\Old$ }}}\frac{\hat{\chi}_{v_\omega}}{2\E[M]+\delta}=\prod\limits_{\substack{\omega\in V(\tree)\\\omega\text{ is $\Old$ }}}\left( \frac{\hat{\chi}_{v_\omega}}{m_{v_{\omega}}+\delta} \right)\left( \frac{m_{v_\omega}+\delta}{\E[M]+\delta} \right)\left( \frac{\E[M]+\delta}{2\E[M]+\delta} \right).}
	Remember that $\chi={(\E[M]+\delta)}/{(2\E[M]+\delta)}$. Therefore from \eqref{for:thm:prelim:density:2}, we see that we obtain a factor $\chi$ for each of the $\Old$ labelled vertices in $\tree$. The first two terms in the RHS of \eqref{for:thm:prelim:density:2} give rise to the size-biasing of the Gamma random variables and the out-edge distribution of the $\Old$ labelled vertices, as we observed in the definition of the $\RPPT$. Now, for $\omega\in\partial V(\tree)$, there is no other term in the LHS of \eqref{for:thm:prelim:density:1} containing $\hat{\chi}_{v_\omega}$ and $m_{v_{\omega}}$ and these are independent of the rest. Therefore taking expectation with respect to these $\hat{\chi}_{v_\omega}$ and $m_{v_{\omega}}$, the first two terms \FC{in \eqref{for:thm:prelim:density:2}} turn out to be $1$. 
	Next, we rewrite
	\eqan{\label{for:thm:prelim:density:3}
		&\prod\limits_{\substack{\omega\in V^\circ(\tree)\\\omega\text{ is $\Young$ }}}m_{v_\omega}!\prod\limits_{\substack{\omega\in\partial V(\tree)\\\omega \text{ is $\Young$ }}} m_{v_{\omega}}\nn\\
		=& \prod\limits_{\substack{\omega\in V^\circ(\tree)\\\omega\text{ is $\Young$ }}}\left(m_{v_\omega}-1\right)!\prod\limits_{\omega \text{ is $\Young$ }}m_{v_{\omega}}\nn\\
		=& \prod\limits_{\substack{\omega\in V^\circ(\tree)\\\omega\text{ is $\Young$ }}}\left(m_{v_\omega}-1\right)!\prod\limits_{\omega \text{ is $\Young$ }}\left( \frac{m_{v_{\omega}}}{\E[M]} \right)\left( \frac{\E[M]}{2\E[M]+\delta} \right)(2\E[M]+\delta)^{n_{\young}}.}
	Again {by} definition, $1-\chi=\frac{\E[M]}{2\E[M]+\delta}$. Therefore, similarly as in \eqref{for:thm:prelim:density:2}, we see that the size-biased out-degree distributions of the $\Young$ labelled vertices arise from the second term in \eqref{for:thm:prelim:density:3} and for every $\Young$ labelled vertex we obtain a factor $1-\chi$ as we observe in the $\RPPT$. There is no size-biasing in $\chi_{v_\emp}$ and $m_{v_\emp}$. The vertices in $V^\circ(\tree)$ can be partitioned in $3$ sets: the root, $\Old$ labelled vertices and $\Young$ labelled vertices. Therefore, using the simplification of the expressions in \eqref{for:thm:prelim:density:2} and \eqref{for:thm:prelim:density:3},
	\eqan{\label{for:thm:prelim:density:4}
		&\mbox{LHS of \eqref{for:thm:prelim:density:1}}\nn\\
		=& ~\chi^{n_{\old}}(1-\chi)^{n_{\young}}(2\E[M]+\delta)^{n_{\young}-\sum\limits_{\omega\in V^\circ(\tree)} d_\omega^{\sss\rm{(in)}}(\tree)}\nn\\
		&\hspace{1cm}\times\E\left[ \prod\limits_{\omega\in V^{\circ}(\tree)}m_{-}(\omega)^\prime!\left( {\hat{\chi}_{\omega}^\prime} \right)^{d_\omega^{\sss\rm{(in)}}(\tree)} \exp{\left( -\hat{\chi}_{\omega}^\prime \lambda(a_\omega) \right)} \right],}
	where $\hat{\chi}_{\omega}^\prime$ is defined as before.
	Note that $d_\omega^{\sss\rm{(in)}}(\tree)$ denotes the number of $\Young$ labelled children of $\omega$ and therefore summing {$d_\omega^{\sss\rm{(in)}}(\tree)$ over} all $\omega\in V^\circ (\tree)$ gives the total number of $\Young$ labelled vertices in $\tree$ and hence the terms in the exponent of $2\E[M]+\delta$ cancel. Therefore, \eqref{for:thm:prelim:density:1} follows immediately from \eqref{for:thm:prelim:density:4}, proving the required result.
\end{proof}
For obtaining the first moment convergence {in} \eqref{firstmoment:aim} using Theorem~\ref{thm:prelimit:density}, we have to sum over the range $\left[ \lceil n(a_\omega-1/r)\rceil,\lceil n(a_\omega+1/r)\rceil \right]^{|V(\tree)|}$. Summing the equality over this range, we obtain
\[
{\frac{1}{n}}\E\left[ N_{r,n}\left(\tree,(a_\omega)_{\omega\in V(\tree)}\right)\right] \to \mu\left( B_r^{\sss(G)}(\emp)\simeq \tree,~\left| A_\omega-a_\omega \right|\leq\frac{1}{r},\forall \omega\in V(\tree) \right).
\]
\subsection{Second moment convergence}\label{subsec:second-moment}
Here we show that
\[
{\frac{1}{n^2}}\E\Big[N_{r,n}\left(\tree,(a_\omega)_{\omega\in V(\tree)}\right)^2\Big]\to \mu\Big( B_r^{\sss(G)}(\emp)\simeq \tree,~\left| A_\omega-a_\omega \right|\leq\frac{1}{r},\forall \omega\in V(\tree) \Big)^2,
\]
or, alternatively, the variance {of $N_{r,n}/n$ vanishes.} {Expanding the double sum yields}
\eqan{\label{eq:second-moment:expansion}
	&N_{r,n}\left(\tree,(a_\omega)_{\omega\in V(\tree)}\right)^2 \nn\\
	=& \sum\limits_{v\in [n]} \one_{\left\{ B_r^{\sss(G_n)}(v)~\simeq~ \tree, ~|v_\omega/n-a_\omega|\leq 1/r~\forall\omega~\in~V(t) \right\}}\nn\\
	&+\sum\limits_{\substack{u\neq v\\ u,v\in[n]}} \one_{\left\{ B_r^{\sss(G_n)}(u)~\simeq~ \tree,~B_r^{\sss(G_n)}(v)\simeq~ \tree, ~\left|\frac{u_\omega}{n}-a_\omega\right|\leq \frac{1}{r},~\left|\frac{v_\omega}{n}-a_\omega\right|\leq \frac{1}{r}~\forall\omega~\in~V(t) \right\}},}
where $v_\omega$ and $u_\omega$ are the vertices in $B_r^{\sss(G_n)}(v)$ and $B_r^{\sss(G_n)}(u)$ corresponding to the vertex $\omega\in V(\tree)$. Note that the first term in the RHS of \eqref{eq:second-moment:expansion} equals $N_{r,n}\left(\tree,(a_\omega)_{\omega\in V(\tree)}\right)$. {By} the results proved earlier in this section, it follows that this term, upon dividing by $n^2$, vanishes.
Therefore, we are {left} to analyse the second term {on} the RHS of \eqref{eq:second-moment:expansion}.
First, we show that $r$-neighbourhoods of two uniformly chosen vertices are disjoint whp. Note that
\eqan{\label{eq:disjoint-neighbourhood}
	&\prob\left( B_r^{\sss(G_n)}\left( o_n^{\sss(1)} \right)\cap B_r^{\sss(G_n)}\left( o_n^{\sss(2)} \right) = \emp \right)\nn\\
	&\hspace{1cm}=~\prob\left( o_n^{\sss(2)}\notin B_{2r}^{\sss(G_n)}\left( o_n^{\sss(1)} \right) \right)
	= 1-\frac{1}{n}\E\left[ \left| B_{2r}^{\sss(G_n)}\left( o_n^{\sss(1)} \right) \right| \right].}
Previously we have proved that $B_{2r}^{\sss(G_n)}\left( o_n^{\sss(1)} \right)$ converges in distribution to $B_r^{\sss(G)}\left( \emp \right)$, where $(G,\varnothing)$ is the random P\'olya point tree with parameters $M$ and $\delta$, so that $\left\{\left| B_{2r}^{\sss(G_n)}\left( o_n^{\sss(1)} \right) \right|\right\}_{n\in\N}$ is a tight sequence of random variables. Therefore the $r$-neighbourhood of two uniformly chosen vertices are whp disjoint. 
Now as we did {for} the first moment, we use a density argument {for the second term:} 
\begin{Theorem}{\label{thm:second-moment:prelim}}
	Fix $\delta>-\infsupp(M)$, and let $G_n=\PArs_n(\boldsymbol{m},\delta)$. Let $\left(\tree_1,(a_\omega)_{\omega\in V(\tree_1)}\right)$ and $\left(\tree_2,(a_\omega)_{\omega\in V(\tree_2)}\right)$ be two vertex-marked trees of depth $r$ with disjoint vertex marks. If $v_\omega$ denotes the vertex in $G_n$ corresponding to $\omega\in V(\tree_1)\cup V(\tree_2)$, then, uniformly for all distinct $v_\omega\geq \eta n$ and $\hat{\chi}_{v_\omega}\leq \left( v_\omega \right)^{1-{\varrho/2}}$,
	\eqan{\label{eq:thm:second-moment:prelim}
		&\prob_{m,\psi}\left( B_r^{\sss(G_n)}\left(o_n^{\sss(1)}\right)\simeq \tree_1,~B_r^{\sss(G_n)}\left(o_n^{\sss(2)}\right)\simeq \tree_2,~v_{\omega}=\lceil na_\omega \rceil,\forall \omega\in V(\tree_1)\cup V(\tree_2) \right)\nn\\
		&\hspace{0.25cm}=(1+o_\prob(1))\frac{1}{n^{|V(\tree_1)\cup V(\tree_2)|}}~g_{r,\tree_1}\left( (a_\omega)_{\omega\in V(\tree_1)}; \left( m_{v_\omega},\hat{\chi}_{v_\omega} \right)_{\omega\in V(\tree_1)} \right)\nn\\
		&\hspace{4.5cm} \times g_{r,\tree_2}\left( (a_\omega)_{\omega\in V(\tree_2)}; \left( m_{v_\omega},\hat{\chi}_{v_\omega} \right)_{\omega\in V(\tree_2)} \right),
	}
	where $g_{r,\tree}(\cdot)$ is as in Theorem~\ref{thm:prelimit:density} and {$o_n^{\sss(1)},o_n^{\sss(2)}\in[n]$ are chosen independently and uniformly at random.}
\end{Theorem}
\begin{proof}
	Since most of the steps in this proof are similar to the proof of Theorem~\ref{thm:prelimit:density}, we will be more concise.
	We restrict our analysis to the case where $B_r^{\sss(G_n)}\left( o_n^{\sss(1)} \right)$ and $B_r^{\sss(G_n)}\left( o_n^{\sss(2)} \right)$ are disjoint graphs. Using the conditional independence of the edge-connection events for $\CPU$, we can {write} the LHS of \eqref{eq:thm:second-moment:prelim} {as}
	\eqan{\label{for:thm:SMP:1}
			&\prob_{m,\psi}\left( B_r^{\sss(G_n)}\left(o_n^{\sss(1)}\right)\simeq \tree_1,~B_r^{\sss(G_n)}\left(o_n^{\sss(2)}\right)\simeq \tree_2,~v_{\omega}=\lceil na_\omega \rceil,\forall \omega\in V(\tree_1)\cup V(\tree_2) \right)\nn\\
			=&~ \prob_{m,\psi}\left( B_r^{\sss(G_n)}\left(o_n^{\sss(1)}\right)\simeq \tree_1,~v_{\omega}=\lceil na_\omega \rceil,\forall \omega\in V(\tree_1) \right)\nn\\
			&\hspace{1.5cm}\times\prob_{m,\psi}\left( B_r^{\sss(G_n)}\left(o_n^{\sss(2)}\right)\simeq \tree_2,~v_{\omega}=\lceil na_\omega \rceil,\forall \omega\in V(\tree_2) \right).
	}
	Theorem~\ref{thm:prelimit:density} {then immediately implies} Theorem~\ref{thm:second-moment:prelim}.
\end{proof}

Considering $\tree_1\simeq\tree_2$ but with different marks, {by} Theorem~\ref{thm:second-moment:prelim}, it {follows} that 
\eqan{\label{eq:second-moment:expectation}
		&{\frac{1}{n^2}}\E\left[N_{r,n}\left( \tree,(a_\omega)_{\omega\in V(\tree)} \right)^2\right]\nn\\
		&\hspace{1cm}=~ (1+o(1))
		\mu\left( B_r^{\sss(G)}(\emp)\simeq \tree,~\left| A_\omega-a_\omega \right|\leq\frac{1}{r},\forall \omega\in V(\tree) \right)^2.
}
{Therefore the {variance of $N_{r,n}\left( \tree,(a_\omega)_{\omega\in V(\tree)} \right)/n$ vanishes} and the local convergence of model (A) to the random P\'olya point tree follows immediately.
	
\begin{remark}[Convergence of models (B) and (D)]\label{remark:conv:models:B,D}
\rm	We have now proved the local convergence of model (A), which is equal in distribution to $\CPU^{\sss\rm{(SL)}}$. The proof of the local convergence of models (B) and (D) follows along the same lines. Theorems~\ref{thm:equiv:CPU:PArt} and \ref{thm:equiv:PU:PAri} show that models (B) and (D) have the same law as $\CPU^{\sss\rm{(NSL)}}$ and $\PU^{\sss\rm{(NSL)}}$, respectively. Recall that, by Remark~\ref{remark:equal-edge-probability}, the edge-connection probabilities for all the models $\CPU^{\sss\rm{(SL)}},\CPU^{\sss\rm{(NSL)}}$, and $\PU^{\sss\rm{(NSL)}}$ behave similarly. Therefore, upon substitution of these in \eqref{eq:conditional_law:CPU:1}, the proofs for Theorems~\ref{thm:prelimit:density} and \ref{thm:second-moment:prelim} for models (B) and (D) follow from the same calculations. 

\hfill$\blacksquare$
\end{remark}

\subsection{Concluding the proof}\label{subsec:proof-conclusion}
In this section, we assemble the results obtained in Sections~\ref{subsec:first-moment} and \ref{subsec:second-moment} to prove Theorem~\ref{theorem:PArs} for models (A), (B), and (D).

\begin{proof}[Proof of Theorem~\ref{theorem:PArs} for models (A), (B), and (D)]
	From Theorems~\ref{thm:prelimit:density} and \ref{thm:second-moment:prelim}, we obtain that the first and second moments of age densities of the vertices in the $r$-neighbourhood of a uniformly chosen vertex in $\PArs_n(\bfm,\delta)$ converge to those of $\RPPT(M,\delta)$. Remark~\ref{remark:conv:models:B,D} shows that the same holds for models (B) and (D) as well. Therefore, by the second moment method, for models (A), (B), and (D), the age densities of the vertices in the $r$-neighbourhood of a uniformly chosen vertex converge in probability to those of $\RPPT(M,\delta)$. This, in turn, proves the local convergence of models (A), (B), and (D) to the random P\'{o}lya point tree.
\end{proof}


\section{Asymptotic degree distribution}\label{chap:convergence:sec:degree-distribution}
We next discuss some consequences of our main result in Theorem \ref{theorem:PArs}, focusing on degree distributions. As an immediate consequence of local convergence theorem, if $G_n$ converges locally in probability to $(G,o)$ having law $\mu$, then for any bounded continuous function $h:\mathcal{G}_\star\mapsto\R,$
\eqn{\label{for:cor:neighbour:1}
	\frac{1}{n}\sum\limits_{u\in[n]}h(G_n,u)\overset{\prob}{\to}\E_{\mu}[h(G,o)]~.}
Let the out-degree distribution $M$ have a power-law exponent $\tau_{\sss M} \geq p$, i.e., for any $k \in \mathbb{N}$,
\begin{equation}\label{eq:def:power-law:M}
	\mathbb{P}(M=k) = L^{\sss(\emp)}(k) k^{-\tau_{\sss M}},
\end{equation}
for some slowly varying function $L^{\sss (\emp)}$. Further, define the preferential attachment power-law exponent $\tau_e$ as
\begin{equation}\label{eq:def:tau_e}
	\tau_e = 3 + \frac{\delta}{\mathbb{E}[M]}.
\end{equation}
In this section, we discuss the power-law exponents of the average degree distribution and the degrees of neighbours of a uniformly chosen vertex in the graph. The next lemma identifies the power-law exponent for the asymptotic degree distribution.

\begin{Lemma}[Power-law exponent for average degree distribution]\label{lem:power-law-root}
	Let $\{G_n\}_{n \geq 1}$ be a sequence of generalised preferential attachment models with out-degree distribution $M$ and $\delta > -\inf\mathrm{supp}(M)$, and let $d_u(n)$ denote the degree of the vertex $u$ in $G_n$. Then, for all $k \geq 1$,
	\begin{equation}\label{eq:lem:power-law:root}
		\frac{1}{n} \sum\limits_{v \in [n]} \mathbb{I}\{d_u(n) = k\} \overset{\mathbb{P}}{\to} p_k = L^{\sss (\emp)}(k) k^{-\tau},
	\end{equation}
	where $\tau = \min\{\tau_{\sss M}, \tau_e\}$.
\end{Lemma}
\begin{proof}
	For proving this lemma, for any $k \in \N$, we consider
	\begin{equation}\label{for:lem:power-law-root:1}
		h(G_n,u) = \one_{\{d_u(n)=k\}}~.
	\end{equation}
	Clearly, $h$ is a bounded continuous function in $\Gcal_\star$, so that
	\begin{equation}\label{for:lem:power-law-root:2}
		\frac{1}{n}\sum\limits_{u}{\one_{\{d_u(n)=k\}}} \overset{\prob}{\to} \prob(d_{\emp}=k),
	\end{equation}
	where $d_{\emp}$ is the degree of the root in the random P\'{o}lya point tree. By definition of $\RPPT$, the degree of the root is distributed as $M+Y\big(M,U_\emp\big)$, where $Y(m,a)$ is a mixed Poisson distribution with parameter $\Gamma_{in}(m)\lambda(a)$, and $\lambda$ is as defined in \eqref{eq:def:lambda-func}. Since $U_\emp \sim \Unif[0,1]$,
	\begin{align}\label{for:lem:power-law-root:3}
		&\prob\Big( M+Y\big(M,U_\emp\big) = k \Big)\nn\\
		=& \sum\limits_{m=1}^k\prob(M=m)\int\limits_0^1\prob\big(Y( m,a )=k-m\big)\,da~.
	\end{align}
	Therefore, explicitly computing the integrand in \eqref{for:lem:power-law-root:3},
	\begin{align}\label{for:lem:power-law-root:4}
		&\prob(Y(m,a)=k-m) \nn\\
		=& \int\limits_0^{\infty} \frac{(\gamma\lambda(a))^{k-m}}{(k-m)!}\exp{(-\gamma\lambda(a))}\frac{1}{\Gamma(m+\delta)}\gamma^{(m+\delta)-1} \exp{(-\gamma)}\,d\gamma \nn\\
		=& \frac{\lambda(a)^{k-m}}{(k-m)!\Gamma(m+\delta)}\int\limits_0^{\infty} \gamma^{(k+\delta)-1}\exp{(-\gamma(\lambda(a)+1))}\,d\gamma \nn\\
		=& \frac{\Gamma(k+\delta)}{(k-m)!\Gamma(m+\delta)} \lambda(a)^{k-m}(1+\lambda(a))^{-(k+\delta)}~.
	\end{align}
	Substituting the value of $\lambda(a) = a^{-(1-\chi)}(1-a^{1-\chi})$, we obtain a compact form of the probability on the LHS of \eqref{for:lem:power-law-root:4} as
	\begin{equation}\label{for:lem:power-law-root:5}
		\prob(Y(m,a)=k-m) = \frac{\Gamma(k+\delta)}{\Gamma(m+\delta)(k-m)!}(1-a^{1-\chi})^{k-m}(a^{1-\chi})^{m+\delta}~.
	\end{equation}
	This result is similar to \cite[Lemma~5.2]{BergerBorgs}. As we can see from \eqref{for:lem:power-law-root:3}, we need to integrate the RHS of \eqref{for:lem:power-law-root:5}.
	\begin{align}\label{for:lem:power-law-root:6}
		&\int_0^1\prob(Y(m,a)=k-m)\,da \nn\\
		=& \frac{\Gamma(k+\delta)}{\Gamma(m+\delta)(k-m)!}\int_0^1 (1-a^{1-\chi})^{k-m}(a^{1-\chi})^{m+\delta}\,da~.
	\end{align}
	Again, performing a change of variable, we obtain
	\begin{align}\label{for:thm:root:2}
		&\int_0^1\prob(Y(m,a)=k-m)\,da \nn\\
		=& \frac{\Gamma(k+\delta)}{(1-\chi)\Gamma(m+\delta)(k-m)!}\int_0^1 u^{(m+\delta+\tau_e-1)-1}(1-u)^{k-m} \,du\nn\\
		=& \frac{(\tau_e-1)\Gamma(k+\delta)\Gamma(k-m+1)\Gamma(m+\delta+\tau_e-1)}{\Gamma(m+\delta)(k-m)!\Gamma(k+\delta+\tau_e)}\nn\\
		=& \frac{(\tau_e-1)\Gamma(k+\delta)\Gamma(m+\delta+\tau_e-1)}{\Gamma(m+\delta)\Gamma(k+\delta+\tau_e)}~.
	\end{align}
	This matches with the expression presented in \cite{DEH09}. Now, plugging this integral value into \eqref{for:lem:power-law-root:3} and using Stirling's approximation, we obtain
	\begin{align}\label{taiL:root:1}
		&\prob(M+Y(M,U_\emp)=k) \nn\\
		= &\frac{(\tau_e-1)\Gamma(k+\delta)}{\Gamma(k+\delta+\tau_e)} \sum\limits_{m=1}^{k}\prob(M=m)\frac{\Gamma(m+\delta+\tau_e-1)}{\Gamma(m+\delta)}\nn\\
		= &(\tau_e-1)k^{-\tau_e}(1+\mathcal{O}(1/k))\sum\limits_{m=1}^{k} L(m)m^{-(\tau_{\sss M}-\tau_e+1)}(1+\mathcal{O}(1/m))~.
	\end{align}
	For $\tau_{\sss M}>\tau_e$, the sum on the RHS of \eqref{taiL:root:1} is finite. On the other hand, for $\tau_{\sss M}\leq\tau_e$, the sum varies regularly as $k^{-(\tau_{\sss M}-\tau_e)}$ and hence
	\begin{equation}\label{tail:root:2}
		\prob(M+Y(M,U_\emp)=k) = L^{(\emp)}(k)k^{-\tau}~,
	\end{equation}
	where ${\tau=\min\{ \tau_{\sss M},\tau_e \}}$ and $L^{(\emp)}(\cdot)$ is a slowly varying function.
\end{proof}
Lemma~\ref{lem:power-law-root} re-establishes the power-law exponent obtained in \cite{DEH09}, using the local limit result of generalised preferential attachment models. We now extend this result to the convergence of degree distribution of older and younger neighbours of uniform vertex as:
\begin{Theorem}[Asymptotic degree distribution for older and younger neighbours]
	\noindent
	\begin{itemize}
		\item[(a)]\label{cor-older-younger-neighbours}
		As $n\rightarrow \infty$, for all $k\geq 1$, the degree of a uniform older neighbour of a uniform vertex satisfies
		\eqn{\label{eq:deg:old}
			\frac{1}{n}\sum\limits_{\substack{u,v,j:\\u\overset{j}{\rightsquigarrow} v}}\frac{\one_{\{d_v(n)=k\}}}{m_u}\overset{\prob}{\to} \Tilde{p}_k^{\sss({\old})} = \prob\Big( 1+ M^{(\delta)}+Y\big( M^{(\delta)}+1,A_{\old} \big) = k\Big),}
		where $Y\big( M^{(\delta)}+1,A_{\old} \big)$ is a mixed Poisson random variable with mixing distribution $\Gamma_{in}\big( M^{(\delta)}+1 \big)\lambda\big( A_{\old} \big)$ and $A_{\old}$ is distributed as $U_\emp U_1^{1/\chi}$, where $U_\emp, U_1$ are independent $Unif(0,1)$ random variables.\\
		Similarly, as $n\rightarrow \infty$, for all $k\geq 1$, the degree of a uniform younger neighbour of a uniform vertex satisfies
		\eqan{\label{eq:deg:young}
				\frac{1}{\sum\limits_v \one_{\{d_v(n)>m_v\}}}\sum\limits_{\substack{u,v,j:\\u\overset{j}{\rightsquigarrow} v}}&\frac{\one_{\{d_v(n)>m_v\}}\one_{\{d_u(n)=k\}}}{d_v(n)-m_v}\nn\\
				&\overset{\prob}{\to} \Tilde{p}_k^{\sss({\young})} =\prob\Big( M^{(0)}+Y\big( M^{(0)},A_{\young} \big) =k \Big),
		}
		where $Y\big( M^{(0)},A_{\young} \big)$ is a {mixed} Poisson random variable with mixing distribution $\Gamma_{in}\big( M^{(0)} \big)\lambda\big( A_{\young} \big)$ and $A_{\young}$ has conditional density
		\eqn{\label{eq:density:A}f_{U_\emp}(x)=\frac{\rho_\emp(x)}{\int_{U_\emp}^1 \rho_\emp(s)\,ds}=\frac{x^{-\chi}}{\int_{U_\emp}^1 s^{-\chi}\,ds}~,}
		where $U_\emp\sim Unif(0,1)$ and $\rho_\emp(s)$ is as defined in \eqref{def:rho_emp}.
		\smallskip\noindent
		\item[(b)]\label{cor-older-younger-neighbours-tail} There exist slowly varying functions $L^{\sss({\old})}(k)$ and $L^{\sss({\young})}(k)$, such that
		\eqn{\label{tail:exponent}
			\Tilde{p}_k^{\sss({\old})}\sim L^{\sss({\old})}(k)k^{-\tau_{\sss({\old})}},
			\qquad
			\Tilde{p}_k^{\sss({\young})}\sim \Theta(L^{\sss({\young})}(k)) k^{-\tau_{\sss({\young})}} \text{ as $k\to\infty$},
		}
		where $\tau_{\sss({\old})}=\min\{\tau_e-1, \tau_{\sss M}-1\}$ and $\tau_{\sss({\young})}=\min\{ \tau_e+1, \tau_{\sss M}-1 \}$.
	\end{itemize}
\end{Theorem}
The convergence is a direct consequence of Theorem \ref{theorem:PArs}. Thus, the degree distribution of random neighbours is asymptotically {\em size-biased} compared to the original degree distribution in the graph, as in Lemma~\ref{lem:power-law-root}. This result is somewhat surprising in the sense that in \eqref{tail:exponent}, both the tail exponents $\tau_{({\old})}$ and $\tau_{({\young})}$ contain the $\tau_{\sss M}-1$ but the dependence on the $\PAM$ power-law exponent is either one larger or one smaller than for the degree distribution of the root.
\begin{proof}[Proof of Theorem~\ref{cor-older-younger-neighbours}]
	Similar to the proof of Lemma~\ref{lem:power-law-root}, we use some bounded continuous function to prove convergence in this corollary. Define for fixed $k\geq 1$, 
	\eqn{\label{for:cor:neighbour:2}
		h(H,u) = \sum\limits_{\substack{v,j:\\u\overset{j}{\rightsquigarrow} v}}\frac{\one_{\{ d_v(H)=k \}}}{m_u}~,} where $d_v(H)$ is the degree of the vertex $v$ in the graph $H$. 
	Clearly $h$ is a bounded continuous function on $\mathcal{G}_\star$. Therefore by \eqref{for:cor:neighbour:1},
	\eqn{\label{for:cor:neighbour:3}
		\frac{1}{n}\sum\limits_{\substack{u,v,j:\\u\overset{j}{\rightsquigarrow} v}}\frac{\one_{\{d_v(n)=k\}}}{m_u}\overset{\prob}{\to} \E_{\mu_\star} \big[ h(G,\emp) \big]~,}
	where $(G,\emp)$ is the rooted $\RPPT(M,\delta)$ with law $\mu_\star$. Now $h(H,o)$ is the fraction of older neighbours of $o$ in $H$ having degree $k$. Therefore, $\E_{\mu_\star} \big[ h(G,\emp) \big]$ is the probability that a random $\Old$ labelled neighbour of the root of $\RPPT(M,\delta)$ has degree $k.$ From the definition of $\RPPT(M,\delta)$, the degree distribution of an $\Old$ labelled node with age $a_\omega$ is $1+M^{(\delta)}+Y\big( M^{(\delta)}+1,a_\omega \big)$, where $Y\big( M^{(\delta)}+1,a_\omega \big)$ is a {mixed} Poisson random variable with mixing distribution $\Gamma_{in}\big( M^{(\delta)}+1 \big)\lambda\big( a_\omega \big)$ and $\Gamma_{in}\big( M^{(\delta)}+1 \big)$ is as defined in Section~\ref{chap:LWC:sec:RPPT}. On the other hand, a random $\Old$ labelled neighbour of the root $\emp$ has age distributed as $U_\emp U_1^{1/\chi}$. Therefore a random $\Old$ labelled neighbour of the root has the degree distribution $1+M^{(\delta)}+Y\big( M^{(\delta)}+1,A_{\old}\big),$ where $Y\big( M^{(\delta)}+1,A_{\old}\big)$ is as defined previously in Corollary~\ref{cor-older-younger-neighbours} and
	\eqn{\label{for:cor:neighbour:4}
		\E_{\mu_\star} \big[ h(G,o) \big] = \prob\big( 1+M^{(\delta)}+Y\big( M^{(\delta)}+1,A_{\old}\big)=k \big)= \Tilde{p}_k^{({\old})}~.}
	We prove \eqref{eq:deg:young} in a similar way with a few adaptations. Instead of considering all vertices, we consider the vertices that have at least one younger neighbour. Next we choose the bounded continuous functions $h_1$ and $h_2$ as
	\eqn{\label{for:cor:neighbour:5}
		\begin{split}
			h_1(H,v)&=\sum\limits_{\substack{u,j:\\u\overset{j}{\rightsquigarrow} v}}\frac{\one_{\{d_v(H)>m_v\}}\one_{\{d_u(H)=k\}}}{d_v(H)-m_v}~,\\
			\mbox{and}\qquad\qquad h_2(H,v)&=\one_{\{v~{\sss\mbox{\rm has at least one younger neighbour}}\}}~.
	\end{split}}
	Similarly by \eqref{for:cor:neighbour:1}, 
	\eqan{
			&\frac{1}{n}\sum\limits_{\substack{u,v,j:\\u\overset{j}{\rightsquigarrow} v}}\frac{\one_{\{d_v(n)>m_v\}}\one_{\{d_u(n)=k\}}}{d_v(n)-m_v} \overset{\prob}{\to} \E_{\mu_\star} \big[ h_1(G,\emp) \big]~,\label{for:cor:neighbour:6-1}\\
			\mbox{and}\hspace{0.2cm} &\frac{1}{n}\sum\limits_{v\in [n]} \one_{\{v~{\sss\mbox{\rm has at least one younger neighbour}}\}}\overset{\prob}{\to} \E_{\mu_\star} \big[ h_2(G,\emp) \big]~.\label{for:cor:neighbour:6-2}
	} 
	where $(G,\emp)$ is the rooted $\RPPT(M,\delta)$ with law $\mu_\star$. Now $h_1(H,o)$ is the fraction of younger neighbours of $o$ in $H$ having degree $k$. If $o$ has no younger neighbour then define $h_1(H,o)=0$. Therefore,
	\eqn{\label{for:cor:neighbour:7}
		\E_{\mu_\star} \big[ h_1(G,\emp) \big] = \E_{\mu_\star} \big[ h_1(G,\emp)\big|\din_\emp>0 \big]\prob(\din_\emp>0)~,}
	and $\E_{\mu_\star} \big[ h_1(G,\emp)\big| \din_\emp>0 \big]$ is the probability that a random $\Young$ labelled neighbour of the root of $\RPPT(M,\delta)$ has degree $k$, conditionally on the event that the root has at least one younger neighbour.
	From the definition of $\RPPT(M,\delta)$, the degree distribution of a $\Young$ labelled node of age $a_\omega$ is $M^{(0)}+Y\big( M^{(0)},a_\omega \big)$, where $Y\big( M^{(0)},a_\omega \big)$ is a {mixed} Poisson random variable with mixing distribution $\Gamma_{in}\big( M^{(0)} \big)\lambda\big( a_\omega \big)$.

	Using the fact that $\din_\emp$ is the total number of points in a Poisson point process, the ages of the $\Young$-labelled neighbours of $\emp$, conditionally on $\din_\emp=n$, are i.i.d.\ random variables with density \eqref{eq:density:A} \cite[Exercise 4.34]{resnick1992adventures}. Hence, conditionally on $\din_\emp=n$, a uniformly chosen younger neighbour of the root $\emp$ has an age distribution given by $A_1$ with density \eqref{eq:density:A}.
	
	Therefore, conditionally on the root having at least one younger neighbour, a random $\Young$-labeled neighbour of the root has a degree distribution given by $M^{(0)} + Y\big(M^{(0)}, A_{\young}\big)$, where $Y\big(M^{(0)}, A_{\young}\big)$ is as defined earlier in Corollary~\ref{cor-older-younger-neighbours}, and
	\begin{align}\label{for:cor:neighbour:8}
		\E_{\mu_\star} \big[ h_1(G,\emp) \big] =& \prob\big(\din_\emp > 0\big)\prob\Big( M^{(0)} + Y\big( M^{(0)}, A_{\young} \big) = k \Big)\nn\\
		=& \prob\big(\din_\emp > 0\big) \Tilde{p}_k^{({\young})}.
	\end{align}
	On the other hand, $\E_{\mu_\star} \big[ h_2(G,\emp) \big]$ is the probability of the root $\emp$ having at least one younger neighbour, which is given by $\prob\big(\din_\emp > 0\big)$.
	 Therefore,
	\eqn{\label{for:cor:neighbour:9}
		\E_{\mu_\star} \big[ h_2(G,\emp) \big] = \prob\big( \din_\emp>0 \big)~.}
	Since $\prob\big(\din_\emp>0\big)$ is non-zero, \eqref{eq:deg:young} follows immediately from \eqref{for:cor:neighbour:8} and \eqref{for:cor:neighbour:9}. This completes the proof of (a).
	
	Now, \emph{we prove part (b) of this theorem}. For all $k\in\N$,
	\eqan{
		&\prob(1+M^{(\delta)}+Y(M^{(\delta)}+1,A_{\old})=k)\nn\\
		&\hspace{1cm}=\sum\limits_{m=1}^{k-1}\prob(M^{(\delta)}=m)\int_0^1\prob(Y(m+1,a)=k-m-1)f_{\old}(a)\,da~,\label{for:appd:deg:001}\\
		&\prob(M^{(0)}+Y(M^{(0)},A_{\young})=k)\nn\\
		&\hspace{1cm}=\sum\limits_{m=1}^{k}\prob(M^{(0)}=m)\int_0^1\prob(Y(m,a)=k-m)f_{\young}(a)\,da~,\label{for:appd:deg:011}}
	where $f_{\old}$ and $f_{\young}$ are the age density functions of a random $\Old$ and $\Young$ child of the root. If we manage to explicitly calculate the densities $f_{\old}$ and $f_{\young}$, then we can use $\prob(Y(m,a)=k-m)$ from \eqref{for:lem:power-law-root:5} to obtain expressions for degree distributions for older and younger neighbours of uniformly chosen vertex.
	\paragraph{Age density computations}
	We proceed to explicitly calculate the densities $f_{\old}$ and $f_{\young}$. From the definition of the $\RPPT(M,\delta)$, we have that $A_{\old}$ is distributed as $U_1U_2^{1/\chi}$, where $U_1$ and $U_2$ are independent uniform random variables on $[0,1]$. Therefore,

	\eqn{\label{for:lem:age:rand:old:1}
	\prob(A_{\old} \leq a) = \prob (U_1U_2^{1/\chi}\leq a) = \int_0^1 \prob(U_2\leq a^{\chi} x^{-\chi})\,dx~.
	}

	Note that for $x\in[0,a]$, the probability in the RHS of \eqref{for:lem:age:rand:old:1} is always $1$. For $x\in(a,1]$, the probability is $a^{\chi} x^{-\chi}$. Therefore, the RHS of \eqref{for:lem:age:rand:old:1} simplifies as

	\eqn{\label{for:lem:age:rand:old:2}
	\prob(A_{\old}\leq a) = a + a^{\chi}\int_a^1 x^{-\chi}\,dx = a+\frac{a^{\chi}}{1-\chi}(1-a^{1-\chi})~.
	}
	Now, differentiating the RHS of \eqref{for:lem:age:rand:old:2} with respect to $a$, we obtain
	\eqan{\label{for:lem:age:rand:old:3}
	f_{\old}(a) &= 1 + \frac{\chi}{1-\chi}a^{-(1-\chi)}(1-a^{1-\chi})-\frac{1-\chi}{1-\chi}a^\chi a^{-\chi}\nn\\
	&=\frac{\chi}{1-\chi}a^{-(1-\chi)}(1-a^{1-\chi})~.
	}
	Similarly, we have seen that conditionally on the existence of at least one $\Young$ child, its age distribution is given in equation \eqref{eq:density:A}. If $U_1$ is a uniform random variable on $[0,1]$, then
	\eqn{\label{for:lem:age:rand:young:1}
	\begin{split}
		f_{\young}(a)= \E\left[(1-\chi)\frac{a^{-\chi}\one_{\{ a\geq U_1 \}}}{1-U_1^{1-\chi}}\right] = (1-\chi)a^{-\chi}\int_0^a (1-y^{1-\chi})^{-1}\,dy~.
	\end{split}}
	Next, we perform a change of variable by substituting $y^{1-\chi}=x$ in the RHS of \eqref{for:lem:age:rand:young:1}:
	\eqn{\label{for:lem:age:rand:young:2}
	f_{\young}(a) = a^{-\chi}\int_0^{a^{1-\chi}}x^{1/(1-\chi)-1}(1-x)^{-1}\,dx~.
	}
	Substituting $\tau_e=1+{1}/{(1-\chi)}$ in the RHS of \eqref{for:lem:age:rand:young:2}, we obtain the density of a randomly chosen $\Young$ neighbour of the root as
	\eqn{\label{eq:den:rand:young}
	f_{\young}(a) = a^{-\chi}\int_0^{a^{1-\chi}}x^{\tau_e-2}(1-x)^{-1}\,dx~.
	}
	We have all the ingredients to prove the tail distributions.
	\paragraph{Tail distribution computation}
	Note that, from the definition of $M^{(0)}$ and $M^{(\delta)}$, and \eqref{eq:def:power-law:M}, we have
	\eqan{
			\prob(M^{(0)}=m)=&~L_1(m) m^{-(\tau_{\sss M}-1)}~,\label{eq:tail:M:1}\\
			\mbox{and}\qquad\prob(M^{(\delta)}=m)=&~L_2(m) m^{-(\tau_{\sss M}-1)}~. \label{eq:tail:M:2}
	}
	From \eqref{for:lem:age:rand:old:3}, we have the expression for $f_{\old}(a)$, and therefore the integral in \eqref{for:appd:deg:001} can be simplified as
	\eqan{\label{for:thm:old:1}
			&\int_0^1\prob(Y(m,a)=k-m)f_{\old}(a)\,da\nn\\
			=& \frac{\Gamma(k+\delta)}{\Gamma(m+\delta)(k-m)!}\int_0^1 (1-a^{1-\chi})^{k-m}(a^{1-\chi})^{m+\delta} \frac{\chi}{1-\chi} a^{-(1-\chi)}(1-a^{1-\chi})\,da\nn\\
			=& (\tau_e-2)\frac{\Gamma(k+\delta)}{\Gamma(m+\delta)(k-m)!}\int_0^1 (1-a^{1-\chi})^{k-m+1}(a^{1-\chi})^{m+\delta-1}\,da~.
	}
	Again we do the same change of variable and simplify the RHS of \eqref{for:thm:old:1} as
	\eqan{\label{for:thm:old:3}
			&\int_0^1\prob(Y(m,a)=k-m)f_{\old}(a)\,da\nn\\
			=& (\tau_e-2)(\tau_e-1)\frac{\Gamma(k+\delta)}{\Gamma(m+\delta)(k-m)!}\int_0^1 (1-u)^{k-m+1}u^{m+\delta+\tau_e-3}\,da\nn\\
			=& (\tau_e-2)(\tau_e-1)\frac{\Gamma(k+\delta)}{\Gamma(m+\delta)(k-m)!} \frac{\Gamma(m+\delta+\tau_e-2)\Gamma(k-m+2)}{\Gamma(k+\delta+\tau_e)}\nn\\
			=& (\tau_e-2)(\tau_e-1){(k-m)}\frac{\Gamma(k+\delta){\Gamma(m+\delta+\tau_e-1)}}{{\Gamma(m+1+\delta)}\Gamma(k+\delta+\tau_e)}.
	}
	Now substituting this integral value in \eqref{for:appd:deg:001}, the sum in the RHS could be simplified as 
	\eqan{\label{tail:old:1}
		&\prob(1+M^{(\delta)}+Y(M^{(\delta)}+1,A_{\old})=k)\nn\\
		=&(\tau_e-2)(\tau_e-1)\frac{\Gamma(k+\delta)}{\Gamma(k+\delta+\tau_e)}\sum\limits_{m=1}^{k-1}\prob(M^{(\delta)}=m){(k-m)}\frac{{\Gamma(m+\delta+\tau_e-1)}}{{\Gamma(m+1+\delta)}}\nn\\
		=& (\tau_e-2)(\tau_e-1) k^{-\tau_e}(1+\mathcal{O}(1/k))\nn\\
		&\hspace{2.8cm}\times\sum\limits_{m=1}^{k-1}L_1(m)m^{-(\tau_{\sss M}-\tau_e+1)}(k-m)(1+\mathcal{O}(1/m))~.}
	If $\tau_{\sss M}\leq \tau_e$, the sum in \eqref{tail:old:1} varies regularly as $k^{(\tau_e-\tau_{\sss M})+1}$. On the other hand for $\tau_{\sss M}>\tau_e$, the sum varies regularly as $k$. Therefore there exists a slowly varying function $L^{({\old})}(k)$ such that
	\eqn{\label{tail:old:2}
		\prob(1+M^{(\delta)}+Y(M^{(\delta)}+1,A_{\old})=k) = L^{({\old})}(k) k^{-\tau_{\sss(\old)}}~,}
	where ${\tau_{\sss(\old)}=\min\{ \tau_{\sss M},\tau_e \}-1}$. 
	\FC{For $\tau_{\sss M}>\tau_e,~L^{\sss(\old)}$ turns out to be a constant. On the other hand, for $\tau_{\sss M}\leq \tau_e,$ using Karamata's Theorem \cite[Proposition~1.5.8]{bingham_goldie_teugels_1987}, $L^{\sss(\old)}$ can be shown to be asymptotically equal to $L_1/\tau_{\sss(\old)}$.}
	
	The calculation for $\Young$ neighbours of the root is more involved. We first substitute the $f_{\young}(a)$ in \eqref{for:appd:deg:011} to obtain
	\eqan{\label{for:thm:young:1}
		&\int_0^1\prob(Y(m,a)=k-m)f_{\young}(a)\,da\nn\\
		=& \frac{\Gamma(k+\delta)}{\Gamma(m+\delta)(k-m)!}\int_0^1(1-a^{1-\chi})^{k-m}(a^{1-\chi})^{(m+\delta)}a^{-\chi}\\
		&\hspace{6cm}\times\int_0^{a^{1-\chi}}x^{\tau_e-2}(1-x)^{-1}\,dx\,da~. \nn}
	We perform a change of variable $u=a^{1-\chi}$, and simplify the above equation as
	\eqan{\label{for:thm:young:2}
			&\int_0^1\prob(Y(m,a)=k-m)f_{\young}(a)\,da\\
			=& (\tau_e-1)\frac{\Gamma(k+\delta)}{\Gamma(m+\delta)(k-m)!}\int_0^1 (1-u)^{k-m}u^{m+\delta}\int_0^u x^{\tau_e-2}(1-x)^{-1}\,dx\,du~.\nn
	}
	Using the fact that $(1-x)^{-1}=\sum\limits_{i\geq 0}x^i$ we simplify the inner integral as
	\eqn{\label{for:thm:young:3}
		\int_0^u x^{\tau_e-2}(1-x)^{-1}\,dx = \sum\limits_{i\geq 0}\int_0^u x^{\tau_e+i-2}\,dx=\sum\limits_{i\geq 0}\frac{u^{\tau_e+i-1}}{\tau_e+i-1}~.}
	Substituting the RHS of \eqref{for:thm:young:3} in \eqref{for:thm:young:2}, we obtain
	\eqan{\label{for:thm:young:4}
		&\int_0^1\prob(Y(m,a)=k-m)f_{\young}(a)\,da\nn\\
		=& \frac{(\tau_e-1)\Gamma(k+\delta)}{\Gamma(m+\delta)(k-m)!}\sum\limits_{i\geq 0}\frac{1}{\tau_e+i-1}\int_0^1 (1-u)^{(k-m+1)-1}u^{(m+\delta+\tau_e+i)-1}\,du\nn\\
		=& (\tau_e-1)\frac{\Gamma(k+\delta)\Gamma(k-m+1)}{\Gamma(m+\delta)(k-m)!}\sum\limits_{i\geq 0}\frac{1}{(\tau_e+i-1)} \frac{\Gamma(m+\delta+\tau_e+i)}{\Gamma(k+\delta+\tau_e+i+1)}\nn\\
		=& (\tau_e-1)\frac{\Gamma(k+\delta)}{\Gamma(m+\delta)}\sum\limits_{i\geq 0}\frac{1}{(\tau_e+i-1)} \frac{\Gamma(m+\delta+\tau_e+i)}{\Gamma(k+\delta+\tau_e+i+1)}\nn\\
		=&\frac{\Gamma(k+\delta)\Gamma(m+\delta+\tau_e)}{\Gamma(m+\delta)\Gamma(k+\delta+\tau_e+1)}\nn\\
		&\hspace{1cm}+(\tau_e-1)\frac{\Gamma(k+\delta)}{\Gamma(m+\delta)}\sum\limits_{i\geq 1}\frac{1}{(\tau_e+i-1)} \frac{\Gamma(m+\delta+\tau_e+i)}{\Gamma(k+\delta+\tau_e+i+1)}~.}
	The RHS of \eqref{for:thm:young:4} is lower bounded by the first term and hence the RHS of \eqref{for:appd:deg:011} can be lower bounded as
	\eqan{\label{for:thm:young:5}
		&\prob(M^{(0)}+Y(M^{(0)},A_{\young})=k)\nn\\
		&\geq\sum\limits_{m=1}^{k}\prob(M^{(0)}=m) \frac{\Gamma(k+\delta)\Gamma(m+\delta+\tau_e)}{\Gamma(m+\delta)\Gamma(k+\delta+\tau_e+1)}\nn\\
		&=k^{-(\tau_e+1)}(1+\mathcal{O}(1/k))\sum\limits_{m=1}^{k}L_2(m)m^{-(\tau_{\sss M}-\tau_e-1)}(1+\mathcal{O}(1/m))~.}
	It can easily be shown that the RHS of \eqref{for:thm:young:5} varies regularly with $k^{-(\tau_e+1)}$ if $\tau_{\sss M}>\tau_e+2$. When $\tau_{\sss M}\leq \tau_e+2$, it varies regularly with $k^{-(\tau_{\sss M}-1)}$. Therefore the lower bound varies regularly with $k^{-\tau_{(\young)}}$, where $\tau_{(\young)}=\min\{\tau_{\sss M}-1,\tau_e+1\}$ and some slowly varying function $L_2^\prime(k)$. For $\tau_{\sss M}>\tau_e+2,~L_2^\prime(k)$ turns out to be a constant, and for $\tau_{\sss M}=\tau_e+2,$
	\eqn{
		L_2'(k)=\sum_{t=1}^k L_2(t)/t,
	}
	and lastly for $\tau_{\sss M}<\tau_e+2$, we use the same Karamata Theorem as before to obtain that $L_2^\prime(k)=\Theta(L_2(k)).$
	We move on to analyse the second term of RHS of \eqref{for:thm:young:4}.
	\eqan{\label{for:thm:young:6}
			&(\tau_e-1)\frac{\Gamma(k+\delta)}{\Gamma(m+\delta)}\sum\limits_{i\geq 1}\frac{1}{(\tau_e+i-1)} \frac{\Gamma(m+\delta+\tau_e+i)}{\Gamma(k+\delta+\tau_e+i+1)}\nn\\
			=& (\tau_e-1)\frac{\Gamma(k+\delta)}{\Gamma(m+\delta)}\sum\limits_{i\geq 1}\frac{1}{(\tau_e+i-1)(m+\delta+\tau_e+i)} \frac{\Gamma(m+\delta+\tau_e+i+1)}{\Gamma(k+\delta+\tau_e+i+1)}\nn\\
			=&(\tau_e-1)\frac{\Gamma(k+\delta)\Gamma(m+\delta+\tau_e+2)}{\Gamma(m+\delta)\Gamma(k+\delta+\tau_e+2)}\nn\\
			&\hspace{1cm}\times\sum\limits_{i\geq 0}\frac{1}{(\tau_e+i)(m+\delta+\tau_e+i+1)}\prod\limits_{j=0}^{i-1}\frac{m+\delta+\tau_e+j+2}{k+\delta+\tau_e+j+2}~.
	}
	\paragraph{\underline{Claim}}For any $m\geq 1,$
	\eqan{\label{eq:lemma:upper:bound:young:tail}
		&\sum\limits_{i\geq 0}\frac{1}{(\tau_e+i)(m+\delta+\tau_e+i+1)}\prod\limits_{j=0}^{i-1}\frac{m+\delta+\tau_e+j+2}{k+\delta+\tau_e+j+2}\nn\\
		&\hspace{3cm}=~ \mathcal{O}(m^{-1}(2-\log (1-2m/(k+m+1))).
	}
	Subject to the claim above, there exists $J_0>0$ such that the RHS of \eqref{for:thm:young:6} can be upper bounded by $J_0 (1+\log (1-2m/(k+m+1))) k^{-(\tau_e+2)}m^{(\tau_e+1)}(1+\mathcal{O}(1/k))(1+\mathcal{O}(1/m))$. Therefore, using \eqref{for:thm:young:4}, \eqref{for:thm:young:5}, \eqref{for:thm:young:6} and \eqref{eq:lemma:upper:bound:young:tail}, the RHS of \eqref{for:appd:deg:011} can be upper bounded as
	\eqan{
		\label{for:thm:young:7}
		&\prob(M^{(0)}+Y(M^{(0)},A_{\young})=k) \nn \\
		\leq& L_2^\prime(k)k^{-\tau_{(\young)}}+J_0k^{-(\tau_e+2)}(1+\mathcal{O}(1/k))\\
		&\hspace{0.52cm}\times\sum\limits_{m=1}^k L_2(m)(2-\log (1-2m/(k+m+1)))m^{-(\tau_{\sss M}-\tau_e-2)}(1+\mathcal{O}(1/m))~.\nn
	}
	When $\tau_{\sss M}>\tau_e+2,$ the second sum is $o(k)$ and hence the second term is $o(k^{-(\tau_e+1)})$. Therefore the tail degree distribution of the $\young$ child varies regularly with $k^{-\tau_{(\young)}}$ when $\tau_{\sss M}>\tau_e+2$.
	
	On the other hand, when $\tau_{\sss M}\leq \tau_e+2$, the upper bound can be shown to vary regularly with $k^{-\tau_{(\young)}}$ and the slowly varying function is again $\Theta(L_2^\prime(k))$. Therefore considering $L_2^\prime(k)$ as $L^{({\young})}(k)$, we can say that 
	\eqn{\label{tail:young:final}
		\prob(M^{(0)}+Y(M^{(0)},A_{\young})=k)=\Theta(L^{({\young})}(k))k^{-\tau_{(\young)}}~.} 
	From our calculation here we could not identify the exact slowly varying function in $k$ when $\tau_{\sss M}\leq \tau_e+2$. We think this could also be shown by tweaking the sum in \eqref{for:thm:young:4} properly. Now, it only remains to prove \eqref{eq:lemma:upper:bound:young:tail} from the claim.
	
	\noindent
	\paragraph{Proof of claim} For $i\geq m-\delta-\tau_e-1$, we bound
	\eqn{
		\prod\limits_{j=0}^{i-1}\frac{m+\delta+\tau_e+j+2}{k+\delta+\tau_e+j+2}\leq 1,
	}
	so that
	\eqan{
		&\sum_{i\geq m-\delta-\tau_e-1}
		\frac{1}{(\tau_e+i)(m+\tau_e+i+\lfloor\delta\rfloor)}\nn\\
		&\hspace{2cm}\leq 
		\sum_{i\geq m-\delta-\tau_e-1}
		\frac{1}{(\tau_e+i)(\tau_e+i+\lfloor\delta\rfloor)}=\Theta(1/m),
	}
	as required.
	For $i<m-\delta-\tau_e-1$, instead, we rewrite the summands in \eqref{eq:lemma:upper:bound:young:tail} as 
	\eqan{
			&\frac{1}{(\tau_e+i)(m+\delta+\tau_e+i+1)}\prod\limits_{j=0}^{i-1}\frac{m+\delta+\tau_e+j+2}{k+\delta+\tau_e+j+2}\nn\\
			=&\frac{1}{(\tau_e+i)(m+\delta+\tau_e+1)}\prod\limits_{j=0}^{i-1}\frac{m+\delta+\tau_e+j+1}{k+\delta+\tau_e+j+2}~.
	}
	Next, we bound the product terms as
	\eqn{
		\frac{m+\delta+\tau_e+j+1}{k+\delta+\tau_e+j+2}
		\leq \frac{m+\delta+\tau_e+i-1+1}{k+\delta+\tau_e+i-1+2}
		\leq 
		\frac{2m}{k+m+1}.
	}
	Therefore,
	\eqn{
		\prod\limits_{j=0}^{i-1}\frac{m+\delta+\tau_e+j+1}{k+\delta+\tau_e+j+2}\leq 
		\Big(\frac{2m}{k+m+1}\Big)^i.
	}
	We conclude that, also using that $\tau_e\geq 1$,
	\eqan{
		&\sum_{i=0}^{m-\delta-\tau_e-1}
		\frac{1}{(\tau_e+i)(m+\delta+\tau_e+i+1)}\prod\limits_{j=0}^{i-1}\frac{m+\delta+\tau_e+j+2}{k+\delta+\tau_e+j+2}\\
		&\qquad \leq \frac{1}{m+\delta+\tau_e+1}+\frac{1}{m}
		\sum_{i=1}^{m-\delta-\tau_e-1}
		\frac{1}{i}\Big(\frac{2m}{k+m+1}\Big)^i\nn\\
		&\qquad\leq \Theta(1/m)
		+\frac{1}{m}
		\sum_{i\geq 1}
		\frac{1}{i}\Big(\frac{2m}{k+m+1}\Big)^i\nn\\
		&\qquad=\Theta(m^{-1}(1-\log(1-2m/(k+m+1)))),\nn
	}
	as required, proving \eqref{eq:lemma:upper:bound:young:tail} in the claim.
\end{proof}

\EndThumbs
\part{Stochastic Processes on Preferential Attachment Model}
\StartThumbs

\chapter{Percolation Threshold for P\'{o}lya point tree}
\label{chap:percolation_threshold_PPT}
\begin{flushright}
\footnotesize{}Based on:\\
\textbf{Percolation on preferential attachment models}
\cite{RRR23}
\end{flushright}
\vspace{0.1cm}
\begin{center}
	\begin{minipage}{0.7 \textwidth}
		\footnotesize{\textbf{Abstract.}
		In this chapter, we prove that the critical percolation threshold for the P\'olya point tree is the inverse of the spectral radius of the mean offspring operator. For positive $\delta$, we use sub-martingales to establish sub-criticality and apply spine decomposition theory to demonstrate super-criticality, thereby completing the proof. For $\delta \leq 0$ and any $\pi > 0$, we prove that the percolated P\'olya point tree dominates a supercritical branching process, showing that the critical percolation threshold equals $0$.
		
		We identify the explicit critical percolation threshold for P\'olya point trees as
		\[
		\pi_c = \frac{\delta}{2\left(m(m+\delta) + \sqrt{m(m-1)(m+\delta)(m+1+\delta)}\right)}
		\]
		for $\delta > 0$, and $\pi_c = 0$ for non-positive values of $\delta$.
		 }
	\end{minipage}
\end{center}
\vspace{0.1cm}

\section{Introduction}\label{chap:percolation_threshold_PPT:intro}
Bond percolation is a fundamental and elementary process in statistical physics and network science that investigates the connectivity of a network under random attack. In bond percolation, each edge of a given graph is independently retained with probability $\pi$ and removed with probability $1-\pi$. 

\chRounak{In Chapter~\ref{chap:LWC}, we observed that generalised preferential attachment models have a universal local limit, the random P\'{o}lya point tree. In this chapter, we study percolation on this multi-type branching process. For any $\pi\in[0,1]$, let $\PPT_{\pi}(m,\delta)$ denote the connected component of the $\pi$-percolated $\PPT(m,\delta)$ that contains the root. The survival probability of the P\'{o}lya point tree is defined as
\eqn{\label{eq:def:survival-probability}
		\zeta(\pi):=\prob\Big( \big|\PPT_{\pi}(m,\delta)\big|=\infty \Big)~.}
Branching processes typically undergo a phase transition in the survival probability, depending on the percolation probability $\pi$. Following the notation commonly used in the branching process literature, we define $\pi_c$ as the critical percolation threshold for the P\'{o}lya point tree, where $\zeta(\pi)>0$ for $\pi>\pi_c$, and $\zeta(\pi)=0$ for $\pi<\pi_c$. Percolation on branching processes has been extensively studied, and we refer the reader to \cite{JN84,AN72,H63} and the references therein.
}

\chRounak{It has been shown in \cite{AN72} that the critical percolation threshold for a Galton-Watson branching process with mean offspring $\lambda (>1)$ is $1/\lambda$. For multi-type branching processes with a finite type space, the critical percolation threshold is given by the inverse of the mean offspring matrix. However, identifying the same for multi-type branching processes with a more general type space is not straightforward. Under certain conditions, it has been shown in \cite{AN72} that the spectral radius of the mean offspring operator of the branching process is its critical percolation threshold. Most of these proof techniques rely on explicit characterization and analysis of the probability generating functions of the number of offspring. Lyons, Peres, and Pemantle in \cite{LPP95} provided an alternative proof of the identification of critical percolation thresholds for Galton-Watson branching processes, using simple probabilistic tools such as martingales and spine decomposition theory. Later, Athreya generalised the proof technique for multi-type branching processes with general type space in \cite{A2000}.
}

\chRounak{In \cite{ABS22}, the authors proved that the P\'{o}lya point tree with $\delta=0$ has a critical percolation threshold of $0$. However, this proof heavily relies on the assumption $\delta=0$, and we found it challenging to extend the result to other admissible values of $\delta$. In this chapter, we identify the critical percolation threshold for P\'{o}lya point trees in general.
}

\begin{center}
	{\textbf{Organisation of the chapter}}
\end{center}
This chapter is organized as follows. 
In Section~\ref{chap:percolation_threshold_PPT:sec:main-theorem}, we have stated the main theorem regarding the critical percolation threshold of the P\'{o}lya point tree, and briefly outline the proof to the theorem. Then, in Section~\ref{chap:percolation_threshold_PPT:sec:non-positive}, we prove the theorem for non-positive $\delta$. In Section~\ref{subsec:positive-delta}, we prove the theorem for positive $\delta$ regime.

\section{Critical percolation threshold of PPT}\label{chap:percolation_threshold_PPT:sec:main-theorem}
In \cite[Appendix~E]{ABS22}, the authors showed that the survival probability is continuous at $\pi_c$ for the $\delta=0$ case of the P\'olya point tree. Additionally, the authors proved in \cite[Corollary~2.2]{ABS22} that $\zeta(\pi)$ is continuous for all $\pi \neq \pi_c$. 

We now proceed to identify the critical percolation threshold $\pi_c$ for the P\'olya point tree, a multi-type branching process with a mixed continuous and discrete type space. Additionally, we demonstrate that $\zeta(\pi)$ is continuous at $\pi_c$ for any $\delta > -m$. Intuitively, $\pi_c$ for such branching processes is the inverse of the spectral norm of the offspring operator, but this requires a proof.

For any $m \geq 1$ and $\delta > -m$, from \cite[(5.4.86)]{vdH2} the kernel function of the mean offspring operator of the P\'olya point tree is given by 
\eqn{\label{eq:kernel:offspring-operator:PPT}
	\kappa\big((x,s), (y,t)\big)=\frac{c_{st}(\one_{\{x>y, t=\old\}}+\one_{\{x<y, t=\young\}})}{(x\vee y)^{\chi}(x\wedge y)^{1-\chi}},
}
with $\chi = (m + \delta)/(2m + \delta) \in (0,1)$ and 
\eqn{
	\label{eq:def-cst}
	c_{st}=
	\begin{cases}
		\frac{m(m+\delta)}{2m+\delta} &\text{ for } st = \Old\Old,\\
		\frac{m(m+1+\delta)}{2m+\delta} &\text{ for } st = \Old\Young,\\
		\frac{(m-1)(m+\delta)}{2m+\delta} &\text{ for } st = \Young\Old,\\
		\frac{m(m+\delta)}{2m+\delta} &\text{ for } st = \Young\Young.
	\end{cases}
}
Let ${\bfT}_\kappa$ denote the mean offspring operator of the P\'olya point tree, defined on $L^2(\Scal,\lambda)$, where $\lambda$ is the Lebesgue measure on the continuous part of $\Scal$.

\begin{Theorem}[Critical percolation threshold for the P\'olya point tree]\label{thm:critical-percolation:PPT}
	Fix $m \geq 1$ and $\delta > -m$. Let ${\bfT}_\kappa$ be the offspring operator of the P\'olya point tree on $\Scal$. Then, for $\delta > 0$, the critical percolation threshold of the P\'olya point tree is $1/r({\bfT}_\kappa)$. For $m \geq 2$ and $\delta \in (-m,0]$, the critical percolation threshold is $0$. Additionally, $\pi \mapsto \zeta(\pi)$, the survival probability of the P\'olya point tree after percolation, is continuous at the critical percolation threshold.
\end{Theorem}

Theorem~\ref{thm:critical-percolation:PPT} identifies the critical percolation threshold for the P\'olya point tree and establishes the continuity of the survival probability at $\pi_c$ for $\delta > -m$.

This theorem holds true for $m = 1$ for positive $\delta$, whereas our proof technique fails for non-positive $\delta$ and $m = 1$. \chRounak{We believe that for non-positive $\delta$, the critical percolation threshold is also $0$.}

Now, we provide a brief outline of the proof of this theorem. We investigate the survival probability of the percolated P\'olya point tree in two regimes: $\delta \leq 0$ and $\delta > 0$.

For $\delta \leq 0$, we prove the following proposition: 
\begin{Proposition}[Non-positive $\delta$]\label{prop:criticality:negative-delta}
	Fix $m \geq 2$ and $\delta \in (-m,0]$. Then, the critical percolation threshold of the P\'olya point tree with parameters $m$ and $\delta$ is $0$.
\end{Proposition}

The proof of this proposition uses an interesting property of the P\'olya point tree, namely the \emph{elbow children}. We do not use any functional analytic tools to prove this proposition, in contrast to \cite{ABS22}. We prove that for non-positive $\delta$ and any $\pi > 0$, the P\'olya point tree, percolated with probability $\pi$, stochastically dominates a supercritical single-type branching process.

On the other hand, for $\delta>0$, the critical percolation threshold turns out to be positive, necessitating a detailed investigation of the mean offspring generator of the P\'olya point tree. Rather than confining our analysis to a restricted space, we choose to extend the type space to $\Scal_e = [0,\infty) \times \{\Old, \Young\}$ while keeping the kernel unchanged. This extension allows us to incorporate vertices that will join the preferential attachment graph after time $n$. By growing the $\PPT$ on this extended type space and subsequently truncating all nodes with an age greater than $1$, we retrieve the $\PPT$.

Let $\bar{\bfT}_{\kappa}$ denote the mean offspring operator defined on $L^2(\Scal_e,\lambda_e)$, where $\lambda_e$ is the Lebesgue measure on the continuous part of $\Scal_e$. The following theorem identifies an eigenvalue of $\bar{\bfT}_{\kappa}$ along with the corresponding eigenfunction, and this eigenvalue is the spectral norm of the mean offspring operator, irrespective of whether the type space is $\Scal$ or $\Scal_e$:

\begin{Theorem}[Spectral norm of the offspring generator]
	\label{thm:operator-norm:PPT}
	Fix $m \geq 1$ and $\delta>0$. Let $r(\bfT_{\kappa})$ and $r(\bar{\bfT}_{\kappa})$ denote the spectral norms of $\bfT_{\kappa}$ and $\bar{\bfT}_{\kappa}$, respectively. Then
	\eqn{
		\label{nu-equality-PAM}
		r({\bfT}_{\kappa})=r(\bar{\bfT}_{\kappa})=2\frac{m(m+\delta)+\sqrt{m(m-1)(m+\delta)(m+1+\delta)}}{\delta}.
	}
	Additionally, the explicit eigenfunction of $\bar{\bfT}_\kappa$ corresponding to $r(\bar{\bfT}_{\kappa})$ is $f(x,s)=\bfp_s/\sqrt{x}$, where $\bfp=(\bfp_{\old},\bfp_{\young})$ is the right eigenvector of $\bfM$, defined by
	\[
	\bfM=\begin{bmatrix}
		c_{\old\old} & c_{\old\young}\\
		c_{\young\old} & c_{\young\young}
	\end{bmatrix}~.
	\]
\end{Theorem}

It can further be shown that the operator norms of these operators are also equal to their spectral norm. Although we do not use this fact to prove the critical percolation threshold of the P\'olya point tree, this result is interesting in its own right. We defer the proof of Theorem~\ref{thm:operator-norm:PPT} to Section~\ref{sec:spectral_radius}.

Next, we prove that for $\pi \leq 1/r(\bar{\bfT}_{\kappa})$, the percolated P\'olya point tree is subcritical, demonstrating the required left continuity of the survival probability function $\zeta$, while the right continuity generally follows for percolation models from \cite[Lemma~8.9]{Grimmet99}:

\begin{Proposition}\label{lem:subcriticality}
	For $\delta>0$, the $\PPT$ dies out almost surely if it is percolated with probability $\pi \leq 1/r(\bar{\bfT}_\kappa)$.
\end{Proposition}

Conversely, for $\pi > 1/r(\bar{\bfT}_{\kappa})$, we prove that the percolated P\'olya point tree is supercritical.

\begin{Proposition}[Supercriticality of $\PPT$]\label{prop:supercriticality}
	For $\delta>0$, the $\PPT$ survives with positive probability when percolated with probability $\pi > 1/r(\bar{\bfT}_{\kappa})$.
\end{Proposition}

We prove this proposition by adapting the proof of \cite[Lemma~3.3]{DM13}, utilising a spine decomposition argument akin to that used by Lyons, Pemantle, and Peres in \cite{LPP95}. Therefore, from Propositions~\ref{prop:criticality:negative-delta}, \ref{lem:subcriticality}, \ref{prop:supercriticality}, and Theorem~\ref{thm:operator-norm:PPT}, the proof of Theorem~\ref{thm:critical-percolation:PPT} follows immediately. In the remainder of the chapter, we shall focus on proving these lemmas and the theorem.

\section{Non-positive regime}\label{chap:percolation_threshold_PPT:sec:non-positive}
In \cite{ABS22}, the authors addressed the case $\delta = 0$ using fixed-point approximation techniques. We propose an alternative approach, establishing that for any $\delta \leq 0$ and $\pi > 0$, the P\'olya point tree stochastically dominates a single-type discrete branching process with a mean offspring greater than $1$. Before defining this branching process, we introduce the concept of \emph{elbow children} as follows:
\begin{Definition}[Elbow children]
	Let $(x,s)$ be a node in $\PPT$. A node $(z,{\Old})$ is called an \emph{elbow child} of $(x,s)$ if there exists a node $(y,{\Young})$ in $\PPT$ such that $(x,s)$ connects to $(y,{\Young})$, and $(y,{\Young})$ connects to $(z,{\Old})$ in $\PPT$. The edge connections involved in forming an elbow child are referred to as \emph{elbow edges}.
\end{Definition}
\begin{wrapfigure}{r}{3.0cm}
\begin{tikzpicture}[line cap=round,line join=round,>=triangle 45,x=1.0cm,y=1.0cm]
	\clip(-7.5,2.2) rectangle (-5,4.7);
	\draw [line width=2.pt,color=wrwrwr] (-7.,3.)-- (-6.5,4.);
	\draw [line width=2.pt,color=wrwrwr] (-6.5,4.)-- (-6.,2.5);
	\draw (-6,2.8) node[anchor=north west] {$(z,{\Old})$};
	\draw (-7,4.7) node[anchor=north west] {$(y,{\Young})$};
	\draw (-7.5,2.9) node[anchor=north west] {$(x,s)$};
	\begin{scriptsize}
		\draw [fill=rvwvcq] (-7.,3.) circle (2.5pt);
		\draw [fill=rvwvcq] (-6.5,4.) circle (2.5pt);
		\draw [fill=rvwvcq] (-6.,2.5) circle (2.5pt);
	\end{scriptsize}
\end{tikzpicture}

	\label{fig:elbow-child}
\end{wrapfigure}
In the figure, the $(z,\Old)$ node is connected to $(x,s)$ through $(y,\Young)$. Thus, $(z,\Old)$ is an elbow child of $(x,s)$. The edge connections $((x,s),(y,\Young))$ and $((y,\Young),(z,\Old))$ are elbow edges. Note that all the elbow children of any node have the label $\Old$ in the tree. Furthermore, the set of elbow children of nodes in the $n$-th generation provides a lower bound on the total number of $\Old$-labeled nodes in the $(n+2)$-th generation in a $\PPT$.

Fix any $\pi > 0$. We now work with the $\PPT$ percolated with probability $\pi$. To construct the branching process, we proceed in two steps. First, we construct a branching process with elbow children, where the offspring generation still depends on the age of the parent node. Next, we further stochastically lower bound this process with another branching process, where the offspring generation no longer depends on the age of the parent node.

Let $h > 0$ be a fixed threshold. Starting from the root in the percolated $\PPT$, there is a positive probability of obtaining a child $(x,\Old)$ of the root such that $x < h$. We consider the branching process $\Tcal(x,\Old)$, starting at $(x,\Old)$. The children of $(x,\Old)$ in $\Tcal(x,\Old)$ are given by the elbow children of $(x,\Old)$ in the percolated $\PPT$ we started with. We continue growing the tree following the same procedure. It is evident that the size of $\Tcal(x,\Old)$ serves as a lower bound for the size of the percolated $\PPT$. We then prune the tree so that all nodes have an age of at most $h$. In this pruned tree, all nodes have the label $\Old$, making the type labels redundant. We denote this pruned tree as $\Tcal^{\sss (h)}(x)$, where the nodes no longer have labels. From its construction, it is evident that the size of $\Tcal^{\sss (h)}(x)$ further lower bounds that of $\Tcal(x,\Old)$, and in turn, lower bounds the size of the percolated $\PPT$.

Before proceeding to the next step of our stochastic lower bounds, we compute the explicit kernel function $\kappa_1$ of the offspring operator of $\Tcal^{\sss (h)}(x)$. Since $(z,\Old)$ is an elbow child of $(y,\Old)$, the latter connects to the former through a $\Young$-labeled node. Therefore, $\kappa_1$ is given by
\begin{equation}\label{def:kernel:lower-bound:1}
	\kappa_1(y,z)=\pi^2\int\limits_{y}^1\kappa((y,\Old),(u,\Young))\kappa((u,\Young),(z,\Old))\,du.
\end{equation}
Here, we crucially use the fact that the node $(u,\Young)$ connecting $(y,\Old)$ and $(z,\Old)$ in the $\PPT$ is unique, if it exists. The $\pi^2$ factor arises from bond percolation on the two elbow edges used to connect $(y,\Old)$ and $(z,\Old)$ in the $\PPT$. The right-hand side of \eqref{def:kernel:lower-bound:1} can be simplified as
\begin{equation}\label{eq:kernel:lower-bound:1}
	\kappa_1(y,z)=\begin{cases}
		c_{\old\young}c_{\young\old}\pi^2 (yz)^{-1+\chi}\frac{1-y^{1-2\chi}}{1-2\chi}, & \text{for } \delta < 0, \\
		c_{\old\young}c_{\young\old}\pi^2 (yz)^{-1+\chi}\log(1/x), & \text{for } \delta = 0,
	\end{cases}
\end{equation}
where $\chi = (m + \delta)/(2m + \delta)$, as defined earlier. 

Now, we move on to the second step of the proof. In this step, we aim to establish a lower bound for $\Tcal^{\sss (h)}(x)$ by comparing it to a single-type branching process where the offspring generation does not depend on the age of the parent node. To achieve this, we carefully choose the offspring distribution to obtain the desired lower bound. To determine the offspring distribution, we use a stochastic domination argument. Let $\Bcal_1^{\sss(h)}(x)$ denote the $1$-neighbourhood of $x$ in $\Tcal^{\sss (h)}(x)$, and for any two random variables $X$ and $Y$, we use $X\preceq Y$ to denote that $Y$ stochastically dominates $X$.

\begin{Lemma}\label{lem:stoch-dom}
	Let $0 < z_1 \leq z_2 < 1$. Then $|\Bcal_1^{\sss(h)}(z_2)| \preceq |\Bcal_1^{\sss(h)}(z_1)|$.
\end{Lemma}

\begin{proof}
	To prove the stochastic domination, we revisit the definition of the $\PPT$. From \eqref{for:pointgraph:poisson:PPT}, the $\Young$-labelled children of a node $(z,\Old)$ are generated by an inhomogeneous Poisson process with intensity
	\begin{equation}\label{for:lem:stoch-dom:1}
		\rho_{(z,\old)}(x) = (1-\chi)\Gamma_{(z,\old)} \frac{x^{-\chi}}{z^{1-\chi}} \one_{\{x \geq z\}},
	\end{equation}
	where $\Gamma_{(z,\old)}$ is a Gamma random variable that is independent of the choice of $z$. 
	Let $z_{ij}$ denote the age of the $j$-th $\Young$-labelled child of $(z_i,\Old)$.
	
	Since $\rho_{(z_2,\Old)}(x) \leq \rho_{(z_1,\Old)}(x)$ for all $x \geq z_1$, $z_{11}$ is stochastically dominated by $z_{21}$. As a consequence of this stochastic domination between the ages, it can similarly be shown that $z_{1j} \preceq z_{2j}$ for all $j$. It is important to remember that the $\Old$-labelled children of $(z_{ij},\Young)$ are the elbow children of $(z_i,\Old)$. From the construction of the $\PPT$, for all $k \in [m-1]$, the age of the $k$-th $\Old$-labelled child of $(z_{ij},\Young)$ is distributed as $U^{1/\chi} z_{ij}$, where $U$ is a uniform random variable on $[0,1]$. Hence, for all $k \in [m-1]$, the age of the $k$-th $\Old$-labelled child of $(z_{1j},\Young)$ is stochastically dominated by that of $(z_{2j},\Young)$. Clearly, all these stochastic domination results also hold true for the percolated $\PPT$.
	
	The lemma follows from the fact that in $\Tcal^{\sss(h)}(z_1)$ and $\Tcal^{\sss(h)}(z_2)$, all nodes with age higher than $h$ are discarded.
\end{proof}

Lemma~\ref{lem:stoch-dom} implies that the offspring distribution of $z_1$ dominates that of $z_2$ in $\Tcal^{\sss(h)}(x)$ for $z_1 \leq z_2$. Since the age of all nodes in $\Tcal^{\sss(h)}(x)$ is bounded from above by $h$, the offspring distribution of the node with age $h$ (denoted as $F_{\pi,h}$) is stochastically dominated by the offspring distribution of each of the nodes in $\Tcal^{\sss(h)}(x)$.

\begin{proof}[Proof of Theorem~\ref{thm:critical-percolation:PPT} for non-positive $\delta$]
	Consider a single-type discrete branching process ${\sf{BP}}(\pi, h)$ with offspring distribution $F_{\pi,h}$. From its construction, the size of ${\sf{BP}}(\pi, h)$, denoted as $|{\sf{BP}}(\pi, h)|$, serves as a stochastic lower bound on the size of the percolated $\PPT$.
	
	The expected number of offspring of every node in ${\sf{BP}}(\pi, h)$ can be expressed as
	\begin{equation}\label{eq:offspring:GWBP:1}
		\int\limits_0^h \kappa_1(h,u) \, du =
		\begin{cases}
			c_{\old\young}c_{\young\old}\pi^2 h^{-1+2\chi} \frac{1-h^{1-2\chi}}{\chi(1-2\chi)} & \text{for } \delta < 0,\\
			c_{\old\young}c_{\young\old}\pi^2 \log(1/h) & \text{for } \delta = 0.
		\end{cases}
	\end{equation}
	
	Since $2\chi < 1$ for $\delta < 0$, and $\log(1/h) \to \infty$ as $h \searrow 0$, for every fixed $\pi > 0$, we can choose the threshold $h$ small enough that the RHS of \eqref{eq:offspring:GWBP:1} is greater than $1$, making ${{\sf{BP}}}(\pi, h)$ a supercritical branching process. Consequently, the percolated P\'olya point tree ($\PPT$) also survives with positive probability after percolation with probability $\pi$.
\end{proof}

\section{Positive $\delta$ regime}\label{subsec:positive-delta}
We present the proof of Theorem~\ref{thm:critical-percolation:PPT} for positive $\delta$ in two crucial steps. First, we establish that the spectral norm of $\bar{\bfT}_{\kappa}$ corresponds to the inverse critical percolation threshold for the P\'olya point tree. Next, we prove the continuity of $\pi \mapsto \zeta(\pi)$ at the previously obtained critical percolation threshold. Our proof technique for the first step enables us to demonstrate the left continuity of $\zeta(\pi)$ at $\pi_c$. Furthermore, we use \cite[Lemma~8.9]{Grimmet99} to establish the right continuity of $\zeta(\pi)$ at $\pi_c$, thereby proving its continuity at $\pi_c$.

We proceed to prove that the inverse of the spectral norm of the offspring generator is identical to the $\PPT$'s critical percolation threshold. This proof consists of two parts. First, we show that when $\pi \leq 1/r(\bar{\bfT}_\kappa)$, the process dies out almost surely. To demonstrate this, we adapt an argument used in \cite[Lemma~3.3]{DM13}.

Throughout this section, we work with the $\PPT$ with type space $\Scal_e$. In Proposition~\ref{lem:subcriticality}, we prove that when percolated with probability $\pi \leq 1/r(\bar{\bfT}_{\kappa})$, the probability that the percolated $\PPT$ has the age of the leftmost node below $1$ at some generation converges to $0$. Therefore, after truncating the tree at nodes with ages greater than $1$, the tree will eventually die out. On the other hand, when the $\PPT$ is percolated with probability $\pi > 1/r(\bar{\bfT}_{\kappa})$, we prove in Proposition~\ref{prop:supercriticality} that with some positive probability, there exists an ancestral line of nodes where the ages of the nodes converge to $0$, and hence the $\PPT$ survives even in the restricted type space $\Scal$.

\begin{proof}[Proof of Proposition~\ref{lem:subcriticality}]
	In Theorem~\ref{thm:operator-norm:PPT}, we derived the eigenfunction of the offspring operator corresponding to its spectral norm, denoted as $f(x, s)$. Let $Y_{(x, s)}^{(n)}(y, t)\,dy$ represent the empirical measures of the type and age of all offspring in the $n$-th generation of a percolated P\'olya point tree initiated at $(x, s)$. With each generation $n \geq 0$, we associate the \emph{score}
	\eqn{\label{eq:def:supermartingale}
		X_\pi^{(n)}(x, s) := \sum\limits_{t \in \{ {\old, \young} \}} \int\limits_0^\infty Y_{(x, s)}^{(n)}(y, t) \frac{f(y, t)}{f(x, s)}\,dy~.}
	Given that $\pi \leq 1/r(\bar{\bfT}_\kappa)$ and $f$ is the eigenfunction of $\bar{\bfT}_{\kappa}$ corresponding to $r(\bar{\bfT}_{\kappa})$, it follows that $(X_\pi^{(n)})_{n \geq 1}$ constitutes a non-negative supermartingale and therefore converges almost surely. Now, let us consider a fixed value of $N \in \R^{+}$. Define the event $\mathcal{A}_n(N)$ as
	\eqn{\label{for:lem:subcriticality:2}
		\Acal_n(N) = \{\exists\text{ particle in the $n$-th generation with age less than }N\}~.}
	To complete the proof of Lemma~\ref{lem:subcriticality}, we prove that $\Acal_n(N)$ happens finitely often almost surely, i.e.,
	\eqn{\label{for:lem:subcriticality:3}
		\prob(\Acal_n(N)~\text{i.o.}) = 0~,}
	where i.o.\ means infinitely often. For simplicity of notation, we shall use $\Acal_n$ to denote $\Acal_n(N)$ throughout this proof. To prove \eqref{for:lem:subcriticality:3}, we define two sequences of stopping times $T_k$ and $S_k$ by
	\eqan{\label{for:lem:subcriticality:02}
		T_k = &\inf\{n > k: |X_\pi^{(n)} - X_\pi^{(n-1)}| > 1\}~, \\
		S_k = &\inf\{n > k: \Acal_n \text{ occurs}\}~,\nn
	}
	for $k \geq 1$. We now prove that $\prob\big(|X_\pi^{(n)} - X_\pi^{(n-1)}| > 1 \mid \mathcal{A}_{n-1}\big)$ is positive and independent of $n$, its value depending only on $N$.
	
	We couple two samples of the offspring of the $(n-1)$-st generation, where we keep all offspring identical except for the leftmost child of the leftmost node in the $(n-1)$-st generation. The two leftmost children of the leftmost node in the $(n-1)$-st generation are sampled independently. Let $(A_1, s_1)$ and $(\bar{A}_1, \bar{s}_1)$ be the two copies of the leftmost child of the leftmost node in the $(n-1)$-st generation.
	
	Note that the distribution of $(A_1, s_1)$ does not depend on $n$, but only on the fact that its parent node has an age less than $N$. Since the parent node of $(A_1, s_1)$ has age at most $N$ conditionally on $\Acal_{n-1}$,
	\eqn{
		\prob\Big(\inf\limits_{s, s_1, \bar{s}_1 \in \{\old, \young\}}\big|\bfp_{s_1}/\sqrt{A_1} - \bfp_{\bar{s}_1}/\sqrt{\bar{A}_1}\big|/\bfp_s > 2/\sqrt{N} \mid \Acal_{n-1}\Big) > \iota(N) > 0~,\nn}
	for some function $\iota(N)$ depending only on $N$, i.e., $\iota(N)$ is independent of $n$. Let $X_\pi^{(n)}$ and $\bar{X}_\pi^{(n)}$ be the \emph{score} of the $n$-th generation when the leftmost child of the leftmost node in the $(n-1)$-th generation is $(A_1, s_1)$ and $(\bar{A}_1, \bar{s}_1)$, respectively. Therefore, adding and subtracting the contribution of other nodes of generation $n$ to the scores, we have
	\eqan{\label{for:lem:subcriticality:04}
		&\prob\Big(\big|\bfp_{s_1}/\sqrt{A_1} - \bfp_{\bar{s}_1}/\sqrt{\bar{A}_1}\big|/\bfp_s > 2/\sqrt{N} \mid \Acal_{n-1}\Big)\nn \\
		=& \prob\big( |{X}_\pi^{(n)} - \bar{X}_\pi^{(n)}| > 2 \mid \Acal_{n-1} \big)~.}
	By the triangle inequality and the union bound, we upper bound the RHS of \eqref{for:lem:subcriticality:04} as
	\eqan{\label{for:lem:subcriticality:05}
		&\prob\big( |{X}_\pi^{(n)} - \bar{X}_\pi^{(n)}| > 2 \mid \Acal_{n-1} \big)\nn \\
		\leq & \prob \Big( |\bar{X}_\pi^{(n)} - {X}_\pi^{(n-1)}| > 1~\text{or } |{X}_\pi^{(n)} - {X}_\pi^{(n-1)}| > 1 \mid \Acal_{n-1}\Big)\nn \\
		\leq & 2\prob \Big( |{X}_\pi^{(n)} - {X}_\pi^{(n-1)}| > 1 \mid \Acal_{n-1}\Big)~.}
	Assume that $\prob(S_k < \infty) > \gamma > 0$ for some fixed $k \in \N$ and $\gamma > 0$. Then
	\eqan{\label{for:lem:subcriticality:06}
		\prob(T_k < \infty) \geq \sum\limits_{n > k}\prob(T_k = n + 1 \mid S_k = n)\prob(S_k = n)~.}
	Since $S_k$ is a stopping time, $\{S_k = n\}$ is $\Fcal_n$ measurable. Therefore, using \eqref{for:lem:subcriticality:04} and \eqref{for:lem:subcriticality:05},
	\eqn{\label{for:lem:subcriticality:07}
		\prob(T_k = n + 1 \mid S_k = n) = \prob(|X_\pi^{(n+1)} - X_\pi^{(n)}| > 1 \mid \Acal_n) > \iota(N)/2~.}
	Hence by \eqref{for:lem:subcriticality:07}, we obtain
	\eqan{\label{for:lem:subcriticality:08}
		\prob(T_k < \infty) > \iota(N)/2 \sum\limits_{n > k}\prob(S_k = n) > \gamma \iota(N)/2~.}
	Since $\big(X_\pi^{(n)}\big)_{n \geq 1}$ converges almost surely, $\prob(T_k < \infty)$ tends to $0$ as $k \to \infty$. Therefore, $\prob(S_k < \infty)$ tends to $0$ as $k \to \infty$. Note that $\{S_k < \infty\} = \bigcup\limits_{n > k}\Acal_n$, and it is a decreasing sequence. From the definition,
	\eqn{\label{for:lem:subcriticality:09}
		\prob(\Acal_n~\text{i.o.}) = \prob\Big(\bigcap\limits_{k \geq 1}\{S_k < \infty\}\Big) = \lim\limits_{k \to \infty}\prob(S_k < \infty) = 0~.}
	Hence, $\Acal_n$ occurs only finitely often, proving that the age of the leftmost node of the $n$-th generation diverges to $\infty$ almost surely. Therefore, upon truncating the nodes with age greater than $1$, the $\PPT$ dies out almost surely.
\end{proof}

Next, in Proposition~\ref{prop:supercriticality}, we prove that the $\PPT$ survives when it is percolated with a probability greater than $1/r(\bar{\bfT}_k)$. This, along with Proposition~\ref{lem:subcriticality}, proves that $1/r(\bar{\bfT}_\kappa)$ is the critical percolation threshold of the P\'olya point tree.

To prove Proposition~\ref{prop:supercriticality}, we suitably truncate the $\PPT$ and obtain the spectral norm of the offspring operator of the truncated $\PPT$. Then, we show that upon percolating with probability $\pi$, the truncated tree has an ancestral line of nodes whose ages converge to $0$ with positive probability, proving that the $\PPT$ also survives the percolation with positive probability.

\paragraph{Truncation}
In the $\PPT$, consider the subtree where the age of the children of a node $(x,s)$ is restricted to the range $(0,bx]$ for some constant $b>1$. We call this truncated subtree the $b-\PPT$. Note that the mean-offspring kernel $\kappa$ changes in the $b-\PPT$ to
\begin{equation}\label{eq:truncated-kernel}
	\kappa_b((x,s),(y,t))=\kappa((x,s),(y,t))\mathbf{1}_{\{y\leq bx\}}~.
\end{equation}

Now we show that with positive probability, there exists an infinite path from the root such that the age of the nodes converges to $0$ in the $b-\PPT$. First, we show that the integral operator $\bar{\mathbf{T}}_{\kappa_b}$ also has an eigenfunction corresponding to its spectral norm:

\begin{Lemma}[Spectral norm of truncated operator]\label{lem:spectral-norm:truncated-operator}
	The spectral norm of the integral operator $\bar{\mathbf{T}}_{\kappa_b}$ is
	\begin{equation}\label{eq:lemma:spectral-radius}
		r(\bar{\mathbf{T}}_{\kappa_b}) = \frac{(1+q)c_{\old\old}+\sqrt{((1-q)c_{\old\old})^2+4qc_{\old\young}c_{\young\old}}}{2\chi-1}~,
	\end{equation}
	where $q=1-b^{1/2-\chi}$. Furthermore, the eigenfunction of $\bar{\mathbf{T}}_{\kappa_b}$ corresponding to its spectral norm is given by
	\begin{equation}\label{eq:lemma:spectral-radius:efunc}
		h(x,s)=\frac{\mathbf{u}_s^b}{\sqrt{x}}~,
	\end{equation}
	where $\mathbf{u}^b=(\mathbf{u}_{\old}^b,\mathbf{u}_{\young}^b)$ is the eigenvector of $\mathbf{M}_b$ corresponding to its largest eigenvalue, and where $\mathbf{M}_b$ is defined as
	\[
	\mathbf{M}_b=\begin{bmatrix}
		c_{\old\old} & c_{\old\young}(1-b^{1/2-\chi})\\
		c_{\young\old} & c_{\young\young}(1-b^{1/2-\chi})
	\end{bmatrix}~.
	\]
\end{Lemma}

The proof of this lemma follows similarly to Theorem~\ref{thm:operator-norm:PPT}. Therefore, we defer the proof of this lemma to Section~\ref{sec:spectral_radius}.

Note that $q\to1$ as $b\to\infty$, and consequently $r(\bar{\mathbf{T}}_{\kappa_b})\to r(\bar{\mathbf{T}}_{\kappa})$ as $b\to\infty$. Therefore, for any $\pi>1/r(\bar{\mathbf{T}}_{\kappa})$, there exists a sufficiently large $b$ such that $\pi r(\bar{\mathbf{T}}_{\kappa_b})>1~.$ We perform a change of measure to prove that the truncated tree survives percolation with positive probability. For this change of measure argument, we define a martingale similar to \eqref{eq:def:supermartingale} as
\begin{equation}\label{eq:def:martingale}
	M^{(n)}_b (x,s) :=\frac{1}{\rho_b^n}\sum\limits_{t\in\{ {\old,\young} \}} \int\limits_0^\infty Y_{(b)}^{(n)}((x,s),(y,t))\frac{h(y,t)}{h(x,s)}\,dy~,
\end{equation}
where $\rho_b=\pi r(\bar{\mathbf{T}}_{\kappa_b})$ and $Y_{(b)}^{(n)}((x,s),(y,t))$ places a Dirac delta mass at the types of the $n$-th generation offspring of $(x,s)$ in the percolated $b-\PPT$. Since $h$ is an eigenfunction of $\bar{\mathbf{T}}_{\kappa_b}$, $M^{(n)}_b (x,s)$ is a martingale for any choice of $(x,s)\in\mathcal{S}_e$. Since ${M}^{(n)}_b (x,s)$ is a non-negative $L^1$ martingale, it converges almost surely to a non-negative martingale limit. We use the following Kesten-Stigum theorem for the $b-\PPT$ to ensure that the martingale limit is non-zero with positive probability:

\begin{Theorem}[Kesten-Stigum Theorem for $b-\PPT$]\label{K-S-multi-type}
	$M_b^{(n)}(x,s)$ converges almost surely to a non-zero limit under the following assumption
	\begin{equation}\label{K-S-cond}
		\sup\limits_{(x,s)\in\mathcal{S}_e}\mathbb{E}\left[ \left(M_b^{(1)}(x,s)\right)^{1+\eta} \right]<\infty~\text{for some }\eta>0.
	\end{equation}
\end{Theorem}

Next, we prove that $M^{(1)}_b (x,s)$ has a finite $(1+\eta)$-th moment for some $\eta>0$, making the martingale uniformly integrable. The following lemma provides us with this uniform finite $(1+\eta)$-th moment for some $\eta>0$. Consequently, our martingale satisfies the $L\log L$ condition for the Kesten-Stigum theorem, proved in Theorem~\ref{K-S-multi-type} for the $b-\PPT$, and hence $M^{(n)}_b (x,s)$ converges to a positive limit with positive probability:

\begin{Lemma}[Finite $(1+\eta)$-th moment]
	\label{thm:finite:higher-moment}
	Fix any $\eta<2\chi-1$. Then
	\begin{equation}\label{eq:thm:finite:higher-moment}
		\sup\limits_{(x,s)\in \mathcal{S}} \mathbb{E}\left[ \left(M_b^{(1)}(x,s)\right)^{1+\eta} \right]<\infty~.
	\end{equation}
\end{Lemma}

First, we proceed to prove Proposition~\ref{prop:supercriticality}, relying on Theorem~\ref{K-S-multi-type} and Lemma~\ref{thm:finite:higher-moment}. We will then establish Theorem~\ref{K-S-multi-type} and Lemma~\ref{thm:finite:higher-moment}. 

Let us denote $M_b(x,s)$ as the limit of $M^{(n)}_b(x,s)$ as $n \to \infty$. Since $M^{(1)}_b(x,s)$ has a uniformly bounded $(1+\eta)$-th moment, it can be shown that the limit $M_b(x,s)$ is positive with positive probability.

We next employ a spine decomposition argument. Let $\prob_{(x,s)}$ denote the law of the $b-\PPT$ rooted at $(x,s)$. Now, consider the $b-\PPT$ under the tilted measure
\eqn{\label{def:tilted-measure}
	\,d\Qbb_{(x,s)}=M_b(x,s)\,d\prob_{(x,s)}~,}
Given a $b-\PPT$ rooted at $(x,s)$, we construct a measure $\mu^\star$ on the set of all infinite paths $(x,s)=\bfx_0,\bfx_1,\ldots$ in the $b-\PPT$, such that, for any permissible sequence,
\eqn{\label{eq:mu:def}
	\mu\big( \{\bfx_0,\bfx_1,\ldots\}:\bfx_0=(x,s),\ldots,\bfx_n=(y_n,t_n) \big)= \rho_b^{-n} \frac{M_b(\bfx_n)}{M_b(\bfx_0)}~.}
Let $(t_n)_{n\geq 0}$ be the sequence of labels in an infinite path $(x_n)_{n\geq 0}$. The next lemma demonstrates that this label sequence is a stationary Markov chain with invariant distribution ${\bf{\upsilon}}^b=({\bf{\upsilon}}_\old^b,{\bf{\upsilon}}_\young^b)$ for some $({\bf{\upsilon}}_\old^b,{\bf{\upsilon}}_\young^b)$. The transition probability of the Markov chain is given by
\eqn{\label{eq:transition-prob:MC}
	p_{s,t}=\frac{(\bfM_b)_{s,t}\bfu_t^b}{\lambda_{\bfM}^{\sss (b)}\bfu_s^b}\qquad\text{for }s,t\in\{\Old,\Young\}~.}
\begin{Lemma}[Markov property of labels]\label{lem:MC:label}
	The random sequence of labels of nodes in a path chosen from $\mu^\star$ is a stationary Markov chain with transition probabilities given by \eqref{eq:transition-prob:MC}.
\end{Lemma}
\begin{proof}
	To prove the Markov property, we consider the probability of $t_{n+1}$ conditional on $\{t_0,t_1,\ldots,t_{n}\}$ and show that this probability depends only on $t_n$. Let $\bfa_{(n)}=\big(a^{(0)},a^{(1)},\ldots,a^{(n)}\big)\in\{\Old,\Young\}^{n}.$ Then, from the definition of conditional probability,
	\eqan{\label{for:lem:MC:01}
		&\Qbb\big(t_{n+1}=\Old|(t_0,t_1,\ldots,t_n)=\bfa_{(n)}\big)\nn\\
		&\hspace{1cm}=\frac{\Qbb\big( (t_0,t_1,\ldots,t_n,t_{n+1})=(\bfa_{(n)},\Old) \big)}{\Qbb\big((t_0,t_1,\ldots,t_n)=\bfa_{(n)}\big)}~.}
	Now, the numerator on the RHS of \eqref{for:lem:MC:01} can be simplified as
	\eqan{\label{for:lem:MC:02}
		&\Qbb\big( (t_0,t_1,\ldots,t_n,t_{n+1})=(\bfa_{(n)},\Old) \big)\nn\\
		=&\rho_b^{-(n+1)}\int\limits_{(0,\infty)^{n}}\int_0^\infty \pi\kappa_b\big((x_n,a^{(n)}),(y,\Old)\big)\frac{h(y,\Old)}{h(x_n,a^{(n)})}\,dy\nn\\ &\qquad\qquad\times\prod\limits_{i=1}^{n}\pi\kappa_b\big((x_{i-1},a^{(i-1)}),(x_{i},a^{(i)})\big)\frac{h(x_i,a^{(i)})}{h(x_{i-1},a^{(i-1)})}\,dx_i\nn\\
		=&p_{a^{(n)},\old}\rho_b^{-n}\int\limits_{(0,\infty)^{n}}\prod\limits_{i=1}^{n}\pi\kappa_b\big((x_{i-1},a^{(i-1)}),(x_{i},a^{(i)})\big)\frac{h(x_i,a^{(i)})}{h(x_{i-1},a^{(i-1)})}\,dx_i\nn\\
		=&p_{a^{(n)},\old} \Qbb\big((t_0,t_1,\ldots,t_n)=\bfa_{(n)}\big)~.
	}
	Since $\Qbb\big( (t_0,t_1,\ldots,t_n,t_{n+1})=(\bfa_{(n)},\Young) \big)=1-\Qbb\big( (t_0,t_1,\ldots,t_n,t_{n+1})=(\bfa_{(n)},\Old) \big)$, by \eqref{for:lem:MC:01} and \eqref{for:lem:MC:02}, the lemma follows immediately.
\end{proof}

We now investigate the age of the $n$-th node of the truncated P\'olya point tree in an infinite path sampled according to $\mu^\star$. These ages turn out to be multiplicative in nature:
\begin{Lemma}[Multiplicative property of the joint law of ages]\label{lem:multiplicative-property}
	Let $\{(X_0,t_0),(X_1,t_1),\ldots\}$ be an infinite path sampled from $\mu$. Then,
	\eqn{\label{eq:lem:multiplicative-property}
		(X_n)_{n\geq 0}\overset{d}{=}\Big(X_0\prod\limits_{i=1}^n R_i(t_i)\Big)_{n\geq 1}~,}
	for some independent random variables $(R_i(t))_{\{t\in\{\Old,\Young\},n\in\N\}}$.
\end{Lemma}
To prove Lemma~\ref{lem:multiplicative-property}, we compute the explicit distribution of the independent random variables, which will also be helpful for the final step in the proof of Proposition~\ref{prop:supercriticality}. First, we calculate the distribution of ${X_n}/{X_{n-1}}$ conditional on $X_{n-1},t_{n-1},t_n$, which turns out to be independent of $\big(X_{n-1},t_{n-1}\big)$. This establishes Lemma~\ref{lem:multiplicative-property}.

\begin{proof}[Proof of Lemma~\ref{lem:multiplicative-property}]
	To prove the lemma, we calculate the joint distribution of $\frac{X_n}{X_{n-1}}$ and $t_n$ conditionally on $X_{n-1}$ and $t_{n-1}$. We compute, for $a \in (0,1)$,
	\eqan{\label{for:lem:multiplicative-property:1}
		&\Qbb\left( \frac{X_n}{X_{n-1}} \geq a, t_n = \Old \mid X_{n-1} = x, t_{n-1} \right)\nn\\
		&\hspace{2cm} = \rho_b^{-1} \int\limits_{ax}^x \pi \kappa_b\left((x,t_{n-1}), (y,\Old)\right)\frac{h(y,\Old)}{h(x,t_{n-1})}\,dy\nn\\
		&\hspace{5cm} = \Qbb(t_n = \Old \mid t_{n-1})(1 - a^{\chi - \frac{1}{2}})~.
	}
	Similarly, for $a \in (1,b)$,
	\eqan{\label{for:lem:multiplicative-property:2}
		&\Qbb\left( \frac{X_n}{X_{n-1}} \geq a, t_n = \Young \mid X_{n-1} = x, t_{n-1} \right)\nn\\
		&\hspace{2cm} = \rho_b^{-1} \int\limits_{ax}^{bx} \pi \kappa_b\left((x,t_{n-1}), (y,\Young)\right)\frac{h(y,\Young)}{h(x,t_{n-1})}\,dy\nn\\
		&\hspace{4.5cm} = \Qbb(t_n = \Young \mid t_{n-1})\frac{a^{\frac{1}{2} - \chi} - b^{\frac{1}{2} - \chi}}{1 - b^{\frac{1}{2} - \chi}}~.
	}
	Therefore, by \eqref{for:lem:multiplicative-property:1} and \eqref{for:lem:multiplicative-property:2}, the tail distribution function of $\frac{X_n}{X_{n-1}}$, conditionally on $(X_{n-1}, t_{n-1}, t_n)$, does not depend on $(X_{n-1}, t_{n-1})$ and is given by
	\eqan{\label{for:lem:multiplicative-property:3}
		\Qbb\left( \frac{X_n}{X_{n-1}} \geq a \mid t_n = \Old, X_{n-1} = x, t_{n-1} \right) &= \Qbb\left( \frac{X_n}{X_{n-1}} \geq a \mid t_n = \Old \right)\nn\\
		&\quad= 1 - a^{\chi - \frac{1}{2}},~\text{for } a \leq 1~,\\
		\Qbb\left( \frac{X_n}{X_{n-1}} \geq a \mid t_n = \Young, X_{n-1} = x, t_{n-1} \right) &= \Qbb\left( \frac{X_n}{X_{n-1}} \geq a \mid t_n = \Young \right)\nn\\
		&\quad= \frac{a^{\frac{1}{2} - \chi} - b^{\frac{1}{2} - \chi}}{1 - b^{\frac{1}{2} - \chi}}~\text{for } a \in (1,b)~.\nn
	}
	By \eqref{for:lem:multiplicative-property:3}, we can express $X_n$ multiplicatively as
	\eqn{\label{for:lem:multiplicative-property:4}
		X_n \overset{d}{=} X_{n-1} R_n(t_n)~,
	}
	where $R_n(t_n)$ has the distribution function defined in \eqref{for:lem:multiplicative-property:3} depending on whether $t_n$ is $\Old$ or $\Young$. By iterating the same argument, we can express $X_n$ conditionally on $(t_0,\ldots,t_n)$ as
	\[
	X_n \overset{d}{=} X_0 \prod\limits_{i=1}^n R_i(t_i)~,
	\]
	where $\big(R_i(t_i)\big)_{i \geq 1}$ are (conditionally) independent random variables. Hence, $X_n$ has the desired multiplicative structure.
\end{proof}

Now we have all the necessary tools to prove Proposition~\ref{prop:supercriticality}. Before presenting the details, we provide an outline of the remaining proof. We show that $\Qbb$-almost surely, $\mu$-almost every path has the property that
\eqn{\label{for:prop:supercriticality:1}
	\lim\limits_{n \to \infty} \frac{\log(X_n)}{n} = c < 0~.
}
This implies that $\Qbb$-almost surely, the logarithmic values of the ages of nodes in an infinite path converge to $-\infty$. Therefore, there exists an ancestral line of nodes whose ages converge to $0$. For the $b-\PPT$ rooted at some $(x,s) \in \Scal$, it follows that there exists an ancestral line of particles whose ages converge to $0$ with positive probability. Therefore, $b-\PPT$ survives the percolation with positive probability.

\begin{proof}[Proof of Proposition~\ref{prop:supercriticality}]
	To prove \eqref{for:prop:supercriticality:1}, we use Lemma~\ref{lem:multiplicative-property} as follows:
	\eqn{\label{for:prop:supercriticality:2}
		\frac{\log(X_n)}{n} \overset{d}{=} \frac{\log(X_0) + \sum\limits_{i=1}^{n_\old} \log(R_i(\Old)) + \sum\limits_{j=1}^{n_\young} \log(R_j(\Young))}{n}~,}
	where $n_\old$ and $n_\young$ are the number of $\Old$ and $\Young$ labels in $\{t_1, t_2, \ldots, t_n\}$, respectively. Since $R_i(\Old)$ and $R_j(\Young)$ are independent random variables, and $\{t_0, t_1, \ldots\}$ is a stationary Markov chain with limiting distribution ${\bf{\upsilon}}^b = ({\bf{\upsilon}}_\old^b, {\bf{\upsilon}}_\young^b)$, by the strong law of large numbers, we have:
	\eqan{\label{for:prop:supercriticality:3}
		\frac{\log(X_n)}{n} \overset{a.s.}{\to} {\bf{\upsilon}}_\old^b \E[\log(R_1(\Old))] + {\bf{\upsilon}}_\young^b \E[\log(R_1(\Young))]~,}
	where ${\bf{\upsilon}}^b = ({\bf{\upsilon}}_\old^b, {\bf{\upsilon}}_\young^b)$ is the stationary distribution of $\{t_0, t_1, \ldots\}$. Using \eqref{for:lem:multiplicative-property:3}, we compute:
	\eqan{
		&\E[\log(R_1(\Old))] = -\frac{1}{\chi - \frac{1}{2}},\label{for:prop:supercriticality:4-1}\\
		\text{and} \qquad &\E[\log(R_1(\Young))] = \frac{1}{1 - b^{1/2 - \chi}} \left[ \frac{1 - b^{1/2 - \chi}}{\chi - \frac{1}{2}} - b^{1/2 - \chi} \log(b) \right]~.\label{for:prop:supercriticality:4-2}
	}
	Next, we show that ${\bf{\upsilon}}_\old^b > {\bf{\upsilon}}_\young^b$. To prove this, we use a general result for a 2-state Markov process.
	Let 
	\eqn{\label{for:clarity:matrix-1}
		A = \begin{bmatrix}
			e_1 & 1 - e_1 \\
			e_2 & 1 - e_2
	\end{bmatrix}}
	be a transition matrix, and let $\theta = (\theta_1, \theta_2)$ be its stationary distribution. Then,
	\eqn{\label{for:clarity:matrix-2}
		\frac{\theta_1}{\theta_2} = \frac{e_2}{1 - e_1}~.}
	Therefore, the necessary and sufficient condition for $\theta_1 > \theta_2$ is $e_1 + e_2 > 1$.
	Thus, ${\bf{\upsilon}}_\old^b > {\bf{\upsilon}}_\young^b$ holds if and only if $p_{\old\old} + p_{\young\old} > 1$. Substituting the values of $p_{\old\old}$ and $p_{\young\old}$ from \eqref{eq:transition-prob:MC}, we have
	\eqn{\label{for:prop:supercriticality:5}
		p_{\old\old} + p_{\young\old} = \frac{c_{\old\old}}{\lambda_{\sss \bfM}^{(b)}} + \frac{c_{\young\old}\bfu_{\old}^b}{\lambda_{\sss \bfM}^{(b)}\bfu_{\young}^b} = \frac{c_{\young\young}}{\lambda_{\sss \bfM}^{(b)}} + \frac{c_{\young\old}\bfu_{\old}^b}{\lambda_{\sss \bfM}^{(b)}\bfu_{\young}^b} > p_{\young\young} + p_{\young\old} = 1~.}
	Therefore, by \eqref{for:prop:supercriticality:3}--\eqref{for:prop:supercriticality:5}, we have that ${\log(X_n)}/{n} \to c$ for some $c < 0$.
	As $b \to \infty$, $\upsilon^{b} \to (1/2, 1/2)$, and the second term on the RHS of \eqref{for:prop:supercriticality:4-2} also converges to $0$, making the RHS of \eqref{for:prop:supercriticality:3} equal to $0$.
	As a result, under $\Qbb_{(x,s)}$, the ages of the nodes in $\PPT$ almost surely converge to $0$. $\Qbb_{(x,s)}$ is a tilted measure of $\prob_{(x,s)}$ with the martingale limit of $M_b^{(n)}(x,s)$ being the Radon-Nikodym derivative. Since $M_b^{(n)}(x,s)$ has a positive limit with positive probability, the ages of the nodes in the $n$-th generation of the $b-\PPT$ converge to $0$ with positive probability. Since $b-\PPT$ is a truncated $\PPT$, $\PPT$ also survives with positive probability.
\end{proof}

Now, we proceed to prove the Kesten-Stigum theorem for the $b-\PPT$ in Theorem~\ref{K-S-multi-type}. First, we prove Theorem~\ref{K-S-multi-type} and then prove Lemma~\ref{thm:finite:higher-moment} to satisfy the conditions for Theorem~\ref{K-S-multi-type}.


\begin{proof}[Proof of Theorem~\ref{K-S-multi-type}]
	We prove this theorem following the same line of proof done by Athreya in \cite{A2000}, adapted for our martingale as done in \cite{KS01}. For the sake of simplicity of the notations we use $W_n(y)$ for $y\in\Scal_e$, in the place of $M_b^{(n)}(x,s)$ for $(x,s)\in\Scal_e$ and $W(y)$ as the $n\to\infty$ limit of $W_n(y)$.
	
	Let $\Mcal\equiv\{ \mu:\mu(\cdot)=\sum_{i=1}^n \delta_{x_i}(\cdot)~\text{for some }n<\infty,~x_1,x_2,\ldots,x_n\in\Scal_e \}$ where $\delta_x(\cdot)$ is the delta measure at $x$, i.e., $\delta_x(A)=1$ if $x\in A$ and $0$ if $x\notin A$. Clearly, $\Mcal$ is closed under addition.
	Let $z_n$ denote the number of particles in the $n$-th generation of the tree, and for any $x\in\Scal_e,$ let $\xi^x$ denote the empirical measures of the children of $x$ in $\Scal_e$. Let $\{x_{ni}:i\in [z_n]\}$ denote the particles in the tree in $n$-th generation of the tree. Then adapting the notions of \cite{A2000}, we define $Z_n$ as
	\eqn{\label{eq:K-S:1}
		Z_n=\sum\limits_{i=1}^{z_n} \xi^{x_{ni}}~,}
	where $\xi^{x_{ni}}$ are independent $\Mcal$-valued random variable, and denote $V(\mu)$ as
	\[
	V(\mu)=\int h\,d\mu~,
	\]
	where $h$ is the eigenfunction defined in \eqref{eq:lemma:spectral-radius:efunc}. 
	For any initial value $z\in \Mcal,$
	\eqn{\label{eq:K-S:02}
		\Qbb(z,\,d\mu)=\frac{V(\mu)\prob(z,\,d\mu)}{\rho_b V(z)}~.} 
	Choosing $z=\delta_{x_0},$ we obtain the same $\Qbb_{x_0}$ measure as defined in \eqref{def:tilted-measure} for some $x_0\in\Scal_e$.
	From Corollary~$1$ of \cite{A2000}, for $x_0\in\Scal_e$,
	\eqn{\label{eq:K-S:021}
		\E_{x_0}[W(x_0)]=1~\text{under}~\prob_{x_0},~\text{if and only if}~\Qbb_{x_0}(W(x_0)=\infty)=0~.}
	The random variable $\tilde{\xi}^x$ has a size-biased distribution of $\xi^x$ in that
	\eqn{\label{eq:K-S:2}
		\prob\big( \tilde{\xi}^x\in \,dm \big)=\frac{V(m)\prob\big( {\xi}^x\in \,dm \big)}{\rho_b h(x)}~.}
	Under $\Qbb$ measure, $b-\PPT$ has an infinite line of nodes. This infinite line of nodes is termed as the spine' of the tree. Offspring distribution of the nodes of spine is the size-biased distribution defined in \eqref{eq:K-S:2}. Rest of the nodes in the tree has usual offspring distribution. Following the calculation in \cite[page~329-330]{A2000}, it can be shown that the markov chain $(Z_n)_{n=0}^\infty$ with transition function $\Qbb$ defined in \eqref{eq:K-S:02}, evolves in the following manner. Given $Z_n=(x_{n1},x_{n2},\ldots,x_{nz_n}),~Z_{n+1}$ is generated as follows:
	\begin{enumerate}
		\item Choose individual $x_{ni}$ with probability $\frac{h(x_{ni})\one_{[x_{ni}\sim x_{(n-1)j^\star}]}}{V(\xi^{x_{(n-1)j^\star}})}$ and name it $x_{nj^\star}$;
		\vspace{3pt}
		\item Choose its offspring $\tilde{\xi}^{x_{nj^\star}}$ according to the one in \eqref{eq:K-S:2};
		\vspace{3pt} 
		\item For all other individuals in the generation, choose the offspring process $\xi^{x_{ni}}$ according to the original probability distribution $\prob(\xi^{x_{ni}}\in\,dm)$;
		\vspace{3pt}
		\item $\tilde{\xi}^{x_{nj^\star}}$ and $\xi^{x_{ni}}$ are all chosen independently for $i\neq j^\star$;
		\vspace{3pt}
		\item Set $Z_{n+1}=\tilde{\xi}^{x_{nj^\star}}+\sum\limits_{i\neq j^\star}\xi^{x_{ni}}~.$
	\end{enumerate}
	Therefore, $\{x_{nj^\star},n=1,2,\ldots\}$ keeps track of the spine' component of the tree. Next we follow the line of proof in \cite[equation~$(15),(19a),(19b)$]{A2000} with the adapted martingale $W_n(x_0)$ to obtain
	\eqan{\label{eq:K-S:3}
		&\E_{\Qbb}\Big[ W_{n+1}(x_0)\mid \tilde{\xi}^{x_{1j^\star}},\tilde{\xi}^{x_{2j^\star}},\ldots \Big]\nn\\
		\leq& 1+\sum\limits_{r=0}^\infty \frac{V(\tilde{\xi}^{x_{rj^\star}})}{\rho_b^{r+1}h(x_0)}\nonumber\\
		=&1+\sum\limits_{r=0}^\infty \left(\frac{V(\tilde{\xi}^{x_{rj^\star}})}{\rho_b^{r+1}h(x_{rj^\star})}\right)\left(\frac{h(x_{rj^\star})}{h(x_0)}\right)
		\equiv W^\star(x_0)~.}
	Now we prove that there exists $\omega\in(0,1)$ such that $$\Dcal_r=\left\{\left(\frac{V(\tilde{\xi}^{x_{rj^\star}})}{\rho_b^{r+1}h(x_{rj^\star})}\right)\left(\frac{h(x_{rj^\star})}{h(x_0)}\right)>\omega^r\right\}$$
	happens at most finitely often, which in turn proves that $W^\star(x_0)$ is finite $\Qbb$ almost surely.	
	For proving that $\Dcal_r$ happens finitely often, define the following:
	\eqan{
		\Bcal_r=&\left\{ \frac{V(\tilde{\xi}^{x_{rj^\star}})}{\rho_b h(x_{rj^\star})}>(\rho_b\omega)^{r(1-\alpha)}  \right\}~,\label{eq:K-S:4}\\
		\text{and}\qquad\Ccal_r=& \left\{ \frac{h(x_{rj^\star})}{h(x_0)}>(\rho_b\omega)^{r\alpha} \right\}~,\label{eq:K-S:5}}
	for some $\alpha\in(0,1)$. Then by the union bound,
	\eqn{\label{eq:K-S:6}\Qbb(\Dcal_r~\text{i.o.})\leq \Qbb(\Bcal_r~\text{i.o.})+\Qbb(\Ccal_r~\text{i.o.})~.}
	The author has proved in \cite{A2000} that a sufficient condition to establish $\Qbb(\Bcal_r~\text{i.o.})=0$ is the finiteness of $\int\limits_1^\infty \bar{f}(\e^t)\,dt$, where 
	\eqn{\label{eq:K-S:07}
		\bar{f}(t)\equiv\sup\limits_x\prob\left(\frac{V(\tilde{\xi}^x)}{\rho_b h(x)}>t\right)~.}
	Although the proof in \cite{A2000} is there for a different denominator, but the argument here still follows exactly the same way as in \cite{A2000}.
	The following claim provides a sufficient condition for the finiteness of $\int\limits_1^\infty \bar{f}(\e^t)\,dt$:
	\begin{Claim}[A sufficient condition]\label{lem:equivalent:condition}
		$\int\limits_1^\infty \bar{f}(\e^t)\,dt$ is finite if for some $\eta>0$,
		\eqn{\label{eq:lem:equivalence}
			\sup\limits_{(x,s)\in\Scal_e}\E\Big[ M_b^{(1)}(x,s)^{1+\eta} \Big]<\infty~,}
		where $f$ is as defined in \eqref{eq:K-S:07}.
	\end{Claim}
	The proof follows immediately from Markov's inequality.
	\paragraph{Proof of Claim~\ref{lem:equivalent:condition}}
	First we simplify $f(\e^t)$ defined in \eqref{eq:K-S:07}. For any $t>1$ and $\eta>0$,
	\eqn{\label{lem:eqv:1}
		\bar{f}(\e^t)=\sup\limits_{x}\prob\Big( \frac{V(\tilde{\xi}^x)}{\rho_b h(x)}>\e^t \Big)=\sup\limits_{x}\prob\Big( \left(\frac{V(\tilde{\xi}^x)}{\rho_b h(x)}\right)^{\eta}>\e^{\eta t} \Big)~.}
	By Markov's inequality, we bound $\bar{f}(\e^t)$ by
	\eqn{\label{lem:eqv:2}
		\sup\limits_x \e^{-\eta t}\E\Big[ \left(\frac{V(\tilde{\xi}^x)}{\rho_b h(x)}\right)^{\eta} \Big]=\e^{-\eta t}\sup\limits_x \E\Big[ \left(\frac{V({\xi}^x)}{\rho_b h(x)}\right)^{1+\eta} \Big]~.}
	The equality follows from the fact that $\tilde{\xi}^x$ is size-biased version of $\xi^x$ as defined in \eqref{eq:K-S:2}. Since, $\sup\limits_x \E\Big[ \left(\frac{V({\xi}^x)}{\rho_b h(x)}\right)^{1+\eta} \Big]$ finite, we bound $\int\limits_1^\infty \bar{f}(\e^t)\,dt$ by
	\eqn{\label{eq:eqv:3}
		\int\limits_1^\infty \bar{f}(\e^t)\,dt \leq \sup\limits_x \E\Big[ \left(\frac{V({\xi}^x)}{\rho_b h(x)}\right)^{1+\eta} \Big] \e^{-\eta}/\eta<\infty~,}
	which completes the proof.

	By Lemma~\ref{lem:equivalent:condition}, finiteness of  $\sup\limits_x\E\left[ \left(\frac{V({\xi}^x)}{\rho_b h(x)}\right)^{1+\eta}\right]$ implies finiteness of $\int\limits_1^\infty \bar{f}(\e^t)\,dt$, and Lemma~\ref{thm:finite:higher-moment} provides us with the uniform bound on the $1+\eta$-th moment of $W_1(y)$. Therefore, 
	\eqn{\label{eq:K-S:7}
		\Qbb(\Bcal_r~\text{i.o.})=0~.}
	For the analysis of $\Ccal_r$, we use the properties of $b-\PPT$ under $\Qbb$, more specifically the proof of Proposition~\ref{prop:supercriticality}.
	Let $x_{rj^\star}$ has label $t_r\in\{\Old,\Young\}$ and age $a_{rj^\star}\in (0,\infty)$. Therefore, $\log h(x_{rj^\star})=\log(\bfu_{t_r}^b)+\log(a_{rj^\star})/2$, where $\bfu^b=(\bfu_\old^b,\bfu_\young^b)$ is as defined in Lemma~\ref{lem:spectral-norm:truncated-operator}.
	By \eqref{for:prop:supercriticality:2}-\eqref{for:prop:supercriticality:4-2}, we can prove that for any fixed $b>0$,
	\eqn{\label{eq:K-S:8}
		\frac{\log(h(x_{rj^\star}))-\log h(x_0)}{r}\overset{a.s.}{\to} \frac{{\bf{\upsilon}}_\old^b-{\bf{\upsilon}}_\young^b}{2\chi-1} + \frac{b^{1/2-\chi}\log(b){\bf{\upsilon}}_\young^b}{2(1-b^{1/2-\chi})}~.}
	Now using \eqref{for:clarity:matrix-1}-\eqref{for:clarity:matrix-2}, the stationary measure ${\bf{\upsilon}}^b=({\bf{\upsilon}}_\old^b,{\bf{\upsilon}}_\young^b)$ can be explicitly calculated as
	\[
	{\bf{\upsilon}}_\old^b=\frac{\frac{c_{\young\old}\bfu_{\old}^b}{\lambda_{\bfM}^{\sss(b)}\bfu_{\young}^b}}{\frac{c_{\young\old}\bfu_{\old}^b}{\lambda_{\bfM}^{\sss(b)}\bfu_{\young}^b}+\frac{c_{\old\young}(1-b^{1/2-\chi})\bfu_{\young}^b}{\lambda_{\bfM}^{\sss(b)}\bfu_{\old}^b}}\quad\text{and}\quad {\bf{\upsilon}}_\young^b=1-{\bf{\upsilon}}_\old^b~,
	\]
	where $\lambda_{\bfM}^{(b)}$ is as defined in \eqref{for:lem:spectra:truncated-operator:2} and $\bfu^b=(\bfu_\old^b,\bfu_{\young}^b)$ is the right eigenvector of $\bfM_b$ defined in Lemma~\ref{lem:spectral-norm:truncated-operator}. For any $b>0~,\bfu^b$ can be explicitly computed as
	\[
	\bfu_{\old}^b=\frac{c_{\old\young}(1-b^{1/2-\chi})}{\lambda_{\bfM}^{(b)}+c_{\old\young}(1-b^{1/2-\chi})-c_{\old\old}}\quad\text{and}\quad\bfu_{\young}^b=1-\bfu_{\old}^b~.
	\]
	Note that all of  $\{\bfu_\old^b,\bfu_{\young}^b,{\bf{\upsilon}}_\old^b,{\bf{\upsilon}}_\young^b\}$ are continuous in $b,$ and as $b\to \infty,~{\bf{\upsilon}}^b\to (1/2,1/2)$, while the RHS of \eqref{eq:K-S:8} decreases to $0$. 
	
	On the other hand, as $b\to\infty,~\rho_b$ increases to $\rho=\pi r(\bfT_{\kappa}),$ which is strictly greater than $1$. Therefore, there exists a $b$ large enough and suitable $\omega\in(1/\rho,1)$ such that $\log(\rho_b\omega)>\log(\rho\omega)/2$.
	Similarly, for $b$ large enough, RHS of \eqref{eq:K-S:8} can be shown to be less than $\log(\rho\omega)/4$.
	Hence $\Ccal_r$ occurs finitely often, and by \eqref{eq:K-S:6}, $\Qbb(\Dcal_r~\text{i.o.})=0$ completes the argument that $\Dcal_r$ happens finitely often $\Qbb_{x_0}$ almost surely.	
	Following the argument in \cite[Proof of Theorem~3]{A2000}, we prove 
	\eqn{\label{eq:K-S:04}
		\Qbb_{x_0}(W(x_0)<\infty)=1,\quad\text{or}\quad\Qbb_{x_0}(W(x_0)=\infty)=0~.}
	Therefore, by \eqref{eq:K-S:021}, we conclude
	\eqn{\label{eq:K-S:041}
		\E_{x_0}[W(x_0)]=1\quad\text{under}\quad\prob_{x_0}~.}
	Hence, $W(x_0)$ is non-zero random variable $\prob_{x_0}$ almost surely. 
\end{proof}

We now proceed to prove Lemma~\ref{thm:finite:higher-moment} to satisfy the sufficiency condition in Claim~\ref{lem:equivalent:condition}. For any $(x,s)$ in the $b$-P\'{o}lya point tree ($b$-PPT), we assign two bins $B_\old(x,s)$ and $B_\young(x,s) \subseteq [0,bx]\times\{\Old,\Young\}$ such that all the $\Old$-labelled children of $(x,s)$ are in $B_\old(x,s)$. Similarly, all the $\Young$-labelled children of $(x,s)$ with ages in $[x,bx]$ are in $B_\young(x,s)$. From the construction of the P\'{o}lya point tree, observe that only finitely many offspring of $(x,s)$ are in $B_t(x,s)$ for both $t \in \{\Old,\Young\}$. Let $n_t$ denote the (possibly random) number of offspring of $(x,s)$ in bin $B_t(x,s)$, and let $(A_{ti})_{i\in[n_t]}$ denote the types of offspring of $(x,s)$ in $B_t(x,s)$ for $t \in \{\Old,\Young\}$.

We express $M^{(1)}_b(x,s)$ as
\begin{equation}\label{eqn:martingale:independent-rv}
	M^{(1)}_b(x,s) = \frac{1}{\rho_b^n}\sum\limits_{t\in\{\old,\young\}}\sum\limits_{i=1}^{n_t} \frac{h(A_{ti})}{h(x,s)}~.
\end{equation}
This representation helps us to prove Lemma~\ref{thm:finite:higher-moment}, proving the finiteness of $(1+\eta)$-th moment of $M^{(1)}_b (x,s)$ for some $\eta>0$:
\begin{proof}[Proof of Lemma~\ref{thm:finite:higher-moment}]
	Fix any $(x,s)\in\Scal$. Then using \eqref{eqn:martingale:independent-rv}, $\big(M^{(1)}_b(x,s)\big)^{1+\eta}$ can be rewritten as
	\eqan{\label{for:lem:finite:higher-moment:1}
		&\big(M^{(1)}_b(x,s)\big)^{1+\eta}\nn\\
		=& (\rho_b h(x,s))^{-(1+\eta)} \Big( \sum\limits_{t\in\{\old,\young\}}\sum\limits_{i=1}^{n_t} {h(A_{ti})}\Big)^{1+\eta}\nn\\
		=&(\rho_b h(x,s))^{-(1+\eta)} \sum\limits_{t\in\{\old,\young\}}\sum\limits_{i=1}^{n_t} h(A_{ti}) \Big(\sum\limits_{l\in\{\old,\young\}}\sum\limits_{j=1}^{n_l} h(A_{lj})\Big)^\eta\\
		\leq&(\rho_b h(x,s))^{-(1+\eta)} \sum\limits_{t\in\{\old,\young\}}\sum\limits_{i=1}^{n_t} h(A_{ti})\Big[ \Big( \sum\limits_{j=1}^{n_\old} h(A_{\old j}) \Big)^\eta + \Big( \sum\limits_{j=1}^{n_\young} h(A_{\young j}) \Big)^\eta \Big]\nonumber~.
	}
	The inequality in \eqref{for:lem:finite:higher-moment:1} follows from the fact that $(a+b)^\eta \leq a^\eta + b^\eta$ for all $\eta < 1$ and $a, b > 0$. Since $\rho_b$ and $h(x, s)$ are constants, we omit these terms in \eqref{for:lem:finite:higher-moment:1} for the time being and focus on the random component.
	More precisely, instead of looking at $\E\Big[ \big(M^{(1)}_b(x,s)\big)^{1+\eta} \Big]$, we upper bound
	\eqan{\label{for:lem:finite:higher-moment:2}
		&(\rho_b h(x,s))^{1+\eta}\E\Big[ \big(M^{(1)}_b(x,s)\big)^{1+\eta} \Big]\nn\\
		\leq& ~\E\Big[\sum\limits_{t\in\{\old,\young\}}\sum\limits_{i=1}^{n_t} h(A_{ti})\Big(\sum\limits_{i=1}^{n_t} h(A_{ti})\Big)^{\eta}\Big]\\
		&\hspace{0.1cm}+\E\Big[\sum\limits_{i=1}^{n_\old} h(A_{\old i})\Big]\E\Big[ \Big( \sum\limits_{j=1}^{n_\young} h(A_{\young j}) \Big)^\eta \Big]+\E\Big[\sum\limits_{i=1}^{n_\young} h(A_{\young i})\Big]\E\Big[ \Big( \sum\limits_{j=1}^{n_\old} h(A_{\old j}) \Big)^\eta \Big]~.\nn
	}
	By the construction of the $b-\PPT$, the ages of the $\Old$ and $\Young$ labelled children of $(x,s)$ are independent. Therefore, the expectations in the second and third terms on the RHS of \eqref{for:lem:finite:higher-moment:2} can be separated. We handle the three sums in \eqref{for:lem:finite:higher-moment:2} separately.
	\paragraph{Upper bounding the first sum in \eqref{for:lem:finite:higher-moment:2}} 
	First we use a similar inequality, splitting between $j=i$ and $j\neq i$ to get 
	\eqan{\label{for:lem:finite:higher-moment:02}
		&\E\Big[ \sum\limits_{t\in\{\old,\young\}}\sum\limits_{i=1}^{n_t} h(A_{ti})\Big( \sum\limits_{j=1}^{n_t} h(A_{tj}) \Big)^\eta  \Big]\nn\\
		\leq& \E\Big[ \sum\limits_{t\in\{\old,\young\}} \sum\limits_{i=1}^{n_t} h(A_{ti})^{1+\eta}\Big] + \sum\limits_{t\in\{\old,\young\}}\E\Big[ \sum\limits_{i=1}^{n_t} h(A_{ti})\Big( \sum\limits_{\substack{j=1\\j\neq i}}^{n_t} h(A_{tj}) \Big)^\eta   \Big]~.}
	Note that $(A_{\old i})_{i\in[n_\old]}$ are independent by the construction of the P\'olya point tree. Therefore, for $k=\Old$, the second expectation in \eqref{for:lem:finite:higher-moment:02} can be simplified as
	\eqan{\label{for:lem:finite:higher-moment:03}
		\E\Big[ \sum\limits_{i=1}^{n_\old} h(A_{\old i})\Big( \sum\limits_{\substack{j=1\\j\neq i}}^{n_\old} h(A_{\old j}) \Big)^\eta   \Big] =& \sum\limits_{i=1}^{n_\old}\E\Big[  h(A_{\old i})\Big]\E \Big[\Big( \sum\limits_{\substack{j=1\\j\neq i}}^{n_\old} h(A_{\old j}) \Big)^\eta   \Big]  \nn\\
		\leq&\E\Big[ \sum\limits_{i=1}^{n_\old} h(A_{\old i})\Big]\E \Big[\Big( \sum\limits_{j=1}^{n_\old}h(A_{\old j}) \Big)^\eta   \Big] ~.
	}
	Since $0 < \eta < 1$, we apply Jensen's inequality for concave functions to obtain
	\begin{equation}\label{for:lem:finite:higher-moment:04}
		\E\left[ \sum\limits_{i=1}^{n_\old} h(A_{\old i})\left( \sum\limits_{\substack{j=1 \\ j \neq i}}^{n_\old} h(A_{\old j}) \right)^\eta \right] 
		\leq \E\left[ \sum\limits_{i=1}^{n_\old} h(A_{\old i})\right]^{1+\eta}~.
	\end{equation}
	
	All of the $(A_{\young i})_{i\in[n_\young]}$ have the label $\Young$, and the ages are the occurrence times given by a mixed-Poisson process on $[x,bx]$ with random intensity $\rho_{(x,s)}(y)$ as defined in \eqref{for:pointgraph:poisson:PPT}. Therefore, $n_\young$ is a mixed-Poisson random variable with intensity $\Gamma_{(x,s)}(b^{1-\chi}-1)$. Conditionally on $n_\young$, the ages of $(A_{\young i})_{i\in[n_\young]}$ are i.i.d.\ random variables independent of $n_\young$. Let $A_\young$ denote the distribution of $A_{\young i}$, conditioned on $n_\young$.
	Then, for $t=\Young$, the second expectation in \eqref{for:lem:finite:higher-moment:02} can be simplified as
	\eqan{\label{for:lem:finite:higher-moment:05}
		&\E\Big[ \sum\limits_{i=1}^{n_\young} h(A_{\young i})\Big( \sum\limits_{\substack{j=1\\j\neq i}}^{n_\young} h(A_{\young j}) \Big)^\eta   \Big]
		=\E\Big[ \sum\limits_{i=1}^{n_\young}\E_{n_\young}\Big[h(A_{\young i})\Big( \sum\limits_{\substack{j=1\\j\neq i}}^{n_\young} h(A_{\young j}) \Big)^\eta \Big]\Big]\nn\\
		&\hspace{4cm}=  \E\Big[ n_\young \E_{n_\young}[h(A_{\young 1})]\E_{n_\young}\Big[ \Big( \sum\limits_{j=1}^{n_\young -1} h(A_{\young j}) \Big)^\eta\Big]\Big]~,}
	where $\E_{n_\young}[\cdot]=\E[\cdot\mid{n_\young}]$. Again by Jensen's inequality for the concave function $x\mapsto x^\eta$, we bound \eqref{for:lem:finite:higher-moment:05} by
	\eqan{\label{for:lem:finite:higher-moment:06}
		&\E\Big[ n_\young \E_{n_\young}[h(A_{\young1})]\E_{n_\young}\Big[ \Big( \sum\limits_{j=1}^{n_\young-1} h(A_{\young j}) \Big)^\eta\Big]\Big]\nn\\
		&\hspace{1.5cm}\leq \E\Big[ n_\young \E[h(A_{\young})]\Big( \sum\limits_{j=1}^{n_\young-1} \E_{n_\young}[h(A_{\young j})] \Big)^\eta\Big]\nn\\
		&\hspace{3.5cm}=\E\Big[ n_\young(n_\young-1)^{\eta} \Big]\E[h(A_{\young})]^{1+\eta} ~.}
	Using Holder's inequality, the first expectation in \eqref{for:lem:finite:higher-moment:06} can be upper bounded as
	\eqan{\label{for:lem:finite:higher-moment:07}
		\E\Big[ n_\young(n_\young-1)^{\eta} \Big]\leq& \E[n_\young]^{1-\eta}\E[n_\young(n_\young-1)]^{\eta}\nn\\
		=&\E[n_\young]^{1-\eta}\E[\Gamma_{(x,s)}^2]^{\eta}(b^{1-\chi}-1)^{2\eta}~.}
	The last equality in \eqref{for:lem:finite:higher-moment:07} follows from the fact that $n_\young$  is a mixed-Poisson random variable with intensity $\Gamma_{(x,s)}(b^{1-\chi}-1)$. Note that, $\Gamma_{(x,s)}$ is a Gamma random variable with parameters $m+\delta+\one_{\{s=\young\}}$ and $1$. Since $m\geq 2$ and $\delta>0$, it can be shown that $\E[\Gamma_{(x,s)}^2]\leq 2\E[\Gamma_{(x,s)}]^2$. Therefore, \eqref{for:lem:finite:higher-moment:07} can be further upper bounded by
	\eqan{\label{for:lem:finite:higher-moment:08}
		\E\Big[ n_\young(n_\young-1)^{\eta} \Big]\leq& 2^\eta\E[n_\young]^{1-\eta}\Big(\E[\Gamma_{(x,s)}](b^{1-\chi}-1)\Big)^{2\eta}\nn\\
		=&2^{\eta}\E[n_\young]^{1+\eta}(b^{1-\chi}-1)^{2\eta}~.}
	By \eqref{for:lem:finite:higher-moment:08}, the LHS of \eqref{for:lem:finite:higher-moment:05} is upper bounded by
	\eqn{\label{for:lem:finite:higher-moment:09}
		\text{LHS of \eqref{for:lem:finite:higher-moment:05}}\leq 2^\eta \Big( \E[n_\young]\E[h(A_{\young})] \Big)^{1+\eta}(b^{1-\chi}-1)^{2\eta}~.}
	Recall that, conditionally on $n_\young,$ the ages of $(A_{\young i})_{i\in[n_\young]}$ are i.i.d.\ random variables, independent of $n_\young$. Let, further, $A_\young$ denote the distribution of $A_{\young i}$ conditionally on $n_\young$. Then, by Wald's identity
	\eqn{\label{for:lem:finite:higher-moment:10}
		\E\Big[ \sum\limits_{i=1}^{n_\young} h(A_{\young i}) \Big]=\E[n_\young ]\E[h(A_{\young})] ~.
	}
	Therefore, by \eqref{for:lem:finite:higher-moment:06}, \eqref{for:lem:finite:higher-moment:09} and \eqref{for:lem:finite:higher-moment:10}, the LHS of \eqref{for:lem:finite:higher-moment:05} can be upper bounded by 
	\eqn{\label{for:lem:finite:higher-moment:11}
		\E\Big[ \sum\limits_{i=1}^{n_\young} h(A_{\young i})\Big( \sum\limits_{\substack{j=1\\j\neq i}}^{n_\young} h(A_{\young j}) \Big)^\eta   \Big] \leq 2^\eta \E\Big[ \sum\limits_{i=1}^{n_\young} h(A_{\young i}) \Big]^{1+\eta}(b^{1-\chi}-1)^{2\eta}~.}
	Now, by \eqref{for:lem:finite:higher-moment:04} and \eqref{for:lem:finite:higher-moment:11}, we upper bound the second term in RHS of \eqref{for:lem:finite:higher-moment:02} by
	\eqan{\label{for:lem:finite:higher-moment:12}
		&\sum\limits_{t\in\{\old,\young\}}^\infty\E\Big[ \sum\limits_{i=1}^{n_t} h(A_{ti})\Big( \sum\limits_{\substack{j=1\\j\neq i}}^{n_t} h(A_{tj}) \Big)^\eta   \Big] \nn\\
		&\hspace{1cm}\leq 2^\eta (b^{1-\chi}-1)^{2\eta}  \sum\limits_{t\in\{\old,\young\}}\E\Big[ \sum\limits_{i=1}^{n_t} h(A_{ti}) \Big]^{1+\eta}\\
		&\hspace{2.5cm}\leq 2^\eta(b^{1-\chi}-1)^{2\eta} \E\Big[ \sum\limits_{t\in\{\old,\young\}}\sum\limits_{i=1}^{n_t} h(A_{ti}) \Big]^{1+\eta}~.\nn}
	Therefore, by \eqref{for:lem:finite:higher-moment:12}, the LHS of \eqref{for:lem:finite:higher-moment:02} is further upper bounded as
	\eqan{\label{for:lem:finite:higher-moment:13}
		&\E\Big[ \sum\limits_{t\in\{\old,\young\}}\sum\limits_{i=1}^{n_t} h(A_{ti})\Big( \sum\limits_{j=1}^{n_t} h(A_{tj}) \Big)^\eta \Big]\\
		\leq &\E\Big[ \sum\limits_{t\in\{\old,\young\}} \sum\limits_{i=1}^{n_t} h(A_{ti})^{1+\eta}\Big] + 2^\eta (b^{1-\chi}-1)^{2\eta} \E\Big[ \sum\limits_{t\in\{\old,\young\}}\sum\limits_{i=1}^{n_t} h(A_{ti}) \Big]^{1+\eta}~.\nn
	}
	\paragraph{Upper bounding the second and third sum in \eqref{for:lem:finite:higher-moment:2}}
	Now we move on to upper bounding the second and third sum in \eqref{for:lem:finite:higher-moment:2}. By Jensen's inequality,
	\eqan{\label{for:lem:finite:higher-moment:14}
		&\E\Big[\sum\limits_{i=1}^{n_\old} h(A_{\old i})\Big]\E\Big[ \Big( \sum\limits_{j=1}^{n_\young} h(A_{\young j}) \Big)^\eta \Big]+\E\Big[\sum\limits_{i=1}^{n_\young} h(A_{\young i})\Big]\E\Big[ \Big( \sum\limits_{j=1}^{n_\old} h(A_{\old j}) \Big)^\eta \Big]\nn\\
		&\hspace{1cm}\leq \E\Big[ \sum\limits_{t\in\{\old,\young\}}\sum\limits_{i=1}^{n_t} h(A_{ti}) \Big] \E\Big[ \Big( \sum\limits_{t\in\{\old,\young\}}\sum\limits_{i=1}^{n_t} h(A_{ti}) \Big)^\eta \Big]\nn\\
		&\hspace{2cm}\leq \E\Big[ \sum\limits_{t\in\{\old,\young\}}\sum\limits_{i=1}^{n_t} h(A_{ti})  \Big]^{1+\eta} ~.}
	\paragraph{Back to \eqref{for:lem:finite:higher-moment:1}}
	By \eqref{for:lem:finite:higher-moment:1}, \eqref{for:lem:finite:higher-moment:13} and \eqref{for:lem:finite:higher-moment:14}, $\E\big[ \big(M^{(1)}_b(x,s)\big)^{1+\eta} \big]$ can be upper bounded by
	\eqan{\label{for:lem:finite:higher-moment:15}
		\E\big[ \big(M^{(1)}_b(x,s)\big)^{1+\eta} \big]
		\leq& \rho_b^{-(1+\eta)}\Big[2\E\Big[ \sum\limits_{t\in\{\old,\young\}}\sum\limits_{i=1}^{n_t} \Big(\frac{h(A_{ti})}{h(x,s)}\Big)^{1+\eta}\Big]\\
		&\hspace{1.75cm} + 2^\eta (b^{1-\chi}-1)^{2\eta} \E\Big[ \sum\limits_{k\in\{\old,\young\}}\sum\limits_{i=1}^{n_k} \frac{h(A_{ki})}{h(x,s)} \Big]^{1+\eta}  \Big]~.\nn}
	Note that $$\E\Big[ \sum\limits_{k\in\{\old,\young\}}\sum\limits_{i=1}^{n_k} \frac{h(A_{ki})}{h(x,s)} \Big]=\rho_b~.$$ Therefore, we are left with upper bounding the first expectation in the RHS of \eqref{for:lem:finite:higher-moment:15}. We calculate the expectation in RHS of \eqref{for:lem:finite:higher-moment:14} explicitly as
	\eqan{\label{for:lem:finite:higher-moment:002}
		&\rho_b^{-(1+\eta)}\E\Big[ \sum\limits_{t\in\{\old,\young\}}^\infty \sum\limits_{i=1}^{n_t} \Big(\frac{h(A_{ti})}{h(x,s)}\Big)^{1+\eta}\Big]\\
		=&\rho_b^{-(1+\eta)}\sum\limits_{t\in\{ \old,\young \}}\int\limits_{0}^\infty \pi\kappa_b((x,s),(y,t))\frac{h(y,t)^{1+\eta}}{h(x,s)^{1+\eta}}\,dy\nn\\
		=&\nu(\Old,\Young)\Big[\frac{c_{s\old}\bfu_{\old}^b}{\bfu_s^b}x^{(1+\eta)/2-\chi}\int\limits_{0}^x  y^{\chi-3/2-\eta/2} \,dy\nn\\
		&\hspace{3cm}+\frac{c_{s\young}\bfu_{\young}^b}{\bfu_s^b}x^{\chi-1/2+\eta/2}\int\limits_{x}^{bx}  y^{-\chi-(1+\eta)/2}\,dy\Big]\nn\\
		=& \nu(\Old,\Young)\Big[\frac{c_{s\old}\bfu_{\old}^b}{\bfu_s^b} \frac{1}{\chi-1/2-\eta/2}+\frac{c_{s\young}(1-b^{1/2-\chi})\bfu_{\young}^b}{\bfu_s^b}\frac{1}{\chi-1/2+\eta/2}\Big]~, \nn
	}
	when $\eta<2\chi-1$ and $$\nu(\Old,\Young)=\pi\rho_b^{-(1+\eta)}\max\{ \big(\frac{\bfu_t^b}{\bfu_s^b}\big)^{\eta}:s,t\in\{\Old,\Young\}\}~.$$ Note that the RHS of \eqref{for:lem:finite:higher-moment:002} does not depend on $(x,s)$. Hence 
	$\E\big[\big(M^{(1)}_b(x,s)\big)^{1+\eta}\big]$ is finite and its upper bound is independent of $(x,s)$.
\end{proof}
Lemma~\ref{lem:subcriticality} and Proposition~\ref{prop:supercriticality} together prove that $1/r(\bar{\bfT}_{\kappa})$ is the critical percolation threshold for the P\'olya point tree, and its survival function is continuous at $1/r(\bar{\bfT}_{\kappa})$.
\section{Spectral norm of mean offspring operator of the P\'olya point tree}\label{sec:spectral_radius}
In \cite[Theorem~1.4]{RRR22}, it was proved that a broad class of preferential attachment models, including those considered in this thesis, converges locally to the P\'olya point tree described in \cite{BergerBorgs}. The P\'olya point tree is a multi-type branching process with a mixed discrete and continuous type space. 

In this section, we investigate certain spectral properties of the mean offspring operator of the P\'olya point tree $\bfT_{\kappa}$. As described in Section~\ref{chap:percolation_threshold_PPT:sec:main-theorem}, instead of working directly with $\bfT_{\kappa}$, we work with $\bar{\bfT}_\kappa$. Ultimately, we demonstrate that the spectral and operator norms of both $\bfT_{\kappa}$ and $\bar{\bfT}_{\kappa}$ are identical. The primary difference between the two is that we identify the explicit eigenfunction for $\bar{\bfT}_{\kappa}$ corresponding to $r(\bar{\bfT}_{\kappa})$, whereas no such eigenfunction is explicitly determined for $\bfT_{\kappa}$. A noteworthy property of these integral operators is that, despite being non-self-adjoint, their spectral and operator norms are equal.

First, we show that the operator norm of the integral operator is bounded above for $\delta > 0$. We then establish the existence of a positive eigenfunction for the integral operator corresponding to this upper bound. By applying Gelfand's formula, we further demonstrate that the spectral norm of a bounded integral operator is bounded above by its operator norm. This approach allows us to explicitly determine both the spectral norm and its associated eigenfunction.

In summary, we provide a comprehensive analysis of $\bar{\bfT}_\kappa$, highlighting its boundedness, the existence of eigenvalues, and the determination of the spectral norm for $\delta > 0$.
\begin{lemma}[Upper bound of the operator norm]\label{lem:upperbound}
	For $\delta>0,$ the operator norm of $\bar{\bfT}_{\kappa}$ and $\bfT_{\kappa}$ is upper bounded by 
	RHS of \eqref{nu-equality-PAM}.
\end{lemma}
\begin{proof}
	We use the Schur test \cite[Theorem (Schur's test, specific version)]{HV21} to upper bound the operator norm. Let $f((x,s))=\frac{\bfp_s}{\sqrt{x}}$ and $g((x,s))=\frac{\bfq_s}{\sqrt{x}},$ where $\bfp$ and $\bfq$ are to be chosen appropriately later on.
	Note that by \eqref{eq:kernel:offspring-operator:PPT},
	\eqan{\label{eq:schur:1}
		&\sum\limits_{t\in\{ {\old,\young} \}}\int\limits_0^\infty \kappa((x,s),(y,t))f((y,t))\,dy\nn\\
		&\hspace{1cm}= c_{s\old}\bfp_{\old}\int\limits_0^x \frac{1}{x^{\chi}y^{1-\chi}}\frac{\,dy}{\sqrt{y}}+ c_{s\young}\bfp_{\young}\int\limits_x^\infty \frac{1}{x^{1-\chi}y^{\chi}}\frac{\,dy}{\sqrt{y}}\nn\\
		&\hspace{5cm}=\frac{2}{2\chi-1}[c_{s\old}\bfp_{\old}+c_{s\young}\bfp_{\young}]\frac{1}{\sqrt{x}}~.}
	In the last equality, we use the fact that $\chi>1/2$ to compute the integral value. Similarly,
	\eqan{\label{eq:schur:2}
		&\sum\limits_{s\in\{ {\old,\young} \}}\int\limits_0^\infty \kappa((x,s),(y,t))g((x,s))\,dx\nn\\
		&\hspace{1cm}=\sum\limits_{s\in\{ {\old,\young} \}} \Big[ c_{s\old}q_s\one_{\{t=\old\}}\int\limits_{y}^\infty \frac{1}{x^\chi y^{1-\chi}}\frac{\,dx}{\sqrt{x}}+c_{s\young}q_s\one_{\{t=\young\}}\int\limits_{0}^y \frac{1}{x^{1-\chi} y^{\chi}}\frac{\,dx}{\sqrt{x}} \Big]\nn\\
		&\hspace{5cm}= \frac{2}{2\chi-1}[\bfq_{\old} c_{\old t}+\bfq_{\young} c_{\young t}]\frac{1}{\sqrt{y}}~.}
	Let us define the matrix $\bfM=(M_{s,t})_{s,t\in \{\old,\young\}}$ by
	\eqn{
		\label{def-bfM-PAM}
		\bfM
		=\left(\begin{matrix}
			c_{\old\old} &c_{\old\young}\\
			c_{\young\old} &c_{\young\young}
		\end{matrix}\right)
		.\nn}
	Therefore, \eqref{eq:schur:1} and \eqref{eq:schur:2} simplifies as
	\eqan{\label{eq:schur:3}
		\sum\limits_{t\in\{ {\old,\young} \}}\int\limits_0^\infty \kappa((x,s),(y,t))f((y,t))\,dy=& \frac{2}{2\chi-1} (\bfM\bfp)_s\frac{1}{\sqrt{x}},\\
		\label{eq:schur:4}
		\text{and}\qquad \sum\limits_{s\in\{ {\old,\young} \}}\int\limits_0^\infty \kappa((x,s),(y,t))g((x,s))\,dx=&\frac{2}{2\chi-1} (\bfM^*\bfq)_t\frac{1}{\sqrt{y}}~,} 
	where $\bfM^*$ is the transpose of $\bfM$. Thus, for the Schur-test in \cite[Theorem~(Schur’s test, specific version)]{HV21} with $p=p'=2$ to apply, we wish that 
	\eqn{ \label{eq:upperbound:requirement}
		\bfM \bfp=\lambda_{\sss\bfM}\bfq,
		\qquad{\text{and}}
		\qquad
		\bfM^* \bfq=\lambda_{\sss\bfM}\bfp.
	}
	Since $c_{\old\old}=c_{\young\young}$, the largest eigenvalue of $\bfM$ is given by
	\eqan{	\lambda_{\sss\bfM}\equiv c_{\old\old}+\sqrt{c_{\old\young}c_{\young\old}}=\frac{m(m+\delta)+\sqrt{m(m-1)(m+\delta)(m+1+\delta)}}{m+\delta} .\nonumber
	}
	By choosing $\bfp$ and $\bfq$ as the right eigenvectors of $\bfM^*\bfM$ and $\bfM\bfM^*$, respectively, corresponding to the eigenvalue $\lambda_{\sss\bfM}^2$, we conclude that \eqref{eq:upperbound:requirement} holds. Consequently, Lemma~\ref{lem:upperbound} follows immediately. The same calculation, adapted to the setting of $\bfT_{\kappa}$, demonstrates that the operator norm of $\bfT_{\kappa}$ is also bounded above by the same value.
\end{proof}

\begin{Remark}[Operator norm of ${\bfT}_{\kappa}$]\label{remark:operator-norm:restricted-space}
	\rm In the restricted type space $\Scal$, it can also be shown that ${\bfT}_{\kappa}$ is a bounded operator from $L^2(\Scal,\lambda)$ to itself. The upper bound on the operator norm has already been obtained in Lemma~\ref{lem:upperbound}. To establish a lower bound, we construct a sequence of $L^2(\Scal,\lambda)$-normalised functions $f_\vep(x,s)=\bfp_s/\sqrt{x\log(1/\vep)}\one_{\{x\geq \vep\}}$ and then use the inequality $\| {\bfT}_{\kappa} \|\geq \langle {\bfT}_{\kappa} f_\vep,{\bfT}_{\kappa} f_\vep \rangle$. By allowing $\vep$ to approach zero, we observe the convergence of $\langle {\bfT}_{\kappa} f_\vep,{\bfT}_{\kappa} f_\vep \rangle$ towards the square of the right-hand side of \eqref{nu-equality-PAM}, proving it to be the $L^2(\Scal,\lambda)$-operator norm of $\bar{\bfT}_{\kappa}$.\hfill$\blacksquare$
\end{Remark}

\begin{proof}[Proof of Theorem~\ref{thm:operator-norm:PPT}]
	We show that there exists an eigenfunction of $\bar{\bfT}_\kappa$ with eigenvalue $r(\bar{\bfT}_\kappa)$. If we choose $\bfp$ to be the right eigenvector of $\bfM$ with eigenvalue $\lambda_{\sss \bfM}$, then \eqref{eq:schur:3} simplifies to
	\eqan{\label{eq:eigenfunction:1}
		&\sum\limits_{t\in\{ {\old,\young} \}}\int\limits_0^\infty \kappa((x,s),(y,t))f((y,t))\,dy\nn\\
		&\hspace{2cm}= \frac{2}{2\chi-1} \lambda_{\sss\bfM}\bfp_s\frac{1}{\sqrt{x}}=r(\bar{\bfT}_\kappa)f((x,s))~.}
	The operator $\bar{\bfT}_\kappa$ possesses a positive eigenfunction with eigenvalue $r(\bar{\bfT}_\kappa)$. Since $\bar{\bfT}_\kappa$ is a bounded linear operator, we can apply Gelfand's theorem to assert that the operator norm of $\bar{\bfT}_\kappa$ is at least its spectral radius. Given that $r(\bar{\bfT}_{\kappa})$ is an eigenvalue of $\bar{\bfT}_{\kappa}$, both the spectral radius and the operator norm of $\bar{\bfT}_{\kappa}$ are bounded below by $r(\bar{\bfT}_{\kappa})$. Moreover, Lemma~\ref{lem:upperbound} establishes that the operator norm is also bounded above by this same value. Consequently, we confirm that $r(\bar{\bfT}_\kappa)$ indeed represents both the operator norm and the spectral radius of $\bar{\bfT}_\kappa$ in the extended type space $\Scal_e=[0,\infty)\times\{\Old,\Young\}$.
	
	Although we cannot explicitly determine the eigenfunction corresponding to $r(\bar{\bfT}_{\kappa})$, we demonstrate that the spectral radius of ${\bfT}_{\kappa}$ is also $r(\bar{\bfT}_{\kappa})$ by using an argument based on the numerical radius of ${\bfT}_{\kappa}$, denoted by $w({\bfT}_{\kappa})$. The numerical radius of any bounded operator is upper bounded by its operator norm. With the sequence $f_\vep$ defined in Remark~\ref{remark:operator-norm:restricted-space}, and taking $\vep$ to $0$, we obtain $r(\bar{\bfT}_{\kappa})$ as a lower bound for $w({\bfT}_{\kappa})$, thereby proving that $w({\bfT}_{\kappa})=\|{\bfT}_{\kappa}\|$. Since ${\bfT}_{\kappa}$ is a bounded operator on the Hilbert space $L^2(\Scal,\lambda)$, this equality implies that the spectral radius and the numerical radius of ${\bfT}_{\kappa}$ are equal.
\end{proof}

\begin{Remark}[Equality of operator norm and spectral norm]
	\rm Although $\bar{\bfT}_{\kappa}$ and $\bfT_{\kappa}$ are not self-adjoint operators, their spectral norm and operator norm are equal. \hfill $\blacksquare$
\end{Remark}


Following similar steps, we now prove Lemma~\ref{lem:spectral-norm:truncated-operator}.

\begin{proof}[Proof of Lemma~\ref{lem:spectral-norm:truncated-operator}]
	We proceed similarly to the proof of Theorem~\ref{thm:operator-norm:PPT}. It can easily be shown that, by the Schur test, we can upper bound the operator norm by $r(\bar{\bfT}_{\kappa_b})$. Now we show that $h(x,s)$ in \eqref{eq:lemma:spectral-radius:efunc} is the eigenfunction of $\bar{\bfT}_{\kappa_b}$ with eigenvalue $r(\bar{\bfT}_{\kappa_b})$. Therefore, for any $(x,s) \in \Scal_e$,
	\eqan{\label{for:lem:spectra:truncated-operator:1}
		&\sum\limits_{t\in\{\old,\young\}} \int\limits_0^\infty \kappa_b((x,s),(y,t))h(y,s)\,dy\nn\\
		&\hspace{2cm}=\sum\limits_{t\in\{\old,\young\}} \int\limits_0^{bx} \kappa((x,s),(y,t))h(y,s)\,dy\nn\\
		&\hspace{3cm}=\frac{x^{-1/2}}{\chi-\frac{1}{2}}\big[ c_{so}\bfu_{\old}^b +c_{s\young}(1-b^{1/2-\chi})\bfu_{\young}^b\big]~.
	}
	Since $\bfu^b$ is the eigenvector of $\bfM_b$ corresponding to its largest eigenvalue $\lambda_{\sss \bfM}^{\sss (b)}$, given by 
	\eqn{\label{for:lem:spectra:truncated-operator:2}
		\lambda_{\sss \bfM}^{\sss (b)} = \frac{(1+q)c_{\old\old}+\sqrt{((1-q)c_{\old\old})^2+4qc_{\old\young}c_{\young\old}}}{2}~,}
	where $q=1-b^{1/2-\chi}$, \eqref{for:lem:spectra:truncated-operator:1} leads to
	\eqn{\label{for:lem:spectra:truncated-operator:3}
		\Big(\bar{\bfT}_{\kappa_b}h\Big)(x,s)=\frac{\lambda_{\sss \bfM}^{\sss (b)}}{2\chi-1}\frac{\bfu_s^b}{\sqrt{x}}=r(\bar{\bfT}_{\kappa_b})h(x,s)~.}
	Hence, $h$ is the eigenfunction of the truncated integral operator $\bar{\bfT}_{\kappa_b}$ corresponding to its spectral norm $r(\bar{\bfT}_{\kappa_b})$ as in \eqref{eq:lemma:spectral-radius}.
\end{proof}


\chapter[Giant in Preferential Attachment model is local]{Giant in Preferential Attachment\\model is local}
\label{chap:local-giant}
\begin{flushright}
	\footnotesize{}Based on:\\
	\textbf{Percolation on preferential attachment models}\cite{RRR23}
\end{flushright}
\vspace{0.1cm}
\begin{center}
	\begin{minipage}{0.7 \textwidth}
		\footnotesize{\textbf{Abstract.}
			In this chapter, we study the percolation phase transition on preferential attachment models, where vertices enter with $m$ edges and attach proportionally to their degree plus $\delta$. We identify the critical percolation threshold as
			\[
			\pi_c = \frac{\delta}{2\left(m(m+\delta) + \sqrt{m(m-1)(m+\delta)(m+1+\delta)}\right)}
			\]
			for $\delta > 0$, and $\pi_c = 0$ for non-positive values of $\delta$. Therefore, the giant component is robust for $\delta \in (-m, 0]$, but not for $\delta > 0$.
			
			First, we show that preferential attachment graphs are large-set expanders, enabling us to verify the conditions outlined in \cite{ABS22}. Under these conditions, the proportion of vertices in the largest connected component of a graph sequence converges to the survival probability of percolation on the local limit. In particular, the critical percolation threshold for both the graph and its local limit are identical.
			
			In Chapter~\ref{chap:percolation_threshold_PPT}, we identified $\pi_c$ as the critical percolation threshold of the P\'olya point tree, which is the local limit of preferential attachment models, proving that $\pi_c$ is indeed the critical percolation threshold for a sequence of preferential attachment models.
			}
	\end{minipage}
\end{center}
\vspace{0.1cm}

\section{Introduction}\label{chap:local-giant:sec:intro}
Random graph models typically undergo a phase transition in their connectivity structure depending on the percolation probability $\pi$. Specifically, there exists a critical percolation threshold $\pi_c \in [0,1]$ such that for $\pi > \pi_c$, the proportion of vertices in the largest connected component is strictly positive, whereas for $\pi < \pi_c$, this proportion converges to zero. In other words,
\begin{equation}\label{def:critical-percolation}
	\frac{|\mathcal{C}_1^{(n)}(\pi)|}{n} \to c(\pi)~,
\end{equation}
where $c(\pi) > 0$ for $\pi > \pi_c$ and $c(\pi) = 0$ when $\pi < \pi_c$. Here, $\mathcal{C}_1^{(n)}(\pi)$ denotes the largest connected component of the percolated graph of size $n$. The value $\pi_c$ is known as the \emph{critical percolation threshold} of the random graph.

Percolation on random graphs has a extensive and diverse literature, to which we refer the reader for more detailed information \cite{BJR07,ER60,vdH1,JL09,JPRR18,MR95}. Local convergence techniques have become valuable tools for studying the size and uniqueness of the giant component in these graphs. However, it is important to note that these properties do not immediately follow from local convergence alone. They hold true under certain well-connectedness conditions of the graph.

Although percolation on static random graphs has been extensively studied, percolation on preferential attachment models is relatively less explored. In \cite{EMO21}, the authors studied the proportion of vertices in the largest connected component of the Bernoulli preferential attachment model defined in \cite{DM13}, where the attachment rule is a function of the number of incoming edges of a vertex. In \cite{ABS22}, the preferential attachment model for the independent edge attachment rule and $\delta = 0$ was studied. It was shown that the model is a large-set expander with bounded average degree, and its critical percolation threshold is $0$. Additionally, it was demonstrated that the phase transition is of infinite order in that case, as proved in \cite{EMO21} for the model in \cite{DM13}. The fact that $\delta = 0$ plays a crucial role in this proof, which mainly relies on functional analysis.

In this chapter, we prove that the giant component of several other preferential attachment models with affine attachment functions and a general $\delta$ parameter is local, and the critical percolation thresholds for these models are all equal and match that of their local limit, the P\'{o}lya point tree.

Throughout this part, we shall only consider the deterministic version of the preferential attachment models. For convenience, we describe the models here again. We focus on sequential preferential attachment models, specifically models (a), (b), and (d) as defined in \cite{vdH1}, excluding their tree cases. We fix $m \in \mathbb{N} \setminus \{1\}$ and $\delta > -m$. In these models, every new vertex is introduced to the existing graph with $m$ edges incident to it. The models are defined by their edge-connection probabilities.

We start from any finite graph with $2$ vertices and finitely many connections between them. Let $a_1$ and $a_2$ denote the degrees of vertices $1$ and $2$, respectively. We also allow for self-loops in the initial graph.
\paragraph{Model (a)} For each new vertex $v$ joining the graph and $j=1,2,\ldots,m$, the attachment probabilities are given by
\eqn{\label{def:model:a}
	\prob\Big( v\overset{j}{\rightsquigarrow} u\mid\PA^{(a)}_{v,j-1}(m,\delta) \Big)= \begin{cases}
		\frac{d_u(v,j-1)+\delta}{c_{v,j}^{(a)}}\hspace{1.15cm}\text{for}~u<v,\\
		\frac{d_v(v,j-1)+{j\delta}/{m}}{c_{v,j}^{(a)}}\hspace{0.65cm}\text{for}~u=v,
	\end{cases}
}
where $v\overset{j}{\rightsquigarrow} u$ denotes that vertex $v$ connects to $u$ with its $j$-th edge, $\PA^{(a)}_{v,j}({m},\delta)$ denotes the graph on $v$ vertices, \FC{with} the $v$-th vertex \FC{having} $j$ out-edges, and $d_u(v,j)$ denotes the degree of vertex $u$ in $\PA^{(a)}_{v,j}({m},\delta)$. We identify $\PA^{(a)}_{v+1,0}({m},\delta)$ with $\PA^{(a)}_{v,m}({m},\delta)$. The normalizing constant $c_{v,j}^{(a)}$ in \eqref{def:model:a} equals
\eqn{\label{eq:normali:1}
	c_{v,j}^{(a)} := a_{[2]}+2\delta+(2m+\delta)(v-3)+2(j-1)+1+\frac{j\delta}{m} \, ,
}
where $a_{[2]}=a_1+a_2$. We denote the above model by $\PA^{(a)}_{v}({m},\delta)$, which is equivalent to $\PA_v^{(m,\delta)}(a)$ as defined in \cite{vdH1}.
\paragraph{Model (b)}
For every new vertex $v$ joining the graph and $j=1,2,\ldots,m$, the attachment probabilities are given by
\eqn{\label{def:model:b}
	\prob\Big( v\overset{j}{\rightsquigarrow} u\mid\PA^{(b)}_{v,j-1}(m,\delta) \Big)= \begin{cases}
		\frac{d_u(v,j-1)+\delta}{c_{v,j}^{(b)}}\hspace{1.8cm}\text{for}~u<v,\\
		\frac{d_v(v,j-1)+{(j-1)\delta}/{m}}{c_{v,j}^{(b)}}\hspace{0.65cm}\text{for}~u=v,
	\end{cases}
}
with all notation as before, but now for model (b), for which
\eqn{\label{eq:normalb}
	c_{v,j}^{(b)} := a_{[2]}+2\delta+(2m+\delta)(v-3)+2(j-1)+\frac{(j-1)\delta}{m} \, .
}
We denote the above model as $\PA^{(b)}_{v}({m},\delta)$, which is equivalent to $\PA_v^{(m,\delta)}(b)$ as defined in \cite{vdH1}.

\begin{remark}[Difference between models (a) and (b)]
	\rm{Models (a) and (b) are different in that the first edge from every new vertex can create a self-loop in model (a) but not in model (b). Note that the edge probabilities are different for all $j$.}\hfill$\blacksquare$
\end{remark}

\paragraph{Model (d)}
For every new vertex $v$ joining the graph and $j=1,2,\ldots,m$, the attachment probabilities are given by
\eqn{\label{def:model:d}
	\prob\Big( v\overset{j}{\rightsquigarrow} u\mid\PA^{(d)}_{v,j-1}(m,\delta) \Big)= \frac{d_u(v,j-1)+\delta}{c_{v,j}^{(d)}}\hspace{1.5cm}\text{for}~u<v~,
}
with all notation as before, but now for model (d), for which
\eqn{\label{eq:normald}
	c_{v,j}^{(d)} := a_{[2]}+2\delta+(2m+\delta)(v-3)+(j-1)\, .
}
We denote the above model as $\PA^{(d)}_{v}({m},\delta)$, which is essentially equivalent to $\PA_v^{(m,\delta)}(d)$ as defined in \cite{vdH1}.

\begin{center}
	{\textbf{Organisation of the chapter}}
\end{center}
This chapter is organised as follows. {In Section~\ref{chap:local-giant:sec:LSE}, we describe a sufficient condition for the large-set expander property as proved in \cite{ABS22}}. In Section~\ref{chap:local-gian:sec:PA:LSE}, we demonstrate that a sequence of preferential attachment models satisfies the large-set expander property with bounded average degree. Lastly, in Section~\ref{chap:local-giant:sec:pi_c:PA}, we identify the critical percolation threshold for a sequence of preferential attachment models.

\section{Percolation on large-set expanders}\label{chap:local-giant:sec:LSE}
We define several notations that will be used throughout the chapter. Let $G=(V(G),E(G))$ be an undirected multigraph with self-loops and multiple edges. The degree of a vertex $u \in V(G)$ is denoted by $d_G(u)$. Each self-loop contributes 2 to the degree of the vertex. For any $S \subset V(G)$, the volume of the subset $S$ is defined as
\begin{equation}
	\label{def:vol:subset}
	\vol_G(S) = \sum_{v \in S} d_G(v),
\end{equation}
and the cutset of $S$ is defined as
\begin{equation}
	\label{def:cutset:subset}
	\cut_G(S, S^c) = \{ e \in E(G) \mid \text{$e$ has one end in $S$ and the other end in $S^c$} \}.
\end{equation}
For every $\varepsilon > 0$, we define the edge-expander constant of the graph $G$ as
\begin{equation}
	\label{def:expander-constant:graph}
	\alpha(G, \varepsilon) = \min_{\substack{S \subset V(G) \\ \varepsilon |V(G)| \leq |S| \leq |V(G)|/2}} \frac{|\cut_G(S, S^c)|}{|S|}.
\end{equation}
We call a graph $G$ an $(\alpha, \varepsilon, d)$-large-set expander if the graph has an average degree bounded above by $d$ and $\alpha(G, \varepsilon) \geq \alpha$. A sequence of (possibly random) graphs $(G_n)_{n \geq 1}$ is called a large-set expander sequence with bounded average degree if there exists a $d < \infty$, and for every $\varepsilon \in (0,1/2)$, there exists $\alpha > 0$ such that
\begin{equation}
	\label{def:large-set-expander-sequence}
	\prob(G_n \text{ is an } (\alpha, \varepsilon, d)\text{-large-set expander}) \to 1 \text{ as } n \to \infty.
\end{equation}

Alimohammadi, Borgs, and Saberi \cite{ABS22} studied percolation on such sequences of large-set expanders with bounded average degree. Let $(G_n)_{n \geq 1}$ be a (possibly random) sequence of large-set expanders with bounded average degree, and let $\mu$ be a probability distribution on $\mathcal{G}_\star$. For any graph $G \in \mathcal{G}_\star$, we use $G(\pi)$ to denote the subgraph of $G$ obtained after performing bond percolation on $G$ with percolation probability $\pi$. Suppose $(G_n)_{n \geq 1}$ converges locally to $(G, o) \in \mathcal{G}_\star$ with law $\mu$. We define $\Ccal_i^{\sss(n)}(\pi)$ as the $i$-th largest component of the graph $G_n(\pi)$ (breaking ties arbitrarily) and $\mathscr{C}(\pi)$ as the cluster containing the vertex $o$ in $G(\pi)$. Let $\zeta(\pi) := \mu\big( |\mathscr{C}(\pi)| = \infty \big)$. According to \cite[Theorem~1.1]{ABS22}, if $\pi \mapsto \zeta(\pi)$ is continuous at all $\pi \in [0,1]$, then
\begin{equation}
	\label{expander:percolation-result:ABS22}
	\frac{|\Ccal_1^{\sss(n)}(\pi)|}{n} \overset{\prob}{\to} \zeta(\pi),
\end{equation}
where $\overset{\prob}{\to}$ denotes convergence in probability with respect to both the random graph and percolation. Furthermore, for all $\pi \in [0,1]$, $\frac{|\Ccal_2^{\sss(n)}(\pi)|}{n} \overset{\prob}{\to} 0$, and this convergence is uniform in $\pi$ if $\pi$ is bounded away from $0$ and $1$.


\section{Preferential attachment models are large-set expanders}\label{chap:local-gian:sec:PA:LSE}
Fix $m \in \mathbb{N} \setminus \{1\}$ and $\delta > -m$. Consider an initial graph of size $2$, where one of the initial vertices has a degree of at most $m$. Without loss of generality, assume $a_2 \leq m$. We then construct a preferential attachment model with parameters $m$ and $\delta$, denoted by $G_{n}$ as $\PA^{(s)}_n(m, \delta)$ with $s \in \{a, b, d\}$. 

In this section, we establish the large-set expander property for models (a), (b), and (d) with bounded average degree, following the strategy of Mihail, Papadimitriou, and Saberi \cite{MPS06}:

\begin{Proposition}\label{prop:Cheeger-value:model-b}
	Let $(G_n)_{n \geq 2}$ be a sequence of preferential attachment models (a), (b), and (d) with $m \geq 2$. Then, for any constant $\epsilon > 0$, there exists $\alpha > 0$ such that $(G_{n})_{n \geq 1}$ is an $(\alpha, \epsilon, m)$-large-set expander. In particular, there exists a constant $c = c(\infty) > 0$ such that
	\begin{equation}\label{eq:prop:Cheeger-value:model-b}
		\Pr(\alpha(G_{n}, \epsilon) < \alpha) = o(e^{-cn})~.
	\end{equation}
\end{Proposition}

To prove this proposition, we adapt the Cheeger value calculation done for model (b) and $\delta = 0$ from \cite{MPS06} to models (a), (b), and (d) with all admissible $\delta$. 

Let us denote $\{1, 2, \ldots, n\}$ by $[n]$ for any $n \in \mathbb{N}$.

According to \cite[Chapter~5]{vdH2}, the graph $G_{n}$ can be constructed by \emph{collapsing} a preferential attachment tree $T_{m(n-2) + a_{[2]}}$, consisting of $m(n-2) + a_{[2]}$ vertices. We denote this as
\begin{equation}\label{eq:def:total-degree}
	m_{[n]} = m(n-2) + a_{[2]}~.
\end{equation}
These vertices of $T_{m_{[n]}}$ are referred to as \emph{mini-vertices}. A mini-vertex $u$ is associated with a vertex $v \in [n]$ if $u$ belongs to the set of mini-vertices that collapse into vertex $v$.

In summary, the construction of $G_{n}$ involves collapsing a preferential attachment tree, leading to mini-vertices being merged into vertices. The initial graph is then obtained by arranging the mini-vertices with self-loops before those creating connections between specific vertices.

In the tree $T_{m_{[n]}}$, a mini-vertex $u$ connects to mini-vertex $n(u)$. We will use this pre-collapsed version of the graph $G_{n}$ to prove Proposition~\ref{prop:Cheeger-value:model-b}. Similar to \cite{MPS06}, we define a set $S \subset [n]$ as \textit{BAD} if
\begin{equation}\label{def:bad-set}
	|\cut_{\sss G_{n}}(S, S^c)| < \alpha |S|~.
\end{equation}
For convenience, we also recall other notations used in \cite{MPS06}:
\begin{Definition}[\textit{FIT} and \textit{ILL} mini-vertices]
	\label{def:fit-ill-minivertex}
	Fix $S \subset [n]$. We call a mini-vertex $t \in [m_{[n]}]$ \textit{FIT} if
	\begin{itemize}
		\item[$\rhd$] $t$ is associated with a vertex in $S$ and $n(t)$ is associated with a vertex in $S^c$; or
		\item[$\rhd$] $t$ is associated with a vertex in $S^c$ and $n(t)$ is associated with a vertex in $S$.
	\end{itemize}
	A mini-vertex is \textit{ILL} if it is not \textit{FIT}. 
\end{Definition}
Note that \textit{FIT} mini-vertices create connections between $S$ and $S^c$, whereas \textit{ILL} mini-vertices are responsible for creating edges between two vertices within the same set, either $S$ or $S^c$. Therefore, bounding the number of \textit{FIT} mini-vertices will provide us with a bound on $|\cut_{\sss G_{n}}(S, S^c)|$. We prove the following intermediate lemma, which provides an upper bound on the probability of observing too few \textit{FIT} mini-vertices and will be useful for proving Proposition~\ref{prop:Cheeger-value:model-b}:
\begin{Lemma}[Upper bound on the probability for mini-vertices to be \textit{ILL} in models (a), (b), and (d)]
	\label{lem:intermediate:Cheeger-value}
	For preferential attachment models (a), (b), and (d), and for a fixed subset $S \subset [n]$ of cardinality $k$, as well as for a fixed subset $A \subset [m_{[n]}]$ such that $|A| \leq \alpha k$, the probability that all mini-vertices in $[m_{[n]}] \setminus A$ are \textit{ILL} and all mini-vertices in $A$ are \textit{FIT} with respect to $S$ is at most $(mn + \delta_0)^{\delta_0} \binom{mk}{\alpha k} / \binom{m_{[n]} - \alpha k}{mk - \alpha k}$ for some fixed $\delta_0$.
\end{Lemma}

\arxiversion{Since the proof of Lemma~\ref{lem:intermediate:Cheeger-value} is a minor extension of the proof of \cite[Lemma~2]{MPS06}, applying to $\delta=0$ and model (b), to models (a) (b) and (d) with general $\delta>-m$, we }
prove Proposition~\ref{prop:Cheeger-value:model-b} subject to Lemma~\ref{lem:intermediate:Cheeger-value} and then prove Lemma~\ref{lem:intermediate:Cheeger-value}. This proof also follows similar line of proof as in \cite[Proof of Theorem 1]{MPS06}. This step requires us to choose $m\geq 2$.
\begin{proof}[Proof of Proposition~\ref{prop:Cheeger-value:model-b}]
	The proof of Proposition~\ref{prop:Cheeger-value:model-b}	 follows the proof of \cite[Theorem~1]{MPS06} almost verbatim. Let $S\subset [n]$ be a set of vertices of size $k$ and $A\subset [m_{[n]}]$ be the set \rm{FIT} mini-vertices such that $|A|\leq \alpha k$ for some $\alpha\in(0,1)$. By Lemma~\ref{lem:intermediate:Cheeger-value},
	\eqan{\label{for:prop:Cheegar-value:model-b:1}
		\prob\Big( \bigcap\limits_{t\in[m_{[n]}]\setminus A}\{t ~\text{is \rm{ILL}}\} \Big)\leq& (mn+\delta_0)^{\delta_0} \binom{mk}{|A|}\binom{m_{[n]}-|A|}{mk-|A|}^{-1}\nn\\
		\leq& (mn+\delta_0)^{\delta_0} \binom{mk}{\alpha k}\binom{m_{[n]}-\alpha k}{mk-\alpha k}^{-1}~,
	}
	where the last inequality follows from the fact that $|A|\leq \alpha k$. Now we have $\binom{n}{k}$ many choices for $S$, whereas for fixed $S$ and $|A|$ being at most $\alpha k$, we have $\alpha k\binom{m_{[n]}}{\alpha k}$ many choices for $A$. Therefore, \eqref{for:prop:Cheegar-value:model-b:1} implies that
	\eqan{\label{for:prop:Cheegar-value:model-b:2}
		\prob\Big( \alpha(G_{n},\vep)<\alpha \Big)\leq& (mn+\delta_0)^{\delta_0} \sum\limits_{k=\vep n}^{n/2} \binom{n}{k}\alpha k\binom{m_{[n]}}{\alpha k}\frac{\binom{mk}{\alpha k}}{\binom{m_{[n]}-\alpha k}{mk-\alpha k}}~.
	}
	In the next few step, we bound the summands in the RHS of \eqref{for:prop:Cheegar-value:model-b:2}.
	Using that $\binom{n}{k}\binom{m_{[n]}-n-\alpha k}{mk-k-\alpha k}\leq \binom{m_{[n]}-\alpha k}{mk-\alpha k}$, we upper bound the summand in the RHS of \eqref{for:prop:Cheegar-value:model-b:2} by
	\eqn{\label{for:prop:Cheegar-value:model-b:03}
		\alpha k\binom{m_{[n]}}{\alpha k}\frac{\binom{mk}{\alpha k}}{\binom{m_{[n]}-n-\alpha k}{mk-k-\alpha k}}~.}
	Since binomial coefficients can be upper and lower bounded by
	\eqn{\label{for:prop:Cheegar-value:model-b:04}
		\Big(\frac{r}{s}\Big)^{s}\leq \binom{r}{s}\leq \Big(\frac{\e r}{s}\Big)^{s}~,}
	we can upper bound the expression in \eqref{for:prop:Cheegar-value:model-b:03} by
	\eqan{\label{for:prop:Cheegar-value:model-b:05}
		&~\alpha k\Big( \frac{\e^2 m_{[n]}m}{\alpha^2 k} \Big)^{\alpha k}\Big( \frac{(m-1)k-\alpha k}{m_{[n]}-n-\alpha k} \Big)^{(m-1-\alpha)k}\nn\\
		&\hspace{2cm}\leq~ \alpha k\Big( \frac{\e^2 m_{[n]}m}{\alpha^2 k} \Big)^{\alpha k}\Big( \frac{(m-1)k}{m_{[n]}-n} \Big)^{(m-1-\alpha)k}~.
	}
	Using the explicit expression for $m_{[n]}$ from \eqref{eq:def:total-degree}, we simplify the RHS of \eqref{for:prop:Cheegar-value:model-b:05} as
	\eqn{\label{for:prop:Cheegar-value:model-b:06}
		\alpha k \Big( \frac{\e m}{\alpha}\Big)^{2\alpha k}\Big( \frac{n+(a_{[2]}-2m)/m}{k} \Big)^{\alpha k}\Big( \frac{k}{n+(a_{[2]}-2m)/(m-1)} \Big)^{(m-1-\alpha)k}~.
	}
	Denoting $\delta_1=\min\{(a_{[2]}-2m)/m,(a_{[2]}-2m)/(m-1)\}$, we upper bound the expression in \eqref{for:prop:Cheegar-value:model-b:06} by
	\eqn{\label{for:prop:Cheegar-value:model-b:07}
		\alpha k \Big( \frac{\e m}{\alpha}\Big)^{2\alpha k}\Big( \frac{k}{n+\delta_1} \Big)^{(m-1-2\alpha)k}~.
	}
	Hence,
	\eqan{\label{for:prop:Cheegar-value:model-b:08}
		&\prob\Big( \alpha(G_{n},\vep)<\alpha \Big)\nn\\
		&\hspace{2cm}\leq (mn+\delta_0)^{\delta_0} \sum\limits_{k=\vep n}^{n/2}\alpha k \Big( \frac{\e m}{\alpha}\Big)^{2\alpha k}\Big( \frac{k}{n+\delta_1} \Big)^{(m-1-2\alpha)k}~.
	}
	The inequalities in \eqref{for:prop:Cheegar-value:model-b:03}-\eqref{for:prop:Cheegar-value:model-b:07} follow exactly from the calculations in \cite[Proof of Theorem~1]{MPS06}. There are $O(n)$ many summands in the sum. Thus, to prove the proposition, it suffices to bound the leading term by $o(\e^{-nc})$ for some $c>0$. Thus, now we study the function
	\eqn{\label{for:prop:Cheegar-value:model-b:3}
		g(x)=\alpha x \Big( \frac{\e m}{\alpha} \Big)^{2\alpha x} \Big( \frac{x}{n+\delta_1} \Big)^{(m-1-2\alpha)x}~,}
	for $x\in[\vep n,n/2]$. We argue that there exists $x_0\in[1,n/2]$ such that $g(x)$ is monotonically decreasing for $x\in[1,x_0]$ and monotonically increasing for $x\in[x_0,n/2]$. For this we first calculate the derivative of $g(x)$ as
	\eqn{\label{for:prop:Cheegar-value:model-b:4}
		g^\prime(x)=\frac{g(x)}{x}\Big[ 1+x\Big( 2\alpha \Big(1+\log\frac{m}{\alpha}\Big)+(m-1-2\alpha)\Big(\log\frac{x}{n+\delta_1}+1\Big) \Big) \Big]~.
	}
	Therefore, the sign of $g^\prime(x)$ is determined by the sign of 
	\eqan{\label{for:prop:Cheegar-value:model-b:5}
		g_1(x)&=\frac{xg^\prime(x)}{g(x)}\nn\\
		&=1+x\Big( 2\alpha \Big(1+\log\frac{m}{\alpha}\Big)+(m-1-2\alpha)\Big(\log\frac{x}{n+\delta_1}+1\Big) \Big)~.}
	Note that $g_1(1)<0$ and $g_1(n/2)>0$ for $n$ sufficiently large, while the first and second derivative of $g_1$ are
	\eqan{\label{for:prop:Cheegar-value:model-b:6}
		g_1^\prime(x)=&2\alpha \Big(1+\log\frac{m}{\alpha}\Big)+(m-1-2\alpha)\Big(\log\frac{x}{n+\delta_1}+1\Big)~,\nn\\
		g_1^{\prime\prime}(x)=&\frac{m-1-2\alpha}{x}>0~.}
	Therefore, $g_1$ is a convex function having a unique root in $[1,n/2]$. Consequently for some $x_0\in[1,n/2],~g(x)$ is decreasing in the interval $[1,x_0]$ and increasing in the interval $[x_0,n/2]$. We conclude that it is enough to show that $\max\{g(\vep n),g(n/2)\}=o(\e^{-nc}).$ Note that
	\eqan{\label{for:prop:Cheegar-value:model-b:7}
		g(\vep n)=&\alpha (\vep n)\Big[\Big( \frac{\e m}{\alpha} \Big)^{2\alpha} \vep^{(m-1-2\alpha)}\Big]^{\vep n}~,\nn\\
		\text{and}\qquad g(n/2)=&\alpha (n/2) \Big[\Big( \frac{\e m}{\alpha} \Big)^{2\alpha} \Big( \frac{1}{2} \Big)^{(m-1-2\alpha)}\Big]^{n/2}~.
	}
	For any $r\in(0,1),~g(r(n+\delta_1))$ drops exponentially first when 
	\eqn{\label{for:prop:Cheegar-value:model-b:8} 
		\Big( \frac{\e m}{\alpha} \Big)^{2\alpha}<r^{2\alpha-(m-1)}~.}
	Taking logarithms on both sides of \eqref{for:prop:Cheegar-value:model-b:8}, the condition simplifies to
	\eqn{\label{for:prop:Cheegar-value:model-b:9}
		\alpha< \frac{(m-1)\log \frac{1}{r}+2\alpha\log \alpha}{\log m-\log r +1}~.
	} 
	Since $2\alpha\log\alpha\to 0$ as $\alpha\to 0$ and for any $\gamma\leq 1/2$, we can choose $\alpha_\gamma>0$ such that 
	\[
	\alpha_\gamma\log \alpha_\gamma =\frac{m-1}{4}\log \gamma<0~. 
	\]
	By \eqref{for:prop:Cheegar-value:model-b:7}-\eqref{for:prop:Cheegar-value:model-b:9}, both $g(\vep n)$ and $g(n/2)$ are $o(\e^{-cn})$, for some $c>0$, if we choose 
	\[
	\alpha<\min\Big\{ \alpha_\vep,\alpha_{1/2},\frac{(m-1)\log \frac{1}{\vep}+2\alpha_\vep\log \alpha_\vep}{\log m-\log \vep +1},\frac{(m-1)\log 2+2\alpha_{1/2}\log \alpha_{1/2}}{\log m-\log \frac{1}{2} +1} \Big\}~.
	\]
	Therefore, the RHS of \eqref{for:prop:Cheegar-value:model-b:2} converges to $0$ as $n\to\infty$ for a suitable choice of positive $\alpha$ (depends only on $\vep$ and $m$). Hence $(G_{n})_{n\geq 2}$ is whp an $(\alpha,\vep,m)$-large-set expander. 
\end{proof}
The proof of Lemma~\ref{lem:intermediate:Cheeger-value} follows by adapting the proof of \cite[Lemma~2]{MPS06} for the case of model (a), (b) and (d) with general initial graph settings and any admissible $\delta$. 
\begin{proof}[Proof of Lemma~\ref{lem:intermediate:Cheeger-value}]
	Since models (a), (b), and (d) can all be derived using the collapsing procedure mentioned in \cite{vdH1}, we will prove the subsequent steps for all of these models. Recall the definitions of the \rm{BAD} set from \eqref{def:bad-set}, as well as \rm{ILL} and \rm{FIT} mini-vertices from Definition~\ref{def:fit-ill-minivertex}.
	
	Let $A_1 \subset A$ be the set of mini-vertices associated with vertices in $S$, and let $A_2 = A \setminus A_1$, where the mini-vertices in $A_2$ are associated with vertices in $S^c$. Denote $|A_1| = k_1$ and $|A_2| = k_2$, with $|A| = k_1 + k_2$. There are $mk - k_1$ mini-vertices associated with vertices in $S$ that are not in $A$. Similarly, there are $m_{[n]} - mk - k_2$ mini-vertices associated with vertices in $S^c$ that are not in $A$. Let $x_1 < x_2 < \cdots < x_{mk-k_1}$ be the mini-vertices associated with vertices in $S$, but not in $A$, and let $\Bar{x}_1 < \Bar{x}_2 < \cdots < \Bar{x}_{m_{[n]} - mk - k_2}$ be the mini-vertices associated with vertices in $S^c$ that are not in $A$. For $i \in [mk - k_1]$,
	\begin{equation}\label{for:lem:intermediate:Cheeger-value:1}
		x_i = y_i + z_i + 1,
	\end{equation}
	where $y_i$ is the number of mini-vertices from $A$ that appeared before $x_i$, and $z_i$ is the number of mini-vertices from $[m_{[n]}] \setminus A$ that appeared before $x_i$. Similarly, for $i \in [m_{[n]} - mk - k_2]$,
	\begin{equation}\label{for:lem:intermediate:Cheeger-value:2}
		\Bar{x}_i = \Bar{y}_i + \Bar{z}_i + 1,
	\end{equation}
	where $\Bar{y}_i$ is the number of mini-vertices from $A$ that appeared before $\Bar{x}_i$, and $\Bar{z}_i$ is the number of mini-vertices from $[m_{[n]}] \setminus A$ that appeared before $\Bar{x}_i$.
	
	Note that $z_i$ or $\Bar{z}_i$ counts the total number of $x_j$ and $\Bar{x}_j$ that appeared before $x_i$ or $\Bar{x}_i$, respectively. If we order the set $\{x_i\}_{i \in [mk - k_1]} \cup \{\Bar{x}_i\}_{i \in [m_{[n]} - mk - k_2]}$ in increasing order, then the corresponding $z_i$ or $\Bar{z}_i$ values increase by $1$ each time, ranging from $0$ to $m_{[n]} - k_1 - k_2$. Therefore,
	\begin{equation}\label{for:lem:intermediate:Cheeger-value:3}
		\bigcup\limits_{i=1}^{mk-k_1} \{z_i + 1\} \cup \bigcup\limits_{i=1}^{m_{[n]} - mk - k_2} \{\Bar{z}_i + 1\} = [m_{[n]} - k_1 - k_2] = [m_{[n]} - |A|].
	\end{equation}
	
	The total volume of the tree when the $t$-th mini-vertex joins is $2(t - 1)$. If $t = x_i$ (or $\Bar{x}_i$), then the total volume equals
	\begin{equation}\label{for:lem:intermediate:Cheeger-value:4}
		2z_i + 2y_i, \quad \text{or} \quad 2\Bar{z}_i + 2\Bar{y}_i.
	\end{equation}
	Now, we consider the following cases:
	\begin{enumerate}
		\item \textbf{Case 1:} Both vertices $1$ and $2$ are in the same set. In this scenario, we consider $S$ to be the set containing both vertices.
		\item \textbf{Case 2:} Vertices $1$ and $2$ are in different sets. In this scenario, we consider $S$ to be the set containing vertex $1$. Note that $a_2 \leq m$.
	\end{enumerate}
	
	Since we have a finite number of mini-vertices with deterministic connections (depending only on the structure of the initial graph), we can safely assume that the upper bound we are about to prove applies to every \rm{ILL} mini-vertex except for the first few. As we aim to establish an exponential bound, these few exceptions do not affect the final result. Therefore, from this point forward, we will focus on upper bounding the non-deterministic edges, assuming that the upper bound also applies to the deterministic edges.
	
	Consider the case where all mini-vertices in $A$ are FIT, while all mini-vertices in $[m_{[n]}]\setminus{A}$ are ILL. When $x_i$ arrives, the total volume of the mini-vertices associated with the vertices in $S$ consists of:
	\begin{enumerate}
		\item[(a)] All ILL mini-vertices that arrived prior to $x_i$ and are associated with a vertex in $S$. Each such mini-vertex contributes $2$ to the total volume of the mini-vertices associated with the vertices in $S$. There are $(i-1)$ such mini-vertices, contributing a total of $2(i-1)$ to the sum.
		\item[(b)] All FIT mini-vertices that arrived prior to $x_i$ and are associated with a vertex in $S$. Each such mini-vertex contributes $1$ to the total volume of the mini-vertices associated with the vertices in $S$. There are $y_i$ such mini-vertices, contributing a total of $y_i$ to the sum.
	\end{enumerate}
	Therefore, when $x_i$ arrives, the volume of the mini-vertices associated with the vertices in $S$ is $2(i-1) + y_i$. Next, we proceed to bound the probability of creating \rm{ILL} mini-vertices, given that all mini-vertices in $A$ are \rm{FIT}, for models (a), (b), and (d). We begin with model (b) and then adapt the bound for models (a) and (d) respectively. Recall that $a_{[2]} = a_1 + a_2$ represents the total degree of the initial graph.
	
	\paragraph{Model (b)}
	The probability that $x_i$ connects to mini-vertices associated with vertices in $S$, and becomes \rm{ILL} given that all mini-vertices in $A$ which arrived prior to $x_i$ are \rm{FIT}, and both $1$ and $2$ are in $S$, is 
	\eqn{\label{for:lem:intermediate:Cheeger-value:5}
		\frac{2(i-1)+y_i+(i-1+y_i-a_{[2]})\delta/m+2\delta}{2(y_i+z_i)+(y_i+z_i-a_{[2]})\delta/m+2\delta }~.}
	For any $\delta>-m$ and $y_i\geq 0$, the probability can be upper bounded by
	\eqan{\label{for:lem:intermediate:Cheegar-value:06}
		&\frac{2(i-1)+y_i+(i-1+y_i)\delta/m+\delta(2-a_{[2]}/m)}{2(y_i+z_i)+(y_i+z_i)\delta/m+\delta(2-a_{[2]}/m)}\nn\\
		&\hspace{2cm}\leq \frac{2(i-1+y_i)+(i-1+y_i)\delta/m+\delta(2-a_{[2]}/m)}{2(y_i+z_i)+(y_i+z_i)\delta/m+\delta(2-a_{[2]}/m) }~. }
	Since $2+\delta/m> 1$ and $|y_i|\leq |A|$, we further upper bound the probability in \eqref{for:lem:intermediate:Cheeger-value:5} by
	\eqan{\label{for:lem:intermediate:Cheeger-value:6}
		&\frac{i+|A|-1+\delta(2-a_{[2]}/m)/(2+\delta/m)}{z_i+|A|+\delta(2-a_{[2]}/m)/(2+\delta/m)}\nn\\
		&\hspace{2cm}\leq  \frac{i+|A|+\delta(2-a_{[2]}/m)/(2+\delta/m)}{z_i+|A|+1+\delta(2-a_{[2]}/m)/(2+\delta/m)}~.
	}
	Instead, for the case when $1\in S$, but $2\in S^c$, the edge-connection probability in \eqref{for:lem:intermediate:Cheeger-value:5} changes to
	\eqn{\label{for:lem:intermediate:Cheegar-value:08}
		\frac{2(i-1)+y_i+(i-1+y_i-a_{1})\delta/m+\delta}{2(y_i+z_i)+(y_i+z_i-a_{[2]})\delta/m+2\delta }~.}
	Since $a_2\leq m$, for $\delta\geq 0$ we bound the edge-connection probability by
	\eqn{\label{for:lem:intermediate:Cheegar-value:09}
		\frac{2(i-1)+y_i+(i-1+y_i-a_{1})\delta/m+\delta(1-a_{1}/m)}{2(y_i+z_i)+(y_i+z_i-a_{[2]})\delta/m+\delta(1-a_{1}/m) }~,}
	whereas, for $\delta\in(-m,0)$, the probability is upper bounded by
	\eqn{\label{for:lem:intermediate:Cheegar-value:10}
		\frac{2(i-1)+y_i+(i-1+y_i-a_{1})\delta/m+\delta(2-a_{[2]}/m)}{2(y_i+z_i)+(y_i+z_i-a_{[2]})\delta/m+\delta(2-a_{[2]}/m) }~.}
	Following a similar upper bound in \eqref{for:lem:intermediate:Cheegar-value:06} and \eqref{for:lem:intermediate:Cheeger-value:6}, we upper bound \eqref{for:lem:intermediate:Cheegar-value:09} by
	\eqn{\label{for:lem:intermediate:Cheegar-value:11}
		\frac{i+|A|-1+\delta(1-a_{1}/m)/(2+\delta/m)}{z_i+|A|+\delta(1-a_{1}/m)/(2+\delta/m)}~,}
	and upper bound \eqref{for:lem:intermediate:Cheegar-value:11} by
	\eqan{\label{for:lem:intermediate:Cheegar-value:12}
		&\frac{i+|A|-1+\delta(2-a_{[2]}/m)/(2+\delta/m)}{z_i+|A|+\delta(2-a_{[2]}/m)/(2+\delta/m)}\nn\\
		&\hspace{2cm}\leq\frac{i+|A|+\delta(2-a_{[2]}/m)/(2+\delta/m)}{z_i+|A|+1+\delta(2-a_{[2]}/m)/(2+\delta/m)}~.}
	Denote $\delta_0=\max\{0,\lceil\delta(2-a_{[2]}/m)/(2+\delta/m)\rceil,\lceil\delta(1-a_{1}/m)/(2+\delta/m)\rceil\}$. Then we bound the connection probability in \eqref{for:lem:intermediate:Cheeger-value:5} and \eqref{for:lem:intermediate:Cheegar-value:09} by
	\eqn{\label{for:lem:intermediate:Cheegar-value:13} \frac{i+|A|+\delta_0}{z_i+|A|+1+\delta_0}~.}
	\paragraph{Adaptation to model (a)}
	In model (a), the total volume of the graph is always one more than that in model (b). For model (a), $x_i$ can become \rm{ILL} either by creating a self-loop or connecting to the mini-vertices in $S$ and hence the probability in \eqref{for:lem:intermediate:Cheeger-value:5} becomes
	\eqn{\label{for:lem:intermediate:Cheeger-value:7}
		\frac{2(i-1)+y_i+1+(i+y_i-a_{[2]})\delta/m+2\delta}{2(y_i+z_i)+1+(y_i+z_i+1-a_{[2]})\delta/m+2\delta }~,}
	which can be bounded by
	\eqn{\label{for:lem:intermediate:Cheeger-value:7-1}
		\frac{2i+y_i+(i+y_i-a_{[2]})\delta/m+2\delta}{2(y_i+z_i+1)+(y_i+z_i+1-a_{[2]})\delta/m+2\delta }~.}
	Following exactly the same bounding technique in \eqref{for:lem:intermediate:Cheegar-value:06}-\eqref{for:lem:intermediate:Cheeger-value:6}, we upper bound \eqref{for:lem:intermediate:Cheeger-value:7-1} by the RHS of \eqref{for:lem:intermediate:Cheeger-value:6}.
	Similarly \eqref{for:lem:intermediate:Cheegar-value:08} is reformulated as
	\eqn{\label{for:lem:intermediate:Cheeger-value:14}
		\frac{2(i-1)+y_i+1+(i+y_i-a_{1})\delta/m+\delta}{2(y_i+z_i)+1+(y_i+z_i+1-a_{[2]})\delta/m+2\delta }~, }
	which can be bounded by
	\eqn{\label{for:lem:intermediate:Cheeger-value:14-1}
		\frac{2i+y_i+(i+y_i-a_{1})\delta/m+\delta}{2(y_i+z_i+1)+(y_i+z_i+1-a_{[2]})\delta/m+2\delta }~.}
	We follow \eqref{for:lem:intermediate:Cheegar-value:09}-\eqref{for:lem:intermediate:Cheegar-value:12} to further upper bound \eqref{for:lem:intermediate:Cheeger-value:14-1} by RHS of \eqref{for:lem:intermediate:Cheegar-value:12}. Therefore with the same $\delta_0$ chosen earlier for model (b), we upper bound the probabilities in \eqref{for:lem:intermediate:Cheeger-value:7} and \eqref{for:lem:intermediate:Cheeger-value:14} by \eqref{for:lem:intermediate:Cheegar-value:13}.
	\paragraph{Adaptation to model (d)}
	In model (d), the difference is that mini-vertices cannot attach to any mini-vertex associated with the same vertex. To address this, we associate a mark to each mini-vertex, counting the number of mini-vertices that appeared before it and are associated with the same vertex in the graph. Define $\{x_i, y_i, z_i\}$ and $\{\bar{x}_i, \bar{y}_i, \bar{z}_i\}$ similarly as before. Note that each mini-vertex $x_i$ (and $\bar{x}_i$) has a score $r_i$ (and $\bar{r}_i$), taking values in $[m]$, representing the edge that the mini-vertex creates from the vertex it is associated with. In fact, every mini-vertex has such a score, but we refrain from defining this explicitly to avoid complicating the notation. Note that in model (d), each mini-vertex $x_i$ and $\bar{x}_i$ can attach to any mini-vertex except for $r_i$ mini-vertices (including itself). 
	
	When $x_i$ (or $\bar{x}_i$) joins the graph, the total volume of the graph to which it can attach to is
	\eqn{\label{for:lem:model-d:1}
		2(y_i+z_i)-r_i\qquad\text{or}\qquad 2(\bar{y}_i+\bar{z}_i)-\bar{r}_i~.}
	When $x_i$ arrives, the total volume of $S$ to which $x_i$ may connect to is similarly calculated.
	For the case $1,2\in S$, the probability that $x_i$ becomes \rm{ILL} is
	\eqan{\label{for:lem:model-d:2}
		&\frac{2(i-1)+y_i-r_i+(i-1+y_i-r_i-a_{[2]})\delta/m+2\delta}{2(y_i+z_i)-r_i+(y_i+z_i-r_i-a_{[2]})\delta/m+2\delta}\nn\\
		&\qquad=\frac{2(i-1)+y_i-r_i(1+\delta/m)+(i-1+y_i-a_{[2]})\delta/m+2\delta}{2(y_i+z_i)-r_i(1+\delta/m)+(y_i+z_i-a_{[2]})\delta/m+2\delta}~,}
	whereas for the case $1\in S,$ but $2\notin S$, the probability is 
	\eqan{\label{for:lem:model-d:3}
		&\frac{2(i-1)+y_i-r_i+(i-1+y_i-r_i-a_{1})\delta/m+\delta}{2(y_i+z_i)-r_i+(y_i+z_i-r_i-a_{[2]})\delta/m+2\delta}\nn\\
		&\qquad=\frac{2(i-1)+y_i-r_i(1+\delta/m)+(i-1+y_i-a_{1})\delta/m+\delta}{2(y_i+z_i)-r_i(1+\delta/m)+(y_i+z_i-a_{[2]})\delta/m+2\delta}~.}
	Note that we also need to subtract a factor of $r_i$ in the denominator to account for the fact that $x_i$ cannot connect to the $r_i$ mini-vertices to become \text{ILL}. Since $\delta > -m$, we can upper bound the probabilities in \eqref{for:lem:model-d:2} and \eqref{for:lem:model-d:3} by \eqref{for:lem:intermediate:Cheeger-value:5} and \eqref{for:lem:intermediate:Cheegar-value:08}, respectively.
	
	\paragraph{Conclusion of the proof for all models}
	Similarly, we calculate the probability that $\Bar{x}_i$ becomes \rm{ILL}, given that all mini-vertices that arrived before $\Bar{x}_i$ and belong to $A$ are \rm{FIT}, while those belonging to $A^c$ are \rm{ILL} for models (a), (b), and (d) of preferential attachment models. We can bound these probabilities by 
	$$
	\frac{i+|A|+\delta_0}{\Bar{z}_i+|A|+1+\delta_0}~.
	$$ 
	The probability that all mini-vertices associated with $[m_{[n]}]\setminus A$ are \rm{ILL}, given that all mini-vertices in $A$ are \rm{FIT}, is at most
	\begin{equation}
		\label{for:lem:intermediate:Cheeger-value:9}
		\Big( \prod\limits_{i=1}^{mk-k_1}\frac{i+|A|+\delta_0}{z_i+1+|A|+\delta_0} \Big)\Big( \prod\limits_{i=1}^{m_{[n]}-mk-k_2}\frac{i+|A|+\delta_0}{\Bar{z}_i+1+|A|+\delta_0} \Big)
	\end{equation}
	Using the relation obtained in \eqref{for:lem:intermediate:Cheeger-value:3}, we simplify \eqref{for:lem:intermediate:Cheeger-value:9} as
	\begin{align}
		\label{for:lem:intermediate:Cheegar-value:15}
		&\Big( \prod\limits_{i=1}^{mk-k_1}\frac{i+|A|+\delta_0}{z_i+1+|A|+\delta_0} \Big)\Big( \prod\limits_{i=1}^{m_{[n]}-mk-k_2}\frac{i+|A|+\delta_0}{\Bar{z}_i+1+|A|+\delta_0} \Big) \nonumber \\
		&\qquad\qquad=\frac{\prod\limits_{i=1}^{mk-k_1}(i+|A|+\delta_0)\prod\limits_{i=1}^{m_{[n]}-mk-k_2}(i+|A|+\delta_0)}{\prod\limits_{i=1}^{m_{[n]}-|A|}(i+|A|+\delta_0)}~.
	\end{align}
	Now following the similar calculation in \cite[equations (8)-(10)]{MPS06}, the RHS of \eqref{for:lem:intermediate:Cheegar-value:15} is further simplified and upper bounded as
	\eqan{\label{for:lem:intermediate:Cheegar-value:16}
		\text{RHS of \eqref{for:lem:intermediate:Cheegar-value:15}}=&~\frac{\prod\limits_{i=1}^{mk-k_1+|A|+\delta_0}i\prod\limits_{i=1}^{m_{[n]}-mk-k_2+|A|+\delta_0}i}{\prod\limits_{i=1}^{m_{[n]}+\delta_0}i\prod\limits_{i=1}^{|A|+\delta_0} i}\nn\\
		=&~\frac{(mk+k_2+\delta_0)!(m_{[n]}-mk+k_1+\delta_0)!}{(m_{[n]}+\delta_0)! (|A|+\delta_0)!}\\
		=&~\prod\limits_{i=0}^{k_2-1}\Big( \frac{mk+k_2+\delta_0-i}{m_{[n]}+\delta_0-i} \Big) \prod\limits_{i=0}^{k_1-1}\Big( \frac{m_{[n]}-mk+k_1+\delta_0-i}{m_{[n]}-k_2+\delta_0-i} \Big)\nn\\
		&\hspace{5cm}\times\frac{(mk+\delta_0)!(m_{[n]}-mk+\delta_0)!}{(m_{[n]}-|A|+\delta_0)!(|A|+\delta_0)!}\nn\\
		\leq&~ \frac{(mk+\delta_0)!(m_{[n]}-mk+\delta_0)!}{(m_{[n]}-|A|+\delta_0)!(|A|+\delta_0)!}~,\nn
	}
	where the inequality follows from the fact that
	\eqan{\label{for:lem:intermediate:Cheegar-value:17}
		\frac{mk+k_2+\delta_0-i}{m_{[n]}+\delta_0-i}\leq&~ 1\qquad\text{for all }i\leq k_2-1~,\nn\\
		\mbox{and}\qquad\frac{m_{[n]}-mk+k_1+\delta_0-i}{m_{[n]}-k_2+\delta_0-i} \leq&~1\qquad\text{for all }i\leq k_1-1~.
	}
	Although the calculations of Proposition~\ref{prop:Cheeger-value:model-b} follow from the bound obtained in \eqref{for:lem:intermediate:Cheegar-value:17}, we upper bound it further to reduce a similar repetitive calculation as in \cite[Proof of Theorem~1]{MPS06} with minor modification. Since we aim to obtain exponential upper bound on the probability in Proposition~\ref{prop:Cheeger-value:model-b}, we use this leverage to further upper bound the probability as
	\eqan{\label{for:lem:intermediate:Cheegar-value:18}
		&\frac{(mk+\delta_0)!(m_{[n]}-mk+\delta_0)!}{(m_{[n]}-|A|+\delta_0)!(|A|+\delta_0)!}\nn\\
		=&~\prod\limits_{i=0}^{\delta_0-1}\Big( \frac{mk+\delta-i}{|A|+\delta_0-i} \Big)\prod\limits_{i=0}^{\delta_0-1}\Big( \frac{m_{[n]}-mk+\delta_0-i}{m_{[n]}-|A|+\delta_0-i} \Big)\frac{(mk)!(m_{[n]}-mk)!}{(m_{[n]}-|A|)!|A|!}\\
		\leq&~(mn+\delta_0)^{\delta_0}
		\binom{mk}{|A|}\binom{m_{[n]}-|A|}{mk-|A|}^{-1}~,\nn
	}
	where the inequality follows from the fact that $|A|\leq \alpha k<mk<mn$.
	Since $|A|\leq \alpha k$, and the upper bound is monotonically increasing in the size of $A$, the bound in the lemma follows immediately from \eqref{for:lem:intermediate:Cheegar-value:18}. 
\end{proof}
\section{Critical percolation threshold}\label{chap:local-giant:sec:pi_c:PA}
Fix $m \in \mathbb{N} \setminus \{1\}$ and $\delta > -m$. Let $(G_n)_{n \geq 2}$ denote $\big(\PA^{(s)}_n(m, \delta)\big)_{n \geq 2}$ with $s \in \{a, b, d\}$, starting from a finite initial graph of size 2, with the degree of at least one vertex at most $m$. We now state the main result of this chapter:

\begin{Theorem}[Critical percolation threshold for $\PA$ models]\label{thm:main:theorem}
	For any $\pi \in [0,1]$, let $\mathcal{C}_i^{(n)}(\pi)$ denote the $i$-th largest connected component of the percolated $G_n$. Then,
	\begin{equation}\label{eq:main:theorem}
		\frac{|\mathcal{C}_1^{(n)}(\pi)|}{n} \overset{\prob}{\to} \zeta(\pi),
		\quad \text{and} \quad \frac{|\mathcal{C}_2^{(n)}(\pi)|}{n} \overset{\prob}{\to} 0,
	\end{equation}
	where $\overset{\prob}{\to}$ denotes convergence in probability with respect to both the random graph and percolation. Furthermore, define 
	\begin{equation}\label{eq:def:pi-c}
		\pi_c = \begin{cases}
			0, &\text{for } \delta \in (-m, 0], \\
			\frac{\delta}{2\big(m(m+\delta) + \sqrt{m(m-1)(m+\delta)(m+1+\delta)}\big)}, &\text{for } \delta > 0
		\end{cases}.
	\end{equation}
	Then, $\zeta(\pi) > 0$ for $\pi > \pi_c$, whereas $\zeta(\pi) = 0$ for $\pi \leq \pi_c$. Therefore, $\pi_c$ in \eqref{eq:def:pi-c} is the critical percolation threshold for the sequence of preferential attachment models.
\end{Theorem}

\begin{proof}[Proof of Theorem~\ref{thm:main:theorem}]
	In Section~\ref{chap:local-gian:sec:PA:LSE}, we established the large-set expander property of models (a), (b), and (d) with bounded average degree, following the strategy of Mihail, Papadimitriou, and Saberi \cite{MPS06}, in Proposition~\ref{prop:Cheeger-value:model-b}. From \cite[Theorem~1.1]{ABS22}, we know that in such a scenario, the convergences in \eqref{eq:main:theorem} follow immediately for preferential attachment models, proving that the critical percolation thresholds for both preferential attachment models and the P\'{o}lya point tree are equal. In Chapter~\ref{chap:percolation_threshold_PPT}, we proved in Theorem~\ref{thm:critical-percolation:PPT} that the critical percolation threshold for the P\'{o}lya point tree is $\pi_c$ as defined in \eqref{eq:def:pi-c}. This, in turn, proves that $\pi_c$ in \eqref{eq:def:pi-c} is also the critical percolation threshold for the sequence of preferential attachment models.
\end{proof}


\paragraph{\textbf{Observations}} We make a few remarks about the above result:
\begin{enumerate}
	\item Our proof critically relies on the local limit of preferential attachment models. In \cite{RRR22} (see also \cite{BergerBorgs}), we observed that a large class of preferential attachment models with affine edge attachment rules have the same local limit. While we have specifically proven Theorem~\ref{thm:main:theorem} for models (a), (b), and (d), it is reasonable to expect that the result holds for other versions of preferential attachment models as well.
	\item We use the condition $m \geq 2$ crucially to prove Theorem~\ref{thm:main:theorem}. Our proof fails for $m = 1$ because trees are not expanders. In \cite{RTV07}, preferential attachment models were investigated from the perspective of continuous-time embeddings, which turns the model into a continuous-time branching process. Percolation on a continuous-time branching process is still a continuous-time branching process, which allows us to investigate the tree setting arising when $m=1$. We abstain from performing the explicit computations for this case in this chapter, and discuss this in Chapter~\ref{chap:subcritical}.
	\item The parameter $1/\pi_c$ can be interpreted as the spectral radius of the mean-offspring operator of the local limit of $\PA$s. It can be expected to play a crucial role in other structural properties of $\PA$s. For example, \cite{vdH2} provides a lower bound on typical distances for $\PA$s with $\delta > 0$ of order $\log n/\log (1/\pi_c)$. See also \cite{DSvdH} for related distance results in $\PA$s.
\end{enumerate}
\chapter{Component sizes in the sub-critical regime}
\label{chap:subcritical}
\begin{flushright}
\footnotesize{}Based on:\\
\textbf{An ongoing work}\\
\end{flushright}
\vspace{1cm}
\begin{center}
	\begin{minipage}{0.7 \textwidth}
		\footnotesize{\textbf{Abstract.}
		In this chapter, we study the sub-critical regime of preferential attachment models. Unlike configuration model with power-law degrees, we observe that the preferential attachment graph with finite variance degrees exhibits a larger largest connected component in the sub-critical regime than its maximum degree. In this ongoing work, we have established a lower bound on the size of the largest connected component by utilising the continuous-time branching process embedding of preferential attachment trees. We also provide a simulation study to support our proof of this lower bound.
		}
	\end{minipage}
\end{center}
\vspace{0.1cm}

\section{Introduction}\label{chap:subcritical:sec:intro}

Sub-critical percolation for a sequence of random graphs refers to the percolation regime when the edge-retention probability is below the critical percolation threshold. In this phase, the graph typically consists of small, isolated sub-linear clusters, and the probability of finding a giant component that spans a significant portion of the graph tends to zero.

The study of sub-critical percolation is crucial for understanding the transition to the critical and super-critical phases, where the largest connected component gradually grows from a sub-linear to a linear scale, see for example \cite{B01,JLR00,BDvdHS20,DvdHvLS17}. This concept also has applications in various fields, such as network resilience in \cite{KY08,LLMSSG22}, where sub-critical percolation models the behaviour of a network under random failures.

In Chapter~\ref{chap:local-giant}, we explicitly computed the critical percolation threshold for a sequence of preferential attachment models with parameters \(m \geq 2\) and \(\delta>-m\). For non-positive \(\delta\), the critical percolation threshold is zero, whereas for positive \(\delta\), the critical percolation threshold is \(1/r(\mathbf{T}_\kappa)\), as defined in Theorem~\ref{thm:main:theorem}. In this chapter, we analyse the size of the largest connected component in percolated preferential attachment models with a percolation probability less than the critical percolation threshold.

In \cite{B01}, Bollobás demonstrated that in an Erd\H{o}s-R\'{e}nyi random graph, the size of the largest connected component in the sub-critical regime is \( \Ocal(\log n) \). \cite{vdH2} provides a comprehensive review of the sizes of the largest and second-largest connected components in Erd\H{o}s-R\'{e}nyi graphs. In \cite{D07}, Durrett observed that this behaviour does not hold when the underlying graph has a power-law degree distribution. In fact, the largest degree in such graphs is of the order \( n^{1/(\tau-1)} \), where the empirical degree distribution follows a power law with exponent \( \tau \). Consequently, the size of the largest connected component in such graphs is at least of the order of the maximum degree. Durrett conjectured in \cite[Conjecture~3.3.1]{D07} that, for a specific class of random graph models with a power-law degree distribution, \( n^{1/(\tau-1)} \) represents the correct order of the largest connected component.

Janson \cite{J09} demonstrated that in configuration models with a given degree sequence following a power-law distribution with exponent \( \tau > 3 \), the size of the largest connected component is of the order \( n^{1/(\tau-1)} \) in the sub-critical percolation regime. Pittel \cite{P08} showed that a similar behaviour can be observed in the largest connected component of the percolated configuration model with a \emph{subpower-law} degree distribution.

For the preferential attachment models discussed in \cite{DM13}, near-critical behaviour was studied in \cite{EMO21}, where it was shown that, in the finite-variance regime of the asymptotic degree distribution, the size of the giant component decays exponentially as one approaches the critical edge density. However, the study of the sub-critical phase in preferential attachment models, described in Chapter~\ref{chap:local-giant:sec:intro}, has been much less explored. Unlike other power-law random graphs discussed earlier, preferential attachment models exhibit different behaviour in the sub-critical regime. In this regime, we have shown that the size of the largest connected component in a preferential attachment model (d) is significantly larger than \( n^{1/(\tau-1)} \), where \( \tau = 3 + \delta/m \). Although the exact exponent is not yet fully understood, simulation studies support our finding of a larger largest connected component.

\begin{center}
	{\textbf{Organisation of the chapter}}
\end{center}
This chapter is organized as follows. In Section~\ref{sec:lower:bound}, prove that the size of the largest connected component has larger exponent than that of the maximum degree in sub-critical preferential attachment models. In Section~\ref{sec:simulation}, we provide some simulation study done to support our claim.

\section{Lower bound}\label{sec:lower:bound}
Let \(D^{(n)}(m,\delta)\) denote the maximal degree of a preferential attachment model of size \(n\) with parameters \(m \geq 1\) and \(\delta > -m\). Recall that \(\chi = \frac{m + \delta}{2m + \delta}\). In \cite[Theorem~8.14]{vdH1}, it is proven that \(D^{(n)}(m,\delta)\) satisfies
\eqn{\label{eq:maximal-degree-exponent}
	D^{(n)}(m,\delta) n^{\chi-1}\overset{\mbox{a.s.}}{\longrightarrow} D^{(\infty)}~,
}
where \(D^{(\infty)}\) is an almost surely positive random variable.

The following theorem establishes that the connected component of a preferential attachment tree containing an initial vertex has a size greater than \( n^{1-\chi} \).

\begin{Theorem}[Lower bound for largest connected component]\label{thm:lower-bound}
	For fixed $m \geq 1$ and $\delta > -m$, the largest connected component in the $\pi$-percolated preferential attachment model (d) is at least $\Ocal(n^{1-\chi+\pi\chi+o(1)})$ for any percolation probability $\pi \in (0,1]$. For preferential attachment trees, the size of the largest connected component is of the order \(\Wcal n^{1-\chi+\pi\chi+o(1)}\), where \(\Wcal\) is a finite, almost surely positive, random variable.
\end{Theorem}

As an immediate consequence of Theorem~\ref{thm:lower-bound}, we can determine the critical percolation threshold for preferential attachment trees for any $\delta > -1$.

\begin{Corollary}[Critical percolation threshold for preferential attachment trees]
	The critical percolation threshold for preferential attachment trees with parameter $\delta > -1$ is $1$.
\end{Corollary}

We first prove the theorem for preferential attachment trees and then extend it to the graph cases using the collapsing process. To prove the theorem in the tree case, we utilise the continuous-time branching process (CTBP) representation of preferential attachment models, a technique introduced in \cite{AGS08}. By leveraging this representation, we show that for any \( \pi \in (0,1) \), the connected component containing vertex \(1\) has a size of order \( n^{1-\chi+\pi\chi} \).

\paragraph{CTBP Representation}

For the CTBP representation of model~(d) with parameter \(\delta\), we consider the rate function \(f(\ell) = \ell + \delta + 1\) for \(\ell \in \mathbb{N} \cup \{0\}\).

We start with an initial graph of size \(2\), consisting of an edge connecting two vertices. Let \(\sigma_\omega\) denote the birth time of vertex \(\omega\) in the CTBP, and \(\Pi_{\omega}[0,t]\) denote the number of children of \(\omega\) at time \(t > 0\) after its birth at \(\sigma_\omega\). Conditionally on \(\Pi_{\omega}[0,t] = \ell\), the birth time of the \((\ell + 1)\)-st offspring of \(\omega\) follows an exponential distribution with rate \(f(\ell)\). All these offspring then initiate their own birth processes with the same rate function.

The birth process for vertex \(1\) is slightly different with respect to the rate function. Conditionally on \(\Pi_1[0,t] = \ell\), the birth time of the \((\ell + 1)\)-st offspring has an exponential distribution with rate \(f(\ell) - 1\). The offspring of vertex \(1\) then start their own birth processes as described earlier, using the rate function \(f\).

Let \((T^{\sss (d)}(t))_{t \geq 0}\) denote the resulting tree, and $Z(t)$ denote the size of \(T^{\sss (d)}(t)\). Define a stopping time 
\[
	\tau_n = \inf \{t \geq 0 : Z(t) \geq n\}~,
\] 
where \(\tau_n\) represents the time at which the CTBP reaches size \(n\). \cite[Theorem~2.1]{AGS08} proves that \(\PA^{\sss\rm{(d)}}_n(1,\delta)\) has the same distribution as \((T^{\sss (d)}(\tau_n))\) for all \(n \geq 1\). We now proceed to compute the \emph{Malthusian parameter} for this birth process.
For any $\lambda \geq 0$, define
\eqn{\label{def:phi:lambda}
	\phi(\lambda)=\E\Big[ \int\limits_0^\infty \e^{-\lambda t}\Pi_1(\,dt) \Big]=\E\Big[ \sum\limits_{k=1}^\infty \e^{-\lambda\gamma_k} \Big]~,}
where $\gamma_k$ is the birth time of the $k$-th child of vertex \(1\) in our birth process. As mentioned in \cite[Section~3]{BDO23}, the Malthusian parameter for the CTBP is defined as the unique solution to $\phi(\lambda)=1$.

\begin{Proposition}[Malthusian Parameter for CTBP]\label{prop:malthusian-parameter}
	Fix any $\delta > -1$ and $\pi \in (0,1]$, and let \((T^{\sss (d)}_{\pi}(t))_{t \geq 0}\) denote the \(\pi\)-percolated \((T^{\sss (d)}(t))_{t \geq 0}\). The Malthusian parameter for \( (T^{\sss (d)}_{\pi}(t))_{t \geq 0} \) is given by
	\eqn{\label{eq:prop:malthusian-parameter}
		\lambda_{\pi}=1+\pi(1+{\delta})~.}
	
\end{Proposition}

To prove this proposition, we make use of the following lemma:

\begin{Lemma}\label{lem:gamma-ratio:sum}
	For $b > a + 1 > 0$, 
	\eqn{\label{eq:lem:gamma-ratio:sum}
		\sum\limits_{k \geq 1}\frac{\Gamma(k+a)}{\Gamma(k+b)}=\frac{\Gamma(a+1)}{(b-a-1)\Gamma(b)}~.
	}
\end{Lemma}

Now, we prove Proposition~\ref{prop:malthusian-parameter} subject to Lemma~\ref{lem:gamma-ratio:sum}.

\begin{proof}[Proof of Proposition~\ref{prop:malthusian-parameter}]
	Note that from the construction of \((T^{\sss (d)}(t))_{t\geq 0},~\gamma_k=\sum_{i=0}^{k-1}E_i\), where \((E_i)_{i\geq 0}\) are independent exponential random variables with rate parameters \(i+1+\delta\). 
	
	Upon percolating \((T^{\sss (d)}(t))_{t\geq 0}\) with probability $\pi \in (0,1]$, let \(\gamma_k(\pi)\) denote the birth time of the $k$-th child of vertex $1$ in \((T^{\sss (d)}_{\pi}(t))_{t\geq 0}\). Then,
	\eqn{\label{for:prop:malthusian:1}
		\gamma_k(\pi)=\sum\limits_{i=1}^{N(k)}E_i~,
	}
	where $N(k)$ follows a negative-binomial distribution with parameters \(\pi\) and \(k\). Therefore, to obtain the Malthusian parameter of \((T^{\sss (d)}_{\pi}(t))_{t\geq 0}\), define
	\eqn{\label{for:prop:malthusian:2}
		\phi(\lambda,\pi)=\sum\limits_{k=1}^\infty \E\Big[ \e^{-\lambda\gamma_k(\pi)} \Big]~.
	}
	Recall that in \((T^{\sss (d)}_{\pi}(t))_{t\geq 0}\), we percolate each edge independently with probability \(\pi\). Therefore, \(\phi(\lambda,\pi)\) can be alternatively written as
	\eqan{\label{for:prop:malthusian:03}
		\phi(\lambda,\pi)= \E\Big[ \sum\limits_{k=1}^{\infty} B_i\e^{-\lambda\gamma_k} \Big]~,
	}
	where \((B_i)_{i\geq 1}\) is a sequence of i.i.d.\ ${\rm Bernoulli}(\pi)$ random variables, and independent of \((\gamma_k)_{k\geq 1}\). 
	Since \(E_i\) are independent exponential random variables, we simplify $\phi(\lambda,\pi)$ as
	\eqan{\label{for:prop:malthusian:3}
		\phi(\lambda,\pi)=&\E\Big[ \sum\limits_{k=1}^\infty B_i\e^{-\lambda\gamma_k} \Big]\nn\\
		=&\sum\limits_{k=1}^{\infty}\E[B_k]\E[\e^{-\lambda\gamma_k}]\nn\\
		=&\pi\sum\limits_{\ell=1}^\infty \frac{\Gamma(\ell+1+\delta)}{\Gamma(1+\delta)}\frac{\Gamma(1+\delta+\lambda)}{\Gamma(\ell+1+\delta+\lambda)}~.
	}
	By Lemma~\ref{lem:gamma-ratio:sum}, we obtain the following explicit form for $\phi(\lambda,\pi)$:
	\eqn{\label{for:prop:malthusian:4}
		\phi(\lambda,\pi)=\pi\frac{\Gamma(1+\delta+\lambda)}{\Gamma(1+\delta)}\times\frac{\Gamma(2+\delta)}{(\lambda-1)\Gamma(1+\delta+\lambda)}=\pi\frac{1+\delta}{\lambda-1}~.
	}
	Now, solving for $\phi(\lambda,\pi)=1$, we obtain $\lambda_\pi=1+\pi(1+\delta)$, as required.
\end{proof}

Note that, by choosing $\pi=1$, we obtain the Malthusian parameter for the non-percolated CTBP, \( (T^{\sss (d)}(t))_{t\geq 0} \), as \(\lambda_1=2+\delta=\tau-1\). 
Nerman proved in \cite[Theorem~5.4]{N81} that 
\eqn{\label{for:martingale:2}
	M(t)=Z(t)\e^{-\lambda_1 t}\overset{\mbox{a.s.}}{\longrightarrow}W_1~,}
where \(W_1\) is the almost surely finite non-negative limit. Rudas, T\'{o}th and Valk\'{o} proved in \cite[Lemma~1]{RTV07} that \((T^{\sss (d)}(t))_{t\geq 0}\) satisfies the regularity condition of \cite[Theorem~5.3]{JN84}, and hence $W_1$ is positive almost surely.
This, in turn, proves that 
\eqn{\label{for:martingale:3}
	\frac{\tau_n}{\log n}\overset{\mbox{a.s.}}{\longrightarrow} \frac{1}{\lambda_1}~,}
where $\lambda_1$ is the Malthusian parameter of the CTBP. 
Now, we move on to prove Theorem~\ref{thm:lower-bound}.
\begin{proof}[Proof of Theorem~\ref{thm:lower-bound}]
For any \(\pi\in(0,1]\), let \( (Z_{\pi}(t))_{t\geq0}\) denote the size of the connected component of vertex \(1\) in \( (T^{\sss (d)}_{\pi}(t))_{t\geq0} \). 
Similar to \eqref{for:martingale:2}, we construct \(M_{\pi}(t)=Z_{\pi}(t)\e^{-\lambda_{\pi} t}\), and by \cite[Theorem~5.4]{N81}, \(M_{\pi}(t)\) converges to an almost surely non-negative limit, denoted by \(W_{\pi}\). The same computation as in \cite[Lemma~1]{RTV07} can be performed to prove the regularity condition of \cite{JN84} for \((T^{\sss (d)}_{\pi}(t))_{t\geq 0}\). Therefore, \(W_\pi\) is almost surely a positive random variable.
Thus,
\eqan{\label{from:martingale:1}
	\frac{M_{\pi}(t)}{M_1(t)}=\frac{Z_{\pi}(t)}{Z(t)}\e^{(\lambda_1-\lambda_{\pi})t}\overset{\mbox{a.s.}}{\longrightarrow}\frac{W_{\pi}}{W_1}~.
}
Substituting \(t=\tau_n\) and using the values of \(\lambda_1\) and \(\lambda_\pi\) from \eqref{eq:prop:malthusian-parameter}, we simplify the LHS of \eqref{from:martingale:1} as
\eqan{\label{from:martingale:2}
	\frac{Z_{\pi}(\tau_n)}{n}\e^{(1+\delta)(1-\pi)\tau_n}=&Z_{\pi}(\tau_n)\e^{\log n\big(-1+(1+\delta)(1-\pi)\frac{\tau_n}{\log n} \big)}\nn\\
	=&Z_{\pi}(\tau_n) n^{-1+\frac{1+\delta}{2+\delta} (1-\pi)+o(1)}\nn\\
	=&Z_{\pi}(\tau_n) n^{-1+\chi (1-\pi)+o(1)}\nn\\
	=&Z_{\pi}(\tau_n) n^{-(\chi-1+\pi\chi)+o(1)}~.
}
From \eqref{for:martingale:3}, we have that \(\tau_n\to\infty\) as \(n\) grows to infinity. Therefore, by \eqref{for:martingale:3} and \eqref{from:martingale:1},
\eqn{\label{from:martingale:3}
	Z_{\pi}(\tau_n) n^{-(\chi-1+\pi\chi)+o(1)}\overset{\mbox{a.s.}}{\longrightarrow} \frac{W_{\pi}}{W_1}~,}
proving that the connected component in \(\PA^{\sss \rm (d)}_n(1,\delta)\) is \(\Ocal(n^{-(\chi-1+\pi\chi)+o(1)})\). Note that from the construction of the CTBP, it immediately follows that the size of any connected component not containing the oldest vertex, \(1\), is always stochastically dominated by the size of the connected component containing the oldest vertex, \(1\). Therefore, in preferential attachment trees, the size of the largest connected component is \(\Wcal n^{-(\chi-1+\pi\chi)+o(1)}\), where $\Wcal$ is an almost sure positive finite random variable.

To return to the graph case, we now perform collapsing on the preferential attachment tree of size \(mn\). Since the collapsing procedure reduces the graph size by a factor of \(m\), the lower bound for the largest connected component still holds true without any correction in the exponent of \(n\). Therefore, the largest connected component is \(\Ocal(n^{1-\chi+\pi\chi+o(1)})\) almost surely.
\end{proof}

\begin{remark}
	\rm {Theorem~\ref{thm:lower-bound} holds true for any $m\in\N$ and $\delta>-m$. For $\pi> \pi_c$, although this is a lower bound for the largest connected component, the size of the largest connected component is $\Ocal(n)$, and hence Theorem~\ref{thm:lower-bound} does not contradict Theorem~\ref{thm:main:theorem}. Theorem~\ref{thm:lower-bound} provides a non-trivial lower bound when $\pi<\pi_c$.}\hfill$\blacksquare$
\end{remark}

Now, we remain to prove Lemma~\ref{lem:gamma-ratio:sum}. This follows by exchanging the sum and integrals using the Fubini's theorem.
\begin{proof}[Prrof Lemma~\ref{lem:gamma-ratio:sum}]
	By properties of $\Beta$ functions, for any \(\alpha,\beta>0\),
	\eqn{\label{for:gamma:sum:1}
		\frac{\Gamma(\alpha)\Gamma(\beta)}{\Gamma(\alpha+\beta)}=\int\limits_{0}^{1} u^{\alpha-1}(1-u)^{\beta-1}\,du~.
	}
	Therefore, for $b>a+1>0$,
	\eqan{\label{for:gamma:sum:2}
		\sum\limits_{k\geq 1}\frac{\Gamma(k+a)}{\Gamma(k+b}=&\frac{1}{\Gamma(b-a)}\sum\limits_{k=1}^{\infty} \frac{\Gamma(k+a)\Gamma(b-a)}{\Gamma(k+b)}\nn\\
		=&\frac{1}{\Gamma(b-a)}\sum\limits_{k=1}^{\infty} \int\limits_{0}^{\infty} u^{{k+a-1}}(1-u)^{b-a-1}\,du~.
	}
	Since the integrand is positive in the RHS of \eqref{for:gamma:sum:2}, by Fubini's theorem,
	\eqan{\label{for:gamma:sum:3}
		\sum\limits_{k\geq 1}\frac{\Gamma(k+a)}{\Gamma(k+b}=&~\frac{1}{\Gamma(b-a)}\int\limits_{0}^{\infty}(1-u)^{b-a-1}\sum\limits_{k=1}^{\infty}  u^{{k+a-1}}\,du\nn\\
		=&~\frac{1}{\Gamma(b-a)}\int\limits_{0}^{\infty}(1-u)^{b-a-2}u^{a}\,du\nn\\
		=&~\frac{\Gamma(a+1)\Gamma(b-a-1)}{\Gamma(b-a)\Gamma(b)},\nn\\
		=&~\frac{\Gamma(a+1)}{(b-a-1)\Gamma(b)}~,
	} 
	as required.
\end{proof}

\section{Simulations}\label{sec:simulation}

We conducted a simulation study to support our theorem for preferential attachment trees. In the simulation, we compared the size of the largest connected component in sub-critical preferential attachment models with the maximum degree in the non-percolated graph. van der Hofstad showed in \cite{vdH1} that the maximum degree of the preferential attachment graph satisfies
\begin{equation}
	D^{\sss(n)} = W_{\infty} n^{1-\chi} + o(1)~,
	\label{eq:max-deg:1}
\end{equation}
where \(D^{\sss(n)}\) denotes the maximum degree of the preferential attachment tree, and \(W_\infty\) is an almost surely finite positive random variable. On the other hand, in Theorem~\ref{thm:lower-bound}, we proved that for preferential attachment models, the size of the largest connected component satisfies
\begin{align}
	|\Ccal^{(n)}_1(p)| =& \Wcal n^{1-\chi+\pi\chi+o(1)}~\text{for }m=1~,\\
	\text{and}\quad|\Ccal^{(n)}_1(p)| \geq& \Wcal n^{1-\chi+\pi\chi+o(1)}~\text{for }m>1~.
	\label{eq:pLCC:1}
\end{align}
We plotted \(\log |\Ccal^{(n)}_1(p)|\) and \(\log D^{\sss(n)}\) against \(\log n\). We then fitted straight lines and computed the slope of each fitted line. To reduce variation in the slope due to the effect of the random variables \(\Wcal\) and \(W_\infty\), we computed their difference and compared it with the theoretical value \(\chi \pi\) observed in \eqref{eq:pLCC:1}. We compiled the simulation results for several values of \(\chi\) and \(\pi\) in this section.

\paragraph{Simulation details} 
We defined the following sequence of integers to represent the graph sizes:
\eqn{\label{eq:size:seq}
	{\bf N} = \{ x \cdot 10^y : x \in [9],~\text{and } y \in \{2,3,4 \}\}.
}

For preferential attachment trees, we considered four percolation probabilities \( \{0.4,~0.5,~0.6,~0.7\} \). For preferential attachment graphs with \(m=2\), we examined two percolation probabilities \(0.85\pi_c\) and \(0.9\pi_c\), where \(\pi_c\) is the critical percolation threshold. Probabilities below \(0.8\pi_c\) were not simulated, as they approach zero and could introduce numerical errors.

For each combination of \(m\), \(\delta\), percolation probability \(\pi\), and \(n \in {\bf N}\), we simulated 200 preferential attachment models of size \(n\). We percolated each graph with probability \(\pi\) and computed the maximum degree and the size of the largest connected component. Averages of these metrics provided estimates for the largest connected component size in the \(\pi\)-percolated graph and the maximum degree in the unpercolated graph.

\paragraph{Preferential attachment tree with $\delta=1.5$}

Table~\ref{tab:slope_differences:0.714} and Figure~\ref{fig:tau=4.5} show simulation results for  $m=1$ and $\chi=0.714~(\tau=4.5)$: 


\begin{figure}[H]
	\centering
	\begin{subfigure}[b]{0.48\textwidth}
		\centering
		\includegraphics[width=\textwidth]{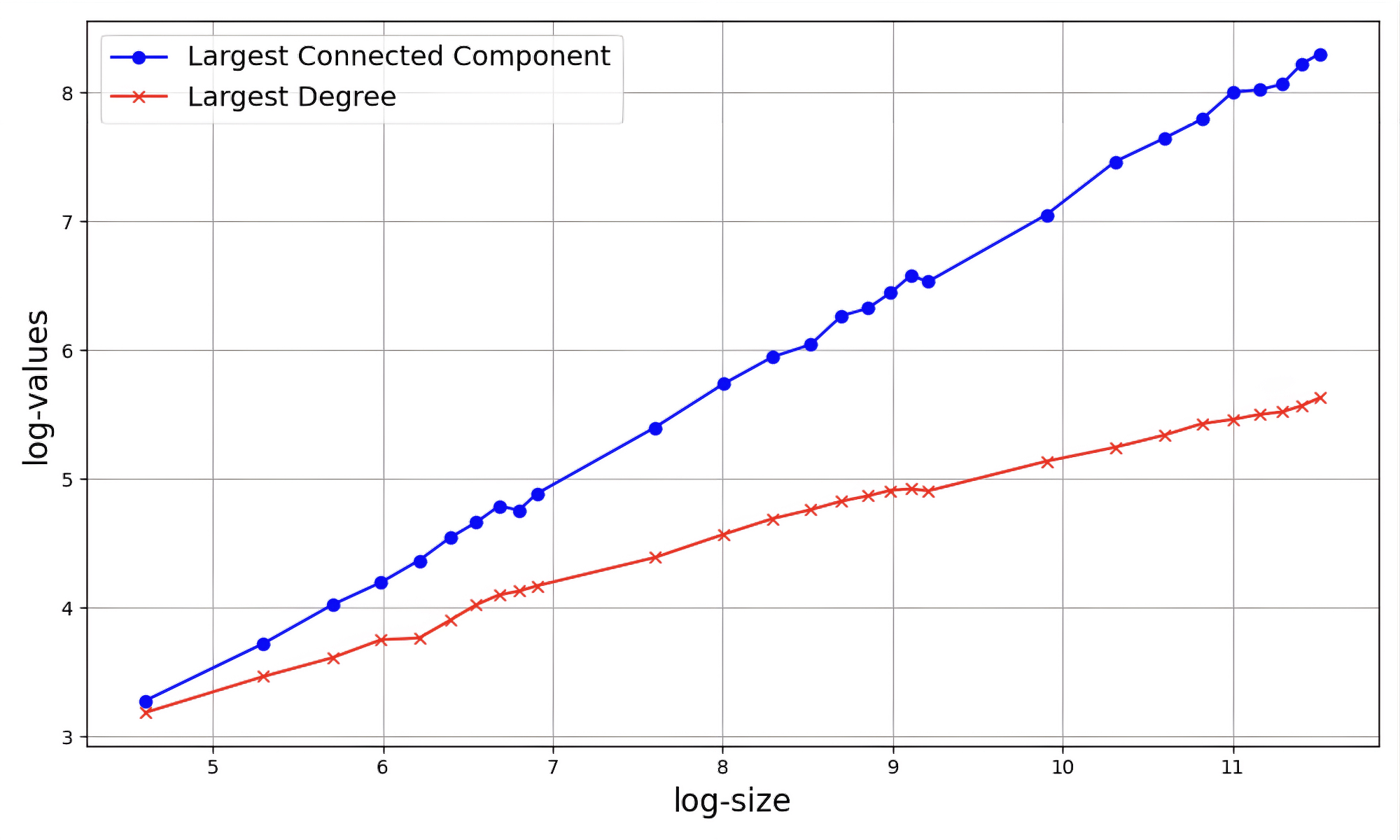}
		\caption{$\pi=0.4$}
		\label{fig:tau=4.5:fig1}
	\end{subfigure}
	\hfill
	\begin{subfigure}[b]{0.48\textwidth}
		\centering
		\includegraphics[width=\textwidth]{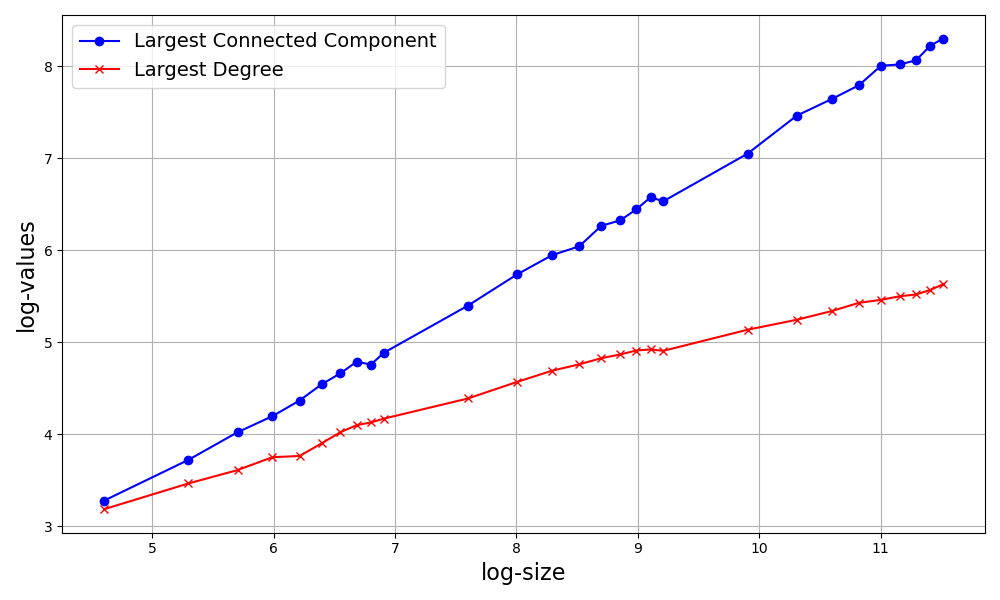}
		\caption{$\pi=0.5$}
		\label{fig:tau=4.5:fig2}
	\end{subfigure}
	
	\vspace{1em} 
	
	\begin{subfigure}[b]{0.48\textwidth}
		\centering
		\includegraphics[width=\textwidth]{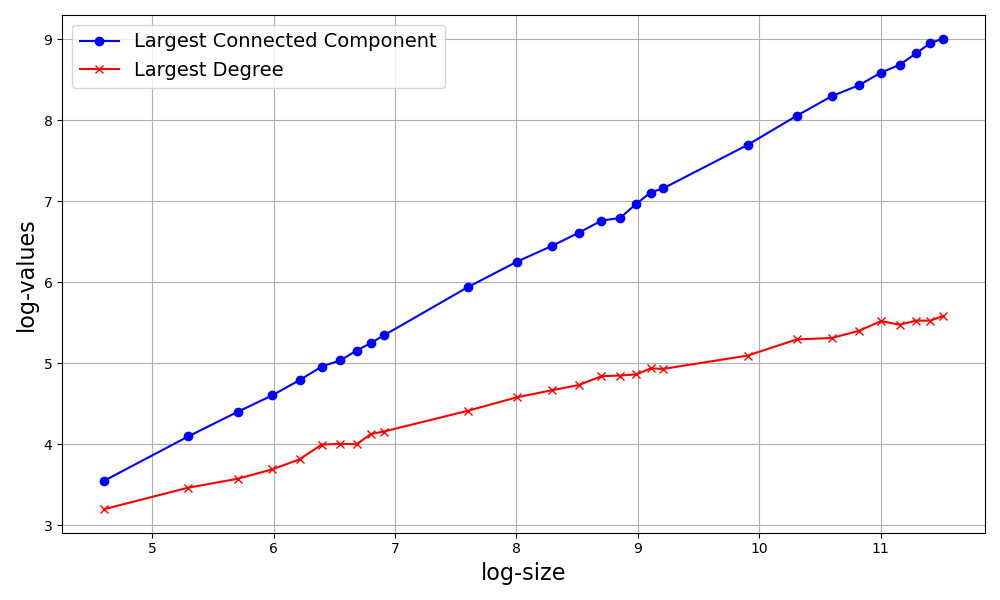}
		\caption{$\pi=0.6$}
		\label{fig:tau=4.5:fig3}
	\end{subfigure}
	\hfill
	\begin{subfigure}[b]{0.48\textwidth}
		\centering
		\includegraphics[width=\textwidth]{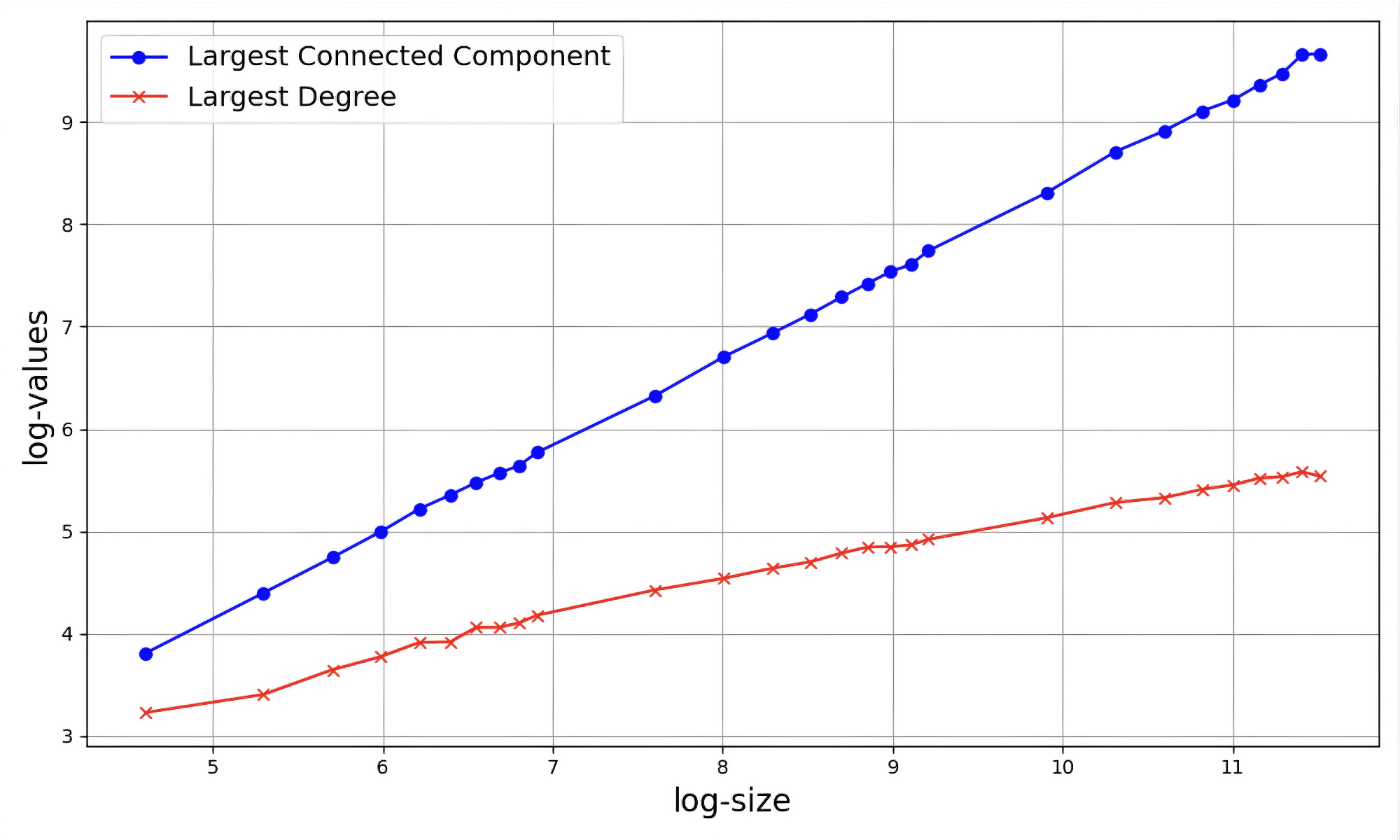}
		\caption{$\pi=0.7$}
		\label{fig:tau=4.5:fig4}
	\end{subfigure}
	
	\caption{Log-log plot of the largest connected component and maximum degree vs. graph size for various percolation probabilities, with $m=1$ and $\chi=0.714$.}
	\label{fig:tau=4.5}
\end{figure}
\begin{table}[H]
	\centering
	\begin{tabular}{|c|c|c|}
		\hline
		$\pi$ & $\pi\chi$ & \text{Slope Difference (Observed)} \\
		\hline
		0.4 & 0.2856 & 0.34842444065633 \\
		0.5 & 0.357 & 0.3865458326732401 \\
		0.6 & 0.4284 & 0.4407357438744405 \\
		0.7 & 0.4998 & 0.5074026514894612 \\
		\hline
	\end{tabular}
	\caption{Comparison of slope differences for $m=1$ and $\chi=0.714$.}
	\label{tab:slope_differences:0.714}
\end{table}

\paragraph{Preferential attachment tree with $\delta=3$}

Table~\ref{tab:slope_differences:0.8} and Figure~\ref{fig:tau=6} show simulation result for $m=1$ and $\chi=0.8~(\tau=6)$:

\begin{figure}[H]
	\centering
	\begin{subfigure}[b]{0.48\textwidth}
		\centering
		\includegraphics[width=\textwidth]{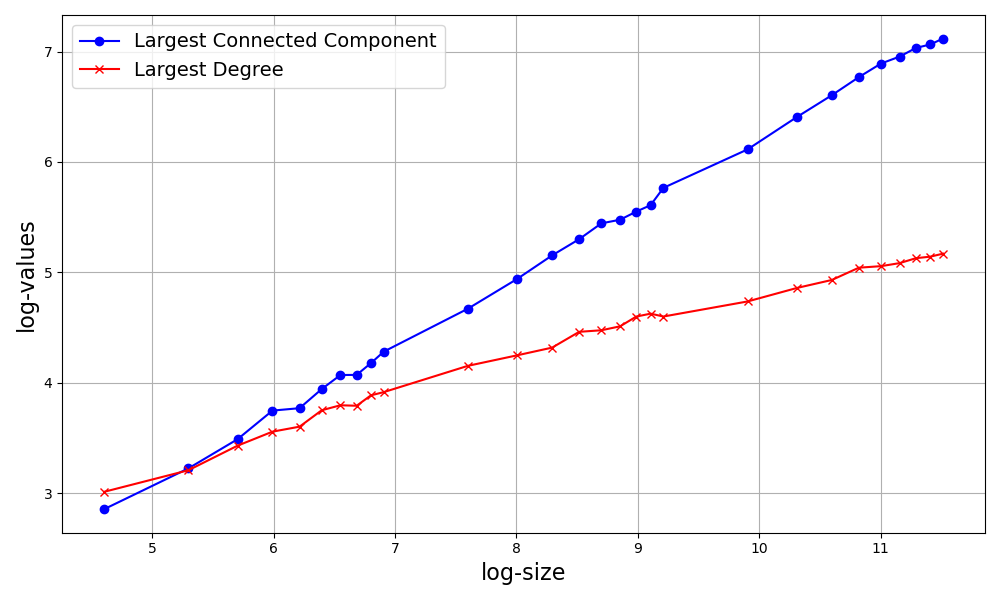}
		\caption{$\pi=0.4$}
		\label{fig:tau=6:fig1}
	\end{subfigure}
	\hfill
	\begin{subfigure}[b]{0.48\textwidth}
		\centering
		\includegraphics[width=\textwidth]{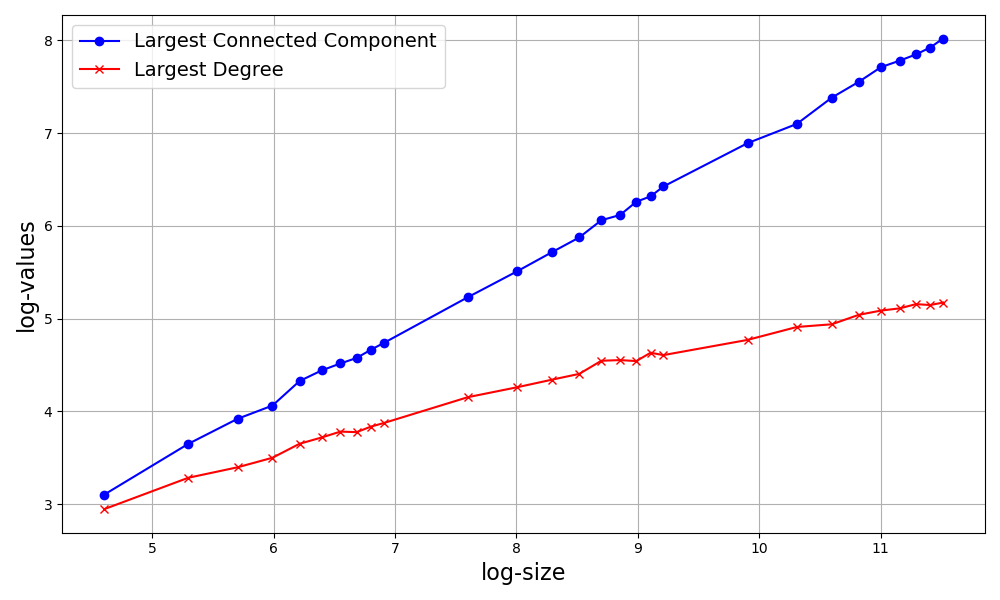}
		\caption{$\pi=0.5$}
		\label{fig:tau=6:fig2}
	\end{subfigure}
	
	\vspace{1em} 
	
	\begin{subfigure}[b]{0.48\textwidth}
		\centering
		\includegraphics[width=\textwidth]{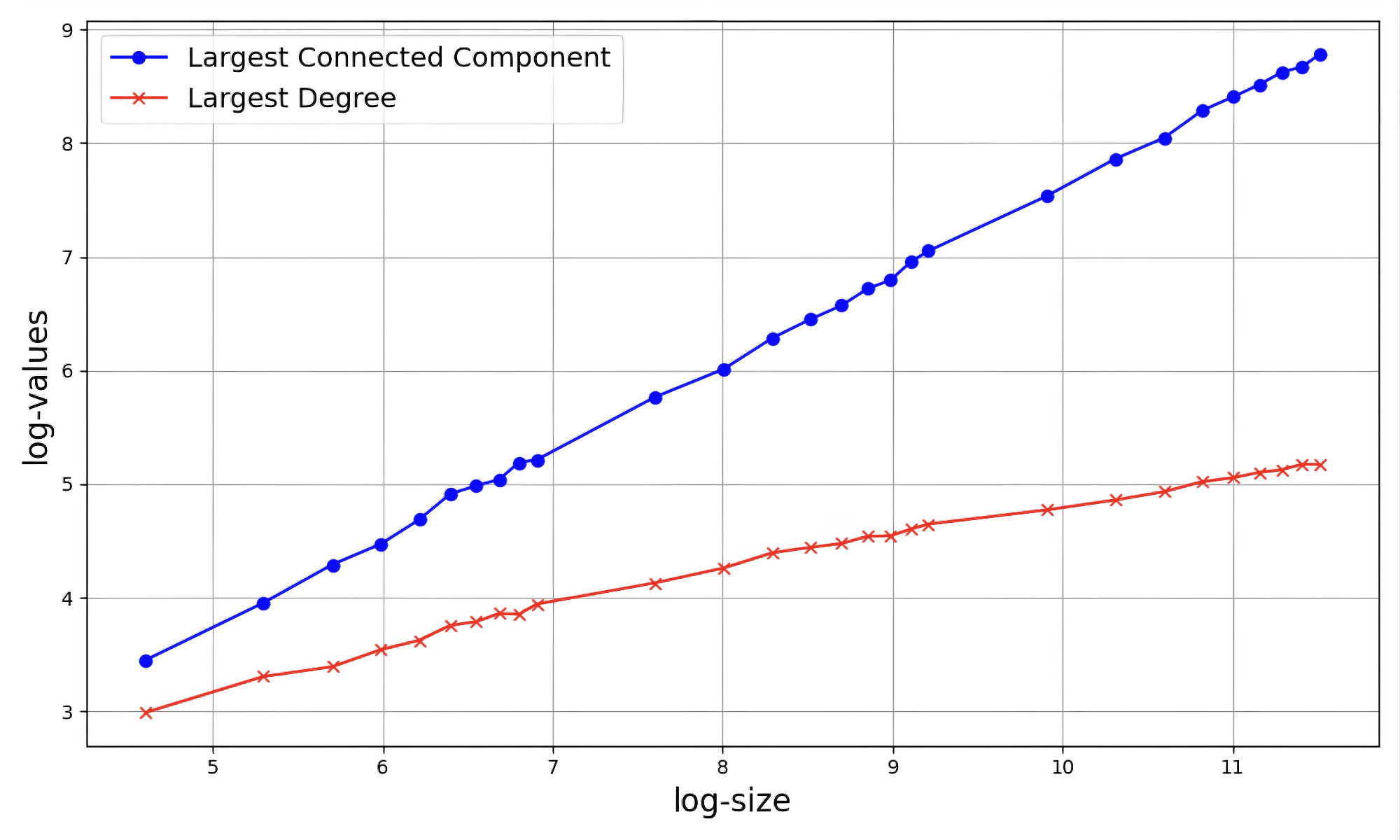}
		\caption{$\pi=0.6$}
		\label{fig:tau=6:fig3}
	\end{subfigure}
	\hfill
	\begin{subfigure}[b]{0.48\textwidth}
		\centering
		\includegraphics[width=\textwidth]{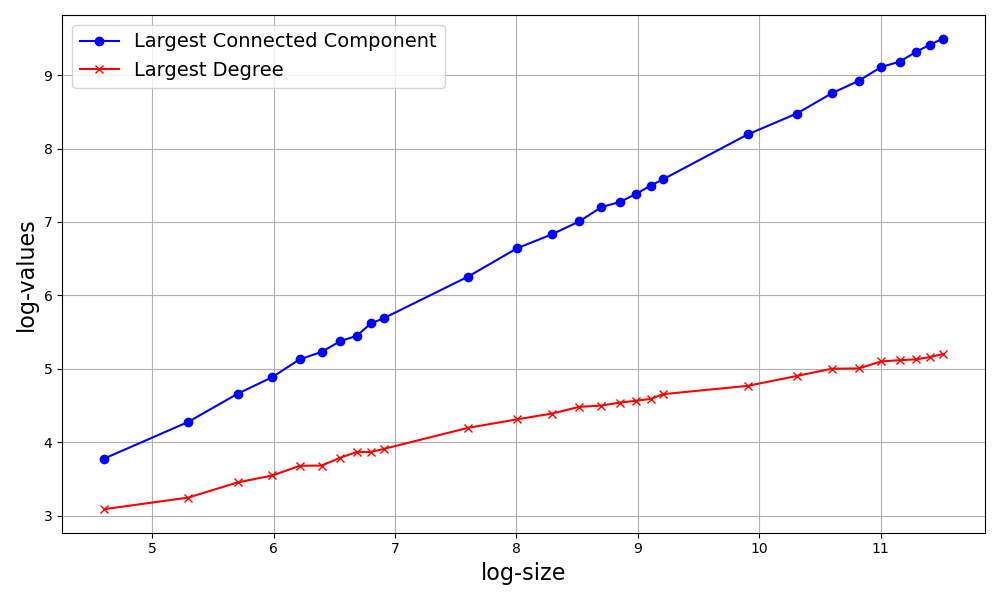}
		\caption{$\pi=0.7$}
		\label{fig:tau=6:fig4}
	\end{subfigure}
	
	\caption{Log-log plot of the largest connected component and maximum degree vs. graph size for various percolation probabilities, with $m=1$ and $\chi=0.8$.}
	\label{fig:tau=6}
\end{figure}

\begin{table}[H]
	\centering
	\begin{tabular}{|c|c|c|}
		\hline
		$\pi$ & $\pi\chi$ & \text{Slope Difference (Observed)} \\
		\hline
		0.4 & 0.32 & 0.32887537637208575 \\
		0.5 & 0.4 & 0.4031197660032112 \\
		0.6 & 0.48 & 0.4705659866292104 \\
		0.7 & 0.56 & 0.5426519248091503 \\
		\hline
	\end{tabular}
	\caption{Comparison of slope differences for $\chi=0.8$.}
	\label{tab:slope_differences:0.8}
\end{table}

\paragraph{Conclusion for the tree model}
Figures~\ref{fig:tau=4.5} and \ref{fig:tau=6} suggest that in the log-log plots, the slopes of the straight lines fitting the size of the largest connected component of the \(\pi\)-percolated tree against the size of the tree are significantly greater than those fitting the maximum degree in the tree against the size of the tree. Furthermore, tables~\ref{tab:slope_differences:0.714} and \ref{tab:slope_differences:0.8} suggest that the differences between these two slopes are close to \(\chi\pi\), as proposed in Theorem~\ref{thm:lower-bound}.


\paragraph{Preferential attachment graph with $m=2$ and $\delta=9$}

Table~\ref{tab:slope_differences:0.846_2} and Figure~\ref{fig:tau=0.846_2} show simulation result for $m=2$ and $\chi=0.846~(\tau=7.5)$:

\begin{figure}[H]
	\centering
	\begin{subfigure}[b]{0.48\textwidth}
		\centering
		\includegraphics[width=\textwidth]{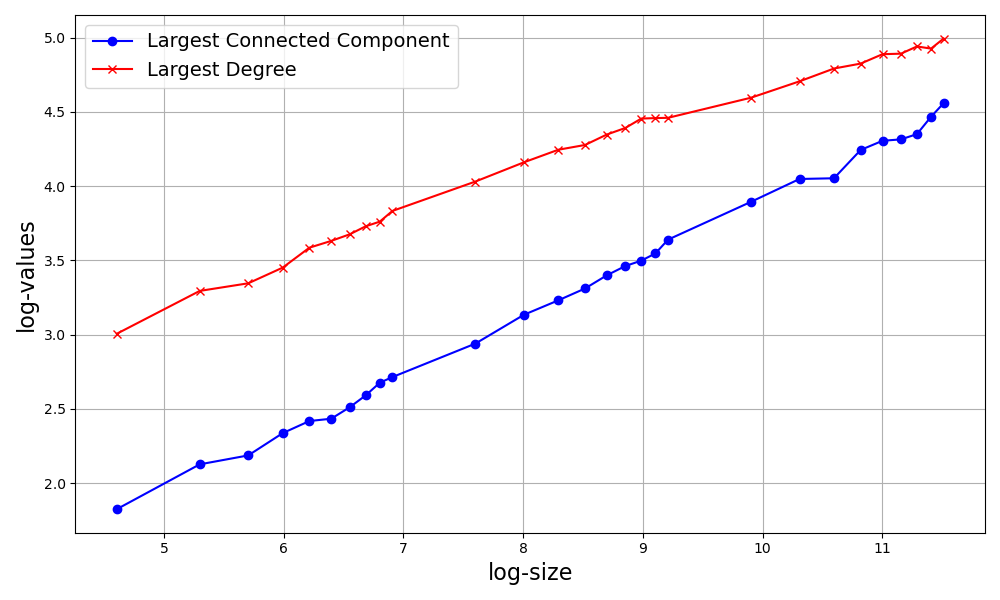}
		\caption{$\pi=0.85\pi_{c}$}
		\label{fig:tau=7.5_2:fig1}
	\end{subfigure}
	\hfill
	\begin{subfigure}[b]{0.48\textwidth}
		\centering
		\includegraphics[width=\textwidth]{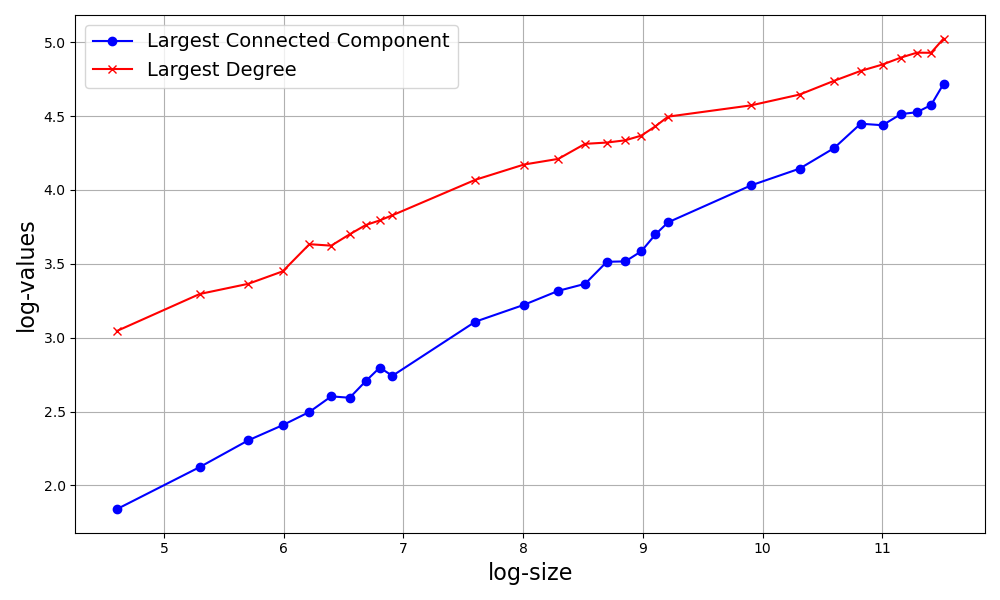}
		\caption{$\pi=0.9\pi_{c}$}
		\label{fig:tau=7.5_2:fig2}
	\end{subfigure}

	\caption{Log-log plot of the largest connected component and maximum degree vs. graph size for various percolation probabilities, with $m=2$ and $\chi=0.846$}
	\label{fig:tau=0.846_2}
\end{figure}

\begin{table}[H]
	\centering
	\begin{tabular}{|c|c|c|}
		\hline
		$\pi$ & $\pi\chi$ & \text{Slope Difference (Observed)} \\
		\hline
		0.85 & 0.08462 & 0.11255188908782127 \\
		0.9 & 0.08960 & 0.14480838267871132 \\
		\hline
	\end{tabular}
	\caption{Comparison of slope differences for $\chi=0.846$.}
	\label{tab:slope_differences:0.846_2}
\end{table}

\paragraph{Preferential attachment graph with $m=2$ and $\delta=6$}

Table~\ref{tab:slope_differences:0.846_2} and Figure~\ref{fig:tau=0.846_2} show simulation result for $m=2$ and $\chi=0.8~(\tau=6)$:

\begin{figure}[H]
	\centering
	\begin{subfigure}[b]{0.48\textwidth}
		\centering
		\includegraphics[width=\textwidth]{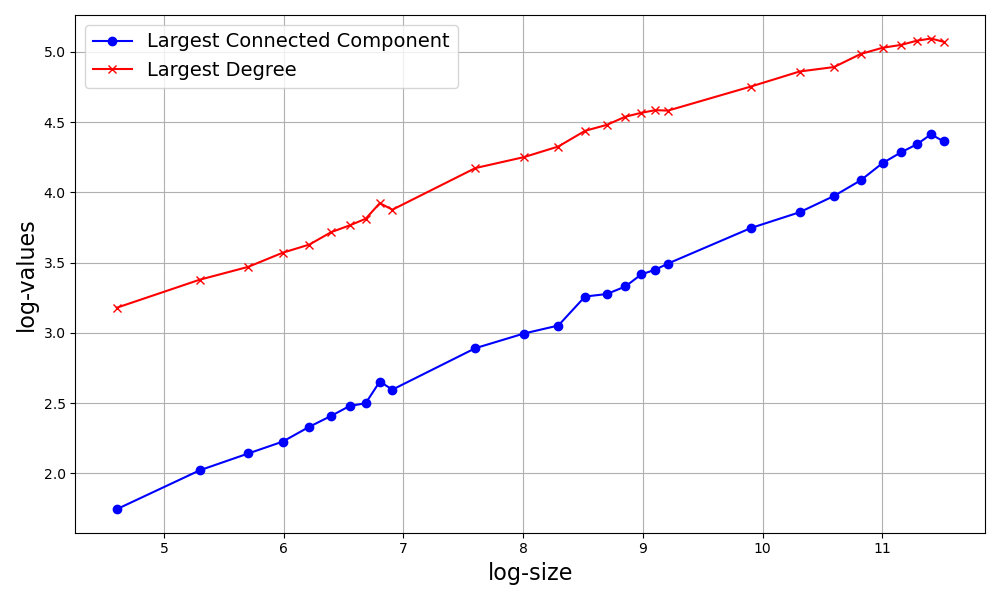}
		\caption{$\pi=0.85\pi_{c}$}
		\label{fig:tau=6_2:fig1}
	\end{subfigure}
	\hfill
	\begin{subfigure}[b]{0.48\textwidth}
		\centering
		\includegraphics[width=\textwidth]{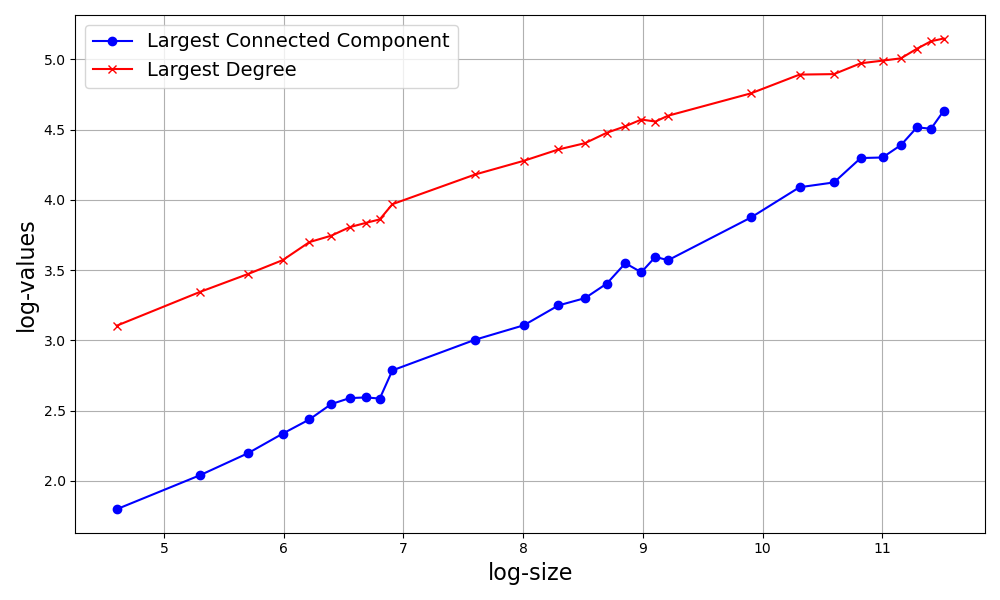}
		\caption{$\pi=0.9\pi_{c}$}
		\label{fig:tau=6_2:fig2}
	\end{subfigure}

	\caption{Log-log plot of the largest connected component and maximum degree vs. graph size for various percolation probabilities, with $m=2$ and $\chi=0.8$}
	\label{fig:tau=0.8_2}
\end{figure}

\begin{table}[H]
	\centering
	\begin{tabular}{|c|c|c|}
		\hline
		$\pi$ & $\pi\chi$ & \text{Slope Difference (Observed)} \\
		\hline
		0.85 & 0.08462 & 0.10040842114736509 \\
		0.9 & 0.08960 & 0.1159294432125863 \\
		\hline
	\end{tabular}
	\caption{Comparison of slope differences for $\chi=0.8$.}
	\label{tab:slope_differences:7.5_2}
\end{table}

\paragraph{Conclusion for non-tree model}
Figures~\ref{fig:tau=0.846_2} and \ref{fig:tau=0.8_2} also indicate that the difference between the slopes of the lines fitting the largest connected component in the \(\pi\)-percolated graph against the size of the graph in log scale is significantly greater than those fitting the maximum degree in the graph against the size of the graph. Furthermore, Tables~\ref{tab:slope_differences:0.846_2} and \ref{tab:slope_differences:7.5_2} show that the difference is considerably large and exceeds $\chi \pi$. This may occur either due to a slower rate of convergence or due to a different exponent of the largest connected component in non-tree settings.

Thus, we may expect a larger largest connected component in non-tree settings compared to tree settings. Although the exact size of the largest connected component in non-tree settings remains inconclusive, the lower bound can still be inferred from the simulation.

\begin{remark}
	\rm 
We have explicitly computed the size of the largest connected component in percolated preferential attachment trees, while for non-tree settings of preferential attachment models in the sub-critical percolation regime, we have obtained a lower bound for this size. 
In our future work, it remains to be checked whether the obtained lower bound is the actual size of the largest connected component in the non-tree settings within the sub-critical percolation regime. The collapsed CTBP representation of preferential attachment models from \cite{BDO23} may be crucial in establishing the upper bound. \hfill$\blacksquare$
\end{remark}

\chapter{Ising critical temperature}
\label{chap:ising}
\begin{flushright}
	\footnotesize{}Based on:\\
	\textbf{Ising model on preferential attachment models}
\end{flushright}
\vspace{0.1cm}
\begin{center}
	\begin{minipage}{0.7 \textwidth}
		\footnotesize{\textbf{Abstract.}
			In this chapter, we study the Ising model on preferential attachment models. Building on the work of Dembo, Montanari, and Sun (2013), we compute the explicit thermodynamic limit of pressure per particle by utilising the local limit of preferential attachment models. We use the convexity properties of the internal energy and magnetisation to determine their thermodynamic limits as the size of the graph tends to infinity. To identify the inverse critical temperature, we prove that the inverse critical temperature for a sequence of graphs and its local limit are equal. Finally, we show that $\beta_c(m,\delta)$ is the inverse critical temperature for the P\'{o}lya point tree with parameters $m$ and $\delta$, using results from Lyons. This part of the proof heavily relies on the critical percolation threshold for P\'{o}lya point trees established earlier.
		}
	\end{minipage}
\end{center}
\vspace{0.1cm}

\section{Introduction}\label{chap:ising:sec:introduction}
The Ising model, originally introduced to describe ferromagnetism in statistical mechanics, represents a system of spins on a lattice, where each spin can take one of two states. The interactions between neighbouring spins determine the system's overall magnetic properties. This model has since found applications beyond physics, including in sociology, biology, and finance, where it helps to explain the collective behaviour of interconnected entities. The Ising model has been extensively studied in the literature; see, for example, \cite{G06}. Recently, the focus in this field has extended to the study of the Ising model on random graphs. Dembo and Montanari studied the Ising model on locally tree-like graphs in \cite{DM10}, restricting their analysis to cases where the average degree has finite variance. Under certain continuity assumptions, the authors examined the Ising measure on uniformly sparse locally tree-like graphs. The Ising model on Erd\H{o}s-Renyi graphs in both the zero and high-temperature regimes was studied in \cite{GdS08}. The Ising model on configuration models has been analysed in \cite{DM10,DGvdH10}. Dembo, Montanari and Sun also studied factor models on locally tree-like graphs, shedding light on the Ising model in these contexts in \cite{DMS13}.

\begin{center}
	{\textbf{Organisation of the chapter}}
\end{center}
This chapter is organized as follows. In Section~\ref{sec:preliminary:results}, we describe some frequently used lemma that we use in this chapter without proof. Proofs to the lemmata described in this section could be obtained in the references provided following the statements of the lemma. In Section~\ref{sec:thermodynamic:limit}, we prove thermodynamic limits of several thermodynamic quantities. Lastly, in Section~\ref{sec:inv:critical-temperature}, we compute the inverse critical temperature for preferential attachment models.

\section{Preliminary results}\label{sec:preliminary:results}
In this section, we state several preliminary results that we frequently use in this section. These lemmas hold true under very mild conditions, and we verify that these conditions are met in our case. We do not provide proofs for these lemmas but instead cite the articles where they are proven in detail.

\begin{lemma}[GKS inequality]\label{lem:GKS}
	Consider two Ising measures $\mu$ and $\mu^\prime$ on graphs $G = (V, E)$ and $G^\prime= (V, E^\prime)$, with inverse temperatures $\beta$ and $\beta^\prime$, and external fields $\underline{B}$ and $\underline{B}^\prime$, respectively. If $E\subseteq E^\prime$, $\beta \leq \beta^\prime$, and $0 \leq B_i \leq B^\prime_i$ for all $i \in V$, then, for any $U \subset V$,
	\begin{equation}\label{eq:l1em:GKS}
		0\leq \left\langle{\prod\limits_{i\in U}\sigma_i}\right\rangle_{\mu}\leq \left\langle{\prod\limits_{i\in U}\sigma_i}\right\rangle_{\mu^\prime}.
	\end{equation}
\end{lemma}
\cite{G67} proves this result in a restricted settings and \cite{KS68} generalised this result. 
The following lemma simplifies the computation of the Ising measure on a tree by reducing it to the computation of Ising measures on subtrees:

\begin{lemma}[Tree pruning]\label{lem:tree-pruning}
	For $U$ a subtree of a finite tree $T$, let $\partial U$ be the subset of vertices in $U$ that connect to a vertex in $W \equiv T \setminus U$. Denote by $\langle \sigma_u \rangle_{\mu_{W,u}}$ the magnetisation of vertex $u \in \partial U$ in the Ising model on $W \cup \{u\}$. Then, the marginal Ising measure on $U$, $\mu_U^{T}$, is equivalent to the Ising measure on $U$ with the magnetic fields
	\begin{equation}\label{eq:lem:tree-pruning}
		B_u^\prime =\begin{cases}
			\arctanh\big( \langle \sigma_u \rangle_{\mu_{W,u}} \big), & u \in \partial U,\\
			B_u, & u \in U \setminus \partial U.
		\end{cases}
	\end{equation}
\end{lemma}
This lemma is proved in details in \cite[Lemma~4.1]{DM10}.
\section{Thermodynamic limits}\label{sec:thermodynamic:limit}

In this section, we state and prove the convergence of the thermodynamic limits of various quantities, such as pressure per particle, magnetisation, and internal energy. We obtain the thermodynamic limit of the pressure using the \emph{belief propagation} method from \cite{DMS13}. Throughout this chapter, we will primarily consider a fixed external magnetic field acting on all the vertices of the graph, resulting in $B_v = B$ for all $v \in V_n$. Note that, under the assumption of a fixed magnetic field, the Boltzmann distribution function in \eqref{eq:def:boltzmann:seq}, is symmetric about the sign of $B$. Thus, without loss of generality, we assume $B > 0$ and denote the partition function and pressure per particle as $Z_n(\beta, B)$ and
\begin{equation}\label{eq:def:pressure-per-particle}
	\psi_n(\beta, B) = \frac{1}{n} \log Z_n(\beta, B),
\end{equation}
respectively. In Section~\ref{sec:thermodynamic-limit:subsec:belief-propagation}, we describe the belief propagation idea, as outlined in \cite{DMS13}, and prove a fixed-point lemma that holds in great generality. In Section~\ref{sec:thermodynamic-limit:subsec:convergence:pressure}, we prove the convergence of pressure per particle. In Section~\ref{sec:thermodynamic-limit:subsec:convergence:thermodynamic-quantity}, we establish the thermodynamic limits of magnetisation and internal energy.


\subsection{Belief propagation on trees}\label{sec:thermodynamic-limit:subsec:belief-propagation}
For a rooted tree $T$ and $t\in\mathbb{N}$, let $T_t$ denote the tree $T$ pruned at height $t$, and let $\nu_T^{t,+}$ denote the root marginal for the Ising model on $T_t$ with $+$ boundary condition. Similarly, we define $\nu_T^{t,f}$ as the root marginal for the Ising model on $T_t$ with free boundary condition. From \cite[Lemma~4.1]{DMS13}, we know that the limits $\nu_T^{t,+}$ and $\nu_T^{t,f}$ exist as $t\to\infty$, and we denote these limiting distributions as $\nu_T^{+}$ and $\nu_T^{f}$, respectively.

Let $T_{x\to y}$ denote the subtree of $T$ obtained by deleting the edge $\{x , y\}$ and rooting it at $x$. For a P\'{o}lya point tree $T$ and $j\in[D(\emp)]$, $T_{\emp\to\emp j}$ represents the subtree of the P\'{o}lya point tree rooted at $\emp$ with the component hanging from $\emp j$ deleted (including $\emp j$). On the other hand, $T_{\emp j\to\emp }$ is a P\'{o}lya point tree rooted at $\emp j$. Essentially, for any $j\in[D(\emp)]$,
\eqn{\label{eq:BP:def:1}
	T=T_{\emp \to \emp j}\sqcup T_{\emp j\to\emp}~,
}
where $\sqcup$ represents the disjoint union operator. We now define the following root marginals:
\eqn{\label{eq:root-marginal}
	\nu_{\emp\to\emp j}^{\dagger}\equiv \nu_{T_{\emp\to\emp j}}^{\dagger}\quad\text{and}\quad\nu_{\emp j\to\emp}^{\dagger}\equiv \nu_{T_{\emp j\to\emp}}^{\dagger}~,
}
for $\dagger\in\{+,f\}$.
By \cite[Lemma~5.15]{RvdHSF},
\[
\nu_T^{+}=\nu_T^{f},~\quad\mu~\text{almost surely}.
\]
Let us denote $\magn{\to\emp j}=2\nu_{\emp\to\emp j}^{+}(+1)-1$, and $\magn{\emp j\to}=2\nu_{\emp j\to\emp}^{+}(+1)-1$. Note that $\magn{\to \emp j}$ is the root magnetisation in $T_{\emp\to\emp j}$. Similarly, $\magn{ \emp j\to }$ is the root magnetisation in the P\'{o}lya point tree $T_{\emp j\to\emp}$.
Root magnetisation of any locally finite tree satisfies the following property:
\begin{lemma}\label{lem:h:magnetization}
	Let $\Tbold$ be a multi-type branching process with a general type-space $\Sbold$, and assume that $\Tbold$ is almost surely locally finite. Furthermore, label the nodes of $\Tbold$ using the Ulam-Harris notation. Consider a distributional functional $\h$ as a fixed point solution to the following recursion:
	\begin{equation}\label{eq:test:1}
		\h^{(\ell+1)}(\omega) = B + \sum\limits_{i=1}^{D_{\ell}(\omega)} \arctanh\Big( \tanh(\beta)\tanh\big( \h^{(\ell)}(\omega i) \big) \Big),
	\end{equation}
	where $\big(D_{\ell}(\omega)\big)_{\ell\geq1}$ is a sequence of $D(\omega)$, the degree of the node $\omega$ in $\Tbold$, for some $B > 0$ and for all $\omega \in \Sbold$ and $\ell \geq 0$. Then, denoting by $\magn{\omega\to}$ the root magnetisation of $\Tbold$, rooted at $\omega$ under the inverse critical temperature $\beta$ and external magnetic field $B > 0$, for all $\omega \in \Sbold$, we have
	\begin{equation}\label{eq:prop:root-magnetization:1}
		\h(\omega) = \magn{\omega\to},\quad \text{a.s.}
	\end{equation}
\end{lemma}
\medskip

Proof to Lemma~\ref{lem:h:magnetization} follows a similar approach to that in \cite[Proposition~5.13]{RvdHSF}, with adaptations made for a more general setting. Our goal is to demonstrate that \eqref{eq:prop:root-magnetization:1} is the unique fixed-point solution to the recursive equations in \eqref{eq:test:1}.

\begin{proof}[Proof of Lemma~\ref{lem:h:magnetization}]
	We prove Lemma~\ref{lem:h:magnetization} for a very general case where the tree is locally finite almost surely. Let $\Tbold_{\ell}(\omega)$ denote $\Tbold$ rooted at $\omega \in \Sbold$, truncated at depth $\ell$, and let $\magn{\omega,\ell,+/f}$ denote the root magnetization given $\Tbold_{\ell}(\omega)$ with an external magnetic field per vertex $B > 0$, under either the $+$ or free boundary condition. Since $\Tbold$ is locally finite almost surely, by \cite[Lemma~3.1]{DGvdH10} there exists $M < \infty$ such that for all $\ell \geq 1$,
	\begin{equation}
		\label{for:lem-h:1}
		|\magn{\omega,\ell,+} - \magn{\omega,\ell,f}| \leq \frac{M}{\ell}, \quad \text{a.s. for all } \omega \in \Sbold.
	\end{equation}
	Therefore, as $\ell \to \infty$, both $\magn{\omega,\ell,+}$ and $\magn{\omega,\ell,f}$ converge to the same limit, which we denote by $\magn{\omega \to}$, defined in Section~\ref{sec:thermodynamic-limit:subsec:belief-propagation}.
	
	Note that, for all $\omega \in \Sbold$, $\underline{\h}^{(\ell)}(\omega) \equiv \arctanh\big( \magn{\omega,\ell,f} \big)$, with initialisation $\underline{\h}^{(0)}(\omega) \equiv B$ for all $\omega \in \Sbold$, satisfies the recursive equation in \eqref{eq:test:1} due to Lemma~\ref{lem:tree-pruning}. By the GKS inequality in Lemma~\ref{lem:GKS}, $\magn{\omega,\ell,f}$ is monotonically increasing in $\ell$. Furthermore, from \eqref{eq:test:1}, for all $\omega \in \Sbold$, and $\ell\geq 1$,
	\[
	B = \underline{\h}^{(0)}(\omega) \leq \underline{\h}^{(\ell)}(\omega) \leq B + D(\emp) < \infty,
	\]
	for all $\ell \geq 1$, almost surely. Therefore, by the monotone convergence theorem, $\underline{\h}^{(\ell)}(\omega)$ converges to some limit $\underline{\h}(\omega)$ almost surely. Hence, $\underline{\h}$ is a fixed point of \eqref{eq:test:1} (from \cite[{Proof of Lemma~2.3}]{DM10}).
	
	Similarly, for all $\omega \in \Sbold$, $\bar{\h}^{(\ell)}(\omega) \equiv \magn{\omega,\ell,+}$ also satisfies \eqref{eq:test:1}, with initialisation $\bar{\h}^{(0)}(\omega) \equiv \infty$ for all $\omega \in \Sbold$. Then, $\bar{\h}^{(\ell)}(\omega)$ is monotonically decreasing and, for all $\omega \in \Sbold$, and $\ell\geq1$,
	\[
	B \leq \bar{\h}^{(\ell)}(\omega) \leq \bar{\h}^{(1)}(\omega) = B + D(\emp) < \infty,
	\]
	almost surely. Therefore, $\bar{\h}^{(\ell)}(\omega)$ also converges to some limit $\bar{\h}(\omega)$ for all $\omega \in \Sbold$.
	Therefore, by \cite[Lemma~3.1]{DGvdH10}, for all $\omega \in \Sbold$,
	\begin{align}
		\label{for:lem:h-uniqueness:8}
		\big| \tanh\big(\underline{\h}^{(\ell)}(\omega)\big) - \tanh\big( \bar{\h}^{(\ell)}(\omega) \big) \big| \leq \big| \magn{\omega,\ell,f} - \magn{\omega,\ell,+} \big|,
	\end{align}
	and hence converges to $0$ as $\ell \to \infty$. This relation holds almost surely, proving that both the limits are unique.
	
	Let $\h$ be a fixed point of \eqref{eq:test:1}, and $\h_\star^{(0)} \equiv \h$. Then, $\h^{(\ell)}_\star$ also converges to some limit $\h_\star$ when applying \eqref{eq:test:1}. Note that for all $\omega \in \Sbold$, $\underline{\h}^{(0)}(\omega) \leq \h_\star^{(0)}(\omega) \leq \bar{\h}^{(0)}(\omega)$. Now, coupling $\underline{\h}^{(\ell)}(\omega)$, $\h_\star^{(\ell)}(\omega)$, and $\bar{\h}^{(\ell)}(\omega)$ with the same sequence of $D_{\ell}(\omega)$ while applying \eqref{eq:test:1}, the inequality is preserved by the GKS inequality as
	\[
	\underline{\h}^{(\ell)}(\omega) \leq \h_\star^{(\ell)}(\omega) \leq \bar{\h}^{(\ell)}(\omega),
	\]
	for all $\omega \in \Sbold$ and $\ell \geq 1$. 
	Since \eqref{for:lem:h-uniqueness:8} holds almost surely and for any realisation of $\h_\star$, for all $\omega \in \Sbold$, $\underline{\h}(\omega)$, $\h_\star(\omega)$, and $\bar{\h}(\omega)$ are equal in distribution, proving that $\h(\omega) = \arctanh\big(\magn{\omega \to}\big)$ is the unique fixed point solution of \eqref{eq:test:1}.
\end{proof} 

\subsection{Convergence of pressure per particle}\label{sec:thermodynamic-limit:subsec:convergence:pressure}

In this section, we state and prove the theorem identifying the explicit thermodynamic limit of the pressure per particle for preferential attachment models. Proof to this theorem follows from the belief propagation proof provided by Dembo, Montannari and Sun in \cite{DMS13}, and some explicit properties of the local limit of the preferential attachment models, the P\'{o}lya point tree.

\medskip

\begin{Theorem}[Thermodynamic limit of the pressure]\label{thm:thermodynamic-limit:pressure}
	Fix $m\geq 2$ and $\delta>0$, and consider $\PA_n(m,\delta)$ (any of the three models (a), (b), and (d) described in Section~\ref{chap:intro:sec:model}). Then for all $0\leq \beta<\infty$ and $B\in\R$, the thermodynamic limit of the pressure exists and is deterministic:
	\eqn{\label{eq:thm:thermodynamic-limit:pressure}
		\psi_n(\beta,B)\overset{\prob}{\to} \varphi(\beta,B)~,}
	where $\varphi(\beta,B)$ is a constant.
	The thermodynamic limit of the pressure satisfies $\varphi(\beta,B)=\varphi(\beta,-B)$ for $B>0$, and $\varphi(\beta,0)=\lim\limits_{B\searrow 0}\varphi(\beta,B)$. For $B>0$, it is given by
	\eqan{\label{eq:thm:thermodynamic-limit:pressure:explicit}
		\varphi(\beta,B)&= \frac{\E[D(\emp)]}{2}\log \cosh(\beta)\nn\\
		&\hspace{0.5cm}-\frac{{\E[D(\emp)]}}{2}\E\Big[\log\big\{1+\tanh(\beta) \tanh(\h(\hat{\emp} 1))\tanh(\h(\hat{\emp})) \big\}\Big]\\
		&\hspace{0.5cm}+\E\Big[ \log\Big( \e^B\prod\limits_{i=1}^{D(\emp)}\big\{1+\tanh(\beta)\tanh(\h(\emp i))\big\}\nn\\
		&\hspace{3.5cm}+\e^{-B}\prod\limits_{i=1}^{D(\emp)}\big\{1-\tanh(\beta)\tanh(\h(\emp i))\big\} \Big) \Big]~,\nn	
	}
	where $\hat{\emp}=(U,\Young)$ and $U\sim\Unif[0,1]$, and $\{\h(\omega):\omega\in \Scal\}$ are independent copies of the fixed point functional $\h^\star(\omega)=\h^\star(\omega,\beta,B)$ of the distributional recursion
	\eqn{\label{eq:distributional-recursion:1}
		\h^{(t+1)}(\omega)\overset{d}{=}B+\sum\limits_{i=1}^{D_t(\omega)}\arctanh \Big(\tanh(\beta)\tanh\big(\h^{(t)}(\omega i)\big)\Big)~,}
	where $\h^{(0)}(\omega)=B$ for all $\omega\in\Scal,$ and $\big(D_t(\omega)\big)_{t\geq 1}$ are i.i.d. random variables with distribution $D(\omega)$, and $\h^{(t)}(\omega)$ is independent of $D_t(\omega)$.
\end{Theorem}


With Lemma~\ref{lem:h:magnetization} established, we now prove Theorem~\ref{thm:thermodynamic-limit:pressure}. The convergence part of the proof follows from \cite[Theorem~1.4]{DGvdH10}, while the identification of the explicit thermodynamic limit follows from \cite[Theorem~1.9]{DMS13}. To derive the explicit formula for the thermodynamic limit of the pressure, we must use two distributional properties of the nodes of the P\'{o}lya point tree. 
By the definition of the P\'{o}lya point tree, it can be shown to be a locally finite tree. Then, Lemma~\ref{lem:h:magnetization} proves that for the P\'{o}lya point tree,
\begin{equation}\label{eq:test:2}
	\h(\emp j) = \magn{\emp j\to},\quad \text{for all}~j\in [D(\emp)]~.
\end{equation}
Further, for any $j\in[D(\emp)]$, let $\h_{\sss -j}(\emp)$ be the fixed point solution to the recursion relation in \eqref{eq:test:1} for $T_{\emp\to\emp j}$. Since $T$ is locally finite, and $T_{\emp\to\emp j}$ is a subgraph of $T$,~$T_{\emp\to\emp j}$ is also locally finite. Therefore, by Lemma~\ref{lem:h:magnetization},
\begin{equation}\label{eq:test:3}
	\h_{\sss -j}(\emp) = \magn{\to\emp j}~.
\end{equation}
The following lemma proves a distributional equivalence related to $\Old$-labelled children of the root in the P\'{o}lya point tree.

\begin{Lemma}[Distributional equivalence of $\Old$-labelled children]
	\label{lem:old:distributional:equivalence}
	Let us denote $\hat{\emp}=(U,\Young)$, where $U\sim\Unif[0,1]$. Then, for any $i\in[m]$,
	\eqn{\label{eq:lem:old:distributional:equivalence}
		\big( \h_{-i}(\emp),\h(\emp i) \big)\overset{d}{=}\big( \h(\hat{\emp}),\h(\hat{\emp} 1) \big)~.}
\end{Lemma}

The next lemma proves a similar distributional equivalence for $\Young$-labelled children of the root.

\begin{Lemma}[Distributional equivalence of $\Young$-labelled children]\label{lem:young:distributional:equivalence}
	Let $\tilde{\emp}$ have label $\Old$ and an age distributed according to $\gamma(a)$, where $\gamma(a)$ is given by
	\eqn{\label{eq:def:gamma(a)} \gamma(a)=\frac{m+\delta}{m}(a^{\chi-1}-1)\quad\mbox{for }a\in[0,1]~.}
	Conditionally on $\tilde{\emp}$ having age $a$, let $\tilde{\emp}\tilde{1}$ have an age distributed according to $f_a(x)$ and label $\Young$, where $f_a(x)$ is defined as
	\eqn{\label{eq:def:f(a)}
		f_a(x)=(1-\chi)\frac{x^{-\chi}}{1-a^{1-\chi}}\one_{\{x\geq a\}}~.}
	Then,
	\eqn{\label{eq:lem:young:distributional:equivalence}
		\big( \h(\tilde{\emp}\tilde{1}),\h(\tilde{\emp}) \big)\overset{d}{=}\big( \h(\hat{\emp}),\h(\hat{\emp} 1) \big)~.
	}
\end{Lemma}

The following lemma establishes a convenient size-biasing argument related to the in-degree of the root:

\begin{Lemma}[Size-biased mixed Poisson]\label{lem:size-biased:poisson}
	Let $\tilde{\emp}$ have label $\Old$ and its age be distributed according to the density $\gamma(a)$ defined in \eqref{eq:def:gamma(a)}. Then, the following size-biasing result regarding the in-degree of the root holds:
	\eqn{\label{eq:lem:size-biased:poisson}
		\ell\prob\Big( d_\emp^{\rm(in)}=\ell\mid A_\emp=a \Big) = \E\big[ d_\emp^{\rm(in)} \big]\prob\Big( d_{\tilde{\emp}}^{\rm(in)}=\ell-1\mid \tilde{\emp}=(a,\Old) \Big)\gamma(a)~.}
\end{Lemma}

We will first prove Theorem~\ref{thm:thermodynamic-limit:pressure} using these lemmas and then proceed to prove the lemmas themselves.

\begin{proof}[Proof of Theorem~\ref{thm:thermodynamic-limit:pressure}]
	The proof of the convergence of the pressure to its thermodynamic limit, as shown in \cite[Theorem~1.4]{DGvdH10}, does not rely on any specific properties of the configuration model, except for the condition $|E_n|/n \leq c$ for some constant $c \in \R$ in \cite[(1.26)]{DGvdH10}, where $E_n$ is the edge set of the configuration model of size $n$. This condition also holds true for preferential attachment models. Since the remainder of the proof is independent of the particular random graph model being considered, it follows for preferential attachment models as well. Alternatively, this part of the proof also follows from \cite[Theorem~1.9]{DMS13}.
	
	We now proceed to prove the explicit expression for the thermodynamic limit of the pressure, which will follow from \cite[Theorem~1.9]{DMS13}. By \cite[(1.6)-(1.8)]{DMS13},
	\begin{align}\label{for:thm:pressure-limit:1}
		&\varphi(\beta,B) \\
		= &\ \mathbb{E}_{\mu}\left[ \log\left\{ \sum\limits_{\sigma\in\{0,1\}} \exp\left(B\sigma\right) \prod\limits_{j=1}^{D(\emp)} \left( \sum\limits_{\sigma_j\in\{0,1\}} \exp\left(\beta\sigma\sigma_j\right) \nu_{\emp j\to\emp}(\sigma_j) \right) \right\}\right.\nonumber \\
		&\hspace{2cm}\left.-\frac{1}{2}\sum\limits_{j=1}^{D(\emp)} \log\left\{ \sum\limits_{\sigma,\sigma_j\in\{0,1\}} \exp\left(\beta\sigma\sigma_j\right) \nu_{\emp j\to\emp}(\sigma_j) \nu_{\emp\to\emp j}(\sigma) \right\} \right]~.\nn
	\end{align}
	To derive an expression similar to \eqref{eq:thm:thermodynamic-limit:pressure:explicit} from \eqref{for:thm:pressure-limit:1}, we simplify the terms on the right-hand side of \eqref{for:thm:pressure-limit:1} step by step.
	We start by writing the first term as
	\eqan{\label{for:thm:pressure-limit:02}
		&\E_\mu\Big[ \log\Big\{ \e^B\prod\limits_{j=1}^{D(\emp)}\Big( \e^\beta\nu_{\emp j\to\emp}(+1)+\e^{-\beta}\nu_{\emp j\to\emp}(-1) \Big)\nn\\
		&\hspace{4cm}+\e^{-B}\prod\limits_{j=1}^{D(\emp)}\Big( \e^{-\beta}\nu_{\emp j\to\emp}(+1)+\e^{\beta}\nu_{\emp j\to\emp}(-1) \Big) \Big\} \Big]\nn\\
		=&\E_\mu\Bigg[ \log\Bigg\{ \e^B\prod\limits_{j=1}^{D(\emp)}\frac{2\Big( \e^\beta\nu_{\emp j\to\emp}(+1)+\e^{-\beta}\nu_{\emp j\to\emp}(-1) \Big)}{\e^{\beta}+\e^{-\beta}}\\
		&\hspace{2.5cm}+\e^{-B}\prod\limits_{j=1}^{D(\emp)}\frac{2\Big( \e^{-\beta}\nu_{\emp j\to\emp}(+1)+\e^{\beta}\nu_{\emp j\to\emp}(-1) \Big)}{\e^{\beta}+\e^{-\beta}} \Bigg\}\nn\\
		&\hspace{8cm}+D(\emp)\log\Bigg\{ \frac{\e^{\beta}+\e^{-\beta}}{2} \Bigg\} \Bigg]~.\nn
	}
	Using that $\big(2\big( \e^\beta\nu_{\emp j\to\emp}(+1)+\e^{-\beta}\nu_{\emp j\to\emp}(-1) \big)\big)/\big(\e^{\beta}+\e^{-\beta}\big)$ equals $\big( 1+\tanh(\beta)\big(2\nu_{\emp j\to\emp}(+1)-1 \big) \big)$, we arrive at
	\eqan{\label{for:thm:pressure-limit:2}
		&\E_\mu[D(\emp)]\log\cosh(\beta)\nn\\
		&\hspace{1cm}+\E\left[ \log\Bigg\{ \e^B\prod\limits_{j=1}^{D(\emp)}\Big( 1+\tanh(\beta)\big(2\nu_{\emp j\to\emp}(+1)-1 \big) \Big)\right.\nn\\
		&\hspace{3cm}\left. +\e^{-B}\prod\limits_{j=1}^{D(\emp)}\Big( 1-\tanh(\beta)\big(2\nu_{\emp j\to\emp}(+1)-1 \big) \Big) \Bigg\} \right]~.
	}
	Note that, $\magn{\emp j\to}=2\nu_{\emp j\to\emp}(+1)-1$ by the definition of the magnetization. Therefore, using Lemma~\ref{lem:h:magnetization}, we simplify \eqref{for:thm:pressure-limit:2} as
	\eqan{\label{for:thm:pressure-limit:3}
		&\E_\mu[D(\emp)]\log\cosh(\beta)\nn\\
		&\hspace{1cm}+\E\Bigg[ \log\Big\{ \e^B\prod\limits_{j=1}^{D(\emp)}\Big( 1+\tanh(\beta)\tanh\big(\h(\emp j)\big) \Big)\nn\\
		&\hspace{3cm} +\e^{-B}\prod\limits_{j=1}^{D(\emp)}\Big( 1-\tanh(\beta)\tanh\big(\h(\emp j)\big) \Big) \Big\} \Bigg]~.
	}
	Similarly, we simplify the second term on the RHS of \eqref{for:thm:pressure-limit:1} as
	\eqan{\label{for:thm:pressure-limit:4}
		&\E_\mu\Bigg[ \sum\limits_{j=1}^{D(\emp)}\log\Big\{ \sum\limits_{\sigma,\sigma_j\in\{0,1\}}\e^{\beta\sigma\sigma_j}\nu_{\emp j\to\emp}(\sigma_j)\nu_{\emp\to\emp j}(\sigma) \Big\} \Bigg]\nn\\
		=&\E_{\mu}\Bigg[ \sum\limits_{j=1}^{D(\emp)} \log\Bigg\{ \e^{\beta}\nu_{\emp j\to\emp}(+1)\nu_{\emp \to\emp j}(+1)+\e^{-\beta}\nu_{\emp j\to\emp}(-1)\nu_{\emp \to\emp j}(+1)\nn\\
		&\hspace{3cm}+\e^{-\beta}\nu_{\emp j\to\emp}(+1)\nu_{\emp \to\emp j}(-1)+\e^{\beta}\nu_{\emp j\to\emp}(-1)\nu_{\emp \to\emp j}(-1) \Bigg\} \Bigg]\nn\\
		=&\E_{\mu}\Bigg[ \sum\limits_{j=1}^{D(\emp)} \log\Bigg( \frac{2}{\e^{\beta}+\e^{-\beta}}\Big\{ \e^{\beta}\nu_{\emp j\to\emp}(+1)\nu_{\emp \to\emp j}(+1)\nn\\
		&\hspace{2.5cm}+\e^{-\beta}\nu_{\emp j\to\emp}(-1)\nu_{\emp \to\emp j}(+1)+\e^{-\beta}\nu_{\emp j\to\emp}(+1)\nu_{\emp \to\emp j}(-1)\nn\\
		&\hspace{4.5cm}+\e^{\beta}\nu_{\emp j\to\emp}(-1)\nu_{\emp \to\emp j}(-1) \Big\}\Bigg) \Bigg]\nn\\
		&\hspace{7cm}+\E_\mu[D(\emp)]\log\Big\{ \frac{\e^{\beta}+\e^{-\beta}}{2} \Big\}\nn\\
		=&\E_{\mu}[D(\emp)]\log\cosh(\beta)\\
		&\hspace{1cm}+\E_{\mu}\Bigg[ \sum\limits_{j=1}^{D(\emp)}\log\Big\{ 1+\tanh(\beta)\big( 2\nu_{\emp j\to\emp}(+1)-1 \big)\big( 2\nu_{\emp\to\emp j}(+1)-1 \big) \Big\} \Bigg]~.\nn
	}
	Again using definition of $\magn{\to \emp j}$ and Lemma~\ref{lem:h:magnetization}, we can rewrite the LHS of \eqref{for:thm:pressure-limit:4} as
	\eqan{\label{for:thm:pressure-limit:5}
		&\E_\mu[D(\emp)]\log\cosh(\beta)\nn\\
		&\hspace{1cm}+\E_{\mu}\Bigg[ \sum\limits_{j=1}^{D(\emp)}\log\Big\{ 1+\tanh(\beta)\tanh\big(\h_{-j}(\emp)\big)\tanh\big(\h(\emp j)\big) \Big\} \Bigg]~.
	}
	Now, plugging in the simplified forms obtained in \eqref{for:thm:pressure-limit:3} and \eqref{for:thm:pressure-limit:5} in \eqref{for:thm:pressure-limit:1}, we obtain
	\eqan{\label{for:thm:pressure-limit:6}
		&\varphi(\beta,B)\\
		=& \frac{\E_\mu[D(\emp)]}{2}\log\cosh(\beta)\nn\\
		&\hspace{0.5cm}-\frac{1}{2}\E_{\mu}\Bigg[ \sum\limits_{j=1}^{D(\emp)}\log\Big\{ 1+\tanh(\beta)\tanh\big(\h_{-j}(\emp)\big)\tanh\big(\h(\emp j)\big) \Big\} \Bigg]\nn\\
		&\hspace{0.75cm}+\E_\mu\Bigg[ \log\Big\{ \e^B\prod\limits_{j=1}^{D(\emp)}\Big( 1+\tanh(\beta)\tanh\big(\h(\emp j)\big) \Big)\nn\\
		&\hspace{3cm}+\e^{-B}\prod\limits_{j=1}^{D(\emp)}\Big( 1-\tanh(\beta)\tanh\big(\h(\emp j)\big) \Big) \Big\} \Bigg]~.\nn
	}
	Next, we use the distributional properties of the nodes in the P\'{o}lya point tree, as outlined in Lemmas~\ref{lem:old:distributional:equivalence} and \ref{lem:young:distributional:equivalence}, to conclude that the RHS of \eqref{for:thm:pressure-limit:6} further simplifies to the RHS of \eqref{eq:thm:thermodynamic-limit:pressure:explicit}. Note that $D(\emp)=m+d_\emp^{\rm(in)}$, where $d_\emp^{\rm(in)}$ is the number of $\Young$-labelled children of the root in the P\'{o}lya point tree. Therefore, the second term on the RHS of \eqref{for:thm:pressure-limit:6} can be divided into two parts: one comprising the contribution from the $\Old$-labelled children of the root, and the other comprising the contribution from the $\Young$-labelled children of the root.
	\eqan{\label{for:thm:pressure-limit:7}
		&\E_{\mu}\Bigg[ \sum\limits_{j=1}^{D(\emp)}\log\Big\{ 1+\tanh(\beta)\tanh\big(\h_{-j}(\emp)\big)\tanh\big(\h(\emp j)\big) \Big\} \Bigg]\\
		=&\E_{\mu}\Bigg[ \sum\limits_{j=1}^{m}\log\Big\{ 1+\tanh(\beta)\tanh\big(\h_{-j}(\emp)\big)\tanh\big(\h(\emp j)\big) \Big\} \Bigg]\nn\\
		&\hspace{2cm}+\E_{\mu}\Bigg[ \sum\limits_{j=m+1}^{m+d_\emp^{\rm(in)}}\log\Big\{ 1+\tanh(\beta)\tanh\big(\h_{-j}(\emp)\big)\tanh\big(\h(\emp j)\big) \Big\} \Bigg]~.\nn
	}
	Using Lemma~\ref{lem:old:distributional:equivalence}, we simplify the first sum in the RHS of \eqref{for:thm:pressure-limit:7} as
	\eqn{\label{for:thm:pressure-limit:8}
		m\E_{\mu}\Big[ \log\Big\{ 1+\tanh(\beta)\tanh\big(\h(\hat{\emp})\big)\tanh\big(\h(\hat{\emp} 1)\big) \Big\} \Big]~.}
	Next, we simplify the second term on the RHS of \eqref{for:thm:pressure-limit:7}. From the definition of the P\'{o}lya point tree, we know that $d_\emp^{\rm(in)}$ is a mixed-Poisson random variable with intensity parameter $\Gamma\big(A_\emp^{\chi-1}-1\big)$, where $\Gamma \sim \text{Gamma}(m+\delta, 1)$ and $A_\emp \sim \text{Unif}(0,1)$. Furthermore, conditionally on $\{A_\emp = a\}$ and $d_\emp^{\rm(in)} = \ell$, the ages of $\{\emp(m+1), \ldots, \emp(m+\ell)\}$ are i.i.d.\ with density $f_a$ defined in \eqref{eq:def:f(a)}. Conditionally on $\{A_\emp = a\}$ and $d_{\emp}^{\rm(in)} = \ell$, the age of $\emp\tilde{1}$ follows the distribution $f_a$ and has the label $\Young$.
	Therefore, second term in the RHS of \eqref{for:thm:pressure-limit:7} simplifies as
	\eqan{\label{for:thm:pressure-limit:9}
		&\int\limits_0^1\sum\limits_{\ell=1}^\infty \prob\big( d_\emp^{\rm(in)}=\ell\mid A_\emp=a \big)\\
		&\hspace{-0.5cm}\times\E_{\mu}\Big[ \sum\limits_{j=m+1}^{m+\ell}\log\big\{ 1+\tanh(\beta)\tanh\big(\h_{-j}(\emp)\big)\tanh\big(\h(\emp j)\big) \big\}\mid A_\emp=a,d_\emp^{\rm(in)}=\ell \Big]\,da\nn\\
		=&\int\limits_0^1\sum\limits_{\ell=1}^\infty \E_{\mu}\Bigg[ \log\Big\{ 1+\tanh(\beta)\tanh\big(\h_{-\tilde{1}}(\emp)\big)\tanh\big(\h(\emp \tilde{1})\big) \Big\}\mid A_\emp=a,d_\emp^{\rm(in)}=\ell \Bigg]\nn\\
		&\hspace{8cm}\times\ell\prob\big( d_\emp^{\rm(in)}=\ell\mid A_\emp=a \big)\,da~.\nn
	}
	Note that, conditionally on $\{A_\emp = a, d_\emp^{\rm(in)} = \ell\}$, $\h(\emp\tilde{1})$ and $\h_{-\tilde{1}}(\emp)$ are independent. Let $\tilde{\emp}$ have label $\Old$ and age with density $\gamma(a)$. Conditionally on $A_{\tilde{\emp}} = a$, $\tilde{\emp}\tilde{1}$ has label $\Young$ and its age is distributed according to $f_a$. From the definitions in \eqref{eq:test:1}, we have
	
	\begin{equation}
		\label{for:thm:pressure-limit:10}
		\begin{aligned}
			\h_{-\tilde{1}}(\emp) \mid \{A_\emp = a, d_\emp^{\rm(in)} = \ell\} ~ &\overset{d}{=} ~ \h(\tilde{\emp}) \mid \{\tilde{\emp} = (a, \Old), d_{\tilde{\emp}}^{\rm(in)} = \ell - 1\}, \\
			\h(\emp\tilde{1}) \mid \{A_\emp = a, d_\emp^{\rm(in)} = \ell\} ~ &\overset{d}{=} ~ \h(\tilde{\emp}\tilde{1}) \mid \{\tilde{\emp} = (a, \Old), d_{\tilde{\emp}}^{\rm(in)} = \ell - 1\}.
		\end{aligned}
	\end{equation}
	
	Now, conditionally on $\tilde{\emp} = (a, \Old)$, drawing $\h(\tilde{\emp})$ and $\h(\tilde{\emp}\tilde{1})$ independently leads to the distributional equality
	\begin{align}
		\label{for:thm:pressure-limit:11}
		&\big(\h(\tilde{\emp}), \h(\tilde{\emp}\tilde{1})\big) \mid \{\tilde{\emp} = (a, \Old), d_{\tilde{\emp}}^{\rm(in)} = \ell - 1\}\nn\\
		&\hspace{2cm}\overset{d}{=} ~ \big(\h_{-\tilde{1}}(\emp), \h(\emp\tilde{1})\big) \mid \{A_\emp = a, d_\emp^{\rm(in)} = \ell\}.
	\end{align}
	
	Hence, using \eqref{for:thm:pressure-limit:11} and Lemma~\ref{lem:size-biased:poisson}, \eqref{for:thm:pressure-limit:10} can be simplified as
	\eqan{\label{for:thm:pressure-limit:12}
		&\E_{\mu}\Bigg[ \sum\limits_{j=m+1}^{m+d_\emp^{\rm(in)}}\log\Big\{ 1+\tanh(\beta)\tanh\big(\h_{-j}(\emp)\big)\tanh\big(\h(\emp j)\big) \Big\} \Bigg]\\
		=&\int\limits_0^1\sum\limits_{\ell=1}^\infty \E\big[ d_{\emp}^{\rm(in)} \big]\prob\big( d_{\tilde{\emp}}^{\rm(in)}=\ell-1\mid \tilde{\emp}=(a,\Old) \big)\gamma(a)\nn\\
		&\hspace{-0.5cm}\times \E_{\mu}\Bigg[ \log\Big\{ 1+\tanh(\beta)\tanh\big(\h(\tilde{\emp}\tilde{1})\big)\tanh\big(\h( \tilde{\emp})\big) \Big\}\mid \tilde{\emp}=(a,\Old),d_{\tilde{\emp}}^{\rm(in)}=\ell-1 \Bigg]\,da\nn\\
		=&\E\big[ d_{\emp}^{\rm(in)} \big]\E_{\mu}\Big[ \log\Big\{ 1+\tanh(\beta)\tanh\big(\h(\tilde{\emp}\tilde{1})\big)\tanh\big(\h( \tilde{\emp})\big) \Big\} \Big]~.
	}
	Now, by Lemma~\ref{lem:young:distributional:equivalence} and \eqref{for:thm:pressure-limit:7},\eqref{for:thm:pressure-limit:8} and \eqref{for:thm:pressure-limit:12}, we conclude that
	\eqan{\label{for:thm:pressure-limit:13}
		&\E_{\mu}\Bigg[ \sum\limits_{j=1}^{D(\emp)}\log\Big\{ 1+\tanh(\beta)\tanh\big(\h_{-j}(\emp)\big)\tanh\big(\h(\emp j)\big) \Big\} \Bigg]\nn\\
		=&\Big(m+\E_{\mu}\big[ d_{\emp}^{\rm(in)} \big]\Big)\E_{\mu}\Big[ \log\Big\{ 1+\tanh(\beta)\tanh\big(\h({\emp}{1})\big)\tanh\big(\h( \hat{\emp})\big) \Big\} \Big]\nn\\
		=&\E_\mu\big[ D(\emp) \big]\E_{\mu}\Big[ \log\Big\{ 1+\tanh(\beta)\tanh\big(\h({\emp}{1})\big)\tanh\big(\h( \hat{\emp})\big) \Big\} \Big]~.}
	Therefore, substituting \eqref{for:thm:pressure-limit:13} in \eqref{for:thm:pressure-limit:6}, we obtain the expression in \eqref{eq:thm:thermodynamic-limit:pressure:explicit} for $\varphi(\beta,B)$.
\end{proof}

Now that we have proved Theorem~\ref{thm:thermodynamic-limit:pressure} subject to Lemma~\ref{lem:old:distributional:equivalence}, \ref{lem:young:distributional:equivalence} and \ref{lem:size-biased:poisson}, we provide the proof to these lemmata one by one. First, we prove Lemma~\ref{lem:old:distributional:equivalence}:
\begin{proof}[Proof of Lemma~\ref{lem:old:distributional:equivalence}]
	Note that in the P\'{o}lya point tree, ages of the $\Old$ labeled children are an exchangeable sequence of random variables, i.e., for any permutation $\Upsilon:\N\to\N$,
	\eqn{\label{for:lem:old:distributional:equivalence:1}
		\big( A_{\emp 1},\ldots,A_{\emp m} \big)\overset{d}{=}\big( A_{\emp \Upsilon(1)},\ldots,A_{\emp \Upsilon(m)} \big)~.}
	Fix any $i\in[m]$ and choose 
	\[
	\Upsilon_i(x)=\begin{cases}
		i \quad \mbox{if }x=1,\\
		1 \quad \mbox{if }x=i,\\
		x \quad \mbox{otherwise.} 
	\end{cases}
	\]
	Therefore, from the distributional recursion in \eqref{eq:test:1} and using the definition of $\h_{\sss -i}(\emp)$ in Section~\ref{sec:thermodynamic-limit:subsec:belief-propagation}, we obtain
	\eqan{\label{for:lem:old:distributional:equivalence:2}
		&\big(\h(\emp i),\h_{-i}(\emp)\big)\nn\\
		&\hspace{0.5cm}\overset{d}{=}\Big( \h(\emp \Upsilon_i(1)), B+\sum\limits_{j=2}^{D(\emp)}\arctanh\big\{\tanh(\beta)\tanh\big(\h(\emp \Upsilon_i(j))\big)\big\} \Big)~.}
	Further, for any $\omega\in\Scal$, the distribution of $h(\omega)$ is dependent only on $\omega$. Hence using \eqref{for:lem:old:distributional:equivalence:1},
	\eqan{\label{for:lem:old:distributional:equivalence:3}
		&\Big( \h(\emp \Upsilon_i(1)), B+\sum\limits_{j=2}^{D(\emp)}\arctanh\big\{\tanh(\beta)\tanh\big(\h(\emp \Upsilon_i(j))\big)\big\} \Big)\nn\\
		&\hspace{2cm}\overset{d}{=}\Big( \h(\emp 1), B+\sum\limits_{j=2}^{D(\emp)}\arctanh\big\{\tanh(\beta)\tanh\big(\h(\emp j)\big)\big\} \Big)~.}
	Note that $\hat{\emp}$ is labelled $\Young$, and from the definition of the P\'{o}lya point tree, we obtain that the root and $\Young$-labelled nodes of the P\'{o}lya point tree are i.i.d. Furthermore, the number of $\Old$-labelled children of the root is one more than the number of $\Young$-labelled children in the P\'{o}lya point tree. Therefore, with $\hat{\emp}$ having the label $\Young$ and age $A_\emp \sim \text{Unif}[0,1]$, we obtain
	\begin{align}
		\label{for:lem:old:distributional:equivalence:4}
		B + \sum_{j=2}^{D(\emp)} \arctanh\left\{ \tanh(\beta) \tanh\left( \h(\emp j) \right) \right\} \overset{d}{=} \h(\hat{\emp}).
	\end{align}
	Furthermore, conditionally on $\{A_\emp = a\}$, $\h(\emp1)$ and $\{\h(\emp2), \ldots, \h(\emp D(\emp))\}$ are independent, and $\h(\emp 1)$ is equal in distribution with $\h(\hat{\emp}1)$. Therefore, conditionally on $\{A_\emp = a\}$, drawing $\h(\hat{\emp})$ and $\h(\hat{\emp}1)$ independently leads to
	\begin{align}
		\label{for:lem:old:distributional:equivalence:5}
		&\left( \h(\emp1), B + \sum_{j=2}^{D(\emp)} \arctanh\left\{ \tanh(\beta) \tanh\left( \h(\emp j) \right) \right\} \right)\nn\\
		&\hspace{1cm}\overset{d}{=} \left( \h(\hat{\emp}1), \h(\hat{\emp}) \right).
	\end{align}
	Hence, Lemma~\ref{lem:old:distributional:equivalence} follows immediately from \eqref{for:lem:old:distributional:equivalence:5}.
\end{proof}
Next, we prove Lemma~\ref{lem:young:distributional:equivalence}:
\begin{proof}
	To prove this lemma, we show that the joint distributions of the types of $\big( \tilde{\emp}\tilde{1},\tilde{\emp} \big)$ and $\big( \hat{\emp},\hat{\emp}1 \big)$ are equal. First, we compute the joint distribution of the types of $\big(\hat{\emp},\hat{\emp}1\big)$ as
	\eqan{\label{for:lem:distributional:equivalence:1}
		\prob\big( A_{\hat{\emp}}\leq x,A_{\hat{\emp}{1}}\leq y \big)&=~\int\limits_0^x\prob\big( A_{\hat{\emp}{1}}\leq y\mid A_{\hat{\emp}}=t \big)\,dt \nn\\
		&=~	\int\limits_0^x	\prob\big( U^{1/\chi}t\leq y\mid A_{\hat{\emp}}=t \big)\,dt~.
	}
	Note that $U$ and $A_{\hat{\emp}}$ are independent. Hence, upon further simplification,
	\eqn{\label{for:lem:distributional:equivalence:2}
		\prob\big( A_{\hat{\emp}}\leq x,A_{\hat{\emp}{1}}\leq y \big)=~\begin{cases}
			x &\mbox{if }x<y~,\\
			\frac{1}{1-\chi}y^{\chi}x^{1-\chi}-\frac{\chi}{1-\chi}y\hspace{3pt} &\mbox{if }x\geq y~.
	\end{cases}}
	Next, we compute the joint distribution of $\big(\tilde{\emp}\tilde{1},\tilde{\emp}\big)$ as
	\eqan{\label{for:lem:distributional:equivalence:3}
		\prob\big(A_{\tilde{\emp}\tilde{1}}\leq x,A_{\tilde{\emp}}\leq y\big)&=~\int\limits_0^y \prob\big(A_{\tilde{\emp}\tilde{1}}\leq x\mid A_{\tilde{\emp}}=t \big)\gamma(t)\,dt\nn\\
		&=~ \int\limits_0^y \int\limits_0^x f_t(s)\gamma(t)\,ds\,dt~.
	}	
	The last equality in \eqref{for:lem:distributional:equivalence:3} is due to the fact that conditionally on $\{\tilde{\emp}=(t,\Old)\},~A_{\tilde{\emp}\tilde{1}}$ has density $f_t(\cdot)$. Therefore, after simplifying \eqref{for:lem:distributional:equivalence:3},
	\eqn{\label{for:lem:distributional:equivalence:4}
		\prob\big(A_{\tilde{\emp}\tilde{1}}\leq x,A_{\tilde{\emp}}\leq y\big)=\begin{cases}
			x &\mbox{if }x<y~,\\
			\frac{1}{1-\chi}y^{\chi}x^{1-\chi}-\frac{\chi}{1-\chi}y\hspace{3pt} &\mbox{if }x\geq y~,
	\end{cases}}
	proving the distributional equivalence of $\big( \tilde{\emp}\tilde{1},\tilde{\emp} \big)$ and $\big( \hat{\emp},\hat{\emp}1 \big)$. Since $\{\h(\omega):\omega\in\Scal\}$ are drawn independently, the joint distributions of $\big( \h(\tilde{\emp}\tilde{1}),\h(\tilde{\emp}) \big)$ and $\big( \h(\hat{\emp}),\h(\hat{\emp}1) \big)$ are equal.
\end{proof}
Lastly, we address the proof of Lemma~\ref{lem:size-biased:poisson}. This lemma involves a size-biasing argument for the in-degree of the root of the P\'{o}lya point tree. Since the in-degree of the P\'{o}lya point tree follows a mixed-Poisson distribution, the size-biasing argument introduces an additional term alongside its expectation:

\begin{proof}[Proof of Lemma~\ref{lem:size-biased:poisson}]
	From the definition of the P\'{o}lya point tree, we know that $d_\emp^{\rm(in)}$ is a mixed-Poisson random variable with a mixing distribution $\Gamma_{\emp}\big(A_\emp^{\chi-1}-1\big)$, where $A_\emp$ is the age of the root of the P\'{o}lya point tree. Therefore, conditionally on $A_\emp = a$, $d_\emp^{\rm(in)}$ is a mixed Poisson random variable with a mixing distribution $\Gamma\big(a^{\chi-1}-1\big)$, where $\Gamma$ is a Gamma random variable with parameters $m + \delta$ and $1$. Thus,
	\eqan{\label{for:lem:size-bias:1}
		\ell \prob\big( d_\emp^{\rm(in)}=\ell \mid A_\emp=a \big)&=\ell \prob\big(\Poi\big( \Gamma(a^{\chi-1}-1) \big)=\ell  \big)\\
		&=\int\limits_0^\infty \ell \prob\big(\Poi\big( \z(a^{\chi-1}-1) \big)=\ell  \big)\z^{m+\delta-1}\frac{\e^{-\z}}{\Gamma(m+\delta)}\,d\z~,\nn
	}
	where $\Gamma(c)$ is the gamma function evaluated at $c$ for any $c \in \mathbb{R}^+$. From a direct computation of the probability mass function of a Poisson random variable with a fixed parameter $\lambda$, the following holds for all $\ell \in \mathbb{N}$:
	
	\begin{equation}
		\label{for:lem:size-bias:2}
		\ell \, \prob(\Poi(\lambda) = \ell) = \lambda \, \prob(\Poi(\lambda) = \ell - 1).
	\end{equation}
	
	Using \eqref{eq:def:gamma(a)} and\eqref{for:lem:size-bias:2}, RHS of \eqref{for:lem:size-bias:1} can be simplified as
	\eqan{\label{for:lem:size-bias:3}
		&\int\limits_0^\infty \ell \prob\big(\Poi\big( \z(a^{\chi-1}-1) \big)=\ell  \big)\z^{m+\delta-1}\frac{\e^{-\z}}{\Gamma(m+\delta)}\,d\z\\
		=&\int\limits_0^\infty \prob\big(\Poi\big( \z(a^{\chi-1}-1) \big)=\ell -1 \big)\frac{m}{m+\delta}\big(a^{\chi-1}-1 \big)\z^{m+\delta}\frac{\e^{-\z}}{\Gamma(m+\delta)}\,d\z~.\nn
	}
	Since $d_\emp^{\rm(in)}$ is a mixed-Poisson random variable with mixing distribution $\Gamma_{\emp}\big(A_\emp^{\chi-1}-1\big)$, 
	\eqan{\label{for:lem:size-bias:4}
		\E\big[d_\emp^{\rm(in)}\big]&=\E\big[\Gamma_{\emp}\big(A_\emp^{\chi-1}-1\big)\big]=\E[\Gamma_{\emp}]\E\big[A_\emp^{\chi-1}-1\big]\nn\\
		&=(m+\delta)(1/\chi-1)=m~.}
	Therefore, using the identity $x\Gamma(x)=\Gamma(x+1)$ and \eqref{for:lem:size-bias:4}, we simplify the RHS of \eqref{for:lem:size-bias:3} further as
	\eqan{\label{for:lem:size-bias:5}
		&\int\limits_0^\infty \prob\big(\Poi\big( \z(a^{\chi-1}-1) \big)=\ell -1 \big)\frac{m}{m+\delta}\gamma(a)\z^{m+\delta}\frac{\e^{-\z}}{\Gamma(m+\delta)}\,d\z\nn\\
		=&\E\big[ d_\emp^{\rm(in)} \big]\gamma(a)\int\limits_0^\infty \prob\big(\Poi\big( \z(a^{\chi-1}-1) \big)=\ell -1 \big)\z^{m+\delta}\frac{\e^{-\z}}{\Gamma(m+\delta+1)}\,d\z\nn\\
		=& \E\big[ d_\emp^{\rm(in)} \big]\gamma(a) \prob\big(\Poi\big( \Gamma_{\rm{in}}(m+1)(a^{\chi-1}-1) \big)=\ell -1 \big)~,
	}
	where $\Gamma_{\rm{in}}(m+1)$ is a $\rm{Gamma}$ random variable with parameters $m+\delta+1$ and $1$. Note that, based on the construction of the P\'{o}lya point tree, conditionally on $\tilde{\emp} = (a, \Old)$, $d_{\tilde{\emp}}^{\rm(in)}$ is a mixed Poisson random variable with mixing distribution $\Gamma_{\tilde{\emp}}(a^{\chi-1} - 1)$, where $\Gamma_{\tilde{\emp}}$ is a $\rm{Gamma}$ random variable with parameters $m + \delta + 1$ and $1$. Therefore,
	\begin{equation}
		\label{for:lem:size-bias:6}
		d_{\tilde{\emp}}^{\rm(in)} \mid \big\{ \tilde{\emp} = (a, \Old) \big\} \quad \overset{d}{=} \quad \Poi\big( \Gamma_{\rm{in}}(m+1)(a^{\chi-1} - 1) \big)~.
	\end{equation}
	Hence, the lemma follows immediately from \eqref{for:lem:size-bias:1}, \eqref{for:lem:size-bias:3}, \eqref{for:lem:size-bias:5}, and \eqref{for:lem:size-bias:6}.
\end{proof}

\subsection{Convergence of thermodynamic quantities}\label{sec:thermodynamic-limit:subsec:convergence:thermodynamic-quantity}

In this section, we state and prove the thermodynamic limits of two thermodynamic quantities, namely internal energy and magnetisation for the Ising model on preferential attachment models:

\medskip

\begin{Theorem}[Thermodynamic quantities]\label{thm:limit:thermodynamic-quantities}
	Let $\{G_n\}_{n\geq 1}$ be a sequence of preferential attachment graphs with parameters $m$ and $\delta$. Then, for all $\beta\geq 0$ and $B\in\R$, 
	\begin{enumerate}
		\item[(a)]\label{thm:magnetization} {\bf Magnetisation.}
		Let $M_n(\beta,B)=\frac{1}{n}\langle\sum_{i\in[n]}\sigma_i\rangle_{\mu_n}$ be the magnetisation per vertex. Then, its thermodynamic limit exists and is given by
		\eqn{\label{eq:def:magnetization}
			M_n(\beta,B)\overset{\prob}{\to}M(\beta,B)=\frac{\partial}{\partial B}\varphi(\beta,B)=\E[\tanh(\h(\emp))]~.
		}
		\item[(b)]\label{thm:internal:energy} {\bf Internal energy.}
		Let $U_n(\beta,B)=-\frac{1}{n} \langle\sum_{(u,v)\in E_n} \sigma_u\sigma_v\rangle_{\mu_n}$ denote the internal energy per vertex. Then, its thermodynamic limit exists and is given by
		\eqan{\label{eq:def:internal:energy}
			U_n(\beta,B)\overset{\prob}{\to}U(\beta,B)=&-\frac{\partial}{\partial\beta}\varphi(\beta,B)\\
			=&-m\E\left[ \frac{\tanh(\beta)+\tanh(\h(\hat{\emp}))\tanh(\h(\hat{\emp} 1))}{1+\tanh(\beta)\tanh(\h(\hat{\emp}))\tanh(\h(\hat{\emp} 1))} \right]~.\nn
		}
	\end{enumerate}
\end{Theorem}

We now proceed to prove Theorem~\ref{thm:limit:thermodynamic-quantities}. To prove this theorem, we utilise the fact that $\varphi(\beta, B)$ is a convex function with respect to both $\beta$ and $B$. The following lemma establishes this convexity property of $\varphi$:

\begin{Lemma}[Convexity of $\psi_n(\beta, B)$]\label{lem:convexity:pressure-limit}
	For any $\beta \in\R$ and $B \in \mathbb{R}$, let $\psi_n(\beta, B)$ be as defined in Theorem~\ref{thm:thermodynamic-limit:pressure}. Then, both $\beta \mapsto \psi_n(\beta, B)$ and $B \mapsto \psi_n(\beta, B)$ are convex functions of $\beta$ and $B$, respectively.
\end{Lemma}

To obtain the explicit expressions in Theorem~\ref{thm:limit:thermodynamic-quantities}, we use the following lemma. This lemma relies on the fact that preferential attachment models locally converge in probability to the P\'{o}lya point tree.

\begin{Lemma}\label{lem:limiting:edge-measure}
	Let $E_n$ denote the edge set of the preferential attachment model of size $n$ (denoted by $G_n$), and let $\mu_n$ denote the Boltzmann distribution on $G_n$. Then,
	\begin{align}
		\frac{1}{|E_n|}\sum_{\{i, j\} \in E_n}\langle \sigma_i \sigma_j \rangle_{\mu_n} &\overset{\prob}{\to} \E\left[ \frac{\tanh(\beta) + \tanh(\h(\hat{\emp})) \tanh(\h(\hat{\emp }1))}{1 + \tanh(\beta) \tanh(\h(\hat{\emp})) \tanh(\h(\hat{\emp }1))} \right]~, \label{eq:lem:limiting:edge-measure-1} \\
		\text{and} \qquad \frac{1}{n}\sum_{v \in [n]}\langle \sigma_v \rangle_{\mu_n} &\overset{\prob}{\to} \E[\tanh(\h(\emp))]~. \label{eq:lem:limiting:edge-measure-2}
	\end{align}
\end{Lemma}

With Lemmas~\ref{lem:convexity:pressure-limit} and \ref{lem:limiting:edge-measure} in place, we now proceed to prove Theorem~\ref{thm:limit:thermodynamic-quantities}:

\begin{proof}[Proof of Theorem~\ref{thm:limit:thermodynamic-quantities}]
	First, we prove the theorem for the internal energy. 
	To prove the convergence part, we use the convexity argument from Lemma~\ref{lem:convexity:pressure-limit}. Since $\beta \mapsto \psi_n(\beta, B)$ is a convex function in $\beta$, for any $\vep > 0$,
	\begin{align}\label{for:prop:limiting:ie:01}
		\frac{1}{\vep}\left[ \psi_n(\beta, B) - \psi_n(\beta - \vep, B) \right] &\leq \frac{\partial}{\partial \beta} \psi_n(\beta, B) \nn\\
		&\hspace{1cm}\leq \frac{1}{\vep}\left[ \psi_n(\beta + \vep, B) - \psi_n(\beta, B) \right]~.
	\end{align}
	Taking limits as $n \to \infty$ in \eqref{for:prop:limiting:ie:01}, and using Theorem~\ref{thm:thermodynamic-limit:pressure}, we obtain, for any $\vep > 0$,
	\begin{align}
		&\lim_{n \to \infty} \frac{1}{\vep} \left[ \psi_n(\beta, B) - \psi_n(\beta - \vep, B) \right] \overset{\prob}{\to} \frac{1}{\vep} \left[ \varphi(\beta, B) - \varphi(\beta - \vep, B) \right], \label{for:prop:limiting:ie:02}\\
		\text{and}~ &\lim_{n \to \infty} \frac{1}{\vep} \left[ \psi_n(\beta + \vep, B) - \psi_n(\beta, B) \right] \overset{\prob}{\to} \frac{1}{\vep} \left[ \varphi(\beta + \vep, B) - \varphi(\beta, B) \right]. \label{for:prop:limiting:ie:03}
	\end{align}
	Since the inequality in \eqref{for:prop:limiting:ie:01} holds for any $\vep > 0$, we take the limit $\vep \to 0$ to obtain $\frac{\partial}{\partial \beta} \varphi(\beta, B)$ as the limit of the right-hand sides of both \eqref{for:prop:limiting:ie:02} and \eqref{for:prop:limiting:ie:03}. Therefore, from \eqref{for:prop:limiting:ie:01}, \eqref{for:prop:limiting:ie:02}, and \eqref{for:prop:limiting:ie:03}, 
	\begin{align}\label{for:prop:limiting:ie:04}
		\frac{\partial}{\partial \beta} \psi_n(\beta, B) \overset{\prob}{\to} \frac{\partial}{\partial \beta} \varphi(\beta, B)~.
	\end{align}
	We now need to prove the second part of \eqref{eq:def:internal:energy}. To do this, we follow the same strategy as in \cite[Lemma~5.2]{DGvdH10}. Let $E_n$ denote the edge set of the preferential attachment graph of size $n$. We simplify
	\begin{equation}\label{for:prop:limiting:ie:05}
		\frac{\partial}{\partial \beta} \psi_n(\beta, B) = \frac{1}{n} \sum_{(i, j) \in E_n} \langle \sigma_i \sigma_j \rangle_{\mu_n} = \frac{|E_n|}{n} \cdot \frac{1}{|E_n|} \sum_{(i, j) \in E_n} \langle \sigma_i \sigma_j \rangle_{\mu_n}~.
	\end{equation}
	Therefore, using \eqref{for:prop:limiting:ie:05} and Lemma~\ref{lem:limiting:edge-measure}, we obtain
	\begin{equation}\label{for:prop:limiting:ie:06}
		\frac{\partial}{\partial \beta} \psi_n(\beta, B) \overset{\prob}{\to} m \E \left[ \frac{\tanh(\beta) + \tanh(\h(\hat{\emp})) \tanh(\h(\hat{\emp} 1))}{1 + \tanh(\beta) \tanh(\h(\hat{\emp})) \tanh(\h(\hat{\emp} 1))} \right]~.
	\end{equation}
	Similarly, from Lemma~\ref{lem:convexity:pressure-limit}, we obtain that $B \mapsto \psi_n(\beta, B)$ is a convex function. Next, using a similar calculation as done in \eqref{for:prop:limiting:ie:01}-\eqref{for:prop:limiting:ie:03}, 
	\begin{equation}\label{for:prop:limiting:ie:07}
		\frac{\partial}{\partial B} \psi_n(\beta, B) \overset{\prob}{\to} \frac{\partial}{\partial B} \varphi(\beta, B)~.
	\end{equation}
	Therefore, by Lemma~\ref{lem:limiting:edge-measure}, we obtain the RHS of \eqref{eq:lem:limiting:edge-measure-2} as the thermodynamic limit of magnetisation per vertex.
\end{proof}
\begin{Remark}[Similarity with Configuration model result]
	\rm Note that, from the description of the P\'{o}lya point tree, the expected degree of the root in the P\'{o}lya point tree is $2m$. Therefore, substituting $m$ with $\E[D(\emp)]/2$ in \eqref{eq:def:internal:energy}, we obtain a result similar in flavour to that obtained in \cite[Corollary~1.6(b)]{DGvdH10}. Similarly, using the distributional recursion property of the $\h$ random variables in \eqref{eq:distributional-recursion:1}, it can be shown that the explicit expression in \eqref{eq:def:magnetization} is the same as the one obtained in \cite[Corollary~1.6(a)]{DGvdH10}.\hfill$\blacksquare$
\end{Remark}
\medskip
Lastly, we turn to the proof of Lemma~\ref{lem:convexity:pressure-limit} and \ref{lem:limiting:edge-measure}. This is a standard result in thermodynamics.
The proof of Lemma~\ref{lem:limiting:edge-measure} follows similarly to the proof provided in \cite[Proof of Lemma~5.2]{DGvdH10}.
\begin{proof}[Proof of Lemma~\ref{lem:limiting:edge-measure}]
	The left-hand side (LHS) of \eqref{eq:lem:limiting:edge-measure-1} can be viewed as the expectation of the correlation $\langle \sigma_u \sigma_v \rangle_{\mu_n}$ with respect to a uniformly chosen edge $\{u,v\}$. For a uniformly chosen edge $\{u,v\}$, denote by $B_{\{u,v\}}(\ell)$ the subgraph of $G_n$ consisting of all vertices at a distance of at most $\ell$ from either $u$ or $v$. Therefore, by the GKS inequality in Lemma~\ref{lem:GKS},
	\begin{equation}
		\label{for:lem:limiting:edge-measure:1}
		\langle \sigma_u \sigma_v \rangle_{B_{\{u,v\}}(\ell)}^f \leq \langle \sigma_u \sigma_v \rangle_{\mu_n} \leq \langle \sigma_u \sigma_v \rangle_{B_{\{u,v\}}(\ell)}^+,
	\end{equation}
	where $\langle \sigma_u \sigma_v \rangle_{B_{\{u,v\}}(\ell)}^{f/+}$ represents the correlation in the Ising model on $B_{\{u,v\}}(\ell)$ with \emph{free} or $+$ boundary conditions.
	
	Note that every new vertex joins the graph with $m$ new edges. Therefore, $G_n$ has $n(m+o(1))$ edges. Hence, a uniformly chosen edge from $G_n$ can be viewed as a uniformly chosen out-edge from a uniformly random vertex. A finite neighbourhood of this uniformly chosen out-edge from a uniformly random vertex converges locally to the neighbourhood of the uniformly chosen out-edge of the root $\emp$ in the P\'{o}lya point tree. For any $\ell \geq 1$, the $\ell$-neighbourhood of a uniformly chosen edge in $G_n$ converges locally to the $\ell$-neighbourhood of $\hat{\emp}$ and $\emp1$, connected by an edge. Let us denote this tree as $\T(\ell)$. Consequently, as a result of local convergence and \cite[Lemma~6.4]{DM10}, for all $\ell \geq 1$, almost surely,
	\begin{equation}
		\label{for:lem:limiting:edge-measure:2}
		\lim_{n \to \infty} \E_n\Big[ \langle \sigma_u \sigma_v \rangle_{B_{(u,v)}(\ell)}^{f/+} \Big] = \E\Big[ \langle \sigma_u \sigma_v \rangle_{\T(\ell)}^{f/+} \Big].
	\end{equation}
	Now, using Lemma~\ref{lem:tree-pruning} and Lemma~\ref{lem:h:magnetization}, as $\ell \to \infty$,
	\begin{equation}
		\label{for:lem:limiting:edge-measure:3}
		\lim_{\ell \to \infty} \E\Big[ \langle \sigma_u \sigma_v \rangle_{\T(\ell)}^{f/+} \Big] = \E\Big[ \langle \sigma_u \sigma_v \rangle_{\nu_2^\prime} \Big],
	\end{equation}
	where $\nu_2^\prime(\sigma)$ is defined as
	\begin{equation}
		\label{for:lem:limiting:edge-measure:4}
		\nu_2^\prime(\sigma_1, \sigma_2) = \frac{1}{Z_2(\beta, \h(\hat{\emp}), \h(\hat{\emp}1))} \exp\Big[ \beta \sigma_1 \sigma_2 + \h(\hat{\emp}) \sigma_1 + \h(\hat{\emp}1) \Big].
	\end{equation}
	Simplifying the right-hand side (RHS) of \eqref{for:lem:limiting:edge-measure:4}, we obtain
	\eqan{\label{for:lem:limiting:edge-measure:5}
		&\E\big[\langle \sigma_1\sigma_2 \rangle_{\nu_2^\prime}\big]\nn\\
		&\hspace{1cm}=\E\left[ \frac{\e^{\beta+\h(\hat{\emp})+\h(\hat{\emp }1)}-\e^{-\beta-\h(\hat{\emp})+\h(\hat{\emp }1)}-\e^{-\beta+\h(\hat{\emp})-\h(\hat{\emp }1)}+\e^{\beta-\h(\hat{\emp})-\h(\hat{\emp }1)}}{\e^{\beta+\h(\hat{\emp})+\h(\hat{\emp }1)}+\e^{-\beta-\h(\hat{\emp})+\h(\hat{\emp }1)}+\e^{-\beta+\h(\hat{\emp})-\h(\hat{\emp }1)}+\e^{\beta-\h(\hat{\emp})-\h(\hat{\emp }1)}} \right]\nn\\
		&\hspace{1cm}=\E\left[ \frac{\tanh(\beta)+\tanh(\h(\hat{\emp}))\tanh(\h(\hat{\emp }1))}{1+\tanh(\beta)\tanh(\h(\hat{\emp}))\tanh(\h(\hat{\emp }1))} \right]~,}
	thus proving \eqref{eq:lem:limiting:edge-measure-1}. Equation~\eqref{eq:lem:limiting:edge-measure-2} can be proved in a similar manner. Following the same steps as in \eqref{for:lem:limiting:edge-measure:1}--\eqref{for:lem:limiting:edge-measure:3} and using Lemma~\ref{lem:h:magnetization}, we can show that
	\begin{equation}
		\label{for:lem:limiting:edge-measure:6}
		\E_n[\langle \sigma_v \rangle_{\mu_n}] \overset{\prob}{\longrightarrow} \E\big[ \sigma_{\emp} \big] = \E[\magn{\emp}] = \E[\tanh(\h(\emp))],
	\end{equation}
	thus completing the proof of the lemma.
\end{proof}

To prove Lemma~\ref{lem:convexity:pressure-limit}, we show that the partial second derivatives of $\psi_n$ with respect to $\beta$ and $B$ are variances of some random variables, which essentially proves their non-negativity. 
\begin{proof}[Proof of Lemma~\ref{lem:convexity:pressure-limit}]
	We first perform the computation with respect to $\beta$; the computation for $B$ follows identically. For any $n \in \N$, differentiating $\psi_n$ with respect to $\beta$ gives
	\eqan{
		\frac{\partial}{\partial\beta}\psi_n(\beta,B)~=&~\frac{1}{n} \E_{\mu_n}\Big[ \sum\limits_{\{i,j\}\in E_n} \sigma_i\sigma_j \Big],\label{for:lem:convexity:1}\\
		\mbox{and}\quad \frac{\partial^2}{\partial\beta^2}\psi_n(\beta,B)~=&~\frac{1}{n}\Big[ \E_{\mu_n}\Big[ \Big( \sum\limits_{(i,j)\in E_n} \sigma_i\sigma_j \Big)^2 \Big]\label{for:lem:convexity:2}\\
		&\hspace{1cm}-\Big( \E_{\mu_n}\Big[ \sum\limits_{(i,j)\in E_n} \sigma_i\sigma_j \Big] \Big)^2 \Big]~,\nn
	}
	where $\E_{\mu_n}$ denotes the expectation with respect to the $\mu_n$ measure. Note that the right-hand side of \eqref{for:lem:convexity:2} simplifies to ${\var}_{\mu_n}\left( \sum_{(i,j) \in E_n} \sigma_i \sigma_j \right)/n$. Therefore, $\frac{\partial^2}{\partial \beta^2} \psi_n(\beta, B)$ is non-negative almost surely, proving that $\beta \mapsto \psi_n(\beta, B)$ is a convex function. Similarly, it can be shown that
	\begin{equation}
		\frac{\partial^2}{\partial B^2} \psi_n(\beta, B) = \frac{1}{n} {\var}_{\mu_n} \left( \sum_{i \in [n]} \sigma_i \right). \label{for:lem:convexity:3}
	\end{equation}
	Hence, using a similar argument, we obtain that $B \mapsto \psi_n(\beta, B)$ is also a convex function.
\end{proof}

\section{Inverse critical temperature}\label{sec:inv:critical-temperature}

In this section, we identify the inverse critical temperature for preferential attachment models with parameters $m$ and $\delta$.
We define $\beta_c(m,\delta)$ as the Ising inverse critical temperature for a sequence of preferential attachment models with parameters $m$ and $\delta$, defined as
\eqn{\label{eq:def:inverse-critical-temperature:PA}
	\beta_c=\inf\limits_{\beta}\Big\{\lim\limits_{B\searrow 0}M(\beta,B)>0\Big\}~.}
The final theorem of this chapter addresses the inverse critical temperature for preferential attachment models. In the following theorem, we explicitly identify the almost sure inverse critical temperature for the preferential attachment model.

\medskip 

\begin{Theorem}[Ising inverse critical temperature]\label{thm:inv:critical-temp:PA}
	Fix $m\geq 2$ and $\delta>0$, and let $\beta_c(m,\delta)$ be the Ising inverse critical temperature for the P\'{o}lya point tree with parameters $m$ and $\delta$. Then, for $\delta>0$
	\eqn{\label{eq:thm:inv:critical-temp:PA}
		\hspace{-0.25cm}\beta_c(m,\delta)= \arctanh\left\{ \frac{\delta}{2\big( m(m+\delta)+\sqrt{m(m-1)(m+\delta)(m+\delta+1)} \big)} \right\},}
	whereas for $\delta\in(-m,0],~\beta_c(m,\delta)=0$ almost surely.
\end{Theorem}

\medskip

The proof of Theorem~\ref{thm:inv:critical-temp:PA} is based on the inverse critical temperature for the P\'{o}lya point tree.
The inverse critical temperature for a rooted tree is defined as follows:

Let $\mu^{\sss (\beta,B)}$ denote the Boltzmann distribution on a tree $(\Tbold, o)$ with inverse temperature parameter $\beta$ and external magnetic field $B$. Define
\begin{equation}
	\label{eq:def:Boltzmann:zero-mag-field}
	\mu^{\sss (\beta,0)}(\cdot) = \lim_{B \searrow 0} \mu^{\sss (\beta,B)}(\cdot)~.
\end{equation}
The \emph{inverse critical temperature} for a rooted tree $(\Tbold, o)$, denoted $\beta_c(\Tbold, o)$, is defined as
\begin{equation}
	\label{eq:def:inverse-critical-temperature}
	\beta_c(\Tbold, o) = \inf\left\{ \beta \mid \mu^{(\beta,0)}(\sigma_o = +1) - \mu^{(\beta,0)}(\sigma_o = -1) > 0 \right\}~,
\end{equation}
where $\mu^{(\beta,0)}(\cdot)$ is the Boltzmann distribution on $(\Tbold, o)$ under zero external magnetic field, as defined in \eqref{eq:def:Boltzmann:zero-mag-field}.

The following proposition identifies the inverse critical temperature for the P\'{o}lya point tree with parameters $m \geq 2$ and $\delta > 0$:

\medskip
\begin{Proposition}[Inverse critical temperature for the PPT]
	\label{prop:inv:critical-temp:PPT}
	Fix $m \in \mathbb{N} \setminus \{1\}$ and $\delta > 0$. Then, the inverse critical temperature for the P\'{o}lya point tree with parameters $m$ and $\delta$ is given by $\beta_c(m, \delta)$ as defined in \eqref{eq:thm:inv:critical-temp:PA}.
\end{Proposition}

First, we prove Theorem~\ref{thm:inv:critical-temp:PA} using Proposition~\ref{prop:inv:critical-temp:PPT}, and then we proceed to prove the proposition in detail.

\begin{proof}[Proof of Theorem~\ref{thm:inv:critical-temp:PA}]
	Proposition~\ref{prop:inv:critical-temp:PPT} identifies $\beta_c(m,\delta)$ as the inverse critical temperature for the P\'{o}lya point tree with parameters $m$ and $\delta$. Therefore,
	\begin{align}
		\mu^{(\beta,0)}(\sigma_o=+1) - \mu^{(\beta,0)}(\sigma_o=-1) &> 0, \quad \text{for all } \beta > \beta_c(m,\delta), \label{for:thm:PA:inv-temp:1-1} \\
		\text{and} \quad \mu^{(\beta,0)}(\sigma_o=+1) - \mu^{(\beta,0)}(\sigma_o=-1) &= 0, \quad \text{for all } \beta < \beta_c(m,\delta). \label{for:thm:PA:inv-temp:1-2}
	\end{align}
	Note that, by Theorem~\ref{thm:magnetization} and Lemma~\ref{lem:h:magnetization},
	\begin{align}\label{for:thm:PA:inv-temp:2}
		&\mu^{(\beta,0)}(\sigma_o=+1) - \mu^{(\beta,0)}(\sigma_o=-1) \nn\\
		=& \lim_{B \searrow 0} \mathbb{E}[\magn{\emp \to}] = \lim_{B \searrow 0} M(\beta,B).
	\end{align}
	Therefore, from \eqref{for:thm:PA:inv-temp:1-1}, \eqref{for:thm:PA:inv-temp:1-2}, and \eqref{for:thm:PA:inv-temp:2}, we obtain
	\begin{equation}\label{for:thm:PA:inv-temp:3}
		\beta_c(m,\delta) = \inf \Big\{ {\beta}: \lim_{B \searrow 0} M(\beta,B) > 0 \Big\},
	\end{equation}
	proving that $\beta_c(m,\delta)$ is the inverse critical temperature for preferential attachment models with parameters $m$ and $\delta$.
\end{proof}

The proof of Proposition~\ref{prop:inv:critical-temp:PPT} follows from \cite[Theorem~2.1]{Lyo89}, with adaptations to our specific setting. The proof employs concepts such as the branching number and the growth of trees. We recall the definitions of the branching number and the growth of trees as established in \cite{Lyo89,Lyo90}:

\begin{definition}[Branching number and growth of tree]\label{def:branching-number:growth}
	The \emph{branching number} of a tree $\Gamma$, denoted by $\rm{br}(\Gamma)$, is defined as
	\eqan{\label{eq:def:branching-number}
		\rm{br}(\Gamma)=&~\inf\Big\{ \lambda>0:\liminf\limits_{\Pi\to\infty}\sum\limits_{u\in\Pi} \lambda^{-|u|}=0 \Big\}\\
		=&~\sup\Big\{ \lambda>0:\liminf\limits_{\Pi\to\infty}\sum\limits_{u\in\Pi} \lambda^{-|u|}=\infty \Big\}\nn\\
		=&~\inf\Big\{ \lambda>0:\inf\limits_{\Pi}\sum\limits_{u\in\Pi} \lambda^{-|u|} \Big\}
		~,\nn}
	whereas the \emph{growth} of $\Gamma$, denoted by $\rm{gr}(\Gamma)$, is defined as
	\eqn{\label{eq:def:growth}
		{\rm{gr}}(\Gamma)=\inf\Big\{ \lambda>0:\liminf\limits_{n\to \infty} {M_n}\lambda^{-n}=0\Big\}=\liminf\limits_{n\to\infty}M_n^{1/n}
		~,}
	where $M_n$ is the size of the $n$-th generation in the tree $\Gamma$.
\end{definition}
Notation used in \cite{Lyo89} differs slightly from those employed here. Adapting to our notation, we can deduce from \cite[Theorem~2.1]{Lyo89} that
\begin{equation}\label{for:thm:inv:critical-temp:1}
	\tanh\big(\beta_c(m,\delta)\big) = \frac{1}{{\rm{br}}\big(\PPT(m,\delta)\big)},
\end{equation}
where ${\rm{br}}\big(\PPT(m,\delta)\big)$ is the branching number of the P\'{o}lya point tree with parameters $m$ and $\delta$. We now identify the explicit expression for ${\rm{br}}\big(\PPT(m,\delta)\big)$. We prove and use the following proposition, which is similar in spirit to \cite[Proposition~6.4]{Lyo89}.
\medskip
\begin{Proposition}[Branching number of the P\'{o}lya point tree]\label{prop:branching-number:PPT}
	Let $\bfT_{\kappa}$ denote the mean offspring operator of the P\'{o}lya point tree with parameters $m$ and $\delta$. Then, for $m \geq 2$ and $\delta > 0$,
	\begin{equation}\label{eq:prop:branching-number:PPT}
		{\rm{br}}\big(\PPT(m,\delta)\big) \geq r(\bfT_{\kappa})~\mbox{a.s.},
	\end{equation}
	where $r(\bfT_{\kappa})$ is the spectral radius of $\bfT_{\kappa}$.
\end{Proposition}
\medskip
To prove the proposition, we follow a strategy similar to that in \cite[Theorem~6.2]{Lyo89}. Theorem~\ref{thm:critical-percolation:PPT} shows that for $m \geq 2$ and $\delta > 0$, the critical percolation threshold for $\PPT(m,\delta)$ is ${1}/{r(\bfT_{\kappa})}$. We now prove that $\PPT(m,\delta)$ almost surely dies out when percolated with $\pi < {1}/{{\rm{br}}\big(\PPT(m,\delta)\big)}$. Therefore, the critical percolation threshold of $\PPT(m,\delta)$ is at least ${1}/{{\rm{br}}\big(\PPT(m,\delta)\big)}$, completing the proof of Proposition~\ref{prop:branching-number:PPT}.

\begin{proof}[Proof of Proposition~\ref{prop:branching-number:PPT}]
	To reduce notational complexity, we shall refer the $\PPT(m,\delta)$ by $\Gamma$ in this proof, and for any $\pi\in[0,1]$, we use $\Gamma(\pi)$ to denote the $\pi$ percolated $\PPT(m,\delta)$.
	For any cutset $\Pi$ and $\lambda>0$, let
	\eqn{\label{for:prop:branching-number:1}
		Z_{\Pi}(\lambda)=\sum\limits_{u\in\Pi}\lambda^{-|u|}~,}
	and for any $\pi\in[0,1],$ define
	\eqn{\label{for:prop:branching-number:2}
		Z_{\Pi(\pi)}(\lambda)=\sum\limits_{u\in\Pi(p)}\lambda^{-|u|}~,}
	where $\Pi(\pi)=\Pi\cap\Gamma(\pi).$ Then,
	\eqn{\label{for:prop:branching-number:3}
		\E_{\pi}\Big[ Z_{\Pi(\pi)}(1) \Big]=\sum\limits_{u\in\Pi} \prob_{\Gamma}\big( u\in\Gamma(\pi) \big)=\sum\limits_{u\in \Pi}\pi^{|u|}=Z_{\Pi}(1/\pi)~,}
	where $\E_{\pi}$ is the expectation with respect to the percolation.
	If $1/\pi>\br{\Gamma}$, then there exists a sequence $\Pi_n\to\infty$ such that $Z_{\Pi_n}(1/\pi)\to0.$ Hence, from Fatou's lemma applied to \eqref{for:prop:branching-number:3}, we obtain conditionally on $\Gamma$,
	\eqn{\label{for:prop:branching-number:4}
		\liminf\limits_{n\to\infty} Z_{\Pi_n(\pi)}(1)=0\quad \text{a.s.}}
	Therefore, conditionally on $\Gamma,~\Gamma(\pi)$ is finite almost surely. From definition of the critical percolation threshold, $\pi_c$, we can upper bound the random variable $1/\br{\Gamma}$ by $\pi_c$. Further Theorem~\ref{thm:critical-percolation:PPT} identifies $\pi_c$ of P\'{o}lya point tree as $1/r(\bfT_{\kappa})$. Hence,
	\eqn{\label{for:prop:branching-number:5}
		\frac{1}{\br{\Gamma}}\leq \pi_c=\frac{1}{r(\bfT_{\kappa})}\quad\text{a.s.}}
	Therefore we are done with proving the almost sure lower bound of $\br{\Gamma}$. From Lemma~\ref{lem:growth:PPT}, we have proved $\gr{\Gamma}$ is at most $r(\bfT_{\kappa})$ almost surely, which serves as an almost sure upper bound for $\br{\Gamma}$ as well. Hence, $\br{\PPT(m,\delta)}$ is $r(\bfT_{\kappa})$ almost surely.
\end{proof}

On the other hand, based on the definitions of growth and branching numbers for a tree, it follows that ${\rm{br}}\big(\PPT(m,\delta)\big)$ is at most ${\rm{gr}}\big(\PPT(m,\delta)\big)$. The next lemma computes ${\rm{gr}}\big(\PPT(m,\delta)\big)$ explicitly.
\medskip
\begin{Lemma}[Growth of the P\'{o}lya point tree]\label{lem:growth:PPT}
	For $m \geq 2$ and $\delta > 0$,
	\begin{equation}\label{eq:lem:growth:PPT}
		{\rm{gr}}\big(\PPT(m,\delta)\big) = r(\bfT_{\kappa}) \text{ a.s.}
	\end{equation}
\end{Lemma}


\medskip
From Proposition~\ref{prop:branching-number:PPT}, we obtain that $\gr{\PPT(m,\delta)} \geq r(\bfT_{\kappa})$ almost surely, and now we prove $\E\big[\gr{\PPT(m,\delta)}\big]$ is at most $r(\bfT_{\kappa})$, proving that $\gr{\PPT(m,\delta)}$ is $r(\bfT_{\kappa})$ almost surely. 
\begin{proof}[Proof of Lemma~\ref{lem:growth:PPT}]
	Let $M_n(x,s)$ denote the number of nodes in $\PPT(m,\delta)$ rooted at $(x,s)\in\Scal$. From the definition in \eqref{eq:def:growth},
	\begin{equation}\label{for:lem:growth:PPT:0}
		\E\big[ \gr{\PPT(m,\delta)} \big] = \E\big[\liminf\limits_{n\to \infty} M_n(\emp)^{1/n}\big] \leq \liminf\limits_{n\to \infty} \E\Big[ M_n(\emp)^{1/n} \Big]~,
	\end{equation}
	where the inequality follows from Fatou's lemma. Using the fact that $x\mapsto x^{1/n}$ is a concave function, we apply Jensen's inequality to obtain
	\begin{equation}\label{for:lem:growth:PPT:1}
		\E\big[ \gr{\PPT(m,\delta)} \big] \leq \liminf\limits_{n\to \infty} \E\Big[ M_n(\emp) \Big]^{1/n}~.
	\end{equation}
	
	We first aim to show that the LHS of \eqref{for:lem:growth:PPT:1} is upper bounded by $r(\bfT_{\kappa})$. To prove this, we demonstrate
	\begin{equation}\label{for:lem:growth:PPT:2}
		M_n(\emp) \overset{d}{=} M_n(U_1, \Young) + M_{n-1}(U_1 U_2^{1/\chi}, \Old)~,
	\end{equation}
	where $U_1, U_2$ are i.i.d.\ $\Unif[0,1]$ random variables. To establish this result, we view the P\'{o}lya point tree rooted at $\emp$ from the perspective of a uniformly chosen out-edge of the root $\emp$. 
	
	Let $\emp u$ be the uniformly chosen $\Old$ neighbour of $\emp$. Let $M^{(-u)}_n(\emp)$ denote the number of nodes in the $n$-th generation of $\PPT$ rooted at $\emp$, ignoring the size of the sub-tree rooted at $\emp u$. Note that, by the construction of the P\'{o}lya point tree, conditionally on the age of $\emp$, $M_{n-1}(\emp u)$ and $M^{(-u)}_n(\emp)$ are independent. Conditionally on the age of $\emp$, $A_\emp=a$, the sub-tree rooted at $\emp$, excluding the sub-tree rooted at $\emp u$, has $m-1$ many $\Old$-labeled children in the first generation. Furthermore, since the strength of any $\Young$-labeled node is identically distributed to the root, it follows that conditionally on $A_\emp=a$, $\emp$ and $(a, \Young)$ have identical offspring distributions. Hence, conditionally on $A_\emp=a$,
	\begin{equation}\label{for:lem:growth:PPT:3}
		M_n(a, \Young) \overset{d}{=} M^{(-u)}_n(\emp)~.
	\end{equation}
	
	On the other hand, conditionally on $A_\emp=a$, the age of $\emp U$ is distributed as $a U_2^{1/\chi}$, where $U_2 \sim \Unif[0,1]$ and labeled $\Old$, and is independent of $A_\emp$. Therefore, conditionally on $A_\emp=a$,
	\begin{equation}\label{for:lem:growth:PPT:4}
		M_{n-1}(\emp u) \overset{d}{=} M_{n-1}(a U_2^{1/\chi}, \Old)~.
	\end{equation}
	
	Since $A_\emp \sim \Unif[0,1]$ and $M^{(-u)}_n(\emp)$ and $M_{n-1}(\emp u)$ are mutually independent and independent of $A_\emp$, we obtain
	\begin{equation}\label{for:lem:growth:PPT:5}
		\left( M_{n-1}(\emp u), M^{(-u)}_n(\emp) \right) \overset{d}{=} \left( M_{n-1}\big(U_1 U_2^{1/\chi}, \Old\big), M_{n}\big(U_1, \Young\big) \right)~,
	\end{equation}
	where $U_1 \sim \Unif[0,1]$ and is independent of $U_2$. Thus, \eqref{for:lem:growth:PPT:1} follows immediately.
	
	Next, for any $(x,s) \in \Scal$ and $n \in \N$,
	\begin{equation}\label{for:lem:growth:PPT:6}
		\E[M_n(x,s)] = \langle \one_{(x,s)}, \bfT_{\kappa}^n \ubar{\one} \rangle_{\Scal},
	\end{equation}
	where $\ubar{\one}(y,t) \equiv 1$ for all $(y,t) \in \Scal$, and $\langle f, g \rangle_\Scal = \int_{\Scal} f(z) g(z) \, dz$ for any functions $f$ and $g$ from $\Scal$ to $\R$. In Section~\ref{chap:percolation_threshold_PPT:sec:main-theorem}, we defined $\kappa$ to be the mean offspring generator for the P\'{o}lya point tree with parameters $m$ and $\delta$. Recall that $\Scal=[0,1]\times\{\Old,\Young\}$. In Theorem~\ref{thm:operator-norm:PPT}, we proved that the integral operator $\bar{\bfT}_{\kappa}$, with kernel $\kappa$, admits an eigenfunction $h$ corresponding to its spectral norm $r(\bfT_{\kappa})$ in the extended type-space $\Scal_e=[0,\infty)\times\{\Old,\Young\}$. The eigenfunction $h$ is given by
	\[
	h(x,s)=\frac{\bfp_s}{\sqrt{x}},~\mbox{for all }(x,s)\in\Scal_e~, 
	\]
	where $\bfp=(\bfp_{\old},\bfp_{\young})$ is as defined in Theorem~\ref{thm:operator-norm:PPT}.
	
	Since $\kappa(u,v) \geq 0$ for all $u, v \in \Scal$, it can further be shown that
	$\bfT_{\kappa} f(u) \geq \bfT_{\kappa} g(u)$ for all $u \in \Scal$, if $f(v) \geq g(v)$ for all $v \in \Scal$. Note that for all $(x,s) \in \Scal$, $\bfc h(x,s) \geq 1$, where $\bfc = \max\{1/\bfp_\old, 1/\bfp_\young\}$. Therefore,
	\begin{align}\label{for:lem:growth:PPT:7}
		\langle \one_{(x,s)}, \bfT_{\kappa}^n \ubar{\one} \rangle_{\Scal} \leq \bfc \langle \one_{(x,s)}, \bfT_{\kappa}^n h \rangle_{\Scal} \leq& \bfc \langle \one_{(x,s)}, \bar{\bfT}_{\kappa}^n h \rangle_{\Scal_e}\nn\\
		=& \bfc r(\bfT_{\kappa})^n h(x,s)~.
	\end{align}
	
	Since $r(\bfT_{\kappa}) > 1$, from \eqref{for:lem:growth:PPT:2}, \eqref{for:lem:growth:PPT:6}, and \eqref{for:lem:growth:PPT:7}, we obtain
	\begin{equation}\label{for:lem:growth:PPT:8}
		\E[M_n(\emp)] \leq \bfc r(\bfT_{\kappa})^n \left[ \E[h(U_1, \Young)] + \E\left[h\big(U_1 U_2^{1/\chi}, \Young\big)\right] \right] \leq \bfc_1 r(\bfT_{\kappa})^n~,
	\end{equation}
	where $\bfc_1 = \bfc/2$. Plugging this upper bound for $\E[M_n(\emp)]$ into \eqref{for:lem:growth:PPT:1}, we obtain $r(\bfT_{\kappa})$ as an almost sure upper bound for $\gr{\PPT(m,\delta)}$. 
\end{proof}

Now that we have all the necessary results, we can prove Proposition~\ref{prop:inv:critical-temp:PPT}. We use \cite[{Theorem~2.1}]{Lyo89} to establish the result. Note that the notations used in \cite{Lyo89} differ from those in this chapter. Therefore, to apply the conclusions from \cite{Lyo89}, we equate \( \frac{J}{kT} \) in \cite{Lyo89} with \( \beta \) in this chapter.

\begin{proof}[Proof of Proposition~\ref{prop:inv:critical-temp:PPT}]
	From definition of branching number and growth of a tree, the following relation holds true almost surely:
	\begin{equation}\label{for:thm:inv:temp:00}
		\br{\Gamma}\leq\gr{\Gamma}~ \mbox{for all infinite tree }\Gamma~.
	\end{equation}
	By Lemma~\ref{lem:growth:PPT}, Proposition~\ref{prop:branching-number:PPT} and \eqref{for:thm:inv:temp:00},
	\eqn{\label{for:thm:inv:temp:01}
		\br{\PPT(m,\delta)}=r(\bfT_{\kappa})
		~.}
	By \cite[Theorem~2.1]{Lyo89},
	\begin{equation}\label{for:thm:inv:temp:1}
		\tanh\left( \frac{J}{kT_c} \right)\br{\PPT(m,\delta)} = 1~,
	\end{equation}
	where \( T_c \) is the critical temperature described in \cite{Lyo89}. Equating these notations to our settings, we obtain
	\begin{equation}\label{for:thm:inv:temp:2}
		\tanh(\beta_c)\br{\PPT(m,\delta)} = 1~.
	\end{equation}
	By \eqref{for:thm:inv:temp:01},
	\begin{equation}\label{for:thm:inv:temp:3}
		\beta_c = \arctanh\left( \frac{1}{r(\bfT_{\kappa})} \right)~.
	\end{equation}
	Substituting the explicit value of \( r(\bfT_{\kappa}) \) from Theorem~\ref{thm:operator-norm:PPT} into \eqref{for:thm:inv:temp:3}, we obtain Proposition~\ref{prop:inv:critical-temp:PPT} immediately.
\end{proof}

\EndThumbs

\StartThumbs

\addtocounter{chapter}{+1}

\cleardoublepage\phantomsection
\addcontentsline{toc}{chapter}{Discussions}

\chapter*{Discussion}

Research in mathematics, rather than narrowing the scope of open questions, tends to exponentially increase them. Our work on preferential attachment models follows this trend, revealing a range of new research directions in the field. Throughout this thesis, we have explored the dynamics of preferential attachment models, focusing on both graph dynamics within these models and the effects of different settings. 

In Part I, we examined how i.i.d.\ out-edge settings influence the dynamics of these models and their impact on the local limit. Although we successfully established the local limit for the generalised models, our study of stochastic processes in Part II is limited to deterministic out-edge settings. Consequently, we believe that a vast array of open questions remains to be explored, presenting a rich landscape for future research.
In this section, we try to put up some of these open directions. We divide these in two directions: local limit direction and stochastic processes direction.

\paragraph{Local limit direction}
Lo studied the case of the preferential attachment tree, where each new vertex joins the graph with an i.i.d.\ random fitness parameter in \cite{tiffany2021}. The author suggested that this result could be generalised to the fixed out-edge preferential attachment graph case as well. This finding is particularly interesting as it provides the rate of convergence in the total variation metric. It would therefore be intriguing to extend this work to random fitness distributions, along with random out-edge distributions and to study the rate of convergence of these generalised preferential attachment models to the local limit, some further generalised version of the random P\'{o}lya point tree.

\paragraph{Percolation direction}

In Chapter~\ref{chap:percolation_threshold_PPT}, we explicitly computed the critical percolation threshold for the P\'{o}lya point tree with parameters \( m\geq 2\) and \(\delta>0\). The computation of the spectral radius of the mean offspring operator for the P\'{o}lya point tree can also be extended to the random P\'{o}lya point tree. Investigating the critical percolation threshold for the random P\'{o}lya point tree using the spectral radius of its mean offspring operator would be very interesting. The main challenge in this computation lies in proving the \emph{Kesten-Stigum condition} for super-criticality. Furthermore, we hypothesise that the critical percolation threshold for the random P\'{o}lya point tree is \(0\) when $\delta$ is non-positive.

We also believe that our proof in Chapter~\ref{chap:local-giant} can be extended to other preferential attachment models, such as the independent models. For instance, a similar technique may prove useful in identifying critical percolation thresholds for preferential attachment models with i.i.d.\ random out-edges, as studied in \cite{DEH09,GV17,RRR22}. Our proof of the large-set expander property for the preferential attachment model extensively relies on the fact that \(m\) is fixed and at least 2. Therefore, subject to the critical percolation threshold of the random P\'{o}lya point tree, the main challenge in determining the critical percolation threshold for generalised preferential attachment models lies in proving the large-set expander property of these models.

Although we have identified the exact critical percolation thresholds for several preferential attachment models, our techniques do not reveal the order of the phase transition. Alimohammadi, Borgs, and Saberi have proved that this phase transition is of infinite order for \(\delta=0\), i.e.\ \(\zeta(\pi)\) is infinitely differentiable at \(\pi_c\). Eckhoff, M\"orters, and Ortgiese proved in \cite{EMO21} that the phase transition in Bernoulli preferential attachment models defined in \cite{DM13} is also of infinite order. Thus, we conjecture that the phase transition is of infinite order for other preferential attachment models and any admissible \(\delta\).

In Chapter~\ref{chap:subcritical}, we discussed our recent ongoing work on the sub-critical percolation regime of preferential attachment models, where we observe a very interesting behaviour in the sub-critical largest connected component. Unlike the typical behaviour in other random graphs with power-law degree distributions, such as configuration models, the sub-critical largest connected component in preferential attachment models is significantly larger than its maximum degree. It would be interesting to study the \emph{critical regime} of preferential attachment models, as we believe that some intriguing phenomena may also emerge within the \emph{critical window} of these models.

\paragraph{Ising model direction}

We studied the Ising model in the quenched setting of preferential attachment models and identified the inverse critical temperature for these models. However, we could not determine the order of the phase transition in this setting. Moreover, the annealed setting on preferential attachment models would also be of significant interest. From the work of Dembo, Montanari, and Sun \cite{DMS13}, we gain insights into the thermodynamic limits of the pressure per particle in various factor models on locally tree-like graphs. Consequently, the study of Potts models on preferential attachment models, in both quenched and annealed settings, would be highly intriguing for the statistical physics community.

\pagestyle{plain}


\cleardoublepage\phantomsection
\addcontentsline{toc}{chapter}{Bibliography}

\bibliographystyle{acm}
\bibliography{bibliofile}


\EndThumbs
\cleardoublepage\phantomsection
\addcontentsline{toc}{chapter}{Summary}
\markboth{Summary}{Summary}
\chapter*{ 
\vspace*{\emptychapterdist}Summary}
\vspace{10pt}

In real life, networks are dynamic in nature; they grow over time and often exhibit power-law degree sequences. To model the evolving structure of the internet, Barab\'{a}si and Albert introduced a simple dynamic model with a power-law degree distribution. This model has since been generalised, giving rise to a large class of affine preferential attachment models, where each new vertex connects to existing vertices with a probability proportional to the current degree of the vertex. While several studies have analysed global and local properties of these random graphs, their dynamic nature and the dependencies in edge-connection probabilities have posed significant analytical challenges.

In this thesis, we explore various properties of these dynamic random graphs in a broad context. First, we extend the preferential attachment models to allow each new vertex to connect with a random number of initial neighbours. To understand the local structure of these generalised preferential attachment models, we aim to identify their local limits as described by Benjamini and Schramm in \(2001\). In \(2014\), Berger et al.\ proved that the standard preferential attachment models are locally tree-like, meaning that the finite neighbourhood of a uniformly chosen vertex resembles a branching process, termed the P\'{o}lya point tree. We demonstrate that, even after generalisation, preferential attachment models retain this tree-like property. The finite neighbourhood of a uniformly chosen vertex in these graphs resembles another branching process, which we name the generalised P\'{o}lya point tree. Additionally, we prove that, despite the strong dependencies in edge connections, certain generalised preferential attachment models exhibit conditional independence. This result facilitates a more detailed analysis of these models.

The second part of this thesis focuses on stochastic processes on preferential attachment models. These stochastic processes introduce an additional layer of randomness to the random graphs. Examples include bond and site percolation, random walks, the Ising and Potts models, and Gaussian processes on random graphs. While these processes are global properties of the graphs and often require more than just local behaviour known from the local limit, under certain conditions, they can be analysed using only local properties. For instance, in vertex-transitive graph sequences, bond percolation can be studied using local properties alone. In this thesis, we specifically examine bond percolation and the Ising model on preferential attachment models. We identify the critical probability \(\pi_c\) such that when the model is percolated with a probability higher than \(\pi_c\), a single largest connected component of the order of the graph's size emerges. Conversely, for percolation probabilities below \(\pi_c\), no such large connected component forms. We show that this property can indeed be studied using the local properties of the preferential attachment models. We further investigate the size of the largest connected component when the graph is percolated with a probability below \(\pi_c\), uncovering surprising results compared to similar static random graphs with power-law degree distributions. Finally, we study the Ising ferromagnetic model on preferential attachment models. Building on the work of Dembo, Montanari, and Sun (\(2014\)) and Lyons (\(1990\)), we demonstrate that various thermodynamic quantities of interest in the Ising model can also be analysed using the local limit of the graph sequence.

Our study indicates that while preferential attachment models satisfy most of the regularity conditions necessary for examining stochastic processes through local limits, their dynamic nature and the strong dependencies in edge-connection probabilities may lead to unexpected results for certain stochastic processes on these models.

\pagestyle{empty}
\cleardoublepage\null\cleardoublepage

\end{document}